\documentclass[reqno]{amsart}
\usepackage{amsmath,amssymb,mathrsfs,amsthm,amsfonts}
\usepackage[inline]{enumitem} 				% to change enum. items\inddown
\usepackage[usenames,dvipsnames]{xcolor}
\usepackage{graphicx} 
\usepackage{wrapfig}
\usepackage[margin=3pt, font=small]{caption}%figure captions
\usepackage{tikz-cd} 						%for commutative diagrams
\definecolor{hypercolor}{HTML}{003399}
\usepackage{hyperref}
 	
\hypersetup{colorlinks=true, linkcolor=hypercolor,citecolor=hypercolor}
\usepackage{xparse}							%richer commands
\usepackage[paper=letterpaper,margin=1in]{geometry}
\newtheorem*{thmfixedtime}{Fixed-time Theorem}
\newtheorem*{thmmain}{Main Theorem}
\newtheorem*{thmfixedtimeup}{Fixed-time Theorem, upper bound}
\newtheorem*{thmfixedtimelow}{Fixed-time Theorem, lower bound}
\newtheorem{lem}{Lemma}[section]
\newtheorem{prop}[lem]{Proposition}
\newtheorem{innercustomlem}{Lemma}
\newenvironment{customlem}[1]
  {\renewcommand\theinnercustomlem{#1}\innercustomlem}%
  {\endinnercustomlem}%
\newtheorem{innercustomprop}{Proposition}
\newenvironment{customprop}[1]
  {\renewcommand\theinnercustomprop{#1}\innercustomprop}%
  {\endinnercustomprop}%
\theoremstyle{definition}
\newtheorem{innercustomex}{Example}
  {\endinnercustomex}%
\newtheorem{ex}[lem]{Example}
\newtheorem{defn}[lem]{Definition}
\newtheorem{assu}[lem]{Assumption}
\theoremstyle{remark}
\newtheorem{con}[lem]{Convention}
\newtheorem{rmk}[lem]{Remark}
\numberwithin{equation}{section}
\allowdisplaybreaks

\newcommand{\tdef}[1]{\textbf{\boldmath{#1}}}

\usepackage{acronym}
\acrodef{TASEP}{Totally Asymmetric Simple Exclusion Process}
\acrodef{LDP}{Large Deviation Principle}
\acrodef{LD}{Large Deviation}
\acrodef{KPZ}{Kardar--Parisi--Zhang}
\acrodef{lsc}{lower-semicontinuous} 

\newcommand{\tL}{\mathrm{L}}				
\newcommand{\tR}{\mathrm{R}}				
\newcommand{\trw}{\mathrm{rw}}
			
\newcommand{\Geo}{\mathrm{Geo}}				

\newcommand{\vect}{\mathbf{t}}
\newcommand{\e}{\varepsilon}

\newcommand{\calC}{\mathcal{C}}
\newcommand{\calO}{\mathcal{O}}
\newcommand{\calK}{\mathcal{K}}
\newcommand{\calM}{\mathcal{M}}

\newcommand{\C}{\mathbb{C}}

\newcommand{\R}{\mathbb{R}}
\newcommand{\Z}{\mathbb{Z}}
\newcommand{\ldq}{\mathsf{LdQ}}				%linear and downward quadratic
\newcommand{\Lip}{\mathsf{1\hspace{-1.2pt}\text{-}\hspace{-.1pt}Lip}} %the 1-Lipschitz space
\newcommand{\HLsp}{\mathsf{HLsp}}			%the Hopf--Lax space
\newcommand{\HLspk}{\mathsf{HLsp}_\star}	%the Hopf--Lax space for the kth wedge
\newcommand{\HLspkk}{\mathscr{F}}			%a finite subset in HLspk
\newcommand{\Elem}{\mathsf{Elem}}			%the space of skeletons
\newcommand{\Csp}{\mathcal{C}}				%the space of conti fns
\newcommand{\Dsp}{\mathcal{D}}				%the space cadlg fns
\newcommand{\lsp}{\ell}						%the l^p space
\newcommand{\Mat}{\mathsf{Mat}} 			%set of matrices

\newcommand{\dist}{\mathrm{dist}}			%the metric
\newcommand{\ppin}{P_{\mathrm{pin}}}%pinning probability, in (s,d)
\newcommand{\punder}{P_\mathrm{und}}		%under probability

\newcommand{\D}{D}							%the D_n in fredholm determinant
\newcommand{\DD}{\Delta}					%D_n(our operator)
\newcommand{\DDprnd}{\DD^\mathrm{pr,nd}}	%DD with only preferred, nondgenerate terms

\newcommand{\RAte}{\mathcal{I}_\star}		%"The" rate function
\newcommand{\rate}{I}						%general notation for fixed-time rate function
\newcommand{\Rate}{\mathcal{I}}				%general notation for process rate function
\newcommand{\rateJV}{\mathcal{I}_{\mathrm{JV}}}%Jensen Varadhan rate
\newcommand{\rateber}{I_{\mathrm{Ber}}}		%rate fn for iid Bernoulli sum
\newcommand{\rateberc}{J_{\mathrm{Ber}}}	%entropy flux for \rateber
\newcommand{\rategeo}{I_{\mathrm{geo}}}		%rate fn for iid Geometric sum

\newcommand{\raterw}{I_{\trw}}				%random walk conditioning rate
\newcommand{\raterwRel}[2]{\raterw(#1\,/\!\!/\raisebox{-1pt}{$#2$})}	% relative raterw
\newcommand{\raterwREl}[2]{\raterw\big(#1\,/\!\!/\raisebox{-2pt}{$#2$}\big)}	% relative raterw
\newcommand{\slope}{\eta}

\newcommand{\parab}[1][]{p_{#1}}			%the parabola function
\newcommand{\hd}{\mathsf{D}}				%the tasep height function in d
\newcommand{\hh}{\mathsf{h}}				%the tasep height function
\newcommand{\uu}{\mathsf{u}}				%the tasep empirical density
\newcommand{\hic}{\mathfrak{h}_\mathrm{ic}}	%the limiting initial height function

\renewcommand{\P}{\mathbb{P}}				%probability
\newcommand{\HL}{\mathrm{HL}^{\mathrm{fw}}}%forward Hopf--Lax
\newcommand{\HLf}{\mathrm{HL}^{\mathrm{fw}}}%forward Hopf--Lax
\newcommand{\HLb}{\mathrm{HL}^{\mathrm{bk}}}%backward Hopf--Lax
\DeclareMathOperator{\hyp}{hyp}				%epigraph
\DeclareMathOperator{\CG}{CG}				%complete graph
\DeclareMathOperator{\tr}{tr}				%trace
\newcommand{\sgn}{\mathrm{sgn}}				%sign function
\newcommand{\parity}{\mathrm{prty}}			%count how many \up in \UD
\newcommand{\parityc}{\mathrm{prty}_\star}	%\parity of \vecUDc
\newcommand{\ind}{\mathbf{1}}				%indicator function
\newcommand{\img}{\mathbf{i}}				%imaginary
\renewcommand{\d}{\mathrm{d}}				%the d in dx
\renewcommand{\Re}{\mathrm{Re}}				%real part	
\newcommand{\set}[1]{ {\{#1\}} }
\newcommand{\kdelta}{\delta}				%Kronecker delta		
\newcommand{\Sumie}{\sum\nolimits_\mathrm{ie}} %inclusion-exclusion sum, displayed
\newcommand{\sumie}{\sum_\mathrm{ie}} 		%inclusion-exclusion sum, inline

\newcommand{\Mid}{\mathrm{mid}}

\newcommand{\norm}[1]{\Vert #1\Vert}

\newcommand{\Id}{\mathrm{I}} 				%the identity operator
\newcommand{\conj}{\theta} 					%conjugation operator
\newcommand{\rwop}{Q}						%random walk operator
\newcommand{\Slop}{S^{\triangleleft}}		%the left S operator
\newcommand{\Srop}{S^{\triangleright}}		%the right S operator
\newcommand{\ica}{\widehat{a}}

\newcommand{\icd}{\widehat{d}}
\newcommand{\ics}{\widehat{s}}
\newcommand{\icx}{\widehat{x}}

\newcommand{\icm}{\widehat{m}}
\newcommand{\icmu}{\widehat{\mu}}
\newcommand{\veca}{\mathbf{a}}
\newcommand{\vecb}{\mathbf{b}}
\newcommand{\vecd}{\mathbf{d}}
\newcommand{\vecs}{\mathbf{s}}
\newcommand{\vecx}{\mathbf{x}}
\newcommand{\vecy}{\mathbf{y}}
\newcommand{\icveca}{\widehat{\mathbf{a}}}

\newcommand{\icvecx}{\widehat{\mathbf{x}}}

\newcommand{\vecrho}{\boldsymbol{\rho}}
\newcommand{\ick}{\hat{k}} %\hat{k}
\newcommand{\Ic}{\hat}	%\hat
\NewDocumentCommand{\opuMQR}{O{} O{} m}{ {}_{#1}[#3]_{#2} } % MQR's version of opu
\NewDocumentCommand{\opu}{O{} O{} m}{ {}_{#1}|\![#3]\!|_{#2} }%the "operator unit"
\NewDocumentCommand{\Opu}{O{} O{} m}{ {\vphantom{\big|}}_{#1}\big|\hspace{-3pt}\big[#3\big]\hspace{-3pt}\big|_{#2} }% big opu
\NewDocumentCommand{\opuc}{O{} O{} m}{ {}_{#1}|\![#3_{\star}]\!|_{#2} }% \opu with canonical UD
\NewDocumentCommand{\opuci}{O{} O{} m}{ {}_{#1}|\![#3_{\star\,\mathrm{isle}}]\!|_{#2} }% \opu with canonical UD and isle-factorized

\newcommand{\alphabet}{\mathsf{Alphabet}}	%alphabets
\newcommand{\alphabetkl}[1][kk']{\alphabet(#1)}%the set of letters for kl
\newcommand{\wordset}{\mathsf{Words}}		%the set of words
\newcommand{\wordsett}{\wordset_{\oslash}}	%the set of nonempty words 
\newcommand{\wordsetkl}[1][kk']{\wordset(#1)}%the set of words for kl
\newcommand{\wordsettkl}[1][kk']{\wordset_{\oslash}(#1)}%the set of nonempty words in kl
\newcommand{\word}{w}
\newcommand{\wordF}{\mathrm{wrd}}			% the word of a function
\newcommand{\wordsub}{v}					%a substitution for word
\newcommand{\wordfull}{\word^{\mathrm{full}}}%the full word for kl
\newcommand{\wordfullkl}[1][kk']{\word^{\mathrm{full},#1}}%the full word for kl
\newcommand{\letterfirst}{j_\mathrm{*}}		%the first letters
\newcommand{\letterlast}{j^\mathrm{*}}		%the last letters 
\NewDocumentCommand{\wordkl}{O{\word} O{k} O{k'}}{ {}_{#2} {#1} {}_{#3} }		

\newcommand{\deep}{\mathsf{Deep}}

\newcommand{\up}{{\scriptstyle\nwarrow}}
\newcommand{\down}{{\scriptstyle\searrow}}
\newcommand{\ups}{{\scriptscriptstyle\nwarrow}}
\newcommand{\downs}{{\scriptscriptstyle\searrow}}
\newcommand{\UD}{\sigma}
\newcommand{\vecUD}{{\boldsymbol{\sigma}}}
\newcommand{\UDc}{\sigma_\star}
\newcommand{\vecUDc}{{\boldsymbol{\sigma}}_\star}
\newcommand{\UDi}{\sigma_\mathrm{inh}}
\newcommand{\vecUDi}{{\boldsymbol{\sigma}}_\mathrm{inh}}

\newcommand{\tree}{\mathsf{Tree}}			%the tree
\newcommand{\treekl}[1][kk']{\tree(#1)}%the kl tree
\DeclareMathOperator{\activated}{\mathsf{Act}}%	%activated letters
\newcommand{\children}{\mathsf{Children}}	%children of a given word
\newcommand{\increment}[1]{{}_{#1}\!\subset}%increament of a tower			 			
\newcommand{\Increment}[1]{{}_{\{#1\}}\!\subset}%singleton increament of a tower		 			
\newcommand{\tower}{\mathsf{Tower}}			%tower of words
\newcommand{\IHC}[1][kk']{\mathrm{IHC}(#1)}% the increasing hypograph condition

\newcommand{\isleset}{\mathsf{Isles}}
\newcommand{\islesetkl}[1][kk']{\mathsf{Isles}(#1)}
\newcommand{\llo}{<}						% i<j for all i\in\word and j \in\word'
\newcommand{\lli}{\ll}						%\ll in the isle sense
\newcommand{\cupi}{\,\rotatebox{90}{\scriptsize$(\!($}\,}%union in the isle sense	
\newcommand{\gterm}{\pmonomial}				%a generic term
\newcommand{\coef}{\alpha}					%the coefficient of \gterm in D_n

\DeclareMathOperator{\fdet}{fdet}			%formal determinant
\newcommand{\fdot}{{\raisebox{-0.5ex}{\scalebox{2.2}{$\cdot$}}}}%the commutative product
\newcommand{\bkk}[1]{[\![#1]\!]}			% [[  ]]
\newcommand{\angg}[1]{\langle\hspace{-2.5pt}\langle#1\rangle\hspace{-2.5pt}\rangle}	%<<  >>
\newcommand{\symb}{\mathsf{Symb}} 			% set of symbols
\newcommand{\Monomials}{\angg{X_1,\ldots,X_\ell}}
\newcommand{\MonomialXY}{\angg{X,Y}}
\newcommand{\cyclic}{\mathrm{cyc}}
\newcommand{\monomial}{\mathfrak{p}}		%an element in \Monomials/\cyclic
\newcommand{\pmonomial}{\mathfrak{q}}		%a product of \monomial's
\newcommand{\fcoef}{\til{\coef}}			%`formal coefficient': the coefficient of \pmonomial in 
\newcommand{\LAmonomial}{\mathfrak{r}}		%a linear algebraic monomial in Lemma l.det
\newcommand{\LAMonomials}{\mathcal{R}}		%a \phi(LAmonomial) in Lemma l.det

\newcommand{\indup}{\ind^{\!\scriptscriptstyle\nwarrow}}
\newcommand{\inddown}{\ind^{\!\scriptscriptstyle\searrow}}
\newcommand{\indsign}[1]{\ind^{\!#1}}

\newcommand{\fncontourr}{\psi}				%not yet fulling specalized
\newcommand{\fncontour}{\phi}				%fully speclized
\newcommand{\fnSll}{\fncontourr^\triangleleft}
\newcommand{\fnSrr}{\fncontourr^\triangleright}
\newcommand{\fnrww}{\fncontourr^\trw}
\newcommand{\fnrw}{\fncontour^\trw}			%the complext fn related to rwop
\newcommand{\fnSl}{\fncontour^\triangleleft}%the complext fn related to Srop
\newcommand{\fnSr}{\fncontour^\triangleright}%the complext fn related to Srop
\newcommand{\expfn}{\Phi}	
\newcommand{\expSl}{\expfn^{\triangleleft}}	
\newcommand{\expSr}{\expfn^{\triangleright}}			
\newcommand{\exprw}{\expfn^{\trw}}
\newcommand{\expSls}[1][]{\expfn^{\triangleleft}_{\star#1}}	
\newcommand{\expSrs}[1][]{\expfn^{\triangleright}_{\star#1}}	
\newcommand{\exprws}[1][]{\expfn^{\trw}_{\star#1}}	
\newcommand{\radius}{r}
\newcommand{\zls}[1][]{z^{\triangleleft}_{\star#1}}
\newcommand{\zrs}[1][]{z^{\triangleright}_{\star#1}}
\newcommand{\zrws}[1][]{z^{\trw}_{\star#1}}
\NewDocumentCommand{\expfnword}{O{\wordkl}}{\Psi_\star(#1)}
\newcommand{\Circ}{\mathcal{C}}

\newcommand{\rwfsymb}{\mathcal{F}_\star}
\NewDocumentCommand{\rwf}{O{\wordkl}}{\rwfsymb(#1)}%the random walk path function in the determintal
\newcommand{\Rwf}{\mathcal{G}_\star}		%the random walk path function in "main thm" context
\newcommand{\xl}{x^{\triangleleft}}			%the left tanget point of \rwf
\newcommand{\xr}{x^{\triangleright}}		%the right tanget point of \rwf

\renewcommand{\bar}{\overline}

\newcommand{\til}{\widetilde}

\newcommand{\shock}{\xi}

\newcommand*{\Cdot}{{\raisebox{-0.5ex}{\scalebox{1.8}{$\cdot$}}}} % largedot

\usepackage{newtxtext}

\graphicspath{{./figs/}}

\title{Hydrodynamic large deviations of TASEP}
\author{Jeremy Quastel and Li-Cheng Tsai}

\address[Jeremy Quastel]{Department of Mathematics, University of Toronto}
\address[Li-Cheng Tsai]{\hspace{1.5pt} Departments of Mathematics, University of Utah}

\subjclass[2010]{%
}

\begin{document}
\begin{abstract}
We consider the large deviations from the hydrodynamic limit of the Totally Asymmetric Simple Exclusion Process (TASEP).
This problem was studied in \cite{jensen00,varadhan04} and was shown to be related to entropy production in the inviscid Burgers equation. 
Here we prove the full large deviation principle.
Our method relies on the explicit formula of \cite{matetski16} for the transition probabilities of the TASEP.
\end{abstract}

\maketitle
\tableofcontents

\section{Introduction}\label{s.intro0}

In the study of hydrodynamic limits of stochastic interacting particle system, those with hyperbolic scaling \cite{rezakhanlou91,seppalainen96,seppalainen96a,fritz04,fritz04a} are distinguished from those with parabolic scaling \cite{guo88, kipnis13, vardhan13}. The latter hold for symmetric and weakly asymmetric particle systems and enjoy the regularity provided by a lasting viscosity; the former hold for strongly asymmetric systems and endure the vanishing of viscosity.

The distinction between the two types of scaling widens at the level of large deviations. The large deviations of symmetric and weakly asymmetric systems under parabolic scaling are of Freidlin--Wentzell type \cite{donsker89,kipnis89,jona93,kipnis13}, ubiquitous in stochastic systems. In their study of the \ac{TASEP} --- a simple and quintessential strongly asymmetric system --- Jensen \cite{jensen00} and Varadhan \cite{varadhan04} showed that under hyperbolic scaling the deviations take place within weak, but not necessarily entropic, solutions of the inviscid Burgers equation. They proposed a rate function given by the positive part of a Kruzhkov entropy production. This description of large deviations differs from the standard Freidlin--Wentzell ones, and is tied to the study of entropies \cite{lax73,dafermos16} in hyperbolic conservation laws. The works \cite{jensen00, varadhan04} obtained the full upper bound, but the lower bound only for specially constructed weak solutions.  This was extended to a larger class of weak solutions in~\cite{vilensky08}; see the paragraph after \hyperref[t.main]{Main Theorem}  in Section~\ref{s.results.spacetime} for a description.  The result was extended to related stochastic PDEs in \cite{mariani10}.

One way to understand the hyperbolic large deviations is to take the vanishing-viscosity limit at the level of rate functions. One can consider the Freidlin--Wentzell-type rate function in \cite{kipnis89} and attempt to show that it converges under hyperbolic scaling to the Jensen--Varadhan rate function. Doing so provides a transparent explanation of the origin of the hyperbolic large deviations, though it does not provide a proof of the large deviation principle. This program was carried out in \cite{bellettini10} for a general class of models, where the full upper bound and the lower bound for special weak solutions were obtained. (The upper/lower bound here corresponds to the lower/upper bound in the $ \Gamma $-convergence language.) The upper bound was extended to higher dimensions in \cite{barre18}. 
The picture of vanishing-viscosity for large deviations also enters the study of non-equilibrium systems in physics, through the macroscopic fluctuation theory \cite{bertini15}, especially in systems with open boundaries \cite{derrida01,derrida03,bodineau06,bahadoran10}.

A major challenge in the study of these large deviations is to approximate weak solutions of nonlinear hyperbolic PDEs with a bounded Kruzhkov entropy production. The problem of the regularity and structure of such solutions has attracted considerable interest in geometric measure theory \cite{delellis03a,delellis03,golse13,lamy18}; the solutions lie in a Besov space and have a rectifiable jump set of dimension one, outside of which they are of vanishing mean oscillation. For such irregular solutions, it is unclear how to construct a nearly optimal perturbation for proving the large deviation lower bound. This issue has obstructed access to the full large deviation principle in existing works.

There has been widespread interest in mathematics and physics in large deviations of strongly asymmetric systems, yet a full understanding remains open. In particular, proving the full large deviation principle has remained one of the key open problems in hydrodynamic limits \cite{varadhan09}. In this paper, we prove the full large deviation principle for the \ac{TASEP}. The rate function is the liminf of the Jensen--Varadhan rate function on a dense set of specially constructed weak solutions, which we term elementary solutions. Our result implies that any weak solution can be optimally approximated by nice, specially constructed weak solutions.

Our proof relies on the integrable structure of the \ac{TASEP}, specifically the exact formula from \cite{matetski16}. In a series of breakthroughs around 2000, exact one-point distributions were computed for \ac{TASEP} with a special initial condition and related models \cite{baik99,borodin00,johansson00} based on the study of Young diagrams, random matrices, and representation theory \cite{logan77,kim96,deuschel99,johansson98,borodin00a}. The results were extended to multi-point distributions and to a few more special initial conditions; see \cite[Section~1]{corwin12} for a review. Characterizing the transition probabilities of the \ac{TASEP} requires starting it from a general deterministic initial condition. Based on a biorthogonal structure obtained in \cite{sasamoto05,borodin07}, the work \cite{matetski16} obtained a determinantal formula for the transition probabilities.

The major challenge in our work is to extract the multi-point large deviation rate function from the determinant of \cite{matetski16}. The asymptotics of Fredholm determinants $ \det(1+A) $ are tractable in the regime where $ A $ is vanishing. For the question at hand, one-point deviations belong to the tractable regime, but multi-point deviations generally belong to the ill-behaved regime where $ A $ diverges; see Section~\ref{s.results.det}. The ill-behaved regime has bedeviled random matrix theory and requires very sophisticated analysis such as the Riemann--Hilbert methods \cite{deift93,deift95,baik01,baik08,deift08}. To the best of our knowledge, existing methods would not apply to what is required by our work (arbitrarily many points and general initial conditions).

Here we tackle the problem by working directly with a Plemelj(-like) expansion, namely expanding the determinant into a series of products of traces. At first glance this approach seems infeasible in the ill-behaved regime. However, we are able to identify and implement many exact cancellations through algebraic manipulations of the determinant. We develop a systematic approach to perform these manipulations to the point where the desired terms become dominant.

A byproduct of our result is a modified Onsager--Machlup principle. Originally proposed for reversible systems \cite{onsager53}, the principle has been generalized to and studied in some irreversible systems \cite{eyink90,gabrielli96,gabrielli97,eyink96,bertini01,bertini02}. The fixed-time rate function that we extract from the determinant is given by a certain random-walk entropy, which can be viewed as the Einstein entropy in statistical mechanics terms. We show that the random-walk entropy can be realized by the Jensen--Varadhan rate function on a weak solution. For a special initial condition, the weak solution is obtained by running Burgers equation backward, but for more general initial conditions the solution needs to be modified according to a max structure of the \ac{TASEP}; see Section~\ref{s.results.matching}. For a general class of scalar conservation laws, \cite{bellettini2010a} showed the matching of the quasipotential and Einstein entropy.

We conclude this introduction by discussing some related works. The work \cite{landim96} first showed that the type of large deviations considered here should have speed $ N $ and take place in weak solutions. The one-point, upper-tail large deviation principles for the \ac{TASEP} and a related model were first obtained in \cite{seppalainen98,seppalainen98a}, and recently for the ASEP in \cite{das21}. Fluctuations in strongly asymmetric systems under hyperbolic scaling were characterized in \cite{rezakhanlou95,rezakhanlou02,seppalainen02}. In the physics literature, large deviations of the \ac{TASEP} have been studied by using the Bethe ansatz \cite{derrida98,derrida99}. The works \cite{chleboun17,chleboun18} studied related models by specially constructed weak solutions. The work \cite{olla19} studied another type of large deviations, related to random tiling, of the \ac{TASEP}.
Recently, there has been much interest in the large deviations of the \ac{KPZ} equation \cite{kardar86} in mathematics and physics. Based on exact formulas and integrable structures \cite{amir11,calabrese10,dotsenko10,sasamoto10,borodin16,quastel19},
the works \cite{ledoussal16a,ledoussal16,krajenbrink17,sasorov17,corwin18,krajenbrink18,krajenbrink18a,krajenbrink18b,tsai18,cafasso19,das19,kim19,krajenbrink19,corwin20,corwin20a,ledoussal20,ghosal20,lin20} obtained results on the one-point tails; using the optimal fluctuation theory, the works \cite{kolokolov07,kolokolov09,meerson16,meerson17,meerson18,kamenev16,hartmann19,lin20a,krajenbrink21} studied the weak-noise large deviations. 
For some related models, the one-point large deviation principles have been established in \cite{georgiou13,janjigian15,basu17,emrah17,janjigian19}.

\subsection*{Outline}

In Section~\ref{s.results}, we give an overview of the proof and state the results.
Besides \hyperref[t.main]{Main Theorem}, a key intermediate result \hyperref[t.rw]{Fixed-time Theorem} is stated.
After Section~\ref{s.results}, the paper is divided into two parts, which can be read \emph{independently}.

\begin{description}[leftmargin=10pt]
\item[Determinantal analysis] This part consists of Sections~\ref{s.overview}--\ref{s.pftrw}.
In Section~\ref{s.overview}, we recall the determinantal formula and related operators from \cite{matetski16}, work out some examples, and introduce two algebraic manipulations that will be carried out later.
In Section~\ref{s.geo} we prepare some notation and highlight the geometric meanings of the notation.
In Section~\ref{s.updown} we carry out one manipulation, the up-down iteration, and establish properties of this iteration that will be relevant later.
Based on the output of the up-down iteration, in Section~\ref{s.isle}, we perform another manipulation, the isle factorization. At the end of this section, we expand the determinant into products of traces.
In Section~\ref{s.asymptotics}, we utilize a contour integral expression and steepest descent to obtain the rates of the traces.
Then, via geometric arguments, we identify the smallest rates among all relevant terms. 
Finally, in Section~\ref{s.pftrw}, based on the results in the preceding sections, we invoke approximation arguments to prove \hyperref[t.rw]{Fixed-time Theorem}.

\item[Hydrodynamic large deviations]
This part consists of Section~\ref{s.matching}. Here, we establish the matching of the random-walk rate function to the Jensen--Varadhan rate function and collect previous results to conclude \hyperref[t.main]{Main Theorem}.

\end{description}

\subsection*{Acknowledgements}
JQ was supported by the Natural Sciences and Engineering Research Council of Canada.
LCT thanks Yao-Yuan Mao for many useful discussions about the presentation of the paper and thanks Kuan-Wen Lai for the discussion about Lemma~\ref{l.det}.
LCT was partially supported by the NSF through DMS-1953407 and DMS-2243112 and by the Alfred P.\ Sloan Foundation through the Sloan Research Fellowship FG-2022-19308.

\section{Overview and results}
\label{s.results}

\subsection{TASEP, Burgers, and Kruzhkov entropy}
\label{s.results.tasepburgers}
The \ac{TASEP} can be thought of as a microscopic particle model for Burgers equation. Particles perform totally asymmetric random walks on $ \Z $ in continuous time, attempting jumps one step to the right using independent unit-rate Poisson clocks. If the target site is unoccupied, the jump is realized; otherwise, the jump is suppressed.
At fixed time $ t $, the configuration of such indistinguishable particles can be described by $ \eta(t) := (\eta(t,x))_{x\in\Z}\in\{-1,+1\}^\Z $, where $ \eta(t,x):=1 $ if a particle occupies site $ x $ and $ \eta(t,x):=-1 $ otherwise.
Note that our $ \{-1,+1\} $-valued convention differs from the standard $ \{0,1\} $-valued one. This $ \eta(t,x) $ solves
\begin{align}
	\label{e.discrete.burgers}
	\d \eta(t,x) 
	= 
	\big( \tfrac12 \Delta \eta(t,x) + \tfrac12 \nabla (\eta(t,x-1)\eta(t,x)) \big) \, \d t + \nabla \d M(t,x),
\end{align}
where $ \d $ acts on the $ t $ variable, $ \Delta $ and $ \nabla $ are the lattice Laplacian and gradient, acting on the $ x $ variable by $ \Delta f(x) := f(x+1) - 2f(x) + f(x-1) $ and $ \nabla f(x) := f(x+1)-f(x) $, and $ M(t,x) $ is a martingale in $ t $.
The $ \eta $-dependent terms in \eqref{e.discrete.burgers} together can be viewed as a microscopic Burgers equation, with the martingale being a noisy driving force.

The connection of the \ac{TASEP} and Burgers equation manifests itself through the hydrodynamic limit. Consider the empirical density field of the particle configuration under hyperbolic scaling: $ \uu_N(t) : = \frac{1}{N}  \sum_{x\in \Z} \eta(Nt,x) \delta_{x/N}(\Cdot) $, where $ \delta $ denotes the Dirac delta. Initiate $ \uu_N $ so that at $ t=0 $ it converges, in a suitable topology, to a fixed $ u_\text{ic} $. The process then $ \uu_N $ converges in probability in a suitable topology to the entropy solution of the inviscid Burgers equation
\begin{align}
	\tag{Burgers}
	\label{e.burgers}
	\partial_t u  = \tfrac12 \partial_x \big(u^2\big),
\end{align}
with the initial condition $ u_\text{ic} $ \cite{rezakhanlou91, seppalainen96}. Note that any relevant solution here must satisfy $ u(t,x) \in [-1,1] $ because $ \eta(t,x)\in\{-1,+1\} $. A non-rigorous way to understand the hydrodynamic limit
is to perform hyperbolic scaling in \eqref{e.discrete.burgers}. Doing so reveals that the viscosity (Laplacian) term and the noise (martingale) term gain prefactors of $ N^{-1} $, and non-rigorously dropping the two terms leads to \eqref{e.burgers}.

Even though the viscosity and noise terms formally vanish in \eqref{e.burgers}, they leave marks on the limiting behaviors of $ \uu_N $. The convergence of $ \uu_N $ to the entropy solution is a remnant of the vanishing viscosity. As we will see in the next paragraph, the vanishing noise term affects the entropy production in \eqref{e.burgers}. To describe it, view a given strictly convex function $ \rate:[-1,1]\to\R $ as a Kruzhkov entropy, and define $ J(u):=\int_0^u \d v\, v \rate'(v) $. The functions $ \rate,J $ form an entropy-entropy flux pair: For a continuous and piecewise smooth solution $ u $ of \eqref{e.burgers} we have $ \partial_t \rate(u(t,x)) - \partial_x J(u(t,x)) = 0 $. The last equality may fail when shocks appear, but is always non-positive for entropy solutions. In fact, entropy solutions can be characterized by the non-positivity condition $ \partial_t \rate(u(t,x)) - \partial_x J(u(t,x)) \leq 0 $ for just one (any one) entropy-entropy flux pair \cite{panov94,delellis04}. Since the \ac{TASEP} has product Bernoulli invariant measures \cite{liggett05}, the following function is the natural choice of a Kruzhkov entropy:
\begin{align}
	\label{e.rateBer}
	\rateber(u) := \tfrac{1+u}{2}\log(1+u) + \tfrac{1-u}{2}\log(1-u),
	\qquad
	u\in[-1,1].
\end{align}

We now describe the large deviations of the \ac{TASEP} studied in \cite{jensen00,varadhan04}. The study of large deviations seeks to characterize the exponentially small probability around given deviations (rare realizations).
More precisely, fixing a topological space $ \mathfrak{X} $, a \ac{lsc} $ \Rate:\mathfrak{X}\to[0,\infty] $, and $ \mathfrak{X} $-valued processes $ \{X_N\} $,
we say
\begin{defn}
The process $ X_N $ satisfies the \tdef{\ac{LDP}} with rate function $ \Rate $ and speed $ N $ if, for closed $ \calC \subset \mathfrak{X} $ and open $ \calO \subset \mathfrak{X} $,
\begin{align*}
	\limsup_{N\to\infty} \frac1{N}\log \P[X_N\in\calC] \leq -\inf_{\calC} \Rate
	\ \ \text{(Upper Bound)},
	\qquad 
	\liminf_{N\to\infty} \frac1{N}\log \P[X_N\in\calO] \geq -\inf_{\calO} \Rate
	\ \ \text{(Lower Bound)}.
\end{align*}
\end{defn}
\noindent{}%
Hereafter, fix $ T $ and observe the \ac{TASEP} within the macroscopic time horizon $ [0,T] $.
It was shown in \cite{jensen00,varadhan04} that off the weak solutions of \eqref{e.burgers} the probability decays superexponentially fast.
Namely, the deviations essentially take place in the set of weak solutions.
By analyzing the microscopic entropy production, the works proposed the rate function 
\begin{align}
	\label{e.rateJV}
	\rateJV(u) := \int_{[0,T]\times\R} \d t \d x \, \big( \partial_t \rateber(u) - \partial_x \rateberc(u) \big)_+\,,
	\qquad
	u=u(t,x) \text{ a weak solution of } \eqref{e.burgers}.
\end{align}
More explicitly, $ \rateberc(u) := \int_0^u \d v \, v \rateber'(v) = -\frac{1-u^2}{4} \log \frac{1+u}{1-u} + \tfrac{u}{2} $, and the positive part $f_+(x)=\max\{f(x),0\}$ of the entropy production is interpreted in the weak sense as
$
	\sup_{\phi} \{  \int_{[0,T]\times\R} \d t \d x \, ( -\partial_t\phi \cdot \rateber(u) + \partial_x\phi \cdot \rateberc(u) ) \},
$
where the supremum runs over $ \phi \in \Csp^\infty([0,T]\times\R) $ such that $ 0 \leq \phi \leq 1 $, $ \phi(0,\Cdot) \equiv 0 $, and $ \phi(T,\Cdot) \equiv 0 $.
Following the discussion in the previous paragraph, we see that in the realm of large deviations, the vanishing noise term steers the entropy production --- it introduces a positive part to $ \partial_t \rateber(u) - \partial_x \rateberc(u) $. In \cite{jensen00,varadhan04} the upper bound was proved with the rate function $ \rateJV $, but the corresponding lower bound was only obtained around special weak solutions; see also \cite{vilensky08}.  

\subsection{The height function and Hopf--Lax space}
\label{s.results.heightfn}
In this paper we work at the level of the height function. Information about the empirical density field can be straightforwardly recovered by taking a spatial derivative of the height function. At $ x=0 $, $ \hh(t,0) := -2( $number of particles that have crossed from site $ -1 $ to $ 0 $ up to time $ t $), and the value of the height at other $ x $ is obtained by integrating $ \eta(t,x) $, namely $ \hh(t,x+1)-\hh(t,x) =: \eta(t,x) $. At any $ t $, the height function $ \hh(t)=\hh(t,\Cdot) $ is piecewise linear with slopes  $ +1 $ or $ -1 $. The dynamics of the \ac{TASEP} translates into that of $ \hh(t) $ as follows: A local maximum (a $ \wedge $) turns into a local minimum (a $ \vee $) at unit rate, independently of other local maxima. In particular, $ \hh=\hh(\Cdot) $ is itself a Markov process. Under hyperbolic scaling we consider $ \hh_N(t,x) := \frac{1}{N}\hh(Nt,Nx) $. %$ \hh_N(t) := \hh_N(t,\Cdot) $, and $ \hh_N := \hh_N(\Cdot)=\hh_N(\Cdot,\Cdot) $.

Let us set up the space and topology for the height function. Let $ \Lip $ denote the set of $ 1 $-Lipschitz functions $ \R \to \R $, namely $ |f(x)-f(x')| \leq |x-x'| $, and equip this space with the uniform over compact metric $ \dist(f_1,f_2) := \sum_{n=1}^\infty 2^{-n} \sup_{x\in[-n,n]} | f_1(x) - f_2(x) | $. The fixed-time heights $ \hh(t) $ and $ \hh_N(t) $ live in $ \Lip $. As a process, $ \hh_N\in\Dsp([0,T],\Lip) $, the space of right-continuous-with-left-limit $ h:[0,T]\to\Lip $.
Equip this space with the uniform metric $ \dist_{[0,T]}(h_1,h_2) := \sup\{ \dist(h_1(t),h_2(t)) : t \in [0,T] \}$. The uniform metric is natural here, because $ \hh_N $ does not make macroscopic jumps up to superexponentially small probabilities. Further, since the uniform topology is stronger than Skorohod's $ J_1 $ topology, proving an \ac{LDP} in the former automatically implies the latter.

The height function enjoys an analogous hydrodynamic limit.
Integrating \eqref{e.burgers} in $ x $ gives the limiting PDE:
\begin{align}
	\tag{HJ Burgers}
	\label{e.intburgers}
	\partial_t h  = -\tfrac12\big( 1 - ( \partial_x h )^2 \big),
\end{align}
the Hamilton--Jacobi equation of Burgers.
We call $ h\in\Dsp([0,T],\Lip) $ a \tdef{weak solution of \eqref{e.intburgers}} if for almost every $ (t,x)\in[0,T]\times\R $ the derivative $ \partial_t h $ exists and \eqref{e.intburgers} holds classically; see Remark~\ref{r.Lip}\ref{r.Lip.express}.
The entropy solutions of Burgers translate into the Hopf--Lax solutions. 
Set
\begin{align}
	\label{e.parabola}
	\parab(t,x) = \big( - \tfrac{t}{2}-\tfrac{x^2}{2t} \big) \ind_{|x|<t} -|x|\ind_{|x|\geq t},
	\qquad
	\parab[\icx,\ica](t,x) := \parab(t,x-\icx) + \ica.
\end{align}
The \tdef{Hopf--Lax solution} of \eqref{e.intburgers} from an initial condition $ f\in\Lip $ is given by
\begin{align}
	\label{e.HLf}
	\big( \HL_t(f) \big)(x) 
	:=
	\sup\big\{ \parab[\icx,f(\icx)](t,x) : \icx\in\R \big\}
	=
	\sup\big\{ \parab(t,x-\icx) + f(\icx) : \icx\in\R \big\},\quad  t>0,
\end{align}
and we call $ \HL_t $ the \tdef{forward (in time) Hopf-Lax evolution}.

While $ \Dsp([0,T],\Lip) $ is the space where $ \hh_N $ lives, the large deviations of $ \hh_N $ actually concentrate around a much smaller subspace.
In \cite{jensen00,varadhan04}, the subspace is taken to be the space of all weak solutions of Burgers.
Here we formulate and work with a more tractable space (and the space is essentially the same as the former: Proposition~\ref{p.HLc.topo.weak.HLc}\ref{p.weak.HLc}).
\begin{defn}
We say $ h \in \Dsp([0,T],\Lip) $ satisfies the \tdef{Hopf--Lax condition} if
\begin{align}
	\label{e.HLc}
	\big( \HL_{t}(h(t_0)) \big) \le  h(t_0+t) \le h(t_0),
	\qquad
	t_0< t_0+t \in[0,T],
\end{align}
and let $ \HLsp $ denote the space of all such functions --- \tdef{the Hopf--Lax space}.
\end{defn}
\noindent{}%
%Proposition~\ref{p.HLc.topo} shows that $ \HLsp $ captures all the relevant deviations;
%Proposition~\ref{p.weak.HLc} compares $ \HLsp $ with almost sure solutions.
The following proposition will be proven in Appendix~\ref{s.a.HLsp}.

\begin{prop}
\label{p.HLc.topo.weak.HLc}
\begin{enumerate}[leftmargin=20pt, label=(\alph*)]
\item[]
\item \label{p.HLc.topo}
Given any open $ \calO\supset\HLsp $, there exists $ c=c(\calO,T)>0 $ such that
$
	\P[ \hh_N \notin \calO ] \leq c\,\exp(-\frac{1}{c}N^{2}).
$
Here the initial condition $ \hh_N(0) $ is arbitrary, possibly $ N $-dependent.
\item \label{p.weak.HLc}
The space $ \HLsp $ is the closure (under the uniform topology) of weak solutions of \eqref{e.intburgers}.
\end{enumerate}
\end{prop}

\begin{rmk}
\label{r.Lip}
\begin{enumerate}[leftmargin=20pt, label=(\alph*)]
\item[]
\item \label{r.Lip.express} 
It is not hard to show that every weak solution of \eqref{e.burgers} can be expressed as $ \partial_x h $ for some weak solution $ h $ of \eqref{e.intburgers}.
It follows from definition that every weak solution $ h $ of \eqref{e.intburgers} is $ \frac12 $-Lipschitz in $ t $.
\item \label{r.Lip.} 
By Proposition~\ref{p.HLc.topo.weak.HLc}\ref{p.weak.HLc}, any $ h\in\HLsp $ is $ \frac12 $-Lipschitz in $ t $, i.e. $|h(x,t)-h(x,t')|\le \tfrac12 |t-t'|$, because weak solutions of \eqref{e.intburgers} are and the Lipschitz bound is preserved under uniform limits.
\end{enumerate}
\end{rmk}

\subsection{Fixed-time large deviations}
\label{s.results.fixedtime}
Our proof of the \ac{LDP} proceeds by discretization in time.
Fix a metric space $ \mathfrak{L} $,
$ \Dsp([0,T],\mathfrak{L}) $-valued Markov processes $ X_N(t) $,
and a $ \rate(g\xrightarrow{\scriptscriptstyle t} f) : (t,g,f) \in(0,T]\times\mathfrak{L}^2 \to [0,\infty] $ that is \ac{lsc} in $ (g,f) $.
\begin{defn}
We say $ X_N $ satisfies \tdef{a fixed-time \ac{LDP} locally uniformly in the initial condition} with rate function $ \rate(g\xrightarrow{\scriptscriptstyle t}f) $ if,
for all $ (t,g,f)\in(0,T]\times\mathfrak{L}^2 $,
\begin{align*}
	\lim_{\delta\to 0}\limsup_{N\to\infty} \sup_{  g'\in B_\delta(g) }
	\Big| 
		\frac{1}{N} \log \P_{g'}\big[ X_N(t) \in B_\delta(f) \big]
		+ \rate\big(g\xrightarrow{t} f\big)
	\Big|
	= 0,
\end{align*}
where $ B_\delta(f) \subset \mathfrak{L} $ denotes the open ball with radius $ \delta $ centered at $ f $,
and $ \P_{g'}[\,\Cdot\,] := \P[\,\Cdot\,| X_N(0)=g'] $. 
\end{defn}

The major step toward the full \ac{LDP} is the following fixed-time result.
The proof will take up Sections~\ref{s.overview}--\ref{s.pftrw}.

\begin{thmfixedtime}
\label{t.rw}
The process $ \hh_N $ satisfies a fixed-time \ac{LDP} locally uniformly in the initial condition with rate function $ \raterw(g\xrightarrow{\scriptscriptstyle t}f) $
given in \eqref{e.raterw.gf} below.
\end{thmfixedtime}

We next define $ \raterw(\Cdot\xrightarrow{\scriptscriptstyle t}\Cdot) $,
and will verify that it is \ac{lsc} in Lemma~\ref{l.raterw.lsc}.
For $ f\in \Lip $, let
\begin{align}
	\tag{Irw f/\!/p}
	\label{e.raterwrel}
	\raterwREl{f}{\parab[\icx,\ica](t)} := 
	\left\{\begin{array}{l@{,\quad}l}
		\int_\R \d x \big( \rateber(\partial_x f) - \rateber(\partial_x \parab[\icx,\ica](t)) \big)& \text{when }  \parab[\icx,\ica](t) \leq f \leq \parab[\icx,\ica](0), 
		\\
		+\infty &\text{otherwise}.
	\end{array}\right.
\end{align}
Note that the integrand is nonzero only when $ |x-\icx|\leq t $, because $\parab[\icx,\ica](t)=\parab[\icx,\ica](0)=\ica-|x-\icx|$ when $ |x-\icx|> t $.
This, together with $\rateber$ being bounded, ensures that the integral in \eqref{e.raterwrel} is well-defined.
The rate function $ \raterw(g\xrightarrow{\scriptscriptstyle t}f) $ will be built from \eqref{e.raterwrel}.
Consider first the discretized setting.
Fix $ (\icvecx,\icveca) = (\icx_{1}<\ldots<\icx_{\icm},\ica_1,\ldots,\ica_{\icm}) \in \R^{2\icm} $
and $ (\vecx,\veca) = (x_1<\ldots<x_m,a_1<\ldots<a_m)\in\R^{2m} $,
which are respectively the discretized initial and terminal conditions.
Let
\begin{align*}
%	\label{e.parab.k}
	\parab[k](t,y) 
	:= 
	\parab[\icx_{k},\ica_{k}](t,y),
	\qquad
	\parab[k_1\ldots k_n](t,y) 
	:= 
	\max\{\parab[k_1](t,y),\ldots,\parab[k_n](t,y)\}.
\end{align*}
We call $ \parab[k_1\ldots k_n](0)\in\Lip $ a \tdef{massif}.
The rate of $ \P_{\parab[1\cdots\icm](0)}[\hh_N(t,x_i)\approx a_i , i=1,\ldots,m ] $ reads
\begin{align}
	\tag{Irw xa-xa}
	\label{e.raterw.xa}
	\raterw\big((\icvecx,\icveca)\xrightarrow{t} (\vecx,\veca)\big)
	:=
	\min \big\{ \raterwRel{ f_k }{\parab[k](t)} : \max_{k=1,\ldots,\icm} f_k(x_j)=a_j, \ j=1,\ldots,m \big\},
\end{align}
where that the minimum is taken over $f_k$ satisfying the post-colon constraint.
Let $ \hyp(f):= \{ (x,a)\in\R^2 : f(x) \leq a \} $ denote the hypograph of $ f $.
By definition, $ \raterw((\icvecx,\icveca)\xrightarrow{\scriptscriptstyle t} (\vecx,\veca)) $ is finite if and only if the points satisfy the following discretized Hopf--Lax condition:
\begin{align}
	\label{e.HLc.xa}
	(x_i,a_i) \in \hyp( \parab[1\cdots\icm](0) ) \setminus \hyp( \parab[1\cdots\icm](t) )^\circ,
	\quad
	i=1,\ldots,m,
\end{align}
where ${}^\circ$ denotes the interior.
Having defined $ \raterw $ in the discretized setting,
we take the continuum limits in the initial and terminal conditions to get
\begin{align}
	\tag{Irw g-xa}
	\label{e.raterw.gxa}
	\raterw\big(g\xrightarrow{t} (\vecx,\veca)\big)
	&:=
	\min \big\{ \raterw\big((\icx_k,g(\icx_k))_{k=1}^{\icm}\xrightarrow{t} (\vecx,\veca)\big) : \icx_1 < \ldots<\icx_{\icm} \big\},
\\
%	\notag
	\raterw\big(g\xrightarrow{t} f\big)
	&:=
	\lim_{\e\to 0} \ \inf \big\{ \raterw\big(g\xrightarrow{t} (x_j,f(x_j))_{j=1}^m\big) : x_1 < \ldots< x_m, M(\vecx)<\e \big\},
%\\
	\tag{Irw g-f}
	\label{e.raterw.gf}	
%	&\qquad\qquad
%	\big( \max_{i=2,\ldots,m}| x_{i}-x_{i-1} | + (-x_1)_+^{-1} + (x_m)_+^{-1} \big) < \e \big\}.
\end{align}
where $ M(\vecx):= \max_{i}| x_{i}-x_{i-1} | + (-x_1)_+^{-1} + (x_m)_+^{-1} $.
It is readily checked that \eqref{e.raterwrel}, \eqref{e.raterw.xa}--\eqref{e.raterw.gf} are mutually consistent.
For example, setting $ g=\parab[\icx,\ica](0) $ in \eqref{e.raterw.gf} reduces the result to \eqref{e.raterwrel},
and setting $ g=\parab[1\ldots\icm](0) $ in  \eqref{e.raterw.gxa} reduces the result to \eqref{e.raterw.xa}.

To gain some intuition for \hyperref[t.rw]{Fixed-time Theorem}, consider the wedge initial condition, $ \hh_N(0) = \parab[\icx_1,\ica_1](0)= \parab[1](0) $. \hyperref[t.rw]{Fixed-time Theorem} states that the probability of $ \hh_N(t)\approx f $ is approximately $ \exp(-N\raterwRel{f}{\parab[\icx,\ica](t)}) $. This result has an interpretation in terms of a random walk. Consider a random walk $ R(n) $,  with the time variable $ n\in[N(-t+\icx_{1}), N(\icx_{1}+t)]\cap\Z $ and i.i.d.\ $ \set{\pm 1} $-valued, mean-zero increments and scale the walk as $ R_N(\Cdot) := \frac{1}{N}R(N\Cdot) $. The walk starts at $ R_N(-t+\icx_{1}) = \parab[1](t,-t+\icx_1) $, ends at $ R_N(\icx_{1}+t)=\parab[1](t,\icx_1+t) $, and is conditioned to stay above $ \parab[1](t) $, namely $ R_N(\Cdot) \geq \parab[1](t,\Cdot) $. Such a walk enjoys an \ac{LDP} with $ \P[ R_N \approx f ] \approx \exp( - N \raterwRel{ f }{ \parab[1](t) }) $.
That is, $ \hh_N(t) $ enjoys the same \ac{LDP} as the conditioned random walk $ R_N $.

This random-walk interpretation originates from an integrable structure of the \ac{TASEP}.
Under the wedge initial condition $ \hh_N(0)=\parab[1](0) $,
the fixed-time height function $ \hh_N(t) $ can be realized as the top curve of a collection of mutually non-intersecting random walks.
This property is featured in a class of models in random matrix theory and the \ac{KPZ} universality class; see \cite[Section~4.6]{anderson10}.
One can think of the second-to-top curve concentrating around $ \parab[1](t) $, and the top curve behaving like a random walk conditioned to lie above it.
\hyperref[t.rw]{Fixed-time Theorem} asserts that at the level of large deviations, this intuition is correct. A second scenario, in which all the curves are lower than expected, corresponds to the  far less likely deviations off $ \HLsp $ described in Proposition~\ref{p.HLc.topo.weak.HLc}\ref{p.HLc.topo}.

In the case of a massif initial condition, the intuition for \eqref{e.raterw.xa} comes from the max structure of the \ac{TASEP}.
Let $ \hh^{k} $ denote the height function initiated from the wedge $ \hh^{k}(0) = \parab[{k}](0) $. 
Couple the dynamics of $ \hh $, $\hh^{1} $, \ldots , $ \hh^{\icm} $ through the basic coupling; see \cite[pp 215--219]{liggett99}.
Doing so gives $ \hh(t)= \max\{ \hh^{k}(t) : k=1,\ldots,\icm\} $.
Non-rigorously dismissing the dependence among $ \hh^{1},\ldots,\hh^{\icm} $,
one treats the \ac{LDP} off a massif as separately obtaining the \ac{LDP} for each wedge and taking the maximum of the result.
This heuristic leads the rate function~\eqref{e.raterw.xa}. 

\subsection{Determinantal analysis}
\label{s.results.det}

While it is plausible that the preceding intuition could be turned into a proof for the wedge initial condition, 
our proof of \hyperref[t.rw]{Fixed-time Theorem} does not utilize the preceding non-intersecting-random-walk structure 
and is instead based on the determinantal formula of \cite{matetski16}.

Our determinantal analysis will be carried out in a discretized (in space) setting. 
Initiating $ \hh_N $ from a massif initial condition $ \hh_N(0) = \parab[1\ldots\icm](0) $,
we seek to extract the rate of the probability that $ \hh_N(t) $ is pinned around $ a_i $ at $ x=x_i $:
\begin{align*}
	\ppin \,\text{``}\!=\!\text{''}\, \P_{\parab[1\ldots\icm](0)} \big[ \hh_N(t,x_i)\approx a_i, \, i=1,\ldots,m \big].
\end{align*}
The precise definition of $ \ppin $ will be given in \eqref{e.ppin}.
The work \cite{matetski16} offers an expression of the \tdef{under probability}:
\begin{align*}
%	\label{e.exact}
	\punder := \P_{\parab[1\ldots\icm](0)}\big[ \hh_N(t,x_i) \leq a_i, i=1,\ldots,m \big] = \det(\Id+A_N),
\end{align*}
where $ A_N $ is an explicit trace-class operator parameterized by $ (\icx_k,\ica_k)_{k=1}^{\icm} $ and $ (x_i,a_i)_{i=1}^m $;
see~\eqref{e.det..} and \eqref{e.det.}.
The desired probability $ \ppin $ can be obtained from $ \punder $
through the inclusion-exclusion formula; see \eqref{e.inclusion.exclusion}--\eqref{e.sumie}.

Given the explicit determinantal formula, one may be tempted to think that the problem can be straightforwardly solved.
This is far from the case.
Our analysis utilizes a Plemelj-like expansion:
\begin{align}
	\label{e.plemelj}
	1- \det(\Id+A_N) =  - \tr(A_N) - \tfrac{1}{2}(\tr(A_N))^2 + \tfrac{1}{2}(\tr(A_N^2)) + \ldots.
\end{align}
For the one-point case $ m=1 $,
the deviations of current interest (those in $ \HLsp $) correspond to the so-called upper (right) tail in the random matrix literature.
For this tail, the operator converges to zero in the trace norm, namely $ \norm{A_N}_1 \to 0 $.
Consequently, $ 1 \gg \tr(A_N) \gg (\tr(A_N))^2, (\tr(A_N^2)) \gg \ldots $,
and extracting the rate amounts to just estimating the leading term $ \tr(A_N) $.
On the other hand, when $ m>1 $, generally we have $ \norm{A_N}_1 \to +\infty $;
this fact will be demonstrated in the last paragraph of Section~\ref{s.overview.updown.isle}.
Ill behaviors of this type are ubiquitous and well-known challenges.
Techniques have been developed to tackle them \cite{tracy94,dean06,dean08,baik08,deift08,ramirez11},
almost always by reformulating the problem and solving the reformulated problem.
To the best of our knowledge, these techniques do not apply to our setting --- arbitrary $ m>1 $ and general initial conditions.
Here, we solve the problem by working directly with \eqref{e.plemelj} despite its ill behavior.
Our approach is to perform exact cancellations through algebraic manipulations of the determinant.

\subsection{Elementary solutions and the Onsager--Machlup principle}
\label{s.results.matching}

To connect the random-walk rate function to the Jensen--Varadhan rate function, we will construct a class of weak solutions, called elementary solutions, and show

\begin{prop}
\label{p.matching}
For any single-layer elementary solution $ h_\star $ on $ [0,t] $ (defined below), 
\begin{align*}
	\raterw\big( h_\star(0)\xrightarrow{t} h_\star(t)\big) = \rateJV|_{[0,t]}(\partial_y {h}_\star).
\end{align*} 
%Therefore, for any multi-layer elementary solution $ h_\star $ with layers $ 0=t_0<t_1<\cdots<t_n=T $ (defined below), 
%\begin{align*}
%	\sum_{i=1}^n \raterw\big( h_\star(t_{i-1})\xrightarrow{t_{i}-t_{i-1}} h_\star(t_i)\big) = \rateJV(\partial_y {h}_\star).
%\end{align*} 
\end{prop}
\noindent{}%
Here $ \rateJV|_{[0,t]} $ is defined by restricting the time integral in \eqref{e.rateJV} to $ [0,t] $.
This proposition identifies $ \rateJV $ as the excess Kruzhkov entropy production, and will be proven in Section~\ref{s.matching}.

Elementary solutions are special solutions similar 
to those considered in the lower bound in \cite{jensen00,varadhan04,vilensky08},
but the solutions enter our proof \emph{in a very different way}.
To demonstrate this fact, we start the \ac{TASEP} from the wedge initial condition $ \hh_N(0,x) = \parab(0,x) = -|x| $,
and for fixed $ a \in (-\frac{T}{2}, 0] $ consider the one-point large deviations
\begin{align}
	\label{e.one-point}
	\lim_{\delta\to 0} \lim_{N\to\infty}
	\tfrac{1}{N}
	\log\P_{\parab(0)}\big[ \hh_N(T,0) \in (-\delta+a,a+\delta) \big].
\end{align}
Under the conditioning $ \hh_N(T,0)\approx a $, a natural weak solution is constructed by having a constant antishock at $ x=0 $:
$
	h_\star(t,x) 
	:=
	( -\tfrac12(1-r^2)t -r|x| ) \ind_{|x|\leq rt }
	+
	\parab(t,x) \ind_{|x|> rt},
$
where $ r:= (1-\frac{2a}{T})^{1/2} $.
The upper bound in \cite{jensen00,varadhan04} implies $ \eqref{e.one-point} \leq -\rateJV(\partial_x h_\star) = -\frac{T}{2}(1-r^2)\log(\frac{1+r}{1-r})+Tr $.
This bound is sharp, but to prove it using analytic methods requires minimizing $ \rateJV $ over \emph{all} weak solutions subject to the relevant constraints.
Our approach is different. We extract the rate $ \eqref{e.one-point} = -\raterw((0,0)\xrightarrow{\scriptscriptstyle T}(0,a))= -\frac{T}{2}(1-r^2)\log(\frac{1+r}{1-r})+Tr $ \emph{directly} from the determinant, and show that the rate coincides with $ \rateJV\circ\partial_x $ evaluated on the natural $ h_\star $.

We now begin the construction of elementary solutions, which is motivated by the Onsager--Machlup principle.
Consider first a wedge initial condition $ \parab[1](0)=\parab[\icx_1,\ica_1](0) $.
Fix any terminal condition $ (\vecx,\veca) $ that satisfies the discretized Hopf--Lax condition~\eqref{e.HLc}.
Recall $ \raterw((\icx_1,\ica_1)\xrightarrow{\scriptscriptstyle t} (\vecx,\veca)) $ from \eqref{e.raterw.xa}, and let $ \Rwf \in \Lip $ be the unique minimizer of it;
the minimizer will be described in detail in Section~\ref{s.results.notation}.
Define the \tdef{backward Hopf--Lax evolution}
\begin{align}
	\label{e.HLb}
	\big( \HLb_{\tau}(f) \big)(y) 
	:= 
	\sup \big\{ f(x) - \parab(\tau,y-x) : x \in \R \big\},
\end{align}
which gives the Hopf--Lax solution of the time-reversed equation $ \partial_\tau h = \frac12(1-(\partial_y h)^2) $.
The Onsager--Machlup principle states that, given a terminal condition in the infinite-time setting, the optimal deviation is achieved by running the system backward in time.
Generalizing this principle to the finite-time setting, we set $ h_\star(\tau) := \HLb_{t-\tau}(\Rwf) $.
This function is indeed an weak solution of \eqref{e.intburgers} and satisfies the initial condition $ h_\star(0)=\parab[1](0) $.
Next we consider massif initial conditions.
Fix a minimizer $ \{ F_k \}_k $ of \eqref{e.raterw.xa}; the minimizer will be characterized in Section~\ref{s.results.notation}.
For a massif initial condition, the Onsager--Machlup principle needs to be modified according to the max structure (described in the last paragraph of Section~\ref{s.results.fixedtime}), namely
$
	h_\star(\tau) := \max\{ \HLb_{t-\tau}(F_k): k=1,\ldots,\icm \}.
$
This $ h_\star $ satisfies the terminal condition $ h_\star(x_j)=a_j $, for all $ j $, and the initial condition $ h_\star(0)=\parab[1\ldots\icm](0) $.
We will show in Section~\ref{s.matching.general} that $ h_\star $ is a weak solution of \eqref{e.intburgers}.
We call the above special weak solutions $ h_\star $ \tdef{single-layer elementary solutions on $ [0,t] $}.

We seek to concatenate, in time, single-layer elementary solutions to produce multi-layer ones.
However, the single-layer elementary solutions previously constructed have different types of initial and terminal conditions.
We hence need to relax the class of initial conditions.
Call $ g\in\Lip $ \tdef{Linear and downward Quadratic (LdQ)} if there exists a finite partition of $ \R $ into intervals such that, 
on each interval $ U $ the function is $ g(y)|_{U}=(\alpha_2 y^2 + \alpha_1 y + \alpha_0)|_{U} $, an $ \alpha_2\leq 0 $ if the interval $ U $ is unbounded.
Let $ \ldq $ denote the set of all LdQ functions.
We now extend the definition of elementary solutions to allow $ h_\star(0)\in\ldq $.
Fix $ t>0 $, $ g\in\ldq $, and $ (\vecx,\veca) $, recall $ \raterw(g\xrightarrow{\scriptscriptstyle t}(\vecx,\veca)) $ from \eqref{e.raterw.gxa},
and let $ \raterw((\icx_k,g(\icx_k))_{k=1}^{\icm}\xrightarrow{\scriptscriptstyle t}(\vecx,\veca)) $ be a minimizer.
View $ (\icx_k,\ica_k) := (\icx_k,g(\icx_k)) $ as a discrete initial condition,
adopt the notation $ \parab[1\ldots\icm](0) $ with $ \ica_k := g(\icx_k) $,
and consider the corresponding elementary solution $ h_\star $ with $ h_\star(0)=p_{1\ldots\icm}(0) $.
The elementary solution with initial condition $ g $ is then defined to be $ \til{h}_\star(\tau) := \max\{ h_\star(\tau), \HLf_\tau(g) \} $.
Indeed, $ \til{h}_\star(0) = g $.
In Section~\ref{s.matching.general}, we will show that $ \til{h}_\star $ is an weak solution of \eqref{e.intburgers},
and $ \til{h}_\star(\tau)\in\ldq $ for all $ \tau\in[0,t] $.
We can now concatenate single-layer elementary solutions to form multi-layer ones.
Let $ h_{\star 1},\ldots, h_{\star n} $ be single-layer elementary solutions that live respectively on $ [t_0,t_1],\ldots,[t_{n-1},t_n] $.
with $ h_{\star i}(t_i) = h_{\star (i+1)}(t_{i}) $.
A \tdef{multi-layer elementary solution} with layers $ t_0<t_1<\ldots<t_n $ is
$ h_\star(\tau) := h_{\star 1}(\tau)\ind_{[t_{0},t_1]}(\tau) + \ldots + h_{\star n}(\tau)\ind_{(t_{n-1},t_n]}(\tau) $.

Let $ \Elem $ denotes the set of all elementary solutions that live on $ [0,T] $.
The set $ \Elem $ is indeed dense in $ \HLsp $.
Proposition~\ref{p.matching} immediately generalizes to

\begin{customprop}{\ref*{p.matching}'}
\label{p.matching.}
For any $ h_\star\in\Elem $ with layers $ 0=t_0<t_1<\cdots<t_n=T $, 
\begin{align*}
	\sum_{i=1}^n \raterw\big( h_\star(t_{i-1})\xrightarrow{t_{i}-t_{i-1}} h_\star(t_i)\big) = \rateJV(\partial_y {h}_\star).
\end{align*} 
\end{customprop}

\subsection{The main result}
\label{s.results.spacetime}

The fixed-time \ac{LDP} in \hyperref[t.rw]{Fixed-time Theorem} can be leveraged into a full \ac{LDP} under suitable assumptions.
Recall the general setup in the first paragraph of Section~\ref{s.results.fixedtime},
let $ \dist^\mathfrak{L} $ denote the metric on $ \mathfrak{L} $,
and equip the space $ \Dsp([0,T],\mathfrak{L}) $ with the uniform metric $ \sup_{t\in[0,T]} \dist^\mathfrak{L}(h_1(t),h_2(t)) $.
\begin{itemize}[leftmargin=10pt]
\item[$ \circ $] Every closed ball in $ \mathfrak{L} $ is compact.
\item [$ \circ $] Assume that $ X_N $ satisfies the fixed-time \ac{LDP} locally uniformly in the initial condition with rate function $ \rate(g\xrightarrow{\scriptscriptstyle t}f) $.
\item [$ \circ $] Fix an $ \hic\in\mathfrak{L} $ and
initiate  $ X_N $ from deterministic initial conditions such that $ \dist^\mathfrak{L}(X_N(0),\hic) \to 0 $.
\item [$ \circ $] Assume that there exists an equicontinuous set $ \calK \subset \Csp([0,T],\mathfrak{L})\subset \Dsp([0,T],\mathfrak{L}) $
such that for any open $ \calO \supset \calK $,
$ \limsup_{N\to\infty} \frac{1}{N} \log \P[ X_N \notin \calO] = -\infty. $
\end{itemize}
For a partition $ \vect =(0=t_0<\ldots<t_n=T) $ of $ [0,T] $, we write $ \norm{\vect} := \max_{i=1}^n |t_i-t_{i-1}| $ for the mesh.

\begin{lem}(from fixed-time to full LDP)
\label{l.discrete.to.continuous}
The process $ X_N $ satisfies the \ac{LDP} with speed $ N $ and rate function
\begin{align*}
	\Rate(h) 
	:=
	\liminf_{ \norm{\mathbf{t}} \to 0 } \ \sum_{i=1}^{n} \rate\big( h(t_{i-1}) \xrightarrow{t_i-t_{i-1}} h(t_i) \big)
	\ \
	\text{ if } h(0)=\hic, % h\in \calK \text{ and } 
	\qquad
	\Rate(h):=+\infty \ \ \text{ otherwise}.
\end{align*}
\end{lem}
\noindent{}%
This lemma can be proven by standard point-set topology arguments.
We omit the proof but note that, given the assumptions on $ \mathfrak{L} $ and $ \calK $, the set $ \calK \cap\{ h : h(0)=\hic \} $ is compact.
For our purpose $ \mathfrak{L} = \Lip $, $ X_N = \hh_N $, and $ \calK=\HLsp $.
The required conditions on $ \calK $ are satisfied thanks to Proposition~\ref{p.HLc.topo.weak.HLc}\ref{p.HLc.topo} and Remark~\ref{r.Lip}\ref{r.Lip.}.

Fix any $ \hic \in \Lip $
and initiate the \ac{TASEP} from possibly $ N $-dependent deterministic initial conditions such that
$
	\lim_{N\to\infty } \dist( \hic ,\, \hh_N(0) ) = 0.
$
Generalization to a random initial condition is straightforward by conditioning on $ \hh_N(0) $, and will add a static part (which depends on the law of the random initial condition) to the rate function. 
Proposition~\ref{p.matching.} shows that the full rate function can be approximated by $ \rateJV $ along elementary solutions.
\begin{defn}
\label{d.Rate}
The rate function for the \ac{TASEP} height function is $ \RAte: \Dsp([0,T],\Lip) \to [0,\infty] $,
\begin{align*}
	\RAte(h) 
	:= 
	\liminf_{ \Elem \ni \til{h} \to h } \ \rateJV(\partial_x \til{h}) \ \ \text{ if }  h(0)=\hic \text{ and } h \in \HLsp,
	\qquad
	\RAte(h) :=+\infty \ \ \text{ otherwise}.
\end{align*}
\end{defn}

\begin{thmmain}
\label{t.main}
Under the preceding setup,
$ \hh_N $ satisfies the \ac{LDP} with rate function $ \RAte $ and speed $ N $.
\end{thmmain}
\noindent{}%
The rate function $ \RAte $ is \ac{lsc} by definition. We will show in Section~\ref{s.matching.pfmain} that $ \RAte|_{\Elem}=\rateJV\circ\partial_x|_{\Elem} $.

It is important to note that, compared with the existing works \cite{jensen00,varadhan04,vilensky08}, the true novelty here is the improved \emph{upper} bound.
In \cite{vilensky08}, a lower bound with the rate function $ \rateJV $ is obtained around what is referred to as a `nice' weak solution,  a solution that has jumps along finitely many piecewise $ \Csp^1 $ curves, is $ \Csp^1 $ off the jumps, and has one-sided limits at both sides of each jump. In fact, such `nice' weak solutions are very similar to the elementary weak solutions that we consider here.
%Those works obtained the lower bound around special solutions, and the upper bound requires examining all weak solutions.
\hyperref[t.main]{Main Theorem} proves a matching upper bound approximated by elementary solutions.

\subsection{Some more notation}
\label{s.results.notation}
We introduce some notation so that those who wish to first skip the determinantal analysis can do so after finishing this section.
For fixed $ t>0 $, $ (\icx_k,\ica_k)_{k=1}^{\icm} $, and $ (x_i,a_i)_{i=1}^{m} $,
we identify the indices $ 1,2,\ldots,m $ in the \emph{terminal} condition as \tdef{letters}, $ \alphabet := \{ 1,\ldots, m\} $,
and call a strictly increasing list of letters a \tdef{word}.
Formally,
$	\wordset 
	:=
	\{ \, \word = \, i_1\cdots i_n : i_1<\ldots<i_n \in \alphabet, \ n\geq 0 \, \}.
$
Note that our definition of words \emph{differs} from the standard one, which does not require letters to increase.
The empty word, denoted by $ \emptyset $, is included in our definition, and
we let $ \wordsett := \wordset \setminus \{\emptyset\} $
denote the set of nonempty words.
We write $ \word_j = i_j $ for the $ j $-th letter and $ |\word|=n $ for the length of $ \word\in\wordset $.
Since letters in  $ \word $ are strictly ordered, slightly abusing notation, we will also view $ \word $ as a \emph{set of letters}.
For example, $ i\in\word $ means ``$ \word $ contains the letter $ i $'', and $ \word\cup\word' $ denotes the word formed by taking the union of letters in $ \word $ and $ \word' $.

Refer to \eqref{e.raterw.xa} and consider first $ \icm=1 $, a wedge initial condition.
Let $ \Rwf=\Rwf( t; \icx_1,\ica_1; \vecx,\veca ) \in \Lip $ denote the unique minimizer in \eqref{e.raterw.xa}.
It is readily verified that the function is $ \Csp^1 $ except at $ y=x_j $, $ j\in\word $,
and is linear except when $ \Rwf(y) = \parab[k](t,y) $ or $ y= x_j $ for $ j\in\word $.
We will often abbreviate $ \Rwf( \wordkl[\word][k][k] ):= \Rwf( t; \icx_k,\ica_k; (x_j,a_j)_{j\in\word} ) $.
Section~\ref{s.overview.example} contains some illustrations of $ \Rwf $.
When $ \icm>1 $, the minimizers of \eqref{e.raterw.xa} may no longer be unique.
For example, take $ \icm=2 $, $ m=1 $ with $ \icx_{1}=-\icx_{2} $, $ \ica_{1}=\ica_{2} $, and $ x_1 =0 $.
By symmetry, $ (F_1,F_2) = ( \Rwf(\wordkl[1][1][1]), \parab[2](t) ) $ and $ (F_1,F_2)=( \parab[1](t), \Rwf(\wordkl[1][2][2])) $ are both minimizers.
We say $ f\in\Lip $ \tdef{passes through} $ j $ if $ f(x_j)=a_j $,
and let $ \wordF(f) $ denote the word formed by all letters $ f $ passes through.
Even though minimizers $ \{F_k\}_k $ of \eqref{e.raterw.xa} may be non-unique when $ \icm>1 $,
they are uniquely characterized by $ \{\wordF(F_k)\}_k $.
More explicitly, $ F_k = \Rwf(\wordkl[\wordF(F_k)][k][k]). $
%To see why, fix any minimizer $ \{F_k\}_k $ of \eqref{e.raterw.xa}.
%Indeed,
%$
%	F_k 
%	= 
%	\mathrm{argmin} \{ \raterwREl{ f }{ \parab[\icx,\ica](t) } : f \in \Lip,\,
%	\parab[k](t) \leq f \leq \parab[k](0), \, f(x_j) \leq a_j\,\forall j \},
%$
%where the condition $ f(x_j) \leq a_j $ is induced from the condition $ \max_{k=1,\ldots,\icm} f_{k}(x_j) = a_j $ in \eqref{e.raterw.xa}.
%Since $ F_k $ is also a minimizer of \eqref{e.raterw.xa},
%it already satisfies $ F_k(x_j) \leq a_j $, for all $ j $.
%Hence we can drop this condition to get
%$
%	F_k
%	= 
%	\mathrm{argmin} \{ \raterwREl{ f }{ \parab[\icx,\ica](t) } : f \in \Lip,\,
%	\parab[k](t) \leq f \leq \parab[k](0), \, \wordF(f) \supset \wordF(F_k) \} = \Rwf(\wordkl[\wordF(F_k)][k][k]).
%$
%
\begin{rmk}
\label{r.Rwf}
For a generic word $ \word \subsetneq 12\ldots m $,
the function $ \Rwf(\wordkl[\word][k][k]) $ may violate the condition $ \Rwf(\wordkl[\word][k][k])|_{y=x_j} \leq a_j $ for $ j=1,\ldots,m $.
For example, in Figure~\ref{f.ex.ud}, $ \Rwf(\wordkl[1][1][1])|_{y=x_2} > a_2 $.
\end{rmk}

\section{Determinantal analysis: formulas and examples}
\label{s.overview}

\subsection{The determinantal formula and related operators}
\label{s.overview.det}
Here we recall the determinantal formula of \cite{matetski16}.
Fix $ t\in\R_{>0} $, $ (\icvecx,\icveca)\in\R^{2\icm} $ and $ (\vecx,\veca)\in\R^{2m} $.
We will be considering the massif initial condition $ \parab[1\cdots\icm](0) $.
The $ x $'s and $ a $'s respectively label the horizontal (space) and vertical (height) coordinates.
It will be convenient to also consider the $\pm$ 45$^\circ$ coordinates, with the variables $ (s,d):=(\frac12(x+a),\frac12(-x+a))\in\R^2 $.
Accordingly, set 
$ s_{i} := \tfrac12(x_i+a_i) $, $ \ics_{i} := \tfrac12(\icx_i+\ica_i) $, $ d_{k} := \tfrac12(-x_{k}+a_{k}) $, $ \icd_{k} := \tfrac12(-\icx_{k}+\ica_{k}) $,
and express the height function in the $ s,d $ coordinates as
\begin{align}
	\label{e.hd}
	\hd_N(t,s) := \inf\big\{ \tfrac12(-x+\hh_N(t,x)) : \tfrac12(x+\hh_N(t,x)) \leq s \big\},
\end{align}
which is decreasing and right-continuous-with-left-limit in $ s $.
Throughout the determinant analysis we assume
\begin{align}
	\label{e.nondeg.sd}
	&\ics_1<\ldots<\ics_{\icm},&
	&\icd_1>\ldots>\icd_{\icm},&
	&s_1<\ldots<s_{m},&
	&d_1\geq \ldots\geq d_{m},&	
\end{align}
and the discretized Hopf--Lax condition \eqref{e.HLc.xa}. Consider the \tdef{under probability}:
\begin{align}
	\label{e.under}
%	\punder=
	\punder(\vecs,\vecd,N)
	:=
	\P_{\parab[1\cdots\icm](0)} \big[ \hh_N(t,x_i) \leq a_i,\,i=1,\ldots,m \big]
	=
	\P_{\parab[1\cdots\icm](0)} \big[ \hd_N(t,s_i^-) \leq d_i,\,i=1,\ldots,m \big].	
\end{align}

\begin{con}
\label{con.lattice}
Strictly speaking, 
$ (x_i,a_i) $ and $ (\icx_{k},\ica_{k}) $ should take values in $ \{ (x,a): x,\frac{1}{2}(x+a-\hh_N(0)) \in\frac{1}{N}\Z \} $.
To alleviate heavy notation, however, we will operate with $ (x_i,a_i)\in\R^2 $ and $ (\icx_{k},\ica_{k})\in\R^2 $
with the consent that proper integer parts are implicitly taken whenever necessary.
\end{con}

We next introduce the relevant operators, all of which act on $ \lsp^2(\Z) $. 
The underlying variable of $ \lsp^2(\Z) $ will be denoted by $ \mu\in\Z $. 
This variable has the same geometric meaning as $ d $ and should be viewed as the variable along the northwest-southeast axis.
Let $ \rwop(\mu,\mu')=\rwop(\mu-\mu'):= 2^{-1-\mu+\mu'} \ind_{\mu-\mu' \geq 0} $ so that, for $ n\in\Z_{> 0} $, $ \rwop^n(\mu,\mu') $ gives the $ n $-step transition probability of a random walk with i.i.d.\ increments $ \sim - $ Exp$ (2) $.
It is readily checked that $ \rwop $ has the inverse $ \rwop^{-1}(\mu-\mu') = \ind_{\mu=\mu'}-2\ind_{\mu-\mu'=1} $ and, for $ n,\mu,\mu'\in\Z $,
\begin{align}
	\label{e.rwop}
	\rwop^n(\mu,\mu')
	=
	\rwop^n(\mu-\mu')
	=
	\oint_{\Circ_r} \frac{\d z}{2\pi\img} \fnrww(z;n,\mu-\mu'), 
	\qquad
	\fnrww(z;n,\mu):=
	\frac{1}{z^{1+\mu}(2-z)^n}.
\end{align}
Hereafter $ \Circ_r:=\{z:|z|=r\} $ denotes a counterclockwise circle, and we assume $ r\in(0,2) $ throughout the paper.
Note that our definition of $ \rwop $ \emph{differs} from that of \cite{matetski16} by a shift in $ \mu $.
Next, define the operators
\begin{align}
	\label{e.Slop}
	\Slop_{-t,-n}(\icmu,\mu)
	=
	\Slop_{-t,-n}(\icmu-\mu)
	&:=
	\oint_{\Circ_r} \frac{\d z}{2\pi\img} \fnSll(z;t,-n,\icmu-\mu),&
	&
	\fnSll(z;t,n,\mu)
	:=
	\frac{e^{\frac{t}{2} (z-1)}}{z^{1+\mu}(2-z)^{n}},
%	\quad
%	n\in \Z_{\geq 0},
\\
	\label{e.Srop}
	\Srop_{-t,n}(\mu,\icmu)
	=
	\Srop_{-t,n}(\mu-\icmu)
	&:=
	\oint_{\Circ_r} \frac{\d z}{2\pi\img} \fnSrr(z;t,n,\mu-\icmu),&
	&
	\fnSrr(z;t,n,\mu)
	:=
	\frac{e^{\frac{t}{2} (z-1)}}{ z^{n} (2-z)^{1+\mu} },
%	\quad
%	n\in \Z_{>0}.
\end{align}
Let $ \ind_{ \geq d} $ and $ \ind_{ < d} $ act on $ \lsp^2(\Z) $ by multiplication, namely $ (\ind_{ \geq d}f)(\mu) := \ind_{\mu\geq d}f(\mu) $ and $ (\ind_{ < d}f)(\mu) = \ind_{\mu< d} f(\mu)$.
Note that the notation $\ind_{ \geq d}$ and $\ind_{ < d}$ are in operator form so the variables $\mu, \icmu$ no longer appear.

To facilitate subsequent presentation, we introduce some shorthand notation.
First, %$ \indup_{\ick} := \ind_{\geq N\icd_{\ick}} $, $ \inddown_{\ick} := \ind_{<N\icd_{\ick}} $, $ \indup_{i} := \ind_{\geq Nd_i} $, and $ \inddown_{i} := \ind_{ < Nd_i}$.
\begin{align*}
%	\label{e.ind.shorthand}
	\indup_{\ick} := \ind_{\geq N\icd_{\ick}},\qquad
	\inddown_{\ick} := \ind_{<N\icd_{\ick}},\qquad
	\indup_{i} := \ind_{\geq Nd_i},\qquad
	\inddown_{i} := \ind_{ < Nd_i}.
\end{align*}
The arrows $ \up $ and $ \down $ hint at the geometric meaning of the $ d $'s and $ \mu $'s, and should be read as `up' and `down', respectively.
Recall letters and words from Section~\ref{s.results.notation}.
For $ \word\in\wordset $ and $ \vecUD\in \{\up,\down\}^\word $, define the shorthand
\begin{align*}
	\opuMQR[k][k']{\word_\vecUD}
	:=
	\Slop_{-Nt,N(-\ics_{k}+s_{1})} (
		\rwop^{N(-s_1+s_{\word_1})}
		\indsign{\UD_{\word_1}}_{\word_1}
		\rwop^{N(-s_{\word_1}+s_{\word_2})}
		\indsign{\UD_{\word_2}}_{\word_2}
		\cdots 
		\indsign{\UD_{\word_{|\word|}}}_{\word_{|\word|}}
		\rwop^{N(-s_{|\word|}+s_{1})}
	) \Srop_{-Nt,N(-s_1+\ics_{k'})},
\end{align*}
with the convention
$
	\opuMQR[k][k']{\emptyset}
	:=
	\Slop_{-Nt,N(-\ics_{k}+s_{1})} \Srop_{-Nt,N(-s_1+\ics_{k'})}.
$  

Under the preceding notation, the determinantal formula of \cite{matetski16} reads
\begin{align}
	\label{e.det..}
	\punder
	&=
	\det\Big( 
		\Id +
		\Big(\inddown_{\ick}\big(
			 \kdelta_{k<k'} \rwop^{N(-\ics_{k}+\ics_{k'})}				
			- \opuMQR[k][k']{\emptyset}
			+ \opuMQR[k][k']{(12\cdots m)_{\downs\downs\cdots\downs} } 
		\big)\inddown_{\ick'}\Big)_{k,k'=1}^{\icm}
	\Big)_{  \lsp^2(\Z)^{\otimes\icm} }.
\end{align}
We use $ \kdelta_{k<k'} $ for the Kronecker delta instead of $ \ind $ to avoid confusion.
The expression $ (\cdots)_{k,k'=1}^{\icm} $ denotes an $ \icm\times\icm $ matrix with entries being operators on $ \lsp^2(\Z) $ or equivalently an operator on $  \lsp^2(\Z)^{\otimes\icm} $.
We call the operator within $ (\cdots)_{k,k'=1}^{\icm} $ the \tdef{\boldmath{$ kk' $-th entry}}.
The formula~\eqref{e.det..} follows from \cite[Theorem~2.6]{matetski16} and the time-reversal symmetry $ \P_{g}[\hh_N(t) \leq f] = \P_{-f(-\Cdot)}[ \hh_N(t,\Cdot) \leq -g(-\Cdot) ] $.

%\begin{rmk}
%\label{r.s1.condition}
%One more assumption $ s_{1} < \ics_{1} $ is actually required to meet the condition $ n>0 $ in \eqref{e.Srop}, but there is no loss of generality. The last paragraph in Section~\ref{s.results.fixedtime} gives $ \hh(t)=$ $ \max\{ \hh^{k}(t) : k=1,\ldots,\icm\} $. Any $ k $ with $ \ics_{k} \leq s_1 $ (which implies $ \ics_{k} \leq s_i $ for all $ i $) satisfies $ \hh^{k}_N(t,x_i) \leq \hh^{k}_N(0,x_i) = \parab[{k}](0,x_i) \leq a_i $ for all $ t $ and $ i $. Forgoing these $ k $'s and redefining $ \hh(t) := \max\{ \hh^{k}(t) : s_1 < \ics_{k} \} $ do not change $ \punder $.
%\end{rmk}

\subsection{The operator $ \opu[k][k']{\ldots} $}
\label{s.det.opu}
As already noted in \cite{matetski16}, $\Srop_{-t,n}$ involve a poorly behaved polynomial extension of the transition probabilities $\rwop^n$.
We now define a well-behaved variant of
the operator $ \opuMQR[k][k']{\word_\vecUD} $.
Specialize the functions $ \fnSl,\fnSr,\fnrw $ into the relevant scaled parameters as
\begin{align}
	\notag
	\fnSll_{ki}(z;\icmu) &:= \fnSll(z;Nt,N(-\ics_k+s_i),\icmu-Nd_i),
	&
	\fnSrr_{jk'}(z;\icmu') &:= \fnSrr(z;Nt,N(-s_j+\ics_{k'}),Nd_j-\icmu'),	
\\
	\label{e.fncontour}
	\fnSl_{ki}(z) &:= \fnSll_{ki}(z;N\icd_{k}),
	&
	\fnSr_{jk'}(z) &:= \fnSrr_{ki}(z;N\icd_{k'}),
\\
	\label{e.fncontour.}
	\fnrw_{ij}(z) &:= \fnrww(z;N(-s_i+s_j),N(d_i-d_j)),
	&	
	\fnrw_{\ick\ick'}(z) &:= \fnrww(z;N(-\ics_k+\ics_{k'}),N(\icd_{k}-\icd_{k'})).
\end{align}
We put hats over the $ k $'s in $ \fnrw_{\ick\ick'} $ to distinguish it from $ \fnrw_{ij} $.
Set $ \sgn(\up):=+1 $ and $ \sgn(\down):=-1 $. 

\begin{defn}
\label{d.opu}
Set $ \opu[k][k']{\emptyset} := \rwop^{N(-\ics_k+\ics_{k'})} $; for $ \word\in\wordsett $ and $ n=|\word| $, set
\begin{align}
	\label{e.opu}
	\hspace{-10pt}
	\opu[k][k']{\word_\vecUD}(\icmu,\icmu')
	:=
	\prod_{i=0}^{n} \oint \frac{\d z_{i}}{2\pi\img} \cdot \fnSll_{k\word_1}(z_0;\icmu)  
	\cdot
	\prod_{i=1}^{n-1}
	\frac{\sgn(\UD_{\word_i})z_{i}}{z_{i}-z_{i-1}} 
	\fnrw_{\word_{i}\word_{i+1}}(z_i)
	\cdot
	\frac{\sgn(\UD_{\word_n})(2-z_n)}{2-z_n-z_{n-1}}
	\fnSrr_{\word_nk'}(z_n;\icmu').
\end{align}
The contours are counterclockwise loops that satisfy the following conditions.
\begin{enumerate}[leftmargin=20pt, label=(\roman*)]
\item \label{d.opu.1} Each $ z_i $ contour encloses $ 0 $ but not $ 2 $.
\item For $ i=1,\ldots,n-1 $, the $ z_{i} $ contour encloses $ z_{i-1} $ if $ \UD_{\word_i} = \up $, and the $ z_{i-1} $ contour encloses $ z_{i} $ if $ \UD_{\word_i} = \down $.
\item \label{d.opu.3} If $ \UD_{\word_n}=\up $, the $ z_n $ contour does not enclose $ 2-z_{n-1} $, and the $ z_{n-1} $ contour does not enclose $ 2-z_{n} $.
\item \label{d.opu.4} If $ \UD_{\word_n}=\down $ and $ n>1 $, the $ z_{n-1} $ contour encloses $ 2-z_n $.
\item \label{d.opu.5} If $ \UD_{\word_n}=\down $ and $ n=1 $, the $ z_{1} $ contour encloses $ 2-z_0 $.
\end{enumerate}
The integrals are iterated in a suitable order so that the preceding conditions make sense.
For example, if $ \word=12 $ and $ \vecUD=(\up\down) $, we can take $ \prod_{i=0}^{3} \oint \frac{\d z_{i}}{2\pi\img}(\Cdot) := \oint \frac{\d z_0}{2\pi \img}(\oint \frac{\d z_2}{2\pi \img}(\oint \frac{\d z_1}{2\pi \img}\Cdot))  $ or $ := \oint \frac{\d z_2}{2\pi \img}(\frac{\d z_0}{2\pi \img}(\frac{\d z_1}{2\pi \img}\Cdot)) $.
\end{defn}

The operator $ \opu[k][k']{\word_\vecUD} $ is a variant of $ \opuMQR[k][k']{\word_\vecUD} $ in the following sense.
First, they enjoy the identities
\begin{align}
	\tag{flip'}
	\label{e.flip.Opu}
	\opuMQR[k][k']{ (\cdots j i j' \cdots)_{\cdots \UD_{j} \downs \UD_{j'} \cdots} }
	&=
	\opuMQR[k][k']{  (\cdots j j' \cdots)_{\cdots \UD_{j}\UD_{j'} \cdots} }
	- \opuMQR[k][k']{  (\cdots j i j' \cdots)_{\cdots \UD_{j} \ups \UD_{j'} \cdots} },
\\
	\tag{flip}
	\label{e.flip}
	\opu[k][k']{ (\cdots j i j' \cdots)_{\cdots \UD_{j} \downs \UD_{j'} \cdots} }
	&=
	\opu[k][k']{  (\cdots j j' \cdots)_{\cdots \UD_{j}\UD_{j'} \cdots} }
	- \opu[k][k']{  (\cdots j i j' \cdots)_{\cdots \UD_{j} \ups \UD_{j'} \cdots} }.
\end{align}
We call these identities flips because they effectively flip a $ \down $ to an $ \up $.
The identity \eqref{e.flip.Opu} follows from $ \inddown_{i} = \Id - \indup_{i} $, and the identity \eqref{e.flip} can be verified from Definition~\ref{d.opu}.
Second, it can be checked from \eqref{e.rwop}--\eqref{e.Srop} that the operators coincide when $ \vecUD=(\up\ldots\up) $ and $ \word\in\wordsett $, namely $ \opuMQR[k][k']{\word_{\ups\ldots\ups} }=\opu[k][k']{\word_{\ups\ldots\ups}} $.
The operators are however different in general, specifically when $ \word=\emptyset $.
It is readily checked that $ \opuMQR[k][k']{\emptyset} = \Srop_{0,N(-\ics_k+\ics_{k'})} $, whose kernel exhibits oscillatory behavior when $ k<k' $.
On the other hand, $ \opu[k][k']{\emptyset} := \rwop^{N(-\ics_k+\ics_{k'})} $, which is well-behaved.

We now express the determinantal formula \eqref{e.det..} in terms of $ \opu[k][k']{\word_{\vecUD}} $.
In \eqref{e.det..}, apply \eqref{e.flip.Opu} to flip every $ \down $ into an $ \up $ to get
$
	- \opuMQR[k][k']{\emptyset}
	+ \opuMQR[k][k']{1\cdots m_{\downs\cdots\downs}}
	=
	\sum_{\word\in\wordsett} (-1)^{|\word|} \opuMQR[k][k']{\word_{\ups\cdots\ups}},
$
use $ \opuMQR[k][k']{\word_{\ups\ldots\ups}}=\opu[k][k']{\word_{\ups\ldots\ups}} $, apply \eqref{e.flip} in reverse, and cancel $ \kdelta_{k<k'} \rwop^{N(-\ics_{k}+\ics_{k'})} $ with $ \opu[k][k']{\emptyset}:= \rwop^{N(-\ics_k+\ics_{k'})} $.
We have
\begin{align}
	\label{e.det.}
	\punder
	&=
	\det\Big( 
		\Id +
		\Big(\inddown_{\ick}\big(
			- \kdelta_{k\geq k'} \cdot \opu[k][k']{\emptyset}
			+ \opu[k][k']{(12\cdots m)_{\downs\downs\cdots\downs}} 
		\big)\inddown_{\ick'}\Big)_{k,k'=1}^{\icm}
	\Big).
\end{align}

For our subsequent analysis, we need to develop a factorization-like identity for $ \opu[k][k']{\word_{\vecUD}} $.
Take $ \word=12 $ for example.
We will construct an operator $ \opu[k][k']{(1_{\UD_1})_{k''}(2_{\UD_2})} $ so that the following identity holds for any $ k'' $:
\begin{align*}
	\opu[k][k']{12_{\vecUD}} = \opu[k][k']{(1_{\UD_1})_{k''}(2_{\UD_2})} + \opu[k][]{1_{\UD_1}}\inddown_{\ick''}\opu[][k']{2_{\UD_2}},
\end{align*}
where we have omitted repeated indices, namely $ \opu[k][]{1_{\UD_1}}\inddown_{\ick''}\opu[][k']{2_{\UD_2}} := \opu[k][k'']{1_{\UD_1}}\inddown_{\ick''}\opu[k''][k']{2_{\UD_2}} $.
We will adopt this convention hereafter.
Note that $ \opu[k][k']{12_{\vecUD}} \neq \opu[k][k'']{1_{\UD_1}}\,\opu[k''][k']{2_{\UD_2}} $, so in particular $ \opu[k][k']{(1_{\UD_1})_{k''}(2_{\UD_2})} \neq \opu[k][]{1_{\UD_1}}\indup_{\ick''}\opu[][k']{2_{\UD_2}} $; one may naively hope otherwise.
The construction of $ \opu[k][k']{(1_{\UD_1})_{k''}(2_{\UD_2})} $ needs to involve contour integrals.

\begin{defn}
\label{d.<<}
For $ \word,\word'\in\wordset $, we write $ \word\llo\word' $ if $ j<j' $ for all $ (j,j')\in\word\times\word' $.

In particular, $ \emptyset \llo \emptyset $, and $ \emptyset \llo \word $ and $ \word\llo\emptyset $ for all $ \word\in\wordsett $.
\end{defn}

\begin{defn}
\label{d.opu.extended}
For $ \wordsub^{(1)}\llo \ldots \llo \wordsub^{(\ell)} $, $ \vecUD^{(i)}\in\{\up,\down\}^{\wordsub^{(i)}} $, and $ k_0,\ldots,k_\ell\in\{1,\ldots,\icm\} $, set
\begin{align}
\label{e.opu.extended}
\begin{split}
	& \Opu[k_0][k_n]{ (\wordsub^{(1)}_{\vecUD^{(1)}})_{k_1} \cdots {}_{k_{n-1}} (\wordsub^{(n)}_{\vecUD^{(n)}}) } (\icmu,\icmu')
	:=
	\prod_{i=1}^{\ell}
	\oint \frac{\d z_{i}}{2\pi\img} \oint \frac{\d z'_{i}}{2\pi\img}
	\cdot
	U_1(z_1,z'_1;\icmu,\icd_{k_1}) V_1(z'_1,z_2)
\\
	&\qquad\qquad
	\cdot\,U_2(z_2,z'_2;\icd_{k_1},\icd_{k_2}) V_2(z'_2,z_3)
	\cdots
	V_{\ell-1}(z'_{\ell-1},z_\ell) U_\ell(z_\ell,z'_\ell;\icd_{k_{\ell}},\icmu').
\end{split}
\end{align}
\begin{description}[leftmargin=10pt]
\item[When \tdef{$ \wordsub^{(i)}\neq\emptyset $}]
$ U_i(z,z';\icmu,\icmu') $ denotes the right side of \eqref{e.opu} with the $ z_0 $ and $ z_n $ integrals removed and with $ (\word,\vecUD,z_0,z_n,k,k')\mapsto (\wordsub^{(i)},\vecUD^{(i)},z,z',k_{i-1},k_i) $, and $ V_i(z_i',z_{i+1}) := \frac{-z_{i+1}}{2-z_i'-z_{i+1}} $.
When $ i<\ell $ and $ \wordsub^{(i+1)}\neq\emptyset $, the $ z'_i $ contour encloses $ 2-z_{i+1} $ ; when $ i<\ell $ and $ \wordsub^{(i+1)}=\emptyset $, the $ z_{i+1} $ contour encloses $ 2-z'_{i} $.

\item[When \tdef{$ \wordsub^{(i)}=\emptyset $}]
We set $ U_i(z_i,z_i';\icmu,\icmu'):= \fnrw(z_i; N(-\ics_{k}+\ics_{k'}), N(\icmu-\icmu')) $, interpret $ z_i=z'_i $ and $ (\oint \frac{\d z_{i}}{2\pi\img} \oint \frac{\d z'_{i}}{2\pi\img}) := \oint \frac{\d z_{i}}{2\pi\img} $, and set $ V_i(z_i',z_{i+1}) := \frac{-z_{i+1}}{z'_i-z_{i+1}} $. When $ i<\ell $, the $ z_{i+1} $ contour encloses $ z'_{i} $.
\end{description}
The contours are counterclockwise loops such that all conditions in Definition~\ref{d.opu} and the preceding ones are satisfied.
\end{defn}

The operator $ \opu[k][k']{ \ldots } $ enjoys a factorization-like identity.
\begin{lem}
\label{l.opu.id}
If $ \wordsub^{(i)} = \wordsub\cup\wordsub' $ with $ \wordsub\llo\wordsub' $, then for $ \vecUD=\vecUD^{(i)}|_{\wordsub} $, $ \vecUD'=\vecUD^{(i)}|_{\wordsub'} $, and any $ k $,
\begin{align*}
	\Opu[k_0][k_n]{ (\wordsub^{(1)}_{\vecUD^{(1)}})_{k_1} \cdots {}_{k_{n-1}} (\wordsub^{(n)}_{\vecUD^{(n)}}) }
	=&
	\Opu[k_0][k_n]{ (\wordsub^{(1)}_{\vecUD^{(1)}})_{k_1} \cdots {}_{k_{i-1}}(\wordsub_{\vecUD})_k (\wordsub'_{\vecUD'})_{k_{i}} \cdots {}_{k_{n-1}} (\wordsub^{(n)}_{\vecUD^{(n)}}) }
\\
	&+
	\Opu[k_0]{ (\wordsub^{(1)}_{\vecUD^{(1)}})_{k_1} \cdots {}_{k_{i-1}}(\wordsub_{\vecUD}) }
	\inddown_{\ick}
	\Opu[][k_n]{ (\wordsub'_{\vecUD'})_{k_{i}} \cdots {}_{k_{n-1}} (\wordsub^{(n)}_{\vecUD^{(n)}}) }.	
\end{align*}
\end{lem}
\noindent{}%
This lemma can be proven from Definition~\ref{d.opu.extended}. We omit the proof.

\subsection{The pinning probability and a Plemelj-like expansion}
\label{s.overview.pinning}

Recall from \eqref{e.hd} that $ \hd_N(t,s) $ denotes the height function in the $ (s,d) $ coordinates.
For $ \vecd^{\pm} = (d_i^{\pm})_{i=1}^m $, where $ d_i^-<d_i^+ $,
we consider the \tdef{pinning probability}
\begin{align}
	\label{e.ppin}
	\ppin = \ppin(\vecs,\vecd^{\pm},N)
	:=
	\P_{\parab[1\cdots\icm](0)}\big[  d_i^{-} < \hd_N(t,s_i) \leq \hd_N(t,s_i^-) \leq d_i^+, \ i=1,\ldots,m \big].
\end{align}
By taking $ |d_i^+-d_i^-| $ and the mesh of $ \vecs $ to zero,
the pinning probability can approximate the probability that $ \hh_N(t)\approx f $, for a generic $ f\in\Lip $.
Further, $ \ppin $ can be expressed by $ \punder $ via the inclusion-exclusion formula:
\begin{align}
	\label{e.inclusion.exclusion}
	\ppin(\vecs,\vecd^\pm,N)
	=
	\Sumie
	\punder(\vecs,\vecd^{\vecrho},N).
\end{align}
Here $ \sumie $ denotes the inclusion-exclusion sum, which acts on a function $ \phi $ of $ \vecd $ by
\begin{align}
	\label{e.sumie}
	\Sumie \big( \phi(\vecd^{\vecrho}) \big)
	:= 
	\sum_{\rho_1=\pm 1} \rho_1 \cdots \sum_{\rho_m=\pm 1} \rho_m 
	\, \phi(d_1^{\rho_1},\ldots,d_m^{\rho_m}).
\end{align}

We will analyze the determinantal formula~\eqref{e.det.} by a Plemelj-like expansion.
Recall that, for a trace-class operator $ A $ on a Hilbert space $ H $, the Fredholm determinant is defined as
\begin{align}
	\label{e.fredholm}
	\det(\Id+A)
	=
	\sum_{n=0}^\infty \frac{1}{n!} \D_n(A),
\end{align}
where $ \D_n(A) := n! \tr(\wedge^n A) $, and $ (\wedge^n A) $ is the natural lifting of $ A $ onto $ \wedge^n H $.
By \cite[Lemma~6.7]{simon77}, 
\begin{align}
	\label{e.simon}
	\D_n(A)
	=
	\begin{vmatrix}
		\tr(A) & n-1 & &&\\
		\tr(A^2) & \tr(A) & n-2 &\\
		\vdots & \ddots & \ddots & \ddots &\\	 
		\tr(A^{n-1}) & \ldots & \ldots & \tr(A) & 1\\
		\tr(A^{n}) & \tr(A^{n-1}) & \ldots & \tr(A^2) &  \tr(A)\\
	\end{vmatrix}.	
\end{align}
Equations~\eqref{e.fredholm}--\eqref{e.simon} are essentially equivalent to Plemelj's expansion,
except that the latter requires $ \Vert A\Vert_1<1 $ while \eqref{e.fredholm} is absolutely convergent for all $ \Vert A\Vert_1<\infty $.
The expansion also holds in the $ \icm\times\icm $ setting for $ A=(A_{kk'})_{k,k'=1}^{\icm} $, 
where each $ \conj^k A_{kk'} \conj^{-k'} $ is trace-class on $ H $ for some operator $ \conj $ on $ H $.
In this case the traces are given by
\begin{align}
	\label{e.trace}
	\tr(A^n)
	:=
	\sum_{k_1,\ldots,k_{n-1}} \tr( A_{k_1k_2} \cdots A_{k_{n-1}k_1} ).
\end{align}

\subsection{Examples}
\label{s.overview.example}
To explain the idea of our determinantal analysis, here we work out a few examples. All examples here will assume $ \icm=1 $, namely starting the \ac{TASEP} from a wedge, and we write $ \P_{\parab[1](0)}=\P $ to simplify notation. For this subsection only, for quantities $ \alpha(N) $ and $ \beta $, we write $ \alpha(N)\approx \exp(N\beta) $ if $ \lim_{N\to\infty} \frac{1}{N} \log \alpha(N) = \beta $ and refer to $ \beta $ as the rate of $ \alpha(N) $.

As the main purpose here is to explain the idea, in the following examples we will not carry out the estimate of the entire expansion~\eqref{e.fredholm}--\eqref{e.simon}, but just the first few terms; the full estimate will be carried out in Section~\ref{s.asymptotics}. Below we list the estimates (which are special cases of the results from Section~\ref{s.asymptotics}) that will be used in the examples here.
Recall $ \Rwf( \wordkl[\word][1][1] ) $ from Section~\ref{s.results.notation} and recall $ \raterwRel{f}{\parab[1](t)} $ from \eqref{e.raterwrel}.
\begin{enumerate}[leftmargin=20pt,label=(\roman*)]
\item \label{e.ex}
Regardless of the configuration, %\\
$
	\tr( (\inddown_{\Ic{1}} \opu{i_{\ups}} \inddown_{\Ic{1}})^n) 
	\approx
	\exp( -N n \raterwRel{ \Rwf(\wordkl[i][1][1]) }{ \parab[1](t) }\,)
$,
for all $ i $ and $ n\geq 1 $.

\item \label{e.ex.uu}
For the configuration depicted in Figure~\ref{f.ex.uu},
$
	\tr( \inddown_{\Ic{1}} \opu{12_{\ups\ups}} \inddown_{\Ic{1}} )
	\approx \exp( -N \raterwRel{ \Rwf(\wordkl[12][1][1]) }{ \parab[1](t) }\,)
$.

\item \label{e.ex.ud}
For the configuration depicted in Figure~\ref{f.ex.ud},
$
	\tr( \inddown_{\Ic{1}} \opu{12_{\ups\downs}} \inddown_{\Ic{1}} )
	\approx
	\exp( -N \raterwRel{ \Rwf(\wordkl[12][1][1]) }{ \parab[1](t) }\,)
$.

\item \label{e.ex.u-u}
For the configuration depicted in Figure~\ref{f.ex.u-u},\\%
$
	\tr( \inddown_{\Ic{1}} \opu{ (1_{\ups})_{1} (2_{\ups}) } \inddown_{\Ic{1}} )
	\approx
	\exp( -N (\raterwRel{ \Rwf(\wordkl[1][1][1]) }{ \parab[1](t) }+\raterwRel{ \Rwf(\wordkl[2][1][1]) }{ \parab[1](t) })\,)
$.
\end{enumerate}
Under the current setup, \hyperref[t.rw]{Fixed-time Theorem} asserts $ \P[ \hh_N(t,x_i) \approx a_i, i=1,\ldots,m ] \approx \exp( -N \raterwRel{\Rwf(\wordkl[12\ldots m][1][1])}{\parab[1](t)} ) $. We will demonstrate the procedure for obtaining this rate in the examples below.

\begin{ex}
\label{ex.m=1}
Consider $ m=1 $. 
The pinning probability can be obtained from $ \P[ \hd_N(t,s_1) > d_1 ] $ by varying $ d_1 $.
For $ \icm=m=1 $ the determinantal formula~\eqref{e.det.} reads
$
	\P[ \hd_N(t,s_1) > d_1 ]
	=
	1 - \det( \Id +
		\inddown_{\Ic{1}}\big(
			-\opu[1][1]{\emptyset} + \opu[1][1]{1_{\downs}}
		)\inddown_{\Ic{1}}
		).
$
Use \eqref{e.flip} in the last expression to write $ -\opu[1][1]{\emptyset} + \opu[1][1]{1_{\downs}} = \opu[1][1]{1_{\ups}}  $, and use~\eqref{e.fredholm}--\eqref{e.simon}.
We have
\begin{align*}
	\P[ \hh_N(t,x_1) > a_1 ]
	&=
	\tr( \inddown_{\Ic{1}}\opu{1_{\ups}}\inddown_{\Ic{1}} )
	-
	\tfrac12 \tr((\inddown_{\Ic{1}} \opu{1_{\ups}} \inddown_{\Ic{1}})^2)
	+
	\tfrac12 (\tr(\inddown_{\Ic{1}}\opu{1_{\ups}}\inddown_{\Ic{1}}))^2
	+
	\ldots.
\end{align*}
Applying the estimate~\ref{e.ex} on the right side shows that $ \P[ \hd_N(t,s_1) > d_1 ] \approx \exp(- N \raterwRel{ \Rwf(\wordkl[1][1][1]) }{ \parab[1](t)}) $.
\end{ex}

\begin{figure}[h]
\centering
\begin{minipage}[t]{.327\linewidth}
	\frame{\includegraphics[width=\linewidth]{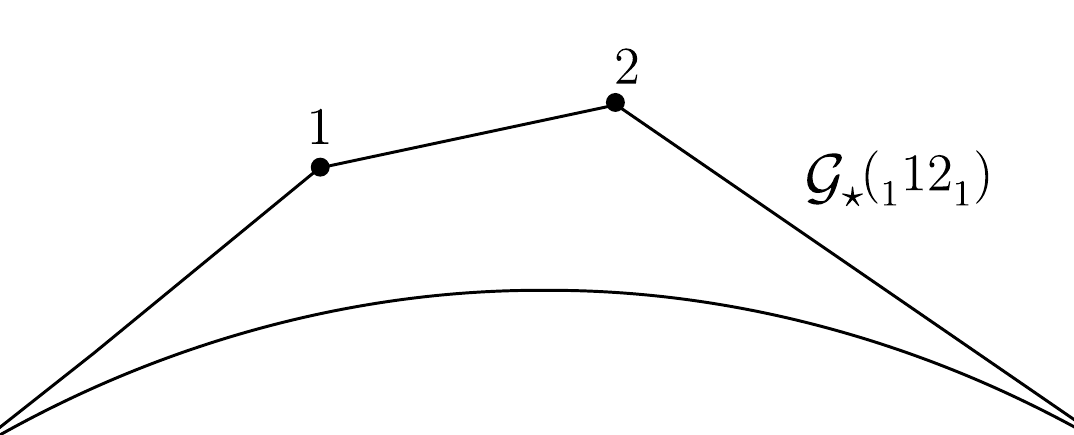}}
	\caption{%
%	Example~\ref{ex.uu}. 
	$ \Rwf(\wordkl[12][1][1]) $ has $ \wedge $ kinks at $ x_1, x_2 $;
	$ \hyp(\Rwf(\wordkl[12][1][1])) \setminus \hyp(\parab[1](t))^\circ $ is connected.%
	}
	\label{f.ex.uu}
\end{minipage}
\hfill
\begin{minipage}[t]{.327\linewidth}
	\frame{\includegraphics[width=\linewidth]{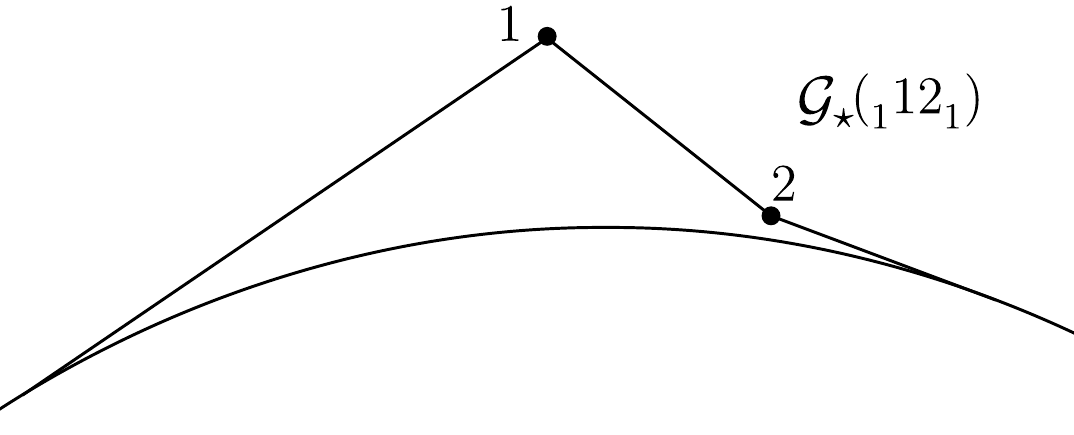}}
	\caption{%
%	Example~\ref{ex.ud}. 
	$ \Rwf(\wordkl[12][1][1]) $ has a $ \wedge $ kink at $ x_1 $ and a $ \vee $ kink at $x_2 $.}%
	\label{f.ex.ud}%
\end{minipage}
\hfill
\begin{minipage}[t]{.327\linewidth}
	\frame{\includegraphics[width=\linewidth]{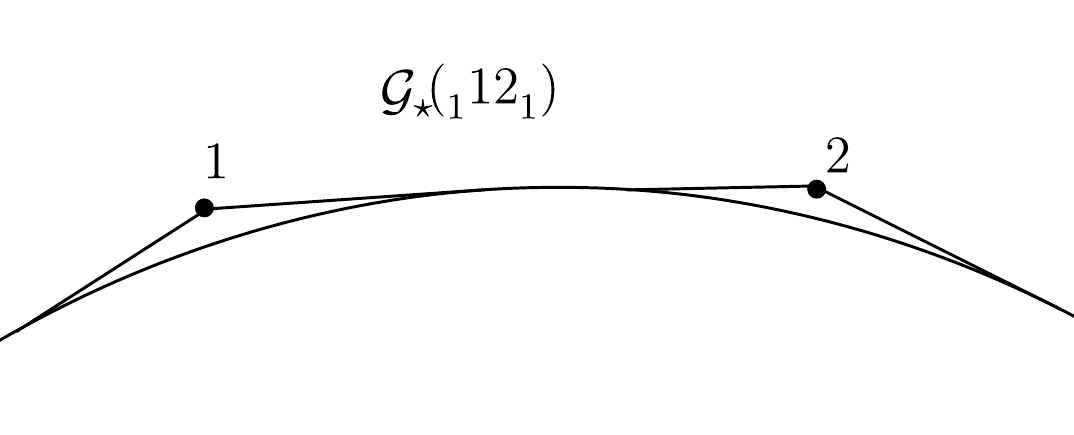}}
	\caption{%
%	Example~\ref{ex.u-u}.
	$ \hyp(\Rwf(\wordkl[12][1][1]))\setminus\hyp(\parab[1](t))^\circ $ is not connected.%
	}%
	\label{f.ex.u-u}%
\end{minipage}
%\end{figure}
%%
%\begin{figure}[h]
\vspace{5pt}
\centering
\frame{\includegraphics[width=.35\linewidth]{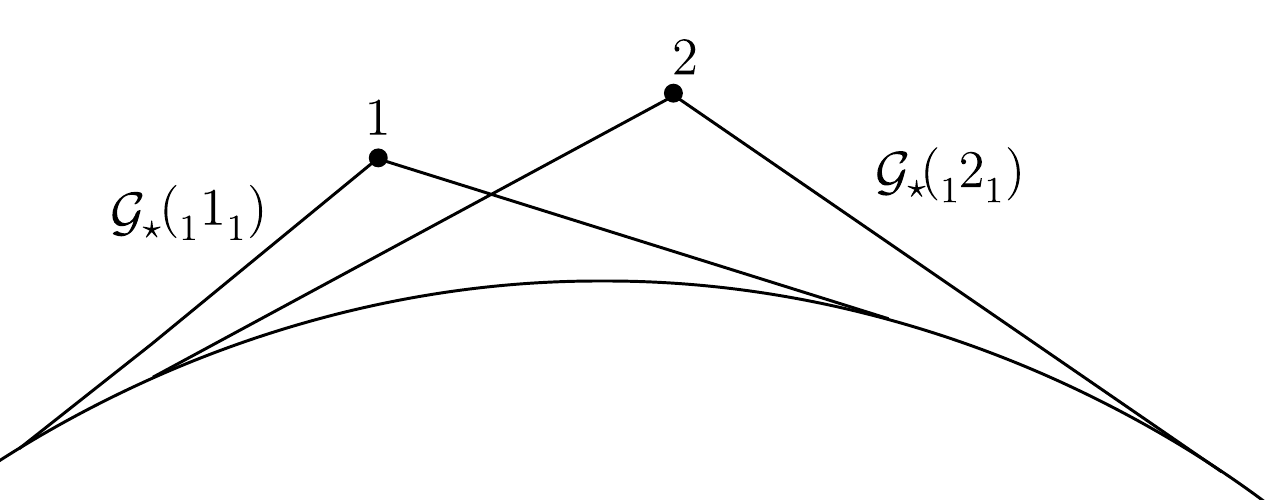}}%
\hspace{5pt}
\frame{\includegraphics[width=.35\linewidth]{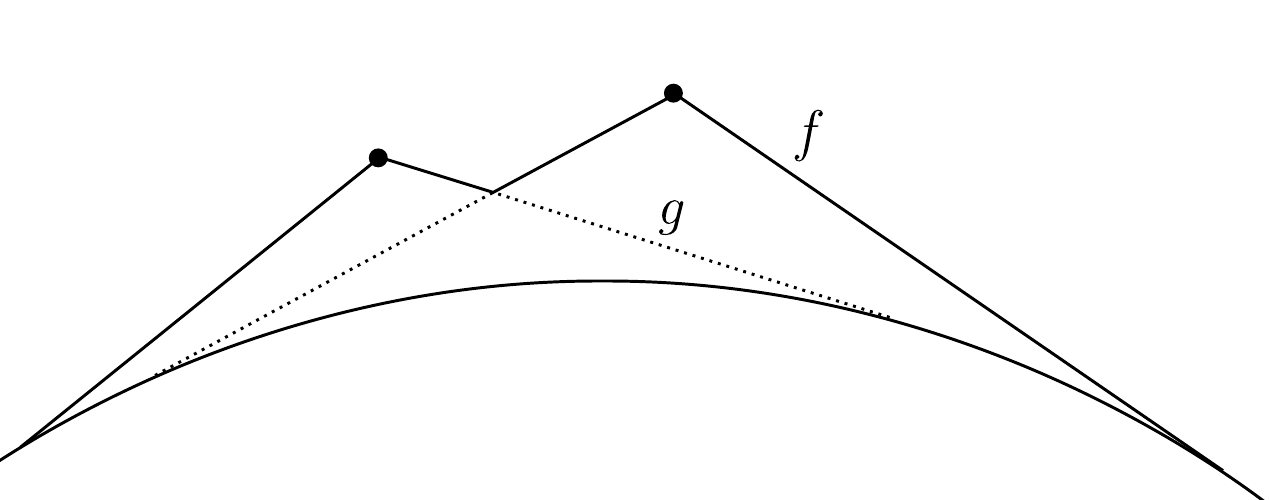}}
\caption{%
Rewiring of $ \Rwf(\wordkl[1][1][1]) $ and $ \Rwf(\wordkl[2][1][1]) $.
In the right figure, the solid line is $ f $; the dashed line is $ g $.%
}
\label{f.ex.uu.rewire}
\vspace{-10pt}
\end{figure}

\begin{ex}
\label{ex.uu}
Consider $ m=2 $ and the configuration depicted in Figure~\ref{f.ex.uu}.
Invoke the determinantal formula \eqref{e.det.} and use \eqref{e.flip} for $ i=1,2 $ to get
\begin{align}
	\label{e.ex.uu.0}
	\punder
	&=
	\det\big( 
		\Id - \inddown_{\Ic{1}} \opu{1_{\ups}} \inddown_{\Ic{1}}
		- \inddown_{\Ic{1}} \opu{2_{\ups}} \inddown_{\Ic{1}}
		+ \inddown_{\Ic{1}} \opu{12{\ups\ups}} \inddown_{\Ic{1}}
	\big)
\\
	\label{e.ex.uu.1}
	&= 1 
	- \tr\big( \inddown_{\Ic{1}} \opu{1_{\ups}} \inddown_{\Ic{1}} \big)
	- \tr\big( \inddown_{\Ic{1}} \opu{2_{\ups}} \inddown_{\Ic{1}} \big)
	+ \tr\big( \inddown_{\Ic{1}} \opu{12_{\ups\ups}} \inddown_{\Ic{1}} \big)
	+ \ldots.
\end{align}
Each term in \eqref{e.ex.uu.1} is a function of $ \vecd=(d_1,d_2) $ through Definition~\ref{d.opu}.
Fix small $ \e>0 $ and apply $ \sumie $ on both sides with $ d_i^\pm = d_i\pm\e $.
Upon the application of $ \sumie $ some terms vanish.
We call a term \tdef{degenerate} if it does not involve all the letters in $ \{1,2\} $. 
For example the first three terms on the right side of \eqref{e.ex.uu.1} are degenerate.
Referring to~\eqref{e.sumie}, one sees that degenerate terms vanish upon the application of $ \sumie $.
Hence
\begin{align}
	\label{e.ex.uu.2}
	\ppin
	=&
	\Sumie \tr\big( \inddown_{\Ic{1}} \opu{12_{\ups\ups}} \inddown_{\Ic{1}} \big)
	+
	\ldots.
\end{align}
The asymptotics of the first term is given in \ref{e.ex.uu}.
The rate $ \exp(-N\raterwRel{ \Rwf(\wordkl[12][1][1]) }{ \parab[1](t) }) $ depends on $ d_1,d_2 $.
Referring to Figure~\ref{f.ex.uu}, we see that among the four possible perturbations $ (d_1,d_2) \mapsto (d_1\pm\e,d_2\pm\e) $ of this rate, 
the perturbation $ (d_1-\e,d_2-\e) $ produces the smallest rate and hence the dominant contribution.
Therefore,
\begin{align}
	\label{e.ex.uu.3}
	\Sumie\tr\big( \inddown_{\Ic{1}} \opu{12_{\ups\ups}} \inddown_{\Ic{1}} \big)
	\approx
	\exp\big( -N \raterwREl{ \Rwf(\wordkl[12][1][1]) }{ \parab[1](t) }\,\big)\big|_{d_1-\e,d_2-\e}.
\end{align}

Other terms in \eqref{e.ex.uu.2} can be shown to be subdominant to \eqref{e.ex.uu.3}. The proof will be carried out in Section~\ref{s.asymptotics} in full generality, and here we work out one term $ \gterm := \tr ( \inddown_{\Ic{1}} \opu{1_{\ups}} \inddown_{\Ic{1}} )\tr(\inddown_{\Ic{1}} \opu{2_{\ups}} \inddown_{\Ic{1}} ) $ to demonstrate the idea.
The estimate~\ref{e.ex} asserts that $ \gterm $ has rate
$ 
	\raterwRel{ \Rwf(\wordkl[1][1][1]) }{ \parab[1](t) } 
	+\raterwRel{ \Rwf(\wordkl[2][1][1]) }{ \parab[1](t) }.
$
Since $ \Rwf(\wordkl[12][1][1])|_{[x_1,x_2]} $ is strictly above $ \parab[1](t)|_{[x_1,x_2]} $, 
the graphs of $ \Rwf(\wordkl[1][1][1]) $ and $ \Rwf(\wordkl[2][1][1]) $ intersect above $ \parab[1](t) $.
Follow Figure~\ref{f.ex.uu.rewire} to `rewire' the functions at the intersection to get $ f $ and $ g $,
or equivalently $ f:= \max\{ \Rwf(\wordkl[1][1][1]), \Rwf(\wordkl[2][1][1]) \} $ and $ g:= \min\{ \Rwf(\wordkl[1][1][1]), \Rwf(\wordkl[2][1][1]) \} $.
We have
\begin{align*}
	\raterwRel{ \Rwf(\wordkl[1][1][1]) }{ \parab[1](t) } 
	+
	\raterwRel{ \Rwf(\wordkl[2][1][1]) }{ \parab[1](t) }
	=
	\raterwRel{ f }{ \parab[1](t) } 
	+
	\raterwRel{ g }{ \parab[1](t) }.
\end{align*}
From Figure~\ref{f.ex.uu.rewire} we see $ \raterwRel{ f }{ \parab[1](t) } >\raterwRel{ \Rwf(\wordkl[12][1][1]) }{ \parab[1](t) } $
and $ \raterwRel{ g }{ \parab[1](t) } >0 $.
Hence
$ \gterm $ is subdominant.
\end{ex}

As shown in Example~\ref{ex.uu}, the basic idea is to develop an expansion with one or few dominant terms.
In general obtaining such a nice expansion requires careful algebraic manipulations, as illustrated in Examples~\ref{ex.ud}--\ref{ex.u-u}.

\begin{ex}
\label{ex.ud}
Consider $ m=2 $ and the configuration depicted in Figure~\ref{f.ex.ud}.
Here the expansion~\eqref{e.ex.uu.1} does not work.
To see why, write
$ 
	\tr( \inddown_{\Ic{1}} \opu{12_{\ups\ups}} \inddown_{\Ic{1}} )
	=
	\tr( \inddown_{\Ic{1}} \opu{1_{\ups}} \inddown_{\Ic{1}} )
	-
	\tr( \inddown_{\Ic{1}} \opu{12_{\ups\downs}} \inddown_{\Ic{1}} )
$
and apply the estimates \ref{e.ex} and \ref{e.ex.ud}.
The strict convexity of $ \rateber $ gives
$
	\raterwRel{ \Rwf(\wordkl[12][1][1]) }{ \parab[1](t) }
	> 
	\raterwRel{ \Rwf(\wordkl[1][1][1]) }{ \parab[1](t) },
$
thereby 
$
	\tr( \inddown_{\Ic{1}} \opu{12_{\ups\ups}} \inddown_{\Ic{1}} )
	\approx
	\exp( -N \raterwRel{ \Rwf(\wordkl[1][1][1]) }{ \parab[1](t) }).
$
On the other hand, assuming the conclusion of \hyperref[t.rw]{Fixed-time Theorem},
we know that the entire series \eqref{e.ex.uu.2} is 
$ \approx \exp( -N \raterwRel{ \Rwf(\wordkl[12][1][1]) }{ \parab[1](t) }\,) $.
Hence, as $ N\to\infty $, the term $ \tr( \inddown_{\Ic{1}} \opu{12_{\ups\ups}} \inddown_{\Ic{1}} ) $ is \emph{exponentially larger} than the entire series \eqref{e.ex.uu.2}.
This fact implies that there must be cancellations within~\eqref{e.ex.uu.1}.

The cure to this issue is to develop a different expansion than~\eqref{e.ex.uu.1}.
As seen in~\ref{e.ex.ud}, 
the term that captures the desired asymptotics is $ \tr( \inddown_{\Ic{1}} \opu{12_{\downs\ups}} \inddown_{\Ic{1}} ) $,
not $ \tr( \inddown_{\Ic{1}} \opu{12_{\ups\up}} \inddown_{\Ic{1}} ) $.
To obtain the former, use \eqref{e.flip} for $ i=1 $ to write
\begin{align*}
	\punder
	=
	\det\big( 
		\Id 
		- \inddown_{\Ic{1}} \opu{\emptyset} \inddown_{\Ic{1}}
		+ \inddown_{\Ic{1}} \opu{2_{\downs}} \inddown_{\Ic{1}}
		- \inddown_{\Ic{1}} \opu{12_{\ups\downs}} \inddown_{\Ic{1}}
	\big),
\end{align*}
and use \eqref{e.flip} for \emph{just} the operator $ \inddown_{\Ic{1}} \opu{2_{\downs}} \inddown_{\Ic{1}} $ to get
\begin{align*}
	\punder
	=
	\det\big( 
		\Id 
		- \inddown_{\Ic{1}} \opu{2_{\ups}} \inddown_{\Ic{1}}
		- \inddown_{\Ic{1}} \opu{12_{\ups\downs}} \inddown_{\Ic{1}}
	\big).
\end{align*}
%The rest of the treatment is similar to the preceding.
Expanding this determinant and estimating the result yield $ \ppin \approx \exp(-N\raterwRel{\Rwf(\wordkl[12][1][1])}{\parab[1](t)}) $.
\end{ex}

\begin{ex}
\label{ex.u-u}
Consider $ m=2 $ and the `two-isle' configuration depicted in Figure~\ref{f.ex.u-u}.
Just like in Example~\ref{ex.ud}, the expansion~\eqref{e.ex.uu.1} does not work here.
Invoke the identity from Lemma~\ref{l.opu.id} for $ \wordsub=1 $ and $ \wordsub'=2 $ in \eqref{e.ex.uu.0} to get
\begin{align}
	\label{e.ex.u-u.0}
	\punder
	&=
	\det\big( 
		\Id - \inddown_{\Ic{1}} \opu{1_{\ups}} \inddown_{\Ic{1}}
		- \inddown_{\Ic{1}} \opu{2_{\ups}} \inddown_{\Ic{1}}
		+ \inddown_{\Ic{1}} \opu{1_{\ups}} \inddown_{\Ic{1}} \opu{2_{\ups}} \inddown_{\Ic{1}}
		+ \inddown_{\Ic{1}} \opu{(1_{\ups})_{1}(2_{\ups})} \inddown_{\Ic{1}}
	\big)
%	.
%\end{align}
%Recognize 
%$ 
%	\inddown_{\Ic{1}} \opu{1_{\ups}} \inddown_{\Ic{1}} \opu{2_{\ups}} \inddown_{\Ic{1}}
%	=
%	(\inddown_{\Ic{1}} \opu{1_{\ups}} \inddown_{\Ic{1}} )
%	(\inddown_{\Ic{1}} \opu{2_{\ups}} \inddown_{\Ic{1}})
%$
%and factorize accordingly
%\begin{align}
\\
	\label{e.islefact.ex1}
%	\punder
	&=
	\det\Big( 
		\big( \Id - \inddown_{\Ic{1}} \opu{1_{\ups}} \inddown_{\Ic{1}} \big)
		\big( \Id - \inddown_{\Ic{1}} \opu{2_{\ups}} \inddown_{\Ic{1}} \big)
		+ \inddown_{\Ic{1}} \opu{(1_{\ups})_{1}(2_{\ups})} \inddown_{\Ic{1}}
	\Big).
\end{align}
It is possible to show that the trace norm of $ \inddown_{\Ic{1}} \opu{i_{\ups}} \inddown_{\Ic{1}} $ tends to $ 0 $ as $ N\to\infty $,
thereby
$	
	(\Id - \inddown_{\Ic{1}} \opu{i_{\ups}} \inddown_{\Ic{1}})^{-1}
	=
	\sum_{n=0}^\infty (\inddown_{\Ic{1}} \opu{i_{\ups}} \inddown_{\Ic{1}})^n
$.
Using this identity gives
\begin{align}
	\notag
	\punder
	=
	&\det\big( \Id - \inddown_{\Ic{1}} \opu{1_{\ups}} \inddown_{\Ic{1}} \big)
	\cdot
	\det\big( \Id - \inddown_{\Ic{1}} \opu{2_{\ups}} \inddown_{\Ic{1}} \big)
\\
	\label{e.islefact.ex2}
	&\cdot	
	\det\Big( 
		\Id
		+ 
		\sum_{n=0}^\infty (\inddown_{\Ic{1}} \opu{2_{\ups}} \inddown_{\Ic{1}})^n
		\sum_{\ell=0}^\infty (\inddown_{\Ic{1}} \opu{1_{\ups}} \inddown_{\Ic{1}})^\ell
		\cdot\,
		\inddown_{\Ic{1}} \opu{(1_{\ups})_{1}(2_{\ups})} \inddown_{\Ic{1}}
	\Big).
\end{align}
Expand these determinants, fix a small $ \e>0 $, and apply $ \sumie $ to the result with $ d_i^\pm=d_i\pm\e $.
We have
\begin{align*}
	\ppin
	=
	\Sumie \tr\big(\inddown_{\Ic{1}} \opu{1_{\ups}}\inddown_{\Ic{1}}\big)
	\tr\big(\inddown_{\Ic{1}} \opu{2_{\ups}}\inddown_{\Ic{1}}\big)
	+
	\Sumie \tr\big( \inddown_{\Ic{1}} \opu{(1_{\ups})_{1}(2_{\ups})} \inddown_{\Ic{1}} \big)
	+
	\ldots.
\end{align*}
By \ref{e.ex} and \ref{e.ex.u-u} the first two terms both produce the expected asymptotics for $ \ppin $. Note that $ \raterwRel{ \Rwf(\wordkl[12][1][1]) }{ \parab[1](t) } = \raterwRel{ \Rwf(\wordkl[1][1][1]) }{ \parab[1](t) }+\raterwRel{ \Rwf(\wordkl[2][1][1]) }{ \parab[1](t) } $ under the configuration depicted in Figure~\ref{f.ex.u-u}.

Another approach is to use the so-called extended kernel formula; see \cite{widom04,corwin14}. This approach leads to the same result in simpler configurations including the one in Figure~\ref{f.ex.u-u}, but does not seem to work generally. 
\end{ex}

\subsection{Up-down iteration and isle factorization}
\label{s.overview.updown.isle}
There are two types of algebraic manipulations that will be needed:
The \tdef{up-down iteration} will be performed in Section~\ref{s.updown}, and the \tdef{isle factorization} will be performed in Section~\ref{s.isle}.
These manipulations respectively ensure that the Plemelj-like expansion
\begin{enumerate}[leftmargin=20pt,label=(\roman*)]
\item \label{goal.updown}
only involves $ \opu[k][k']{\word_{\vecUD}} $ with $ \vecUD = \vecUDc(\wordkl) $, which will be defined in Section~\ref{s.geo.rwf.updown.isle}, and
\item \label{goal.isle}
only involves preferred terms, which will be defined in Definition~\ref{d.prefer}.
\end{enumerate}
These are the terms that are controllable as $ N\to\infty $. Condition~\ref{goal.updown} requires the $ \vecUD $ to be consistent with the kinks, as illustrated in Example~\ref{ex.ud}. Condition~\ref{goal.isle} is tied to the isle geometry, as illustrated in Example~\ref{ex.u-u}.

Let us note a subtlety regarding Example~\ref{ex.u-u}. There, we manipulate the determinant by the factorization \eqref{e.islefact.ex1}--\eqref{e.islefact.ex2}. This factorization, however, does not work directly for more complicated geometric configurations, and we need a more flexible variant of the factorization procedure. The key is to return to the pre-factorized determinant \eqref{e.islefact.ex1}. Expanding this determinant using~\eqref{e.fredholm}--\eqref{e.simon} gives
\begin{align*}
	\punder 
	= 
	(\text{degenerate terms})
	+\tr\big((\inddown_{\Ic{1}} \opu{1_{\ups}} \inddown_{\Ic{1}})(\inddown_{\Ic{1}} \opu{2_{\ups}} \inddown_{\Ic{1}})\big)
	-\tr\big((\inddown_{\Ic{1}} \opu{1_{\ups}} \inddown_{\Ic{1}})(\inddown_{\Ic{1}} \opu{2_{\ups}} \inddown_{\Ic{1}})\big)
	+ \ldots.
\end{align*}
The last two terms are non-preferred (defined in Definition~\ref{d.prefer}),
which we would like to avoid. Here, they cancel exactly, which is unsurprising given that \eqref{e.islefact.ex1} and \eqref{e.islefact.ex2} are the same. This observation suggests that, in fact, we can just work with the prefactorized determinant \eqref{e.islefact.ex1}, and use the factorization procedure \emph{only to argue} that any non-preferred term has a zero net coefficient in the prefactorized determinant. For this idea to work, we need a new notion of determinants where traces are viewed as \emph{indeterminate variables} in the sense of abstract algebra. We call such determinants formal determinants and develop them in Section~\ref{s.isle.fdet}.

The use of formal determinants has a technical bonus: We will only need to take inverses at the level of formal power series. More explicitly, the step of taking the actual inverse of $ (\Id - \inddown_{\Ic{1}} \opu{i_{\ups}} \inddown_{\Ic{1}}) $ in Example~\ref{ex.u-u} will be replaced by taking formal inverses, which do not require conditions for convergence.  
Note that we do still need to show convergence later in our proof.
What just said means that at the step of exhibiting cancellation---which is combinatorial in nature---we can do things without worrying about convergence.
After all required cancellations are exhibited, we will bound all the remaining terms via steepest descent analysis of the contour integrals, as in done in Section~\ref{s.asymptotics}.

Let us finally note the ill-behaved nature of the Plemelj-like expansion, using \eqref{e.ex.u-u.0} as an example.
(As said, \eqref{e.islefact.ex1} works in Example~\ref{ex.u-u} but not in general.)
Refer to Figure~\ref{f.ex.u-u}. Consider the function $ G\in\Lip $ such that $ G|_{\R\setminus(x_1,x_2)} = \Rwf(\wordkl[12][1][1])|_{\R\setminus(x_1,x_2)} $ and $ G|_{[x_1,x_2]} $ is linear.
This function, unlike $ \Rwf(\wordkl[12][1][1]) $, cuts through $ \parab(t) $ in $ (x_1,x_2) $.
It is possible to show that $ \tr( \inddown_{\Ic{1}} \opu{1_{\ups}} \inddown_{\Ic{1}} \opu{2_{\ups}} \inddown_{\Ic{1}} ) \approx \exp( -N \raterwRel{ G }{\parab[1](0)} $. The rate is \emph{negative} when the points in Figure~\ref{f.ex.u-u} are close to $ \parab(t) $. When this happens, at least one of the operators in \eqref{e.ex.u-u.0} is exponentially divergent in the trace norm.

\section{Determinantal analysis: geometric inputs}
\label{s.geo}
Here we introduce some notation and basic properties, with an emphasis on their geometric meanings.

\subsection{Trimming}
\label{s.geo.trim}

In the determinantal formula~\eqref{e.det.},
each $ kk' $-th entry involves all letters $ 1,\ldots,m $. 
On the other hand, based on geometric intuition one would expect that only those letters in
\begin{align}
	\label{e.alphabetkl}
	\alphabetkl
	=
	\big\{ j\in\{1,\ldots,m\} : d_j \leq \icd_{k} \text{ and } s_j \leq \ics_{k'} \big\}
\end{align}
should matter for the $ kk' $-th entry.
Here we show that indeed the determinantal formula~\eqref{e.det.} can be `trimmed' so that only letters in $ \alphabetkl[kk'] $ get involved in the $ kk' $-th entry.

Let us prepare some notation.
Let 
%\begin{align*}
$
	\wordsetkl
	:=
	\{ \, \word = \, i_1\cdots i_n : i_1<\ldots<i_n \in \alphabet(kk'), \ n\geq 0 \, \}
$
%\end{align*}
denote the set of words formed by letters in $ 	\alphabetkl[kk'] $.
As before, the empty word $ \emptyset $ is included in $ \wordsetkl[kk'] $, and
we let $ \wordsettkl := \wordsetkl \setminus \{\emptyset\} $.
Let $ \wordfullkl $ denote the \tdef{$ kk' $-th full word} formed by all letters in $ \alphabetkl $.
Note that $ \wordfullkl = \emptyset $ is possible, whence $ \wordsetkl = \{ \emptyset \} $.

The key property to perform the trimming is the following, which is readily verified from Definition~\ref{d.opu}.
\begin{align}
	\label{e.trimming}
	\inddown_{\ick}\opu{\word_{\vecUD}}\inddown_{\ick'} = 0
	\quad
	\text{if }\
	 d_{\word_1} > \icd_{k} \text{ and } \UD_{\word_1} = \up,
	\qquad
	\text{or } \
	s_{\word_{|\word|}} > \ics_{k'} \text{ and } \UD_{\word_{|\word|}} = \up.
\end{align}
\begin{lem}
\label{l.trimming}
Fix $ k,k' $. For any $ \letterfirst \leq \min\{ j : d_j \leq \icd_k \} $
and $ \letterlast \geq \max\{ j : s_j \leq \ics_{k'} \} $,\\
$
	\inddown_{\ick}\opu{ (1\cdots m)_{\downs\ldots\downs} }\inddown_{\ick'}
	=
	\inddown_{\ick}\opu{\ [\letterfirst,\letterlast]_{\downs\ldots\downs}}\inddown_{\ick'}\,.
$
\end{lem}
\begin{proof}
Use \eqref{e.flip} for each $ i\notin\alphabetkl $ to get
$
	\inddown_{\ick}\opu{(1\cdots m)_{\downs\ldots\downs}}\inddown_{\ick'}
	=
	\sum\nolimits \inddown_{\ick}\opu{\word_{\vecUD(\word)}}\inddown_{\ick'}.
$
Here the sum goes over all $ \word\in\wordset $ such that $ \word \supset [\letterfirst,\letterlast]  $,
and $ \UD(\word)_j := \down $ for $ j\in[\letterfirst,\letterlast]  $ and $ \UD(\word)_j := \up $ for $ j\notin[\letterfirst,\letterlast]  $.
Each $ j\notin[\letterfirst,\letterlast]  $ has either $ s_j>\ics_{k'} $ or $ d_j>\icd_{k} $.
Hence, by \eqref{e.trimming}, each term in the last sum is zero except when $ \word=[\letterfirst,\letterlast]  $.
The desired result follows.
\end{proof}

Invoking Lemma~\ref{l.trimming} for $ [\letterfirst,\letterlast] = \wordfullkl $ and inserting the result into the determinantal formula~\eqref{e.det.} give
\begin{align}
	\label{e.det}
	\punder
	=
	\det\big( \Id + \inddown_{\ick} \big(
		- \kdelta_{k\geq k'} \cdot \opu[k][k']{\emptyset}  
		+ \opu[k][k']{(\wordfullkl)_{\downs\ldots\downs}} 
	\big)\inddown_{\ick'} \big)_{k,k'=1}^{\icm}\big).
\end{align}

Hereafter $ \wordkl $ denotes a triplet $ (k,k',\word) $ that respects Convention~\ref{con.wordkl}.
This convention may appear to be the opposite of the $ kk' $-th entry in \eqref{e.det}, but actually is \emph{not}.
Later in Section~\ref{s.updown}, we will perform the up-down iteration to $ \opu[k][k']{(\wordfullkl)_{\downs\cdots\downs}} $.
When $ k \geq k' $, the iteration produces one copy of $ \opu[k][k']{\emptyset} $, which cancels the same term in \eqref{e.det}.
As a result, the empty word $ \emptyset $ is irrelevant when $ k \geq k' $.

\begin{con}
\label{con.wordkl}
The notation $ \wordkl $ assumes $ \word\in\wordsettkl[kk'] $ for $ k\geq k' $
and  $ \word\in\wordsetkl[kk'] $ for $ k<k' $.
\end{con}

\subsection{The function $ \rwf $ and related notation and properties}
\label{s.geo.rwf.updown.isle}
We begin by defining a function $ \rwf \in \Lip $,
which should be viewed as a generalization of $ \Rwf $ (defined in Section~\ref{s.results.notation}) but tailored to our algebraic manipulations.
Fix $ k,k'=1,\ldots,\icm $ and $ \word \in \wordsettkl $ and consider $ f\in\Lip $ that satisfy
\begin{subequations}
\label{e.rwfn.space}
\begin{align}
	\label{e.rwfn.space.<}
	&f(y) = \parab[k](t,y), \quad y \in (-\infty, -t + \icx_{k}],
\\
	\label{e.rwfn.space.<1}
	&f(y) \geq \parab[k\cdots\icm](t,y), \quad y\in[-t+\icx_{k}, x_{\word_1}],
\\	
	&
	\label{e.rwfn.space.1n}
	f(y) \geq \parab[1\cdots\icm](t,y), \quad y\in[x_{\word_1}, x_{\word_{|\word|}}],
	\qquad
	f(x_j) = a_j, \quad j\in\word,
\\
	\label{e.rwfn.space.>n}
	&f(y) \geq \parab[1\cdots k'](t,y), \quad y\in[x_{\word_{|\word|}},\icx_{k'}+t],
\\
	\label{e.rwfn.space.>}
	&f(y) = \parab[k'](t,y), \quad y \in [\icx_{k'}{+t},+\infty).
\end{align}
\end{subequations}
\begin{defn}
\label{d.rwf}
For $ \word\neq\emptyset $,
we define $ \rwf\in\Lip $ to be the unique minimizer of $ \int_{-t+\icx_{k}}^{\icx_{k'}+t} \d y \rateber(\partial_y f) $ among all $ f $'s that satisfy \eqref{e.rwfn.space};
when $ k<k' $ and $ \word=\emptyset $, we define $ \rwf[\wordkl[\emptyset]]\in\Lip $ to be the unique minimizer of $ \int_{-t+\icx_{k}}^{\icx_{k'}+t} \d y \rateber(\partial_y f) $ among 
all $ f\in\Lip $ that satisfy \eqref{e.rwfn.space.<}, \eqref{e.rwfn.space.>},
and $ f(y) \geq \parab[k\cdots k'](t,y) $, for all $ y\in\R $.
\end{defn}
\noindent{}%
Some realizations of $ \rwf $ are shown in Figures~\ref{f.rwf-1}--\ref{f.rwf-3}.
There the graphs of $ \parab[1](t) $, $ \parab[2](t) $, and $ \parab[3](t) $ 
are colored green, blue, and red, respectively;
the graph of $ \rwf $ is the black solid curve;
the dotted lines on the left and right are respectively $ s=\ics_{k} $ and $ d=\icd_{k'} $.

\begin{rmk}
The functions $ \Rwf(\wordkl[\word][k][k]) $ and $ \rwf[\wordkl[\word][k][k]] $ are different in general.
The latter imposes more constraints than the former; compare \eqref{e.raterw.xa} and \eqref{e.rwfn.space}.
Yet, as will be shown in Appendix~\ref{s.a.raterw}, these two functions are actually the same in any minimizer of \eqref{e.raterw.xa}.
Let us mention one related subtle point.
The definition of $ \rwf[\wordkl[\word][k][k]] $ does \emph{not} impose the condition $ f(x_j) \leq a_j $ for all $ j $,
which is required by any minimizer of \eqref{e.raterw.xa}.
In general, it is possible that $ \rwf[\wordkl[\word][k][k]]|_{y=x_j} > a_j $ for some $ j \notin\word $.
Nevertheless, as will be shown later in Section~\ref{s.updown} that,
for those $ \word $ that are relevant to our determinant analysis, $ \rwf[\wordkl[\word][k][k]]|_{y=x_j} \leq a_j $ for all $ j $;
specifically, this statement follows from Proposition~\ref{p.tree.prop}\ref{p.IHC}.
\end{rmk}

\begin{figure}[h]
\begin{minipage}[t]{.40\linewidth}
	\frame{\includegraphics[height=93pt]{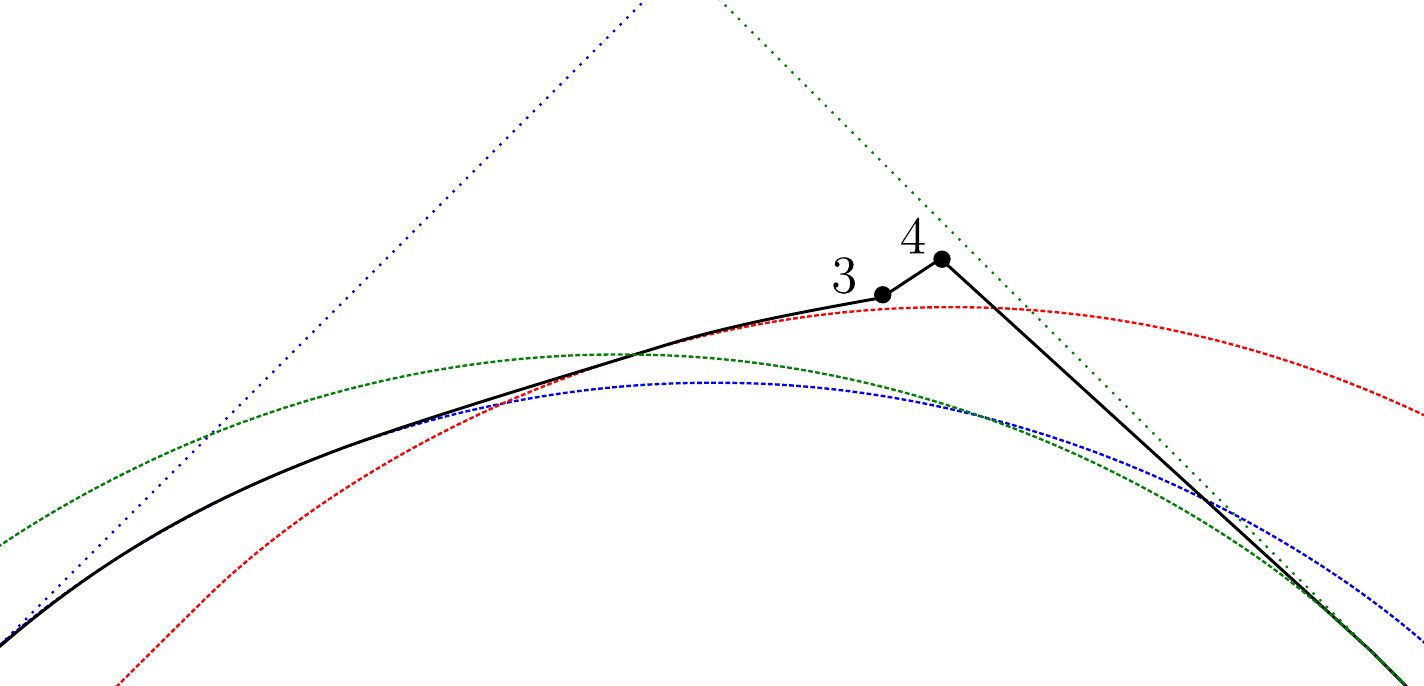}}
	\caption{An example of $ \rwf[\wordkl[34][2][1]] $, $ \icm=3 $.}
	\label{f.rwf-1}
\end{minipage}
\hfill
\begin{minipage}[t]{.58\linewidth}
	\frame{\includegraphics[height=93pt]{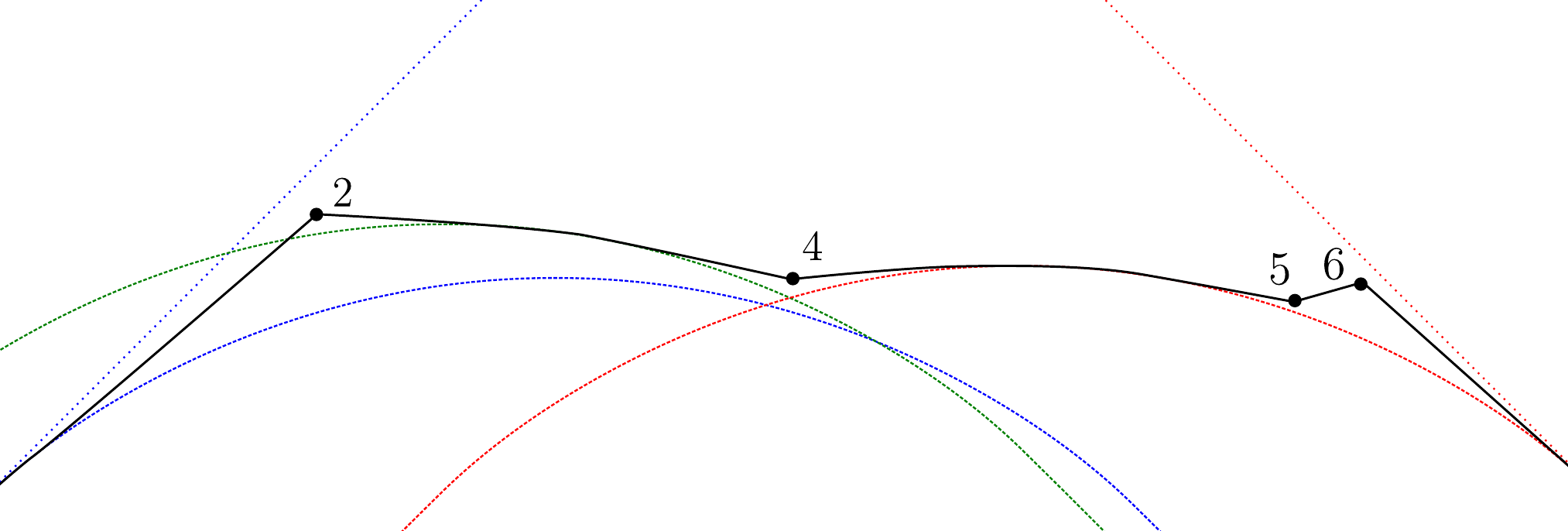}}
	\caption{An example of $ \rwf[\wordkl[2456][2][3]] $, $ \icm=3 $.}
	\label{f.rwf-2}%
\end{minipage}\\
%\end{figure}
%
%\begin{figure}[h]
\begin{minipage}[t]{.66\linewidth}
	\frame{\includegraphics[height=93pt]{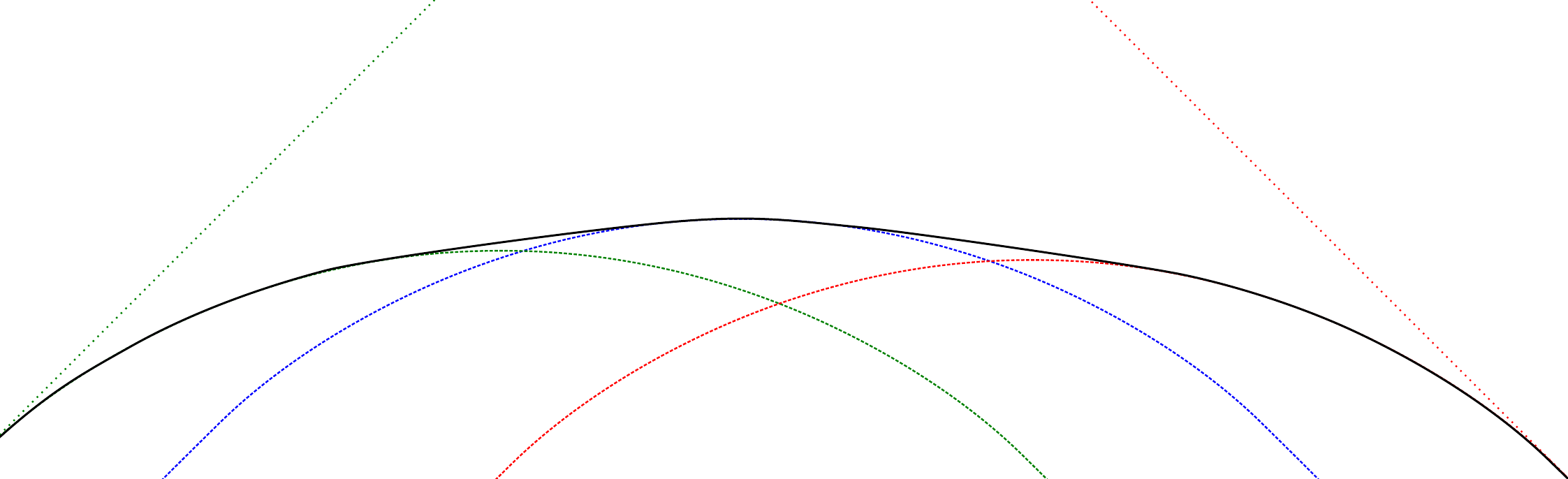}}
	\caption{An example of $ \rwf[\wordkl[\emptyset][1][3]] $, $ \icm=3 $.}
	\label{f.rwf-3}
\end{minipage}
\hfill
\begin{minipage}[t]{.32\linewidth}
	\frame{\includegraphics[height=93pt]{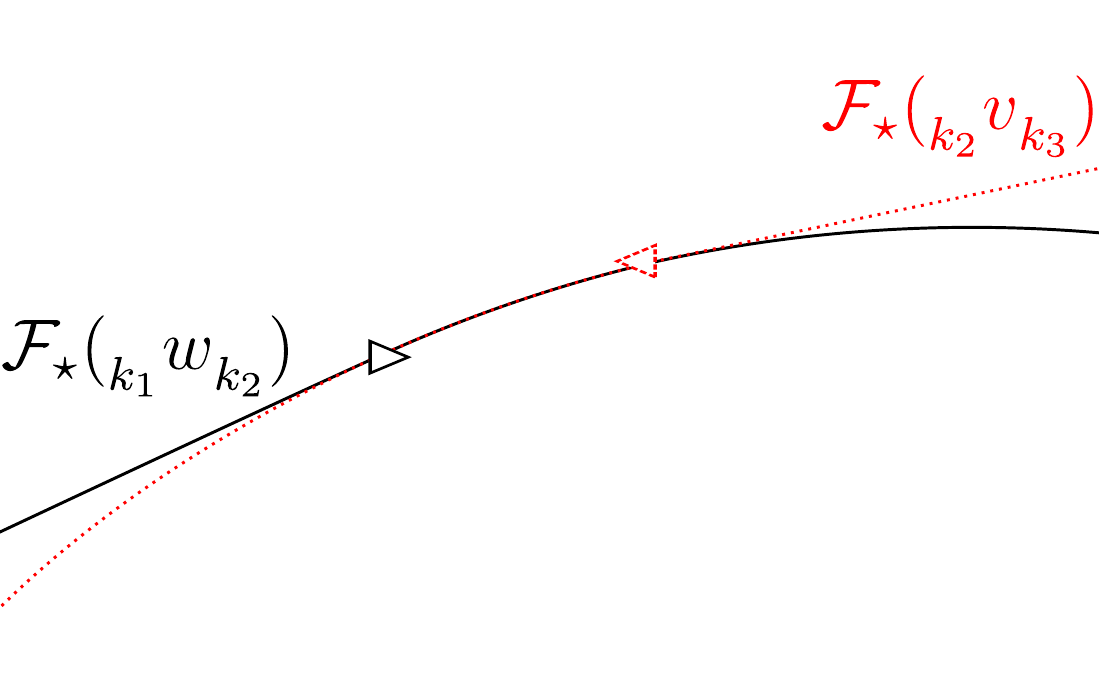}}
	\caption{$ \wordkl[\word][k_1][k_2] \lli \wordkl[\wordsub][k_2][k_3] $.}
	\label{f.isle-sep}%
\end{minipage}
\vspace{-10pt}
\end{figure}

The function $ \rwf $ induces certain geometric information, which we introduce next.
First,
\begin{align}
	\label{e.UDc}
	\UDc(\wordkl)_j := 
	\left\{\begin{array}{l@{}l}
		\up\,, \text{ if } \partial_y \rwf(y)|_{y=x_j^-} \geq \partial_y \rwf(y)|_{y=x_j^+},
	\\
		\down\,, \text{ if } \partial_y \rwf(y)|_{y=x_j^-} < \partial_y \rwf(y)|_{y=x_j^+}.
	\end{array}\right.	
\end{align}
Namely, $ \UDc(\wordkl)_j := \up $ when the graph of $ \rwf $ has a $ \wedge $ kink or is flat at $ y=x_j $ ($ j=2,6 $ in Figure~\ref{f.rwf-2})
and $ \UDc(\wordkl)_j := \down $ when the graph of $ \rwf $ has a $ \vee $ kink at $ y=x_j $ ($ j=4,5 $ in Figure~\ref{f.rwf-2}).
Next, let
\begin{align}
	\label{e.xL.xR}
	{\xl}(\wordkl) := \inf\big\{ y\in\R \, \big| \, \rwf(y)\neq \parab[k](t,y) \big\},
	\
	{\xr}(\wordkl) := \sup\big\{ y\in\R \, \big| \, \rwf(y)\neq \parab[k'](t,y) \big\}
\end{align}
denote where $ \rwf $ merges with $ \parab[k](t) $ and $ \parab[k'](t) $, respectively. Recall $ \word\llo\wordsub $ from Definition~\ref{d.<<}.
\begin{defn}
\label{d.<<i}
Under Convention~\ref{con.wordkl},
write $ \wordkl[\word][k_1][k_2] \lli \wordkl[\wordsub][k_2][k_3] $ if
one of the following equivalent conditions holds:
\begin{enumerate}[leftmargin=20pt, label=(\alph*)]
\item \label{d.<<i.x}
${\xr}(\wordkl[\word][k_1][k_2]) \leq {\xl}(\wordkl[\wordsub][k_2][k_3])$.
\item \label{d.<<i.slope}
$ 
	\partial_y \rwf[\wordkl[\word][k_1][k_2]](y)|_{y={\xr}(\wordkl[\word][k_1][k_2])} 
	\geq
	\partial_y \rwf[\wordkl[\word][k_1][k_2]](y)|_{y={\xl}(\wordkl[\word][k_2][k_3])}.
$
\item \label{d.<<i.hypo}
$ \word \llo \wordsub $  and
$
	( \hyp( \rwf[\wordkl[\word][k_1][k_2]] ) \setminus (\hyp(\parab[k_2](t))^\circ) )
	\cap
	( \hyp( \rwf[\wordkl[\wordsub][k_2][k_3]] ) \setminus (\hyp(\parab[k_2](t))^\circ))
	=
	\emptyset.
$
\end{enumerate}
We write $ \wordkl = \wordkl[\wordsub][k][k''] \cupi \wordkl[\wordsub'][k''] $ if $ \word=\wordsub\cup\wordsub' $ and $  \wordkl[\wordsub][k][k''] \lli  \wordkl[\wordsub'][k''] $.
\end{defn}
\noindent{}%
An illustration of $ \wordkl[\word][k_1][k_2] \lli \wordkl[\wordsub][k_2][k_3] $ is shown in Figure~\ref{f.isle-sep}.
There, the black solid curve is the graph of $ \rwf[\wordkl[\wordsub][k_1][k_2]] $,
with $ \triangleright $ being the location of $ \xr(\wordkl[\word][k_1][k_2]) $;
the red dashed curve is the graph of $ \rwf[\wordkl[\wordsub][k_2][k_3]] $,
with $ \triangleleft $ being the location of $ \xl(\wordkl[\wordsub][k_2][k_3]) $.	

\begin{defn}
\label{d.isle}
We say $ \wordkl $ is a \tdef{$ kk' $-th isle}, or just an \tdef{isle}, if one of the following equivalent conditions holds:
\begin{enumerate}[leftmargin=20pt, label=(\alph*)]
\item \label{d.isle.1}
It is impossible to decompose the word as  $ \wordkl = \wordkl[\wordsub][k][k''] \cupi \wordkl[\wordsub'][k''] $,
under Convention~\ref{con.wordkl}.
\item \label{d.isle.2}
The function $ \rwf[\wordkl]\in\Lip $ is piecewise linear for $ y \in [{\xl}(\wordkl),{\xr}( \wordkl)] $.
\end{enumerate}
We let $ \islesetkl $ denote the set of $ kk' $-th isles.
\end{defn}

\begin{rmk}
When $ \icm=1 $, $ \wordkl[\word][1][1] $ is an isle if and only if $ \hyp(\rwf[\wordkl[\word][1][1]]) \setminus \hyp(\parab[1](t))^\circ $ is a connected set. When this happens the set looks like an `isle', hence the name.
\end{rmk}

\begin{rmk}
Note that $ \wordkl[\emptyset][k][k'] $ is not necessarily an isle. For example, in Figure~\ref{f.rwf-3}, $ \wordkl[\emptyset][1][3]=\wordkl[\emptyset][1][2]\cupi \wordkl[\emptyset][2][3] $.
\end{rmk}

Any $ \wordkl[\word][k_0][k_n] $  can be uniquely decomposed into isles.
Inspect where of $ \rwfsymb=\rwf[\wordkl[\word][k_0][k_n]] $ touches $ \parab[k](t) $, $ k=1,\ldots,\icm $.
By touch we mean the values of the two functions coincide on a nontrivial interval.
When such a touch happens, `cut' the graph of $ \rwfsymb $ there.
These cuts then produce a decomposition.
For example, in Figure~\ref{f.rwf-1}, $ \wordkl[34][2][1] = \wordkl[\emptyset][2][3] \cupi \wordkl[34][3][1] $,
and in Figure~\ref{f.rwf-2}, $ \wordkl[2456][2][3] = \wordkl[2][2][1] \cupi \wordkl[4][1][3] \cupi \wordkl[56][3][3] $.
In general,
$
	\wordkl[\word][k_0][k_n] 
	= 
	\wordkl[\word^{(1)}][k_0][k_1] 
	\cupi 
	\wordkl[\word^{(2)}][k_1][k_2] 
	\cupi \cdots 
	\cupi 
	\wordkl[\word^{(n)}][k_{n-1}][k_{n}],
$
where each $ \wordkl[\word^{(i)}][k_{i-1}][k_{i}] \in \islesetkl[k_{i-1}k_i] $.

The following properties of $ \vecUDc $ are proven from the geometry of $ \rwf $.

\begin{lem}
\label{l.rwf.fulldown}
Consider the statement ``$ \vecUDc(\wordkl) = (\down\ldots\down) $ for some $ \word \in \wordsettkl $''.
\begin{enumerate}[leftmargin=20pt,label=(\alph*)]
\item \label{l.rwf.fulldown.k>=k'}
	When $ k \geq k' $, the statement is false.
\item \label{l.rwf.fulldown.k<k'}
	When $ k < k' $,  the statement holds only if
	$ (x_j,a_j) \in \hyp( \rwf[\wordkl[\emptyset]] )^\circ $ for some $ j\in\alphabetkl $.
\end{enumerate}
\end{lem}
\begin{proof}
Throughout the proof  $ \rwfsymb := \rwf $, $ \rwf[\wordkl[\emptyset]] := \mathcal{F}_\emptyset $, $ n:=|\word| $, $ {\xl}:={\xl}(\wordkl) $, and $ {\xr}:={\xr}(\wordkl) $.

First assume $ \wordkl $ is an isle.
By Definition~\ref{d.isle}\ref{d.isle.2} the graph of $ \rwfsymb(y) $ is piecewise linear between $ y\in[\xl,\xr] $.
List the slopes as $ \slope_0,\slope_1,\ldots,\slope_n $ from left to right.
The condition $ \vecUDc(\wordkl) = (\down\ldots\down) $ forces $ \slope_0<\slope_1<\ldots<\slope_n $ so
\begin{align}
	\label{e.l.rwf.fulldown}
	\slope_0 = \partial_y \parab[k](t,y)|_{y={\xl}} < \partial_y \parab[k'](t,y)|_{y={\xr}} = \slope_{n}.
\end{align}
When $ k \geq k' $, \eqref{e.l.rwf.fulldown} is impossible because
$
	\partial_y \parab[k](t,y)|_{y={\xl}}
	=
	-(\xl-\icx_{k})/t 
	> 
	-(\xr-\icx_{k'})/t 
	=
	\partial_y \parab[k'](t,y)|_{y={\xr}}.
$
When $ k > k' $, consider the (unique) line $ L $ in $ \R^2 $ that intersects both $ \parab[k](t) $ and $ \parab[k'](t) $ tangentially,
and let $ y^\triangleleft $ and $ y^\triangleright $ denote the respective horizontal coordinates of the intersections.
It is straightforward to check that \eqref{e.l.rwf.fulldown} implies $ y^\triangleleft<\xl<\xr<y^\triangleright $.
This property together with $ \slope_0<\slope_1<\ldots<\slope_n $ implies that
the graph of $ \rwfsymb|_{[y^\triangleleft,y^\triangleright]} $ lies strictly below $ L $.
We claim that, between $ y\in [y^\triangleleft,y^\triangleright] $, 
the line $ L $ actually coincides with the graph of $ \mathcal{F}_\emptyset $.
Otherwise there exists $ \parab[k''](t) $ that either forces the graph of $ \mathcal{F}_\emptyset $ to bulge above $ L $ or touches $ L $,
but either scenario contradicts the statement that $ \rwfsymb|_{[y^\triangleleft,y^\triangleright]} $ lies strictly below $ L $.
The desired property now follows by taking any $ j\in\word $.

When $ \wordkl $ is not an isle,
decompose it into a union of $ \lli $-ordered isles and use the results proven for isles.
\end{proof}

\section{Determinantal analysis: the up-down iteration}
\label{s.updown}

We begin by preparing some notation and basic properties.
Recall that $ \sgn(\up) := +1 $ and $ \sgn(\down) := -1 $.
For $ \vecUD \in \{\up,\down\}^\word $, define the parity $ \parity(\vecUD) := \prod_{j\in\word} (-\sgn(\UD_j)) $ for the number of $ \up $'s in $ \vecUD $.
Set
\begin{align*}
	\opuc[k][k']{\word}
	:=	
	\opu[k][k']{\word_{\vecUDc(\wordkl)}}\,,
	\qquad
	\parityc(\wordkl)
	:=
	\parity(\vecUDc(\wordkl)).
\end{align*}
The following lemma will be frequently used throughout this section.
For $ \word\in\wordsetkl $ and a subset $ \alpha\subset\word $ of letters in $ \word $, 
we write $ \vecUDc(\wordkl)|_{\alpha} := (\vecUDc(\wordkl)_j)_{j\in\alpha} $.
\begin{lem}
\label{l.up.criterion}
Consider $ \word\subset\word'\in\wordsetkl $ such that $ \vecUDc(\wordkl)|_{\word'\setminus\word} = (\up\ldots\up) $.
Then for any $ j\in\word\subset\word' $ with $ \vecUDc(\wordkl[\word'])_j = \up $, we have $ \vecUDc(\wordkl)_j = \up $.

In words, deleting `up letters' in $ \word' $ keeps other `up letters' (those $ j $'s) up.
\end{lem}
\begin{proof}
Assume $ \word'\setminus \word $ consists of one letter, denoted $ j_* $. The general case follows by induction on $ |\word'\setminus \word| $.
View $ \rwf[\wordkl[\word']] \mapsto \rwf[\wordkl] $ as a transformation by deleting the letter in $ j_* $ from $ \word' $.
By assumption, $ \rwf[\wordkl[\word']] $ has a $ \wedge $ kink (or is flat) at $ y=x_{j_*} $.
Hence deleting $ j_* $ only decreases the hypograph: $ \hyp(\rwf[\wordkl[\word']] ) \supset \hyp(\rwf[\wordkl] ) $.
Fix any $ j\in\word\subset\word' $ with $ {\UDc(\wordkl[\word'])}_j = \up $.
The graph of $ \rwf[\wordkl[\word']] $ has $ \wedge $ kink (or is flat) at $ y=x_j $.
Since $ \hyp(\rwf[\wordkl[\word']] ) \supset \hyp(\rwf[\wordkl] ) $, it is impossible that this $ \wedge $ kink turns into a $ \vee $ kink upon the transformation.
\end{proof}

The iteration and related properties depend on $ k,k' $. 
To alleviate heavy notation, however, in this section we will often omit the dependence on $ k,k' $.
For clarity we list the omissions here. Some notation will be defined later.
\begin{align*}
	&\wordfull = \wordfullkl,&
	&\opu{\cdots} = \opu[k][k']{\cdots},&
	&\vecUDc(\word) = \vecUDc(\wordkl),&
	&\parityc(\word) = \parityc(\wordkl),
\\
	&\activated(\word) = \activated(\wordkl),&
	&\children(\word) = \children(\wordkl),&
	&\vecUDi(\word) = \vecUDi(\wordkl).
\end{align*}

\subsection{The iteration}
\label{s.updown.iteration.}

As was explained in Section~\ref{s.overview.updown.isle},
in order to control the determinant we would like to \emph{only} have operators of the form
$ 
	\opu[k][k']{\word_{\vecUDc(\wordkl)}}
	=	
	\opuc{\word}	
$.
The goal of the iteration is hence to express the operator
$ \opu{(\wordfull)_{\downs\ldots\downs}} = \opu[k][k']{(\wordfullkl)_{\downs\ldots\downs}} $
in \eqref{e.det} as a linear combination of $ \opu[k][k']{\word_{\vecUDc(\wordkl)}}=  \opuc{\word} $.

Starting with $ \opu{(\wordfull)_{\downs\ldots\downs}} $, we seek to convert $ (\down\ldots\down) $ into $ \vecUDc(\wordfull) $.
To this end, examine which letters $ i\in\wordfull $ have $ \vecUDc(\wordfull)_{j} = \up $.
Call them the \tdef{activated letters} of $ \wordfull $ and let $ \activated(\wordfull) $ denote the set of activated letters of $ \wordfull $.
Apply \eqref{e.flip} for each $ i \in \activated(\wordfull) $.
We have
\begin{align*}
	\opu{(\wordfull)_{\downs\ldots\downs}}
	=
	\parityc({\wordfull}) \
	\opuc{\wordfull}
	+
	\hspace{-10pt}
	\sum_{\word\in\children(\wordfull)}
	\hspace{-10pt}
	\parity(\vecUD_\mathrm{inh}(\word)) \
	\opu{\word_{\vecUD_\mathrm{inh}(\word)}}\,.
\end{align*}
The sum goes over all 
$
	\children(\wordfull) 
	:=
	\{  \word\subsetneq\wordfull : (\wordfull\setminus \word) \subset \activated(\wordfull)  \},
$
and each child inherits a $ \vecUDi(\word) \in \{\up,\down\}^\word $ given by $ \UDi(\word)_j = \UDc(\wordfull)_j $ for all $j\in\word$.
%$ \UDi(\word)_j = \up $ if $ j\in\activated(\wordfull) $ and $ \UDi(\word)_j = \down $ if $ j\notin\activated(\wordfull) $.

The iteration proceeds by taking each child $ \word' $ of $ \wordfull $ as a new parent.
For each such $ \word' $, examine which letters $ i\in\word' $ have 
$ \vecUDc(\word')_i = \up $ and $ \vecUD_\mathrm{inh}(\word')_i = \down $.
These $ i $'s are called the \tdef{activated letters of $ \word' $}: 
\begin{align*}
	\activated(\word')
	:=
	\big\{ i\in\word' : \,\vecUDc(\word')_i = \up \text{ and } \vecUDi(\word')_i = \down \big\}.
\end{align*}
Applying \eqref{e.flip} for each $ i\in\activated(\word') $ in $ \opu{\word'_{\vecUDi(\word')}} $ gives
\begin{align}
	\label{e.iteration}
	\parity(\vecUDi(\word')) \
	\opu{\word'_{\vecUDi(\word')}}
	=
	\parityc(\word') \
	\opuc{\word'}
	+
	\hspace{-10pt}
	\sum_{\word\in\children(\word')}
	\hspace{-10pt}
	\parity(\vecUDi(\word)) \
	\opu{\word_{\vecUDi(\word)}},
\end{align}
where the sum goes over 
$
	\word \in
	\children(\word') 
	:=
	\{  \word\subsetneq\word' : (\word'\setminus \word) \subset \activated(\word')  \}
$
and $ \vecUDi(\word) \in \{\up,\down\}^{\word} $ is given by $ \UDi(\word)_j = \UDc(\word')_j $ for all $j\in\word$.
%$ \UDi(\word)_j = \up $ if $ j\in\activated(\word') $ and $ \UDi(\word)_j = \down $ if $ j\notin\activated(\word') $.

The following lemma follows immediately from Lemma~\ref{l.up.criterion}.
\begin{lem}
\label{l.no.reactivate}
For any finite sequence $ (\word,\word',\word'',\ldots) $ in $ \wordsetkl $ such that consecutive words are parent-child, namely
$ \word \in \children(\word') $, $ \word' \in \children(\word'') $, \ldots,
we have $ (\activated(\word) \cap \activated(\word') \cap \activated(\word'') \cap \ldots) = \emptyset $.
\end{lem}
\begin{proof}
Take any word $\word_1$, take any $i\in\activated(\word_1)$, and consider all the descendants of $\word_1$, namely the children of $\word_1$, the children of children $\word_1$, and so on.
By the definition of activated letters, $\vecUDc(\word_1)_i=\up$.
Any child is obtained from its parent by deleting some letters that are ``up'' for the parent.
The deletion procedure can change the up-down type, but Lemma~\ref{l.up.criterion} ensures that any up letter for the parent remains up for the child, if not deleted.
Therefore, for all descendants that contains the letter $i$, the letter remains up for those descendants, so $\vecUDi(\word_\mathrm{d})_i=\up$ for all descendant $\word_\mathrm{d}$ such that $i\in\word_\mathrm{d}$.
By the definition of activated letters, $i\notin\activated(\word_\mathrm{d})$.
This proves $\activated(\word_1)\cap\activated(\word_d)=\emptyset$, for any descendant $\word_\mathrm{d}$ of any word $\word_1$.
\end{proof}

The iteration proceeds as described and terminates when $ \activated(\word) = \emptyset $, or equivalently $ \vecUDi(\word)=\vecUDc(\word) $. 
Lemma~\ref{l.no.reactivate} ensures that in each step of the iteration, the children have length strictly less than their parent.
Hence the iteration must terminate in finitely many steps.

\begin{ex}
\label{ex.updown.1}
Take the geometry depicted in Figure~\ref{f.iter.ex1}, with $ \icm=3 $, $ k=2 $, $ k'=1 $, and $ \wordfullkl=\wordfullkl[21]=1234 $.
The iteration starts with $ \opu[2][1]{1234_{\downs\downs\downs\downs}}=\opu{1234_{\downs\downs\downs\downs}} $.
Figure~\ref{f.iter.ex1} gives $ \vecUDc( \wordkl[1234][2][1])=\vecUDc(1234)=(\up\down\down\up) $,
so we need to activate the letters in $ \activated(1234) = \{1,4\} $.
Doing so gives 
\begin{align*}
	 \opu{1234_{\downs\downs\downs\downs}}
	 =
	 (-1)^2\opuc{1234}
	 +
	 (-1)\opu{234_{\downs\downs\ups}}
	 +
	 (-1)\opu{123_{\ups\downs\downs}}
	 +
	 \opu{23_{\downs\downs}}.
\end{align*}
From Figure~\ref{f.iter.ex1}, $ \vecUDc(234)=(\up\down\up) $, so $ \activated(234)=\{2\} $;
$ \vecUDc(123)=(\up\down\up) $, so $ \activated(234)=\{3\} $;
$ \vecUDc(23)=(\up\up) $, so $ \activated(23)=\{2,3\} $.
Accordingly,
\begin{align*}
	(-1)\opu{234_{\downs\downs\ups}}
	&=
	(-1)^2\opuc{234}
	+
	(-1)\opu{34_{\downs\ups}},
\\
	(-1)\opu{123_{\ups\downs\downs}}
	&=
	(-1)^2\opuc{123}
	+
	(-1)\opu{12_{\ups\downs}},
\\
	\opu{23_{\downs\downs}}
	&=
	(-1)^2\opuc{23} + (-1)\opu{2_{\ups}} + (-1)\opu{3_{\ups}} + \opu{\emptyset}.
\end{align*}
All $ \opu{\word_{\vecUDi(\word)}} $ on the right sides have $ \vecUDi(\word) = \vecUDc(\word) $, so the iteration is completed.
Altogether
\begin{align*}
	 \opu{1234_{\downs\downs\downs\downs}}
	 =
	 \opuc{1234} -\opuc{234} -\opuc{123} + \opuc{23}
	 -\opuc{34} -\opuc{12} - \opuc{2} - \opuc{3} + \opu{\emptyset}.
\end{align*}
\end{ex}

\begin{figure}[h]
\begin{minipage}[b]{.58\linewidth}
\begin{minipage}[t]{.495\linewidth}
	\frame{\includegraphics[width=\linewidth]{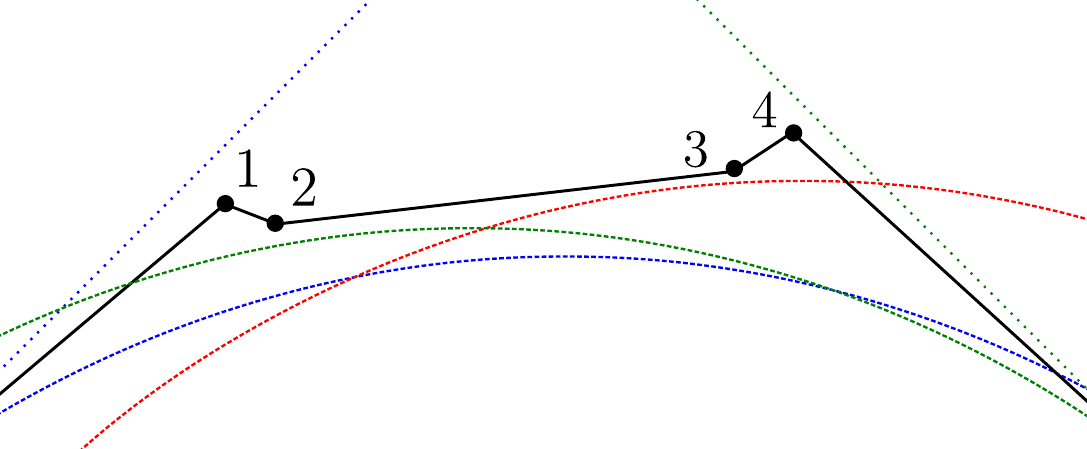}}
\end{minipage}
\hfill
\begin{minipage}[t]{.495\linewidth}
	\frame{\includegraphics[width=\linewidth]{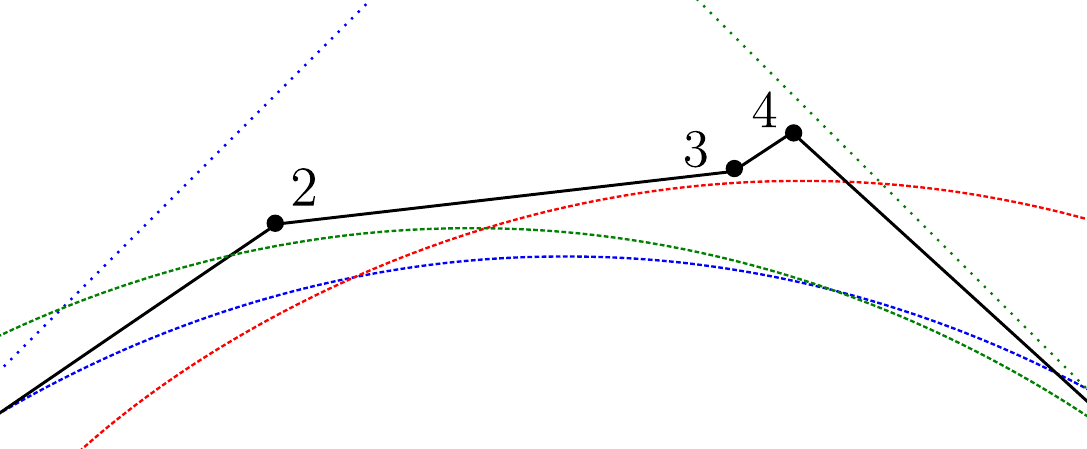}}
\end{minipage}
\\
~\vspace{-10pt}
\\
\begin{minipage}[t]{.495\linewidth}
	\frame{\includegraphics[width=\linewidth]{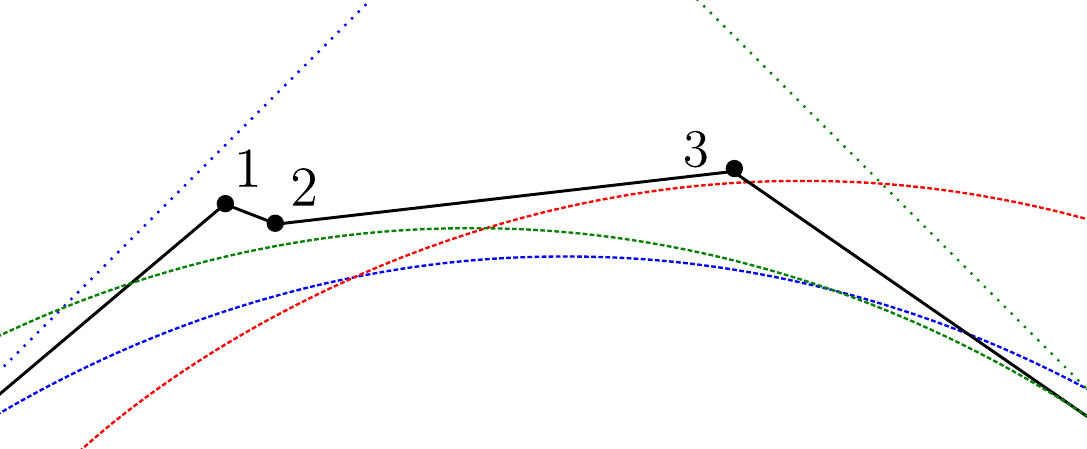}}
\end{minipage}
\hfill
\begin{minipage}[t]{.495\linewidth}
	\frame{\includegraphics[width=\linewidth]{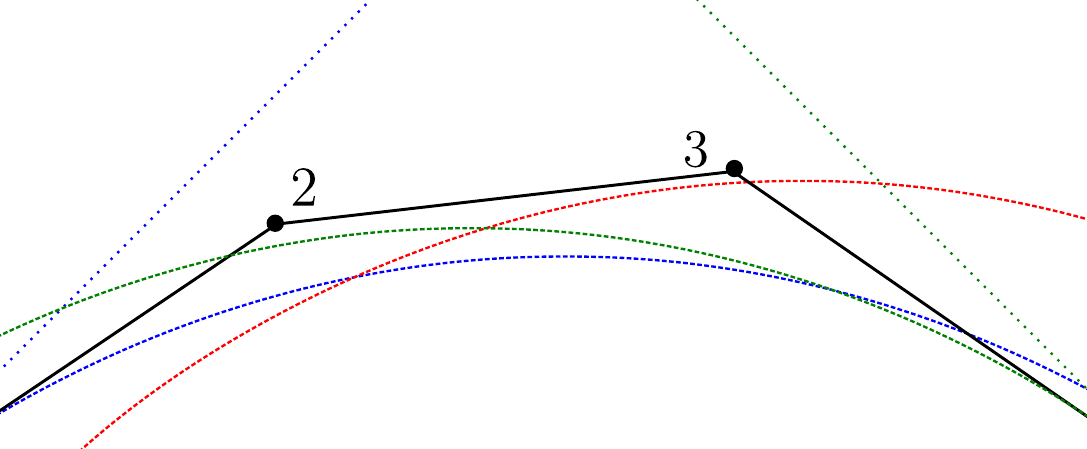}}
\end{minipage}
\\
~\vspace{-10pt}
\\
\begin{minipage}[t]{.495\linewidth}
	\frame{\includegraphics[width=\linewidth]{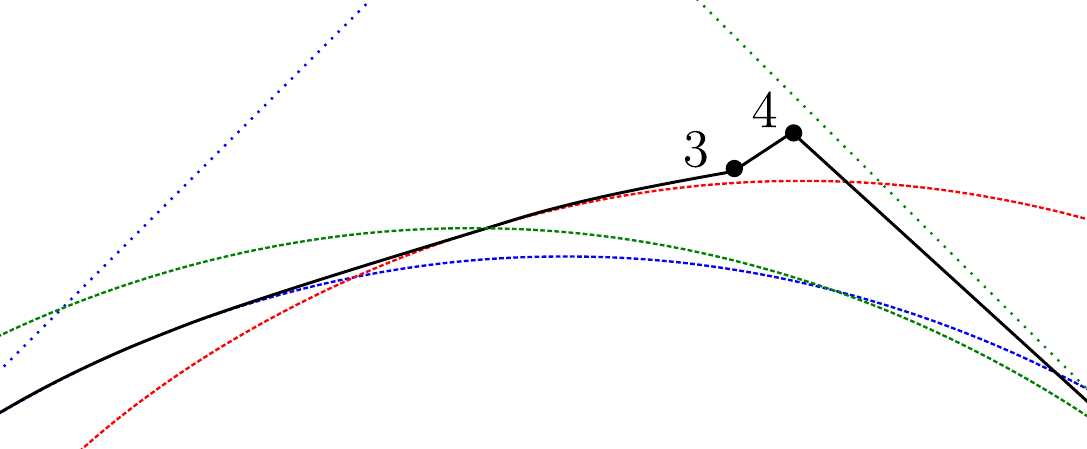}}
\end{minipage}
\hfill
\begin{minipage}[t]{.495\linewidth}
	\frame{\includegraphics[width=\linewidth]{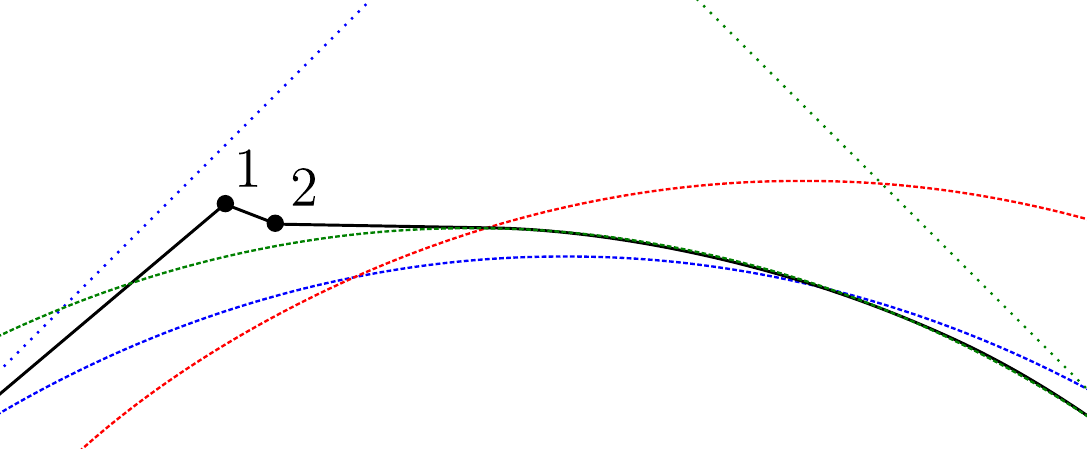}}
\end{minipage}
\\
~\vspace{-10pt}
\\
\begin{minipage}[t]{.495\linewidth}
	\frame{\includegraphics[width=\linewidth]{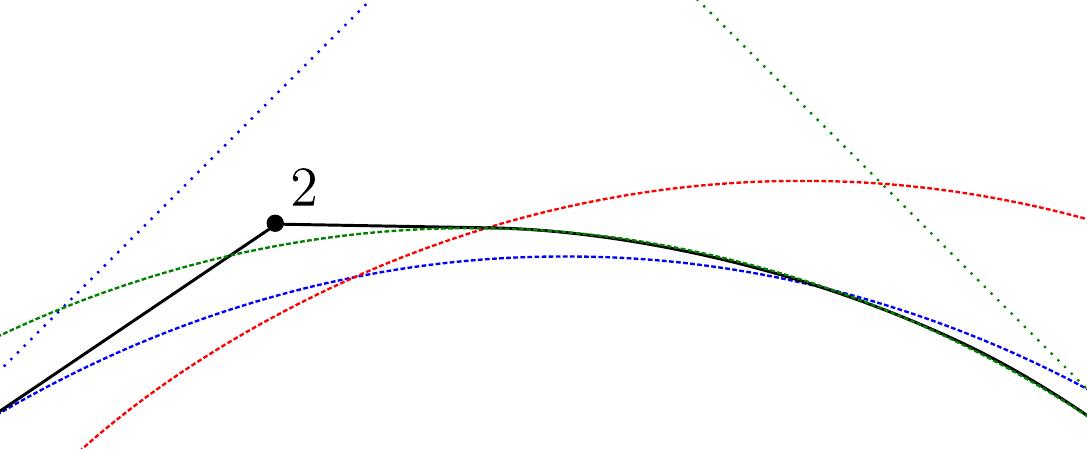}}
\end{minipage}
\hfill
\begin{minipage}[t]{.495\linewidth}
	\frame{\includegraphics[width=\linewidth]{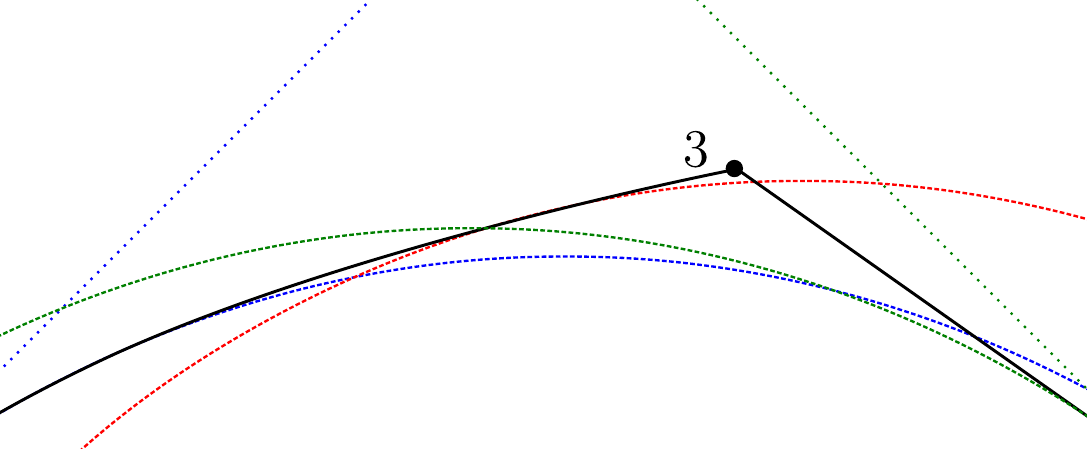}}
\end{minipage}
\caption{A example with $ \icm=3 $, $ k=2 $, $ k'=1 $.}
\label{f.iter.ex1}
\end{minipage}
\hfill
\begin{minipage}[b]{.39\linewidth}
\includegraphics[width=\linewidth]{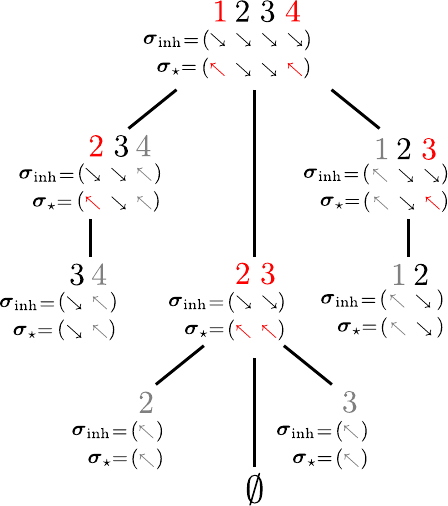}
\caption{The tree of Example~\ref{ex.updown.1}.}
\label{f.iter.ex1.tree}
\end{minipage}
\end{figure}

\subsection{The tree}
\label{s.updown.tree}
It is instructive to describe the process of the iteration.
We say $ \wordsub $ is a \tdef{descendant} of $ \word $
if there exists a tower of words $ \wordsub = \word^{(1)} \subset \word^{(2)} \ldots \subset \word^{(n)} = \word $
such that $ \word^{(i)} \in \children(\word^{(i+1)}) $ for all $ i $.
The iteration consists of steps given by \eqref{e.iteration}.
In each step, a word $ \word' $ activates a number of letters,
and each $ \word\in\children(\word') $ decides to keep some of those activated letters and delete at least one activated letter.
By Lemma~\ref{l.no.reactivate}, the kept letters will be kept in all descendants of $ \word $, and the deleted letters will be permanently deleted from all descendants of $ \word $.
Furthermore, since each step of the iteration~\eqref{e.iteration} only converts $ \down $'s to $ \up $'s,
any kept letter $ j $ of $ \word $ will have $ \vecUDi(\word_\mathrm{desc})_j=\up $ and $ \vecUDc(\word_\mathrm{desc})_j=\up $ for any descendant $ \word_\mathrm{desc} $ of $ \word $.

Let us encode the iteration by a graph.
The vertices are words that are parents and/or children in the iteration, with the full word $ \wordfull $ being the root, and the bonds connect parent-child pairs.
\emph{The graph is a tree}.
To see why, consider any $ \word' $ in the iteration and its children.
The description in the last paragraph shows that
any descendant $ \word_\text{desc} $ of $ \word' $ can just go through the letters $ j\in \activated(\word') $, 
and see which has $ \UDc(\word_\text{desc})_j = \up $ and which has been deleted from $ \word_\text{desc} $.
The result uniquely reconstructs which child $ \word $ of $ \word' $ the descendant $ \word_\text{desc} $ belongs to.
We let $ \treekl $ denote this rooted tree. %and let $ \treeekl := \treekl \setminus\{\emptyset\} $.
The tree for Example~\ref{ex.updown.1} is shown in Figure~\ref{f.iter.ex1.tree}.
Activated letters are colored red.
Letters that were previously activated and kept are colored gray.

\begin{lem}
\label{l.tree.empty}
\begin{enumerate}[leftmargin=20pt, label=(\alph*)]
\item[]
\item \label{l.tree.empty.k>=k'} When $ k \geq k' $, $ \emptyset\in\treekl $.
\item \label{l.tree.empty.k<k'} When $ k < k' $, $ \emptyset\in\treekl $ if and only if $ (x_i,a_i) \notin \hyp( \rwf[\wordkl[\emptyset]] )^\circ $ for all $ i=1,\ldots,m $.
\end{enumerate}
\end{lem}
\begin{proof}
In this proof we restore the independence on $ k,k' $ for clarity.
We begin by preparing an equivalent statement of $ \emptyset \in \treekl $.
Refer to the description in the first paragraph of this subsection.
If a letter $ j\in\alphabetkl $ has never been activated throughout the iteration, 
the letter must belong to all $ \word\in\treekl $ and $ \vecUDc(\wordkl)_{j} = \down $.
Gathering such $ j $'s forms the \tdef{terminal word} $ \word^{\text{term},kk'} \in \treekl $ such that $ \word^{\text{term},kk'} \subset \word $ for all $ \word\in\tree $.
Having $ \emptyset \in \treekl $ is equivalent to $ \word^{\text{term},kk'}=\emptyset $.
This is so because $ \word^{\text{term},kk'} $ is the (necessarily unique) word in $ \treekl $ that is a subword of any word in $ \treekl $, and because $ \emptyset $ is a subword of any word.

\ref{l.tree.empty.k>=k'}\ 
By construction $ \UDc(\wordkl[ \word^{\text{term},kk'}] )_j = \down $, for all $ j\in \word^{\text{term},kk'} $.
Hence Lemma~\ref{l.rwf.fulldown}\ref{l.rwf.fulldown.k>=k'} 
implies $ \word^{\text{term},kk'} = \emptyset $.

\ref{l.tree.empty.k<k'}\
Assume $ \word^{\text{term},kk'}\neq\emptyset $.
Lemma~\ref{l.rwf.fulldown}\ref{l.rwf.fulldown.k<k'} implies the existence of $ (x_i,a_i) \in \hyp( \rwf[\wordkl[\emptyset]] )^\circ $.
Conversely, assume the existence of such points.
Geometric considerations show that, among all such points, there exists one $ (x_{i_*},a_{i_*}) $ such that $ \UDc(\wordkl[\word^{\text{term},kk'}])_{i_*} = \down $.
Having $ \UDc(\wordkl[\word^{\text{term},kk'}])_{i_*} = \down $ implies $ i_*\in\alphabetkl $ can never be activated throughout the up-down iteration.
Hence $ \word^{\text{term},kk'}\neq\emptyset $.
\end{proof}

Sum \eqref{e.iteration} over all $ \word'\in\treekl $, insert the result into the determinantal formula \eqref{e.det}, use Lemma~\ref{l.tree.empty}.
We have
\begin{align}
	\label{e.det.iter}
	\punder
	=
	\det\Big( \Id + \Big( 
		\sum_{\wordkl \in\,\treekl} 
		\parityc(\wordkl) \ \inddown_{\ick}\opuc{\word} \inddown_{\ick'}
	\Big)_{k,k'=1}^{\icm}\Big),
\end{align}
where, with a slight abuse of notation, $ \wordkl \in \treekl $ means $ \word\in\treekl $ and $ \word\neq\emptyset $ when $ k \geq k' $.

\subsection{The increasing hypograph condition and the isle structure}
\label{s.updown.IHC}
Here we state some useful properties of $ \treekl $.
The proof is deferred to Sections~\ref{s.updown.pf1}--\ref{s.updown.pf2}.
For words $ \word\subset\word' $,
we call $ \delta = \word'\setminus\word $ the \tdef{increment} and write $ \word \ \increment{\delta} \word' $.
For $ \word\in\wordsetkl $,
consider a tower of words 
$
	\word^{(0)} 
	\increment{\delta_0}
	\word^{(1)} 
	\increment{\delta_1}
	\word^{(2)} 
	\increment{\delta_2}
	\cdots
	\word^{(n)}
$
that ascends from $ \word^{(0)} $ to $ \word^{(n)} $.
We say the tower \tdef{increases by one letter at a time} if $ |\delta_0|=|\delta_1| = \ldots = 1 $.
We say the tower satisfies the \tdef{$ kk' $-th increasing hypograph condition}, denoted \tdef{$ \IHC $}, if 
\begin{enumerate}[leftmargin=20pt, label=(\roman*)]
\item the tower increases by one letter at a time, 
\item each word in this tower belongs to $ \wordsetkl $, and 
\item $ \hyp( \rwf[\wordkl[\word^{(0)}]] ) \subset \hyp( \rwf[\wordkl[\word^{(1)}]] ) \subset \hyp( \rwf[\wordkl[\word^{(2)}]] ) \subset \ldots  $.
\end{enumerate}
Recall $ \cupi $ from Definition~\ref{d.<<i}.

\begin{prop}
\label{p.tree.prop}
\begin{enumerate}[leftmargin=20pt, label=(\alph*)]
\item[]
\item \label{p.IHC}
We have $ \word \in \treekl $ if and only if there exists a tower that ascends from $ \word $ to $ \wordfullkl $ and satisfies $ \IHC $.
\item
\label{p.tree.isle}
For any $ \wordkl[\wordsub][k_0][k_n] = \wordkl[\wordsub^{(1)}][k_0][k_1] \cupi\cdots\cupi \wordkl[\wordsub^{(n)}][k_{n-1}][k_n] $ that respects Convention~\ref{con.wordkl},\\
$ \wordsub\in\treekl[k_0k_n] $ 
if and only if $ (\wordsub^{(1)},\ldots,\wordsub^{(n)})\in\treekl[k_{0}k_{1}]\times\cdots\times \treekl[k_{n-1}k_{n}] $.
\end{enumerate}
\end{prop}
\begin{rmk}
Note that $ \word\in\treekl $ does \emph{not} require $ \word\neq\emptyset $ when $ k \geq k' $,
which \emph{differs} from $ \wordkl \in \treekl $.
\end{rmk}

\subsection{Proof of Proposition~\ref{p.tree.prop}\ref{p.IHC}}
\label{s.updown.pf1}
Throughout the proof we will label towers top-down instead of bottom-up
and will continue to omit most $ k,k' $ dependence as declared at the beginning of this section.
We call 
$
	\word \subset \cdots \subset \word^{[2]} \subset \word^{[1]} \subset \word^{[0]} = \wordfullkl = \wordfull
$
a \tdef{tree tower} if consecutive levels $ \word^{[i]}\subset\word^{[i-1]} $ in this tower are parent-child in $ \treekl=\tree $, but \emph{not} necessarily $ |\word^{[i-1]}\setminus\word^{[i-1]}| = 1 $.
Indeed, $ \word\in\tree $ is equivalent to having a tree tower that ascends from $ \word $ to $ \wordfull $.

Referring to the definitions of $ \rwfsymb $ and $ \vecUDc $ in Section~\ref{s.geo.rwf.updown.isle},
we see that $ \IHC $ is equivalent to 
\begin{align}
	\tag{IHC'}
	\label{e.IHC'}
	(\cdots \subset \word^{(i+1)} \Increment{j_{i+1}} \word^{(i)} \subset \cdots),
	\qquad
 	\UDc(\word^{(i)})_{ j_{i+1} } = \up \ \text{for all } i.
\end{align}

To prove the `only if part', take a tree tower
%\begin{align*}
%	\label{e.family.tower}
$
	(\word \subset \cdots \subset \word^{[2]} \subset \word^{[1]} \subset \word^{[0]} = \wordfull).
$
%\end{align*}
We seek to modify it so that the resulting tower satisfies \eqref{e.IHC'}.
Examine each consecutive levels $ \word^{[i]} \subset \word^{[i-1]} $ in this tree tower.
When $ |\word^{[i-1]}\setminus\word^{[i]} | > 1 $, insert some children of $ \word^{[i-1]} $ in between
\begin{align}
	\label{p.criteria.insert}
%$
	\word^{[i]} \subset \word' \subset \word'' \subset \ldots \subset \word^{[i-1]}.
%$
\end{align}
The choices of the children $ \word',\word,'',\ldots $ are arbitrary as long as they form a tower that increases by one letter at a time.
That $ \word^{[i-1]} \subset \word^{[i]} $ being parent-child implies $ \word^{[i]}\setminus\word^{[i-1]} \subset\activated(\word^{[i-1]}) $.
Hence any increment in \eqref{p.criteria.insert} belongs to $ \activated(\word^{[i-1]}) $.
Since $ \word^{[i]}, \word',\word'',\ldots $ are children of $ \word^{[i-1]} $, the desired condition~\eqref{e.IHC'} holds.

To prove the `if part', fix a tower
%\begin{align*}
$
	(\word
	\
	\cdots
	\Increment{j_3}
	\word^{(2)} 
	\Increment{j_2}
	\word^{(1)} 
	\Increment{j_1}
	\word^{(0)} = \wordfull)
$
%\end{align*}
that satisfies \eqref{e.IHC'}.
Note that such a tower may fail to be a tree tower in an essential way.
For example, in Figure~\ref{f.iter.ex1.tree}, the tower $ (2 \Increment{1} 12 \Increment{3} 123 \Increment{4} 1234) $ satisfies \eqref{e.IHC'},
but the words $ 2 $ and $ 12 $ belong to different branches in $ \tree $: the middle and right branches in Figure~\ref{f.iter.ex1.tree}.
Still, our strategy is to modify this tower, inductively top-down, 
so that the end result is a tree tower that ascends from $ \word $ to $ \wordfull $.
To begin, the condition $ \UDc(\word^{(0)})_{j_1}=\up $ from \eqref{e.IHC'} together with $ \word^{(0)}=\wordfull $ guarantees that $ \word^{(1)} $ is a child of $ \wordfull $.
No modification is needed at this point.
Next, the condition $ \UDc(\word^{(1)})_{j_2}=\up $ from \eqref{e.IHC'} implies that
$ j_2 $ is from $ \activated(\word^{(1)}) $ or $ \activated(\word^{(0)}) $.
In the former case $ \word^{(2)} $ is a child of $ \word^{(1)} $, so no modification is needed.
In the latter case, $ \word^{(2)} $ is a child of $ \wordfull $, and we modify the tower by deleting $ \word^{(1)} $:
\begin{align}
	\label{e.p.criteria.n=2}
	\cdots \
	\word^{(3)}
	\Increment{j_3}
	\big(
		\word^{(2)} 
		\subset
		\word^{(1)} 
		\subset 
		\word^{(0)}
	\big)
\quad
\longmapsto
\quad
	\cdots \
	\word^{(3)}
	\Increment{j_3}
	\big(
		\word^{(2)} 
		\subset 
		\word^{(0)}
	\big).
\end{align}

Proceed inductively. 
Assume, for some $ n\in \Z_{> 0} $, we have modified the tower from the top to just before $ \word^{(n+1)} $:
\begin{align*}
	\cdots \
	\word^{(n+1)}
	\Increment{j_{n+1}}
	\word^{[\ell_n]} 
	\subset
	\cdots 
	\subset 
	\word^{[0]} = \wordfull
\end{align*}
so that
$
	(\word^{[\ell_n]} 
	\subset
	\cdots 
	\subset 
	\word^{[0]})
$ 
is itself a tree tower.
As seen in \eqref{e.p.criteria.n=2}, the modifications may change the length of the tower, so we allow $ \ell_n\neq n $.
To further the induction,
note that the condition $ \UDc(\word^{(n)})_{j_{n+1}} = \up $ from \eqref{e.IHC'} implies that 
$ j_{n} $ is from one of $ \activated(\word^{[0]}) $, $ \activated(\word^{[1]}) $, \ldots, or $ \activated(\word^{[\ell_n]}) $.
If $ j_{n+1} \in \activated(\word^{[\ell_n]}) $, then $ \word^{(n+1)} $ is a child of $ \word^{[\ell_n]} $, and the induction proceeds by setting $ \ell_{n+1}:=n+1 $ and $ \word^{[\ell_{n+1}]} := \word^{(n+1)} $.
Otherwise $ j_{n+1} \in \activated(\word^{[i_0]}) $, for some $ i_0\in\{1,\ldots,\ell_n-1\} $.
In this case we seek to modify the segment within the following parentheses.
\begin{align}
	\label{e.replace.tree}
	\cdots \
	\big( \word^{(n+1)} \Increment{j_{n+1}} \word^{[\ell_n]} \subset \cdots \subset \word^{[i_0]}  \big)
	\subset
	\cdots
	\subset
	\word^{[0]} = \wordfull.
\end{align}
At this point we invoke Lemma~\ref{l.tree.tower}, which is stated and proven after the current proof.
Apply Lemma~\ref{l.tree.tower} with $ \tower \mapsto (\word^{[\ell_n]} \subset \ldots \subset \word^{[i_0]} ) $ and $ j_* \mapsto j_{n+1} $.
The result provides a tree tower $ \tower' $ that ascends from $ \word^{(n+1)}=\word^{[\ell_n]}\setminus\{j_{n+1}\} $ to $ \word^{[i_0]} $.
Replacing the segment in the parentheses in \eqref{e.replace.tree} with $ \tower' $ completes the induction.
The proof is completed contingent on proving Lemma~\ref{l.tree.tower}.

Before proving Lemma~\ref{l.tree.tower}, let us reformulate Lemma~\ref{l.up.criterion} in the following form.
\begin{customlem}{\ref*{l.up.criterion}'}
\label{l.up.criterion.}
Given words $ \word,\word'\in\wordsetkl $ and a set of letters $ \alpha $, if
\begin{center}
\begin{enumerate*}[label=(\alph*)]
\item \label{l.up.criterion.0} $ \word\subset\word' $, \quad
\quad
\item \label{l.up.criterion.1} $ \vecUDc(\word')|_{\alpha\cap\word'} = (\up\ldots\up) $, \quad and \quad \
\item \label{l.up.criterion.2} $ \vecUDc(\word')|_{\word'\setminus\word} = (\up\ldots\up) $, 
\end{enumerate*}
\end{center}
then $ \vecUDc(\word)|_{\alpha\cap\word} = (\up\ldots\up) $.
\end{customlem}

\begin{lem}[Finishing the proof of Proposition~\ref{p.tree.prop}\ref{p.IHC}]
\label{l.tree.tower}
For any tree tower
$
	\tower =
	(\word^{[n]}
	{\subset}
	\ldots
%	{\subset}
%	\word^{[1]} 
	{\subset} 
	\word^{[0]})
$
and $ j_*\in\activated(\word^{[0]}) $,
there exists a tree tower that ascends from $ (\word^{[n]} \setminus \set{j_*}) $ to $ \word^{[0]} $.
\end{lem}
\begin{proof}
The strategy is to use induction.
For each $ \ell=1,\ldots,n $, we will construct a tower
\begin{align*}
	\tower^{[\ell]}
	=
	\big(
	(\word^{[n]} \setminus \set{j_*}) = \wordsub^{[n]}
	\
	\increment{\epsilon_n}
	\ldots
%	\tensor*[ _{\epsilon_{\ell+2}} ]{\subset}{}
	\wordsub^{[\ell+1]}
	\
	\increment{\epsilon_{\ell+1}}
	\wordsub^{[\ell]}
	\
	\increment{\epsilon_{\ell}}
	\cdots
	\
%	\tensor*[ _{\epsilon_{2}} ]{\subset}{}
	\wordsub^{[1]}
	\
	\increment{\epsilon_{1}}
	\wordsub^{[0]}
	=
	\word^{[0]}
	\big).
\end{align*}
Throughout the induction,
the top level $ \word^{[0]} $ and bottom level $ (\word^{[n]}\setminus\{j_*\}) $ will remain unchanged.
On the other hand, 
the intermediate levels $ \wordsub^{[1]},\ldots,\wordsub^{[n-1]} $ and the increments $ \epsilon_1,\ldots,\epsilon_n $ will change as $ \ell $ varies.
We however omit the dependence on $ \ell $ in the notation $ \{\wordsub^{[i]}\}_{i=1}^{n-1} $ and $ \{\epsilon_i\}_{i=1}^n $ for better readability.
The tower $ \tower^{[\ell]} $ will be constructed to satisfy the induction hypotheses:
\begin{enumerate}[leftmargin=20pt,label=(\roman*-$\ell$) ] %
\item \label{l.tree.induction.1} $ \epsilon_i \subset \activated(\wordsub^{[i-1]}) $, for $ i=1,\ldots,\ell $.
\item \label{l.tree.induction.2} $ \epsilon_{i'} \cap \activated(\wordsub^{[i-1]}) = \emptyset $, for $ i=1,\ldots,\ell $ and $ i'=\ell+1,\ldots,n $.
\item \label{l.tree.induction.3} $ \vecUDc(\wordsub^{[i]})|_{\epsilon_{i+1}} = (\up\ldots\up) $, for $ i=\ell,\ldots,n $.
\end{enumerate} 
The hypothesis~\ref{l.tree.induction.1} implies that the top $ \ell $ levels in $ \tower^{[\ell]} $ form a tree tower, except for possible repetitions, namely $ \epsilon_{i} = 0 $ so that $ \wordsub^{[i]} = \wordsub^{[i-1]} $.
Such repetitions can be removed by deleting words from the tower.
Once completed, the induction produces the desired tree tower $ \tower^{[n]} $.

We now begin the induction, starting with $ \ell=1 $.
Construct $ \tower^{[1]} $ by setting $ \wordsub^{[i]}:=\word^{[i]}\setminus\{j_*\} $ for $ i=1,\ldots,n $. Note that this construction gives
\begin{align}
	\label{e.towerl=1}
	\epsilon_{1}:=
	\wordsub^{[0]} \setminus \wordsub^{[1]} = (\word^{[0]} \setminus \word^{[1]})\cup\{j_*\},
	\qquad	
	\epsilon_{i}:=
	\wordsub^{[i-1]} \setminus \wordsub^{[i]} = \word^{[i-1]} \setminus \word^{[i]},
	\
	i=2,\ldots,n-1.
\end{align}
We next check the hypotheses for $ \ell=1 $.
\begin{enumerate}[leftmargin=20pt]
\item [(i-1)] The fact that $ \tower $ is a tree tower implies $ (\word^{[0]} \setminus \word^{[1]}) \subset \activated(\word^{[0]}) $.
	This fact together with \eqref{e.towerl=1} and the assumption $ j_*\in\activated(\word^{[0]}) $ 
	verifies the desired hypothesis $ \epsilon_{1} \subset \activated(\word^{[0]}) $.
\item [(ii-1)] 
	By \eqref{e.towerl=1} and the fact that $ \tower $ is a tree tower, 
	we have $ \epsilon_{i'} \subset \activated(\word^{[i'-1]}) $, for $ i'=2,\ldots,n $.
	Further, Lemma~\ref{l.no.reactivate} implies that $ \activated(\word^{[i'-1]}) \cap \activated(\word^{[0]})  = \emptyset $.
	Hence the hypothesis (ii-1) holds.
\item [(iii-1)]
Fix $ i\in\{1,\ldots,n\} $ and apply Lemma~\ref{l.up.criterion.} with $ (\word,\word',\alpha) \mapsto(\wordsub^{[i]},\word^{[i]},\epsilon_{i+1}) $. 
The result of this application verifies the hypothesis (iii-1).
For Lemma~\ref{l.up.criterion.} to apply let us check the required conditions:
	\begin{enumerate}
	\item This follows by construction: $ \wordsub^{[i]}:=\word^{[i]}\setminus\{j_*\} $.
	\item
	Use $ \vecUDc(\word^{[i]})|_{\epsilon_{i+1}} = \vecUDc(\word^{[i]})|_{\word^{[i]}\setminus \word^{[i+1]}} $
	and $ \word^{[i+1]} \in \children(\word^{[i]}) $.
	\item
	Referring to description in the first paragraph in Section~\ref{s.updown.tree}, we see that $ \word^{[i]} $ being a descendant of $ \word^{[0]} $ implies $ \vecUDc(\word^{[i]})|_{ \activated(\word^{[0]})\cap\word^{[i]}} = (\up\ldots\up) $.
	Combining this fact with $ j_*\in\activated(\word^{[0]}) $ verifies the condition~\ref{l.up.criterion.2}.
	\end{enumerate}
\end{enumerate}

Assume $ \tower^{[\ell]} $ has been constructed and satisfies the induction hypotheses.
Let $ \epsilon_{i} := \wordsub^{[i-1]} \setminus \wordsub^{[i]} $ denote the increments of $ \tower^{[\ell]} $.
Define a new set of increments $ \til\epsilon_{\ell+1}, \ldots, \til\epsilon_{n} $, as
\begin{align}
	\label{e.tree.replace}
	\til\epsilon_{\ell+1} 
	&:= 
	\epsilon_{\ell+1} \cup \bigcup\nolimits_{i>\ell+1} \big( \epsilon_{i}\cap \activated(\wordsub^{[\ell]}) \big),
\\
	\label{e.tree.replace.}
	\til\epsilon_{i} 
	&:=
	\epsilon_{i} \setminus \activated(\wordsub^{[\ell]}),
	\quad
	\text{for } i = \ell+2,\ldots,n.
\end{align}
Namely, we examine the letters in $ \epsilon_{\ell+2},\epsilon_{\ell+3},\ldots,\epsilon_n $ 
to see which belong to $ \activated(\wordsub^{[\ell]}) $, and transport all such letters into $ \epsilon_{\ell+1} $.
Now use these new increments to construct $ \tower^{[\ell+1]} $:
\begin{align}
	\label{e.tree.replace..}
	\til\wordsub^{[i]} &:= \wordsub^{[\ell]} \setminus ( \til\epsilon_{\ell+1} \cup \cdots \cup \til\epsilon_i ), 
	\quad 
	i=\ell+1,\ldots,n,
\\
	\notag
	\tower^{[\ell+1]}
	&:=
	\big(
	\wordsub^{[n]}
	=
	(\word^{[n]} \setminus \set{j_*})
	\
	\Increment{\til\epsilon_n}
	\ldots
	\
	\Increment{\til\epsilon_{\ell+2}}
	{\til\wordsub}^{[\ell+1]}
	\
	\Increment{\til\epsilon_{\ell+1}}
	{\wordsub}^{[\ell]}
	\
	\Increment{\epsilon_{\ell}}
	\cdots
	\
	\wordsub^{[1]}
	\
	\Increment{\epsilon_{1}}
	\
	\word^{[0]}
	\big).
\end{align}
We next verify the hypotheses.
\begin{enumerate}[leftmargin=20pt]
\item [(i-($\ell$+1))] Since the top $ \ell $ levels were not altered, it suffices to check the hypothesis for $ i=\ell+1 $, namely
\begin{align}
	\label{e.inductioni.goal}
	\til\epsilon_{\ell+1} \subset \activated(\wordsub^{[\ell]}).
\end{align}
 Recall that \ref{l.tree.induction.1} implies that the top $ \ell $ levels in $ \tower^{[\ell]} $ form a tree tower (with possible repetitions). This property together with \ref{l.tree.induction.3} for $ i=\ell $ implies that
$ \epsilon_{\ell+1} \subset \cup_{i=0}^{\ell} \activated(\wordsub^{[i]}) $.
On the other hand \ref{l.tree.induction.2} for $ i'=\ell+1 $ asserts that 
$ \epsilon_{\ell+1} \cap (\cup_{i=0}^{\ell-1} \activated(\wordsub^{[i]}))=\emptyset $.
These properties together imply $ \epsilon_{\ell+1} \subset \activated(\wordsub^{[\ell]}) $.
By construction, $ \til{\epsilon}_{\ell+1} $ is obtained by taking a union of $ \epsilon_{\ell+1} $ and some letters that are already in $ \activated(\wordsub^{[\ell]}) $.
Hence \eqref{e.inductioni.goal} holds.

\item [(ii-($\ell$+1))] 
Since the top $ \ell $ levels were not altered, it suffices to check the hypothesis for $ i=\ell $, namely
$ \til\epsilon_{i'} \cap \activated(\wordsub^{[\ell]}) = \emptyset $ for all $ i'=\ell+2,\ldots,n $.
This statement holds by construction, \eqref{e.tree.replace.}.

\item [(iii-($\ell$+1))] 
Apply Lemma~\ref{l.up.criterion.} with $ (\word,\word',\alpha) \mapsto (\wordsub^{[\ell+1]},\wordsub^{[\ell]},\activated(\wordsub^{[\ell]})) $.
The conditions~\ref{l.up.criterion.0}--\ref{l.up.criterion.1} therein clearly hold, while the condition~\ref{l.up.criterion.2} holds because of \ref{l.tree.induction.3} for $ i=\ell $.
The result asserts that 
\begin{align}
	\label{e.multiple.1}
	\vecUDc(\wordsub^{[\ell+1]})|_{\wordsub^{[\ell+1]}\cap\activated(\wordsub^{[\ell]})}=(\up\ldots\up).
\end{align}
Next, apply Lemma~\ref{l.up.criterion.} with $ (\word,\word',\alpha) \mapsto (\wordsub^{[\ell+2]},\wordsub^{[\ell+1]},\activated(\wordsub^{[\ell]})) $.
The condition~\ref{l.up.criterion.0} therein clearly holds; 
the condition~\ref{l.up.criterion.1} holds because of \eqref{e.multiple.1};
the condition~\ref{l.up.criterion.2} holds because of \ref{l.tree.induction.3} for $ i=\ell+1 $.
The result asserts that 
$
	\vecUDc(\wordsub^{[\ell+2]})|_{\wordsub^{[\ell+2]}\cap\activated(\wordsub^{[\ell]})}=(\up\ldots\up).
$
Continue this procedure inductively.
Namely apply Lemma~\ref{l.up.criterion.} with $ (\word,\word',\alpha) \mapsto (\wordsub^{[i+1]},\wordsub^{[i]},\activated(\wordsub^{[\ell]})) $ inductively for $ i=\ell+2,\ell+3,\ldots,n $. 
The result gives
\begin{align}
	\label{e.multiple.2}
	\vecUDc(\wordsub^{[i]})|_{\wordsub^{[i]}\cap\activated(\wordsub^{[\ell]})}=(\up\ldots\up),
	\quad
	\text{for } i = \ell+1, \ell+2, \ldots, n.
\end{align}

Fix $ i \in [\ell+1,n] $.
Apply Lemma~\ref{l.up.criterion.} with $ (\word,\word',\alpha) \mapsto (\til\wordsub^{[i]},\wordsub^{[i]},\til\epsilon_{i+1}) $.
Contingent on validation of the conditions therein,
the result of this application verifies the hypothesis (iii-($\ell$+1)).
We now check the conditions:
	\begin{enumerate}
	\item This follows by construction, \eqref{e.tree.replace..}.
	\item This follows from \ref{l.tree.induction.3} and $ \til{\epsilon}_{i+1}\subset\epsilon_{i+1} $ (by \eqref{e.tree.replace.}).
	\item Referring to the construction in \eqref{e.tree.replace}--\eqref{e.tree.replace..}, 
	we see that $ \wordsub^{[i]} \setminus \til\wordsub^{[i]} \subset \activated(\wordsub^{[\ell]}) $.
	This property together with \eqref{e.multiple.2} checks the condition~\ref{l.up.criterion.2}.
	\end{enumerate}
\end{enumerate}
We have completed the induction and hence the proof.
\end{proof}

\subsection{Proof of Proposition~\ref{p.tree.prop}\ref{p.tree.isle}}
\label{s.updown.pf2}
It suffices to prove the statement for two words:

\begin{customprop}{\ref{p.tree.prop}\ref{p.tree.isle}'}
\label{p.tree.isle.}
For any
$
	\wordkl[\wordsub][k_\tL][k_\tR] 
	=
	\wordkl[\wordsub^\tL][k_\tL][k] 
	\cupi
	\wordkl[\wordsub^\tR][k][k_\tR]
$
that respects Convention~\ref{con.wordkl}, $ \wordsub\in\treekl[k_\tL k_\tR] $ if and only if
$ (\wordsub^\tL,\wordsub^\tR)\in\treekl[k_\tL k]\times\treekl[kk_\tR] $.
\end{customprop}
\noindent{}%
This proposition immediately implies Proposition~\ref{p.tree.prop}\ref{p.tree.isle}.

Let us prepare some notation and tools for the proof of Proposition~\ref{p.tree.isle}.
For clarity we will \emph{restore the dependence on $ k,k' $, etc}.
For the rest of this subsection we will only consider towers that increase by one letter at a time.
Given such a tower $ \tower = (\word^{(0)} \Increment{i_0} \word^{(1)} \Increment{i_1} \cdots ) $ and a set $ \alpha $ of letters, let $ \tower|_{\alpha} := ((\word^{(0)}\cap\alpha) \subset (\word^{(1)}\cap\alpha) \subset \cdots ) $ denote the restricted tower.
This new tower may contain repeated words, and we operate under the consent that repeated words are removed from $ \tower|_{\alpha} $ so that the resulting tower also increases by one letter at a time.
We further adopt the shorthand $ \tower|_{[1,j]} := \tower|_{\leq j} $, $ \tower|_{[j,m]} := \tower|_{\geq j} $, and $ \hyp(\wordkl) := \hyp( \rwf ) $.

The following lemma will come in handy.

\begin{lem}
\label{l.IHC.cut}
Let $ \tower $ be a tower that satisfies $ \IHC $ and suppose all words in $ \tower $ contain the letter $ j $.
For all $ k'' $ such that $ \ics_{k''}\geq s_j $, $ \tower|_{\leq j} $ satisfies $ \IHC[kk''] $;
for all $ k'' $ such that $ \icd_{k''} \geq d_j $, $ \tower|_{\geq j} $ satisfies $ \IHC[k''k'] $.
\end{lem}
\begin{proof}
Consider the regions $ D_{\tL}:=(-\infty,x_j]\times \R $ in the $ (x,a) $ plane. % and $ D_{\tR}:=(x_j,\infty)\times \R $, to the left and right of $ x=x_j $.
For $ k'' $ such that $ \ics_{k''} \geq s_j $ and any $ \word\in\tower $, we have $ \hyp( \wordkl )\cap D_{\tL} = \hyp( \wordkl[(\word\cap[1,j])][k][k''] ) \cap D_{\tL} $.
This property ensures that if $ \tower $ satisfies $ \IHC $ then $ \tower|_{\leq j} $ satisfies $ \IHC[kk''] $.
The other statement is proven similarly by considering $ D_{\tR}:=[x_j,\infty)\times \R $.
\end{proof}

\begin{proof}[Proof of the only if part of Proposition~\ref{p.tree.isle.}]
Assume $ \wordsub\in\treekl[k_\tL k_\tR] $. We seek to prove $ \wordsub^\tL\in\treekl[k_\tL k] $.

Consider first the case $ \wordsub^\tL=\emptyset $.
We claim that $ (x_j,a_j) \notin \hyp(\wordkl[\emptyset][k_\tL][k])^\circ $, for all $ j=1,\ldots,m $.
Once this claim is proven, Lemma~\ref{l.tree.empty}\ref{l.tree.empty.k<k'} gives the desired result $ \wordsub^\tL=\emptyset\in\treekl[k_\tL k] $.
To prove the claim, first note that by Proposition~\ref{p.tree.prop}\ref{p.IHC}, the word $ \wordsub $ satisfies $ \IHC[k_\tL k_\tR] $,
so $ (x_j,a_j) \notin \hyp(\wordkl[\wordsub][k_\tL][k_\tR])^\circ $, for all $ j\notin\wordsub $.
Next, under the current assumptions, $ \wordkl[\wordsub][k_\tL][k_\tR] = \wordkl[\emptyset][k_\tL][k] \cupi \wordkl[\wordsub][k][k_\tR] $,
so $ \hyp(\wordkl[\wordsub][k_\tL][k_\tR]) \setminus \hyp(\parab[k](t))^\circ $
is a disjoint union of $ \hyp(\wordkl[\emptyset][k_\tL][k]) \setminus \hyp(\parab[k](t))^\circ $
and $ \hyp(\wordkl[\wordsub][k][k_\tR]) \setminus \hyp(\parab[k](t))^\circ $.
Recall that the discretized Hopf--Lax condition \eqref{e.HLc.xa} forbids any $ (x_j,a_j) $ to be in $ \hyp(\parab[k](t))^\circ $.
Combining the preceding properties gives the claim.

Next consider $ \wordsub^\tL\neq\emptyset $.
Let $ \letterlast := (\wordsub^\tL)_{|\wordsub^\tL|} $ denote the last letter in $ \wordsub^\tL $,
and consider the word
$
	\word^* := \wordfullkl[k_\tL k]|_{\leq \letterlast}
$
formed by all letters in $ \alphabetkl[k_\tL k] $ that are $ \leq \letterlast $.
We next show that $ \word^* \in \treekl[k_\tL k] $.
Examine the up-down iteration corresponding to $\treekl[k_\tL k] $.
Start from $ \wordfullkl[k_\tL k] $.
In each step of the iteration, follow the child that chooses to delete every activated letter that is $ > \letterlast $ and chooses to keep every activated letter that is $ \leq \letterlast $.
At the end of this procedure, we arrive at a word $ \til\word^{*}\in\treekl[k_\tL k] $ that contains $ \word^* $.
In order to conclude $ \word^* \in \treekl[k_\tL k] $, it suffices to show $ \til\word^{*} = \word^* $.
Assume the contrary: $ \til\word^{*} = \word^* \cup\{j_1,\ldots\} $, with $ \letterlast < j_1 <\ldots $.
Consider the region $ \Omega := \{ (y,a): x_{\letterlast} \leq y \leq \xr(\wordkl[\wordsub^{\tL}][k_\tL][k]), \, \parab[k](t,y) \leq a < \rwf[ \wordkl[\wordsub^{\tL}][k_\tL][k] ](y) \} $
as depicted in Figure~\ref{f.wing}, and note that $ \Omega \subset \hyp( \wordkl[\wordsub^\tL][k_\tL][k] )^\circ \subset \hyp( \wordkl[\wordsub][k_\tL][k_\tR] )^\circ $.
Those excess letters $ j_1,\ldots $ have never been been activated throughout the construction of $ \til\word^* $, so $ \vecUDc(\wordkl[\til\word^{*}][k_\tL][k])|_{\set{j_1\ldots }} = (\down\ldots\down) $.
Referring to Figure~\ref{f.wing-under},
we see that the last property forces $ (x_{j_1},a_{j_1}),\ldots \in \Omega \subset \hyp( \wordkl[\wordsub][k_\tL][k_\tR] )^\circ  $.
However, the fact that $ \wordsub $ satisfies $ \IHC[k_\tL k_\tR] $ (by Proposition~\ref{p.tree.prop}\ref{p.IHC}) 
forbids the existence of any $ (x_j,a_j) \in \hyp( \wordkl[\wordsub][k_\tL][k_\tR] )^\circ $.
Hence those excess letters $ j_1,\ldots $ do not exist, and $ \til\word^{*} = \word^* $.

We now show $ \wordsub^\tL\in\treekl[k_\tL k] $.
The current assumption $ \wordsub\in\treekl[k_\tL k_\tR] $ gives (by Proposition~\ref{p.tree.prop}\ref{p.IHC}) a tower $ \tower $ that satisfies $ \IHC[k_\tL k_\tR] $ and ascends from $ \wordsub $ to $ \wordfullkl[k_\tL k_\tR] $. 
Applying Lemma~\ref{l.IHC.cut} with $ (j,k,k'',k')\mapsto(\letterlast,k_\tL,k,k_\tR ) $ gives a tower $ \tower' $ that satisfies $ \IHC[k_\tL k] $ and ascends from $ \word^\tL $ to $ \word^* $. 
On the other hand, the conclusion $ \word^* \in \treekl[k_\tL k] $ from the last paragraph gives (by Proposition~\ref{p.tree.prop}\ref{p.IHC}) a tower $ \tower'' $  that satisfies $ \IHC[k_\tL k ] $ and ascends from $ \word^* $ to $ \wordfullkl[k_\tL k] $.
Concatenating $ \tower' $ and $ \tower'' $ proves that $ \wordsub^{\tL}\in\treekl[k_\tL k] $.

The proof of $ \wordsub^\tR\in\treekl[kk_\tR] $ is similar, which we omit.
\end{proof}

\begin{figure}[h]
\begin{minipage}[t]{.5\linewidth}
	\frame{\includegraphics[height=82pt]{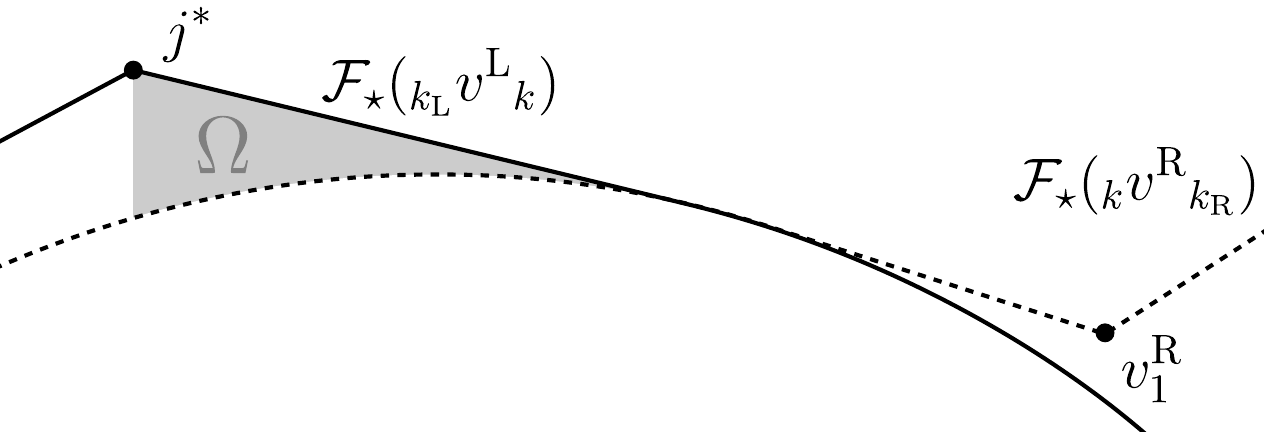}}
	\caption{The region $ \Omega $. %
	The solid line is the graph of $ \rwf[\wordkl[\wordsub^\tL][k_\tL][k]] $, %
	and the dashed line is the graph of $ \rwf[\wordkl[\wordsub^\tR][k][k_\tR]] $.}
	\label{f.wing}
\end{minipage}
\hfill
\begin{minipage}[t]{.48\linewidth}
	\frame{\includegraphics[height=82pt]{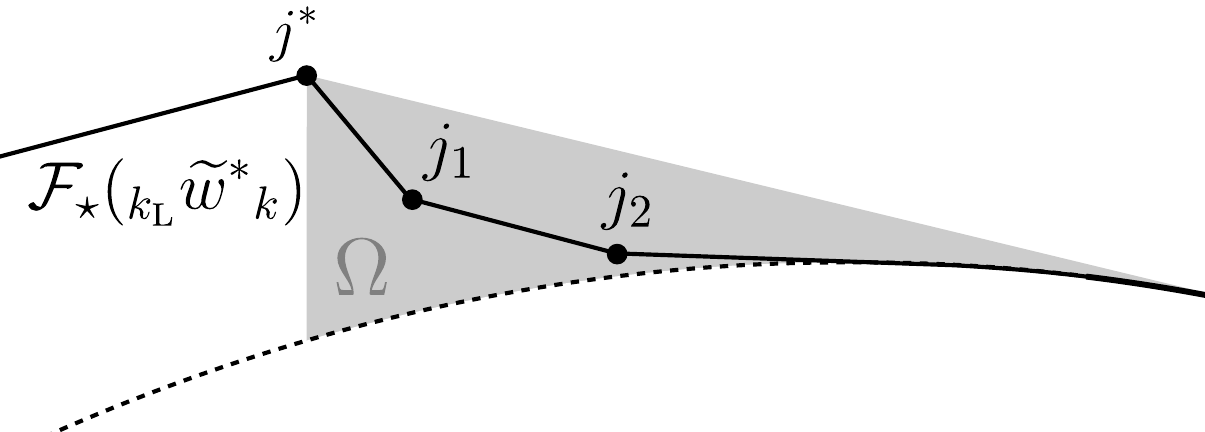}}
	\caption{An illustration of the graph of $ \rwf[\wordkl[\til\word^*][k_\tL][k]] $.}
	\label{f.wing-under}%
\end{minipage}
\vspace{-10pt}
\end{figure}

\begin{proof}[Proof of the if part of Proposition~\ref{p.tree.isle.}]
Assume $ (\wordsub^\tL,\wordsub^\tR)\in\treekl[k_\tL k]\times\treekl[kk_\tR] $. We seek to prove $ \wordsub\in\treekl[k_\tL k_\tR] $.

Consider first the case $ \wordsub^\tL = \wordsub^\tR = \emptyset $.
Applying Lemma~\ref{l.tree.empty}\ref{l.tree.empty.k<k'} with $ (k,k')\mapsto (k_\tL,k) $ and with $ (k,k')\mapsto (k,k_\tR) $ 
shows that $ (x_j,a_j)\notin\hyp( \wordkl[\emptyset][k_\tL][k] )^\circ $ and $ (x_j,a_j)\notin\hyp( \wordkl[\emptyset][k][k_\tR] )^\circ $, for all $ j = 1,\ldots,m $.
These properties together with $ \wordkl[\emptyset][k_\tL][k] \cupi \wordkl[\emptyset][k][k_\tR] = \wordkl[\emptyset][k_\tL][k_\tR] $ imply that $ (x_j,a_j)\notin\hyp( \wordkl[\emptyset][k_\tL][k_\tR] )^\circ $, for all $ j = 1,\ldots,m $ (see Definition~\ref{d.<<}\ref{d.<<i.hypo}).
Applying Lemma~\ref{l.tree.empty}\ref{l.tree.empty.k<k'} with $ (k,k')\mapsto (k_\tL,k_\tR) $ concludes $ \empty \in \treekl[k_\tL k_\tR] $.

Consider next the case $ \wordsub^\tL \neq \emptyset $ and $ \wordsub^\tR \neq \emptyset $.
Let $ \letterlast := (\wordsub^\tL)_{|\wordsub^\tL|} $ denote the last letter in $ \word^\tL $
and let  $ \letterfirst := (\wordsub^\tR)_{1} $ denote the first letter in $ \word^\tR $.
Consider the words
$
	\word^* := \wordfullkl[k_\tL k_\tR]|_{\leq \letterlast}
$
and
$
	\word_* := \wordfullkl[k_\tL k_\tR]|_{\geq \letterfirst}
$
formed by all letters in $ \alphabetkl[k_\tL k_\tR] $ 
that are respectively  $ \leq \letterlast $ and $ \geq \letterfirst $.
Similarly arguments in the `only if part' show that $ \word^*\cup \word_* \in \treekl[k_\tL k_\tR] $.
Hence, by Proposition~\ref{p.tree.prop}\ref{p.IHC}, there exists a tower 
\begin{align}
	\label{e.tree.isle.if.1}
	((\word^*\cup \word_*)  \subset \ldots \subset \wordfullkl[k_\tL k_\tR]) \text{ that satisfies } \IHC[k_\tL k_\tR].
\end{align}
Next, the current assumption $ (\wordsub^\tL,\wordsub^\tR)\in\treekl[k_\tL k]\times \treekl[kk_\tR] $ gives the towers
\begin{align*}
	(\wordsub^\tL \subset \ldots \subset \wordfullkl[k_\tL k]) \text{ that satisfies } \IHC[k_\tL k],
	\qquad
	(\wordsub^\tR \subset \ldots \subset \wordfullkl[kk_\tR]) \text{ that satisfies } \IHC[kk_\tR].
\end{align*}
Applying Lemma~\ref{l.IHC.cut} with $ (j,k,k'',k')\mapsto (\letterlast,k_\tL,k,k_\tR) $ and with $ (j,k,k'',k') \mapsto (\letterfirst,k_\tL,k,k_\tR) $ gives
\begin{align}
	\label{e.tree.isle.if.2}
	(\wordsub^\tL \subset \ldots \subset \word^*) \text{ that satisfies } \IHC[k_\tL k],
\\	
	\label{e.tree.isle.if.3}
	(\wordsub^\tR \subset \ldots \subset \word_*) \text{ that satisfies } \IHC[kk_\tR].
\end{align}
Take the union of each word in \eqref{e.tree.isle.if.2} with $ \word^\tR $.
Under the condition $ \wordsub^\tL \lli \wordsub^\tR $, referring to Definition~\ref{d.<<i}\ref{d.<<i.hypo},
one sees that the resulting tower %$ ((\wordsub^\tL\cup \wordsub^\tR) \subset \ldots \subset (\word^*\cup\wordsub^\tR)) $ 
satisfies $ \IHC[k_\tL k_\tR] $:
\begin{align}
	\label{e.tree.isle.if.2'}
	((\wordsub^\tL\cup \wordsub^\tR) \subset \ldots \subset (\word^*\cup\wordsub^\tR)) \text{ satisfies } \IHC[k_\tL k_\tR].
\end{align}
Similarly, taking the union of each word in \eqref{e.tree.isle.if.3} with $ \word^* $ gives a tower $ ((\word^*\cup\wordsub^\tR)\subset \ldots \subset (\word^*\cup \word_*)) $ that satisfies $ \IHC[k_\tL k_\tR] $.
Concatenating the last two towers and the tower in \eqref{e.tree.isle.if.1} yields $ \wordsub=\wordsub^{\tL}\cup\wordsub^{\tR}\in\treekl[k_\tL k_\tR] $.

The case $ \wordsub^\tL \neq\emptyset $ or $ \wordsub^\tR \neq \emptyset $ can be treated similarly.
Say $ \wordsub^\tL \neq\emptyset $ and $ \wordsub^\tR =\emptyset $.
Similarly arguments in the `only if part' show that $ \word^* \in \treekl[k_\tL k_\tR] $.
Hence, by Proposition~\ref{p.tree.prop}\ref{p.IHC}, there exists a tower $ (\word^*  \subset \ldots \subset \wordfullkl[k_\tL k_\tR]) $ that satisfies $ \IHC[k_\tL k_\tR] $.
The same argument that leads up to \eqref{e.tree.isle.if.2'} produces a tower $ (\wordsub^\tL \subset \ldots \subset\word^*) $ that satisfies $ \IHC[k_\tL k_\tR] $.
Concatenating this tower with the previous one gives the desired result.
\end{proof}

\section{Determinantal analysis: the isle factorization}
\label{s.isle}

\subsection{Expanding the determinant, generic terms, and preferred terms}
\label{s.isle.processing}
We begin with a definition.
\begin{defn}
\label{d.opuci}
For any $ \wordkl $, decompose it into $ \lli $-ordered isles as 
$
	\wordkl
	= 
	\wordkl[\wordsub^{(1)}][k][k_1] 
	\cupi 
	\wordkl[\wordsub^{(2)}][k_1][k_2] 
	\cupi \cdots 
	\cupi 
	\wordkl[\wordsub^{(n)}][k_{n-1}],
$
where each $ \wordkl[\wordsub^{(i)}][k_{i-1}][k_{i}] \in \islesetkl[k_{i-1}k_i] $.
Recall $ \opu[k][k']{(\ldots)_{k_1}(\ldots)_{k_2}\ldots} $ from Definition~\ref{d.opu.extended}.
Set
\begin{align}
	\label{e.opuci}
	\opuci[k][k']{\word} 
	&:=
	\Opu[k][k']{% 
		(\wordsub^{(1)} {}_{\vecUDc(\wordkl[\wordsub^{(1)}][k][k_1]) })_{k_1} 
		(\wordsub^{(2)} {}_{\vecUDc(\wordkl[\wordsub^{(2)}][k_1][k_2]) })_{k_2}
		\cdots
		{}_{k_{n-1}} (\wordsub^{(n)} {}_{\vecUDc(\wordkl[\wordsub^{(n)}][k_{n-1}][k']) })  
	}\ .
\end{align}
\end{defn}

The first step is to convert the relevant operators into $ \opuci[k][k']{\word} $.
Recall the parity from the beginning of Section~\ref{s.updown}.
\begin{lem}
\label{l.isle.prep}
We have $ \punder = \det(\Id + (K_{kk'\star})_{k,k'=1}^{\icm}) $, where
\begin{align}
	\label{e.det.islse}
	K_{kk'\star}
	&:= 
	\sideset{}{^\lli_{\tree}}{\sum}
	\parityc(\wordkl[\word^{(1)}][k][k_1]) \cdots \parityc(\wordkl[\word^{(\ell)}][k_\ell][k'])
	\
	\inddown_{\ick}
	\opuci{\word^{(1)}} 
	\inddown_{\ick_1}
	\cdots 
	\inddown_{\ick_\ell}
	\opuci{\word^{(\ell)}}
	\inddown_{\ick'}\ ,	
\end{align}
the sum goes over all $ \wordkl[\word^{(1)}][k][k_1] \lli \cdots \lli \wordkl[\word^{(\ell)}][k_\ell][k'] $, and each $ \word^{(i)} \in \treekl[k_{i-1}k_i] $, with the convention $ k_0:=k $ and $ k_n:=k' $.
\end{lem}
\begin{proof}
The starting point is the output~\eqref{e.det.iter} of the up-down iteration.
Fix $ \wordkl\in\treekl $ and decompose this word into isles, namely
$
	\wordkl[\word][k][k'] 
	= 
	\wordkl[\wordsub^{(1)}][k][k_1] 
	\cupi 
	\wordkl[\wordsub^{(2)}][k_1][k_2] 
	\cupi \cdots 
	\cupi 
	\wordkl[\wordsub^{(n)}][k_{n-1}][k'],
$
and each $ \wordkl[\wordsub^{(i)}][k_{i-1}][k_{i}] \in \islesetkl[k_{i-1}k_i] $, with the convention $ k_0:=k $ and $ k_n:=k' $.
Invoke Lemma~\ref{l.opu.id} to factorize $ \inddown_{\ick} \opuc{\word} \inddown_{k'} $. 
For example, when $ n=3 $, 
\begin{align*}
	\inddown_{\ick} \opuc{\word} \inddown_{\ick'} 
	=
	&\inddown_{\ick} \opuci{\word} \inddown_{\ick'}
	+ 
	\inddown_{\ick} \opuci{\wordsub^{(1)}} \inddown_{\ick_1} \opuci{(\wordsub^{(2)} \cupi \wordsub^{(3)}) } \inddown_{\ick'}
\\
	&+ 
	\inddown_{\ick} \opuci{ (\wordsub^{(1)}\cupi \wordsub^{(2)}) } \inddown_{k_2} \opuci{\wordsub^{(3)} } \inddown_{\ick'}
	+ 
	\inddown_{\ick} \opuci{\wordsub^{(1)} } \inddown_{\ick_1} \opuci{\wordsub^{(2)} } \inddown_{k_2} \opuci{\wordsub^{(3)} } \inddown_{\ick'}.
\end{align*}
Sum the result over $ \wordkl\in\treekl $.
By Proposition~\ref{p.tree.prop}\ref{p.tree.isle}, the result can be written as \eqref{e.det.islse}.
\end{proof}

Lemma~\ref{l.trace-class} verifies that $ \inddown_{\ick}\opuci{\word} \inddown_{\ick'} $ is trace-class upon suitable conjugation.

We now invoke~\eqref{e.fredholm}--\eqref{e.trace} to expand the determinant into products of traces.
\begin{defn}
\label{d.gterm}
We call 
%\begin{align}
%	\label{e.gtrace}
$
	\tr( \inddown_{\ick_0} \opuci{ \word^{(1)} } \inddown_{\ick_1} \opuci{ \word^{(2)} } \cdots \inddown_{\ick_{n-1}} \opuci{ \word^{(n)} } \inddown_{\ick_0} )
$
%\end{align}
a \tdef{generic trace}, where $ k_n := k_0 $, each $ \word^{(i)} \in \islesetkl[k_{i-1}k_i] $, but the $ \word $'s are \emph{not} necessarily $ \lli $-ordered. 
We call a product of generic traces a \tdef{generic term}.
We call a generic term \tdef{degenerate} if the words involved (in the $ \word $'s) do not exhaust all the letters $ 1,\ldots,m $.
Let $ \norm{\gterm} $ count the total number of $ \opuci{\ldots} $ involved in $ \gterm $.
\end{defn}
\noindent{}%
For example, for $ m=3 $, the generic term $ \gterm = \tr(\inddown_{\Ic{1}}\opuci{ 13 }\inddown_{2} \opuci{ 13 }\inddown_{\Ic{1}}) \tr(\inddown_{3}\opuci{1}\inddown_{3}) $ is degenerate because the letter $ 2 $ has not been involved, and $ \norm{\gterm} =3 $.
Inserting $ K_{kk'\star} $ from \eqref{e.det.islse} into \eqref{e.simon}--\eqref{e.trace} gives
\begin{align}
	\label{e.DD}
	\DD_n := \D_n( (K_{kk'\star})_{k,k'=1}^{\icm}) = \sum_{\gterm} \coef_{n,\gterm} \, \gterm.
\end{align}
Here $ \coef_{n,\gterm} \in \Z $ is the coefficient, and, for each fixed $ n $, is nonzero for finitely many generic terms $ \gterm $.
We have
\begin{align}
	\tag{Expansion}
	\label{e.det.expanded}
	\punder
	=
	\sum_{n= 0}^\infty \frac{\DD_n}{n!} = \sum_{n=0}^\infty \Big( \sum_{\gterm} \frac{1}{n!}\coef_{n,\gterm} \, \gterm \Big),
\end{align}
where the sum over $ n $ converges absolutely.

We now define the type of terms that are amenable for the $ N\to\infty $ analysis.

\begin{defn}
\label{d.prefer}
We call a generic trace in Definition~\ref{d.gterm} \tdef{preferred} if
$ \wordkl[{\word^{(i)}}][k_{i-1}][k_{i}] \not\lli \wordkl[{\word^{(i+1)}}][k_{i}][k_{i+1}] $ for all $ i $,
under the cyclic convention $ \wordkl[{\word^{(n+1)}}][k_{n}][k_{n+1}] := \wordkl[{\word^{(1)}}][k_{0}][k_{1}] $.
We call a product of preferred traces a \tdef{preferred term}.
\end{defn}

The goal of the isle factorization is to show that all non-preferred terms in \eqref{e.det.expanded} cancel exactly:
\begin{prop}
\label{p.isle}
For any non-preferred term $ \gterm $, $ \sum_{n \geq 0} \tfrac{1}{n!} \coef_{n,\gterm} = 0 $.
\end{prop}
\noindent{}%
The proof of Proposition~\ref{p.isle} is in Section \ref{s.isle.}.
The sum in Proposition~\ref{p.isle} contains only finitely many non-zero terms. 
To see why, recall that $ \norm{\gterm} $ counts the total number of $ \opuci{\ldots} $ involved in $ \gterm $.
From \eqref{e.DD}, \eqref{e.det.islse}, and \eqref{e.simon}--\eqref{e.trace}, 
one sees that each generic term $ \gterm' $ in $ \D_n $ has $ \norm{\gterm'} \geq n $,
so $ \coef_{n,\gterm} =0 $ for all $ n > \norm{\gterm} $.

\subsection{Formal determinants}
\label{s.isle.fdet}
As explained in Section~\ref{s.overview.updown.isle}, to prove Proposition~\ref{p.isle} we need to develop and invoke a more flexible notion of determinants --- formal determinants. 

For a symbol $ X $, we view $ X $, $ XX $, $ XXX $, \ldots as \emph{distinct} indeterminates, and consider formal power series in these variables $ \C\bkk{X,XX,\ldots} = \C\bkk{ \angg{X} } $.
The \tdef{formal determinant} $ \fdet(\Id+X) $ is an element in $ \C\bkk{ \angg{X} } $ given as
\begin{align*}
%	\label{e.fdet}
	\fdet(\Id+X) 
	:=
	\sum_{n=0}^\infty 
	\frac{1}{n!}
	\begin{vmatrix}
		X & n-1 & &&\\
		XX & X & n-2 &\\
		\vdots & \ddots & \ddots & \ddots \\
		\vdots & \vdots & \ddots & X & 1 \\
		\cdots & \cdots & \cdots & XX & X
	\end{vmatrix}
%	\begin{vmatrix}
%		X & n-1 & &&&\\
%		XX & X & n-2 &&\\
%		XXX & XX & X & n-3 &\\
%		\vdots & \ddots & \ddots & \ddots & \ddots & \\
%		\vdots & \vdots & \ddots & XX & X & 1\\
%		\cdots & \cdots & \cdots & XXX & XX & X
%	\end{vmatrix}
	=
	1 + X + \tfrac12 X\fdot X - \tfrac12 XX + \cdots.
%	\in
%	\C\bkk{X,XX,\ldots}.
\end{align*} 
To avoid confusion, we will always use a $ \fdot $ for the (commutative) multiplication in $ \C\bkk{\angg{X} } $.
For example $ X \fdot X \neq XX = 1\fdot XX = XX \fdot 1 $.

Next we extend this definition of formal determinants to allow more symbols and to allow matrix-valued inputs.
Fix symbols $ X_1,\ldots,X_\ell $.
Consider \emph{non-commutative} formal series
\begin{align*}
	\C\Monomials
	=
	\Big\{ %\sum_{j\geq 0} \sum_{i_1,\ldots,i_j=1}^\ell 
	\sum
	\alpha_{i_1\cdots i_j} 
	X_{i_1}\cdots X_{i_j} 
	:
	\alpha_{i_1\cdots i_j} \in \C
	\Big\}.
%	=
%	\Big\{ \sum_{j\geq 0} \sum% \sum_{i_1,\ldots,i_j=1}^\ell \alpha_{i_1\cdots i_j} 
%	X_{i_1}\cdots X_{i_j} \Big\}	
\end{align*}
The sum runs over \emph{finitely many} $j$ tuples $(i_1,\cdots,i_j)\in\{1,\ldots,\ell\}^j$, with $j\in\Z_{\geq 0}$ and $\alpha_\emptyset X_\emptyset:=\alpha_\emptyset\in\C$.%
Let us clarify that the `$ \cdots $' in $ X_{i_1}\cdots X_{i_j} := \prod_{j'=1}^j X_{i_{j'}} $ means omission of some symbols, and is not to be confused with the multiplication `$ \fdot $' in $ \C\bkk{\angg{X} } $. 
The non-commutative multiplication in $ \C\Monomials $ is done by juxtaposing symbols and assuming the associative law.
For example, $(X_1X_2)(X_1):=X_1X_2X_1\neq X_1X_1X_2$ in $ \C\angg{X_1,X_2} $.
This is not to be confused with $(X_1X_2)\Cdot X_1 = X_1\Cdot(X_1X_2)$ in $\C\bkk{\angg{X_1,X_2}}$.%
Fix $ u=(u_{kk'})_{k,k'=1}^{\icm} $ with $ u_{kk'} \in \C\Monomials $. Such a $ u $ induces an algebra homomorphism
\begin{align}
	\label{e.fdet.phiu}
	\varphi_u: \C\angg{X} \longrightarrow \C\Monomials,
	\qquad
	\underbrace{X\cdots X}_{n} \mapsto \sum u_{k_0k_1} \cdots u_{k_{n-1}k_0},
\end{align}
where the sums go over $ k_0,\ldots,k_{n-1} \in \{1,\ldots,\icm\} $.
This homomorphism can be lifted to
\begin{align}
	\label{e.fdet.phiu.}
	\bar\varphi_u: \C\bkk{\angg{X}} \longrightarrow \C\bkk{\Monomials}.
\end{align}
To incorporate the cyclic nature of traces,
we mod out cyclic relations in $ \C\bkk{\Monomials} $ by setting $ X_{i_1}\cdots X_{i_j} \sim X_{i_{\sigma(1)}} \cdots X_{i_{\sigma(j)}} $ for any cyclic permutation $ \sigma \in \mathbb{S}_j $.
Doing so produces the quotient algebra $ \C\bkk{\Monomials/\cyclic} $ along with the projection 
$
	\pi: \C\bkk{\Monomials} \to \C\bkk{\Monomials/\cyclic}.
$
\begin{defn}
%\begin{align}
%	\label{e.fdet.}
$
	\fdet(\Id + u)
	:=
	\big( \pi \circ \bar\varphi_u \big)( \fdet(\Id+X) )
	\in
	\C\bkk{ \Monomials/\cyclic }.
$
%\end{align}
\end{defn}

We next state the main properties of formal determinants that will be used later.
%Let us first explain how the formal determinant is related to the usual, numerical determinant given in \eqref{e.det.expanded}.
In view of Lemma~\ref{l.isle.prep}, we take 
$ \symb:=\{ X_1,\ldots,X_\ell\} = \{ \wordkl \in \wordsetkl : k,k'=1,\ldots,\icm \} $, with $\ell=\sum_{k,k'=1,\ldots,\icm}|\wordsetkl|$,
%and choose a $ u $ of the form
%$	
%	u_{kk'} 
%	= 
%	\sum 
%	\alpha_{\cdots} \ \wordkl[\word][k][k_1] \cdots \wordkl[\word'][k_\ell][k'],
%$
and consider 
\begin{align}
	\label{e.ustar}
	u_\star
	:=
	\Big( \sideset{}{^{\lli}_\symb}\sum
		\parityc(\wordkl[\word][k][k_1]) \cdots \parityc(\wordkl[\word'][k_\ell][k'])
		\
		\wordkl[\word][k][k_1]
		\cdots 
		\wordkl[\word'][k_\ell][k']	
	\Big)_{k,k'=1}^{\icm},
\end{align}
which is an $ \icm\times\icm $ matrix with $ \C\Monomials $-valued entries.
%
%where the sum is finite, and $ \alpha_{\cdots}=\alpha_{\wordkl[\word][k][k_1] \cdots \wordkl[\word'][k_\ell][k']}\in\C $.
%The $ \cdots $ on the left denotes a complicated-looking dependence that is irrelevant.
%Such a $ u $ satisfies the conditions:
%\begin{enumerate}[leftmargin=20pt]
%\item Each $ u_{kk'} $ has a zero constant coefficient.
%	This condition is necessary for the corresponding \emph{numerical} determinant to be well-defined in any infinite dimensional space.
%\item Each $ u_{kk'} $ is a linear combination of products of the $ \word $'s with indices beginning at $ k $ and ending at $ k' $.
%\end{enumerate}
%These conditions are not required for the general treatment in this subsection,
%but we remark that the $ u $'s satisfying these conditions form a subalgebra in $ \C\Monomials $.

\begin{prop}
\label{p.det.prop}
\begin{enumerate}[leftmargin=20pt,label=(\alph*)]
\item[]
\item \label{p.det.formal}
For the $ u_\star $ given in \eqref{e.ustar},
\begin{align}
	\label{e.isle.fdet.}
	\fdet(\Id+u_\star)
	=
	\sum
	\fcoef_{\pmonomial} \, \pmonomial
	\in
	\C\bkk{\angg{\symb}/\cyclic}.
\end{align}
Here $ \fcoef_\pmonomial \in \C $;
the sum goes over $ \pmonomial=\monomial_1\fdot\cdots \fdot \monomial_\ell \in \C\bkk{ \angg{\symb}/\cyclic } $, where each $ \monomial_i\in \angg{\symb}/\cyclic $ is of the form
%\begin{align}
$
%	\label{e.monomial}
	\wordkl[\word^{(1)}][k_\ell][k_1] \cdots \wordkl[\word^{(\ell)}][k_{\ell-1}][k_\ell].
%	\in \C\bkk{ \angg{\symb}/\cyclic }.
$
%\end{align}
\item \label{p.det.sub}
The formal determinant \eqref{e.isle.fdet.} becomes the numerical determinant \eqref{e.det.expanded}  upon the substitution
\begin{align*}
%	\label{e.evaluation}
	\wordkl[{\word^{(1)}}][k_0][k_1] \cdots \wordkl[{\word^{(\ell)}}][k_\ell][k_0]
	\longmapsto
	\tr( \inddown_{\ick_0}\opuci{{\word^{(1)}}}\inddown_{\ick_1} \cdots  \inddown_{\ick_{n-1}}\opuci{{\word^{(\ell)}}}\inddown_{\ick_0} ).
\end{align*}

\item \label{p.det.factorize}
For any $ u,v \in \C\Monomials $,
$	
	\fdet(\Id + u) \fdot \fdet(\Id + v) = \fdet(\Id + (u + v + uv)).
$
\end{enumerate}
\end{prop}
\noindent{}%
Parts~\ref{p.det.formal}--\ref{p.det.sub} follow by definition.
The proof of Part~\ref{p.det.factorize} is deferred to Section~\ref{s.isle.pf1}.

\subsection{Cancellation of non-preferred terms --- proof of Proposition~\ref{p.isle}}
\label{s.isle.}
%Recall \eqref{e.det.expanded} and recall the definition of non-preferred terms from Definition~\ref{d.prefer}. The goal in this section is to prove
%%We will prove Proposition~\ref{p.isle} by applying Proposition~\ref{p.det.prop}.
%%Take $ \symb = \{  \wordkl: \word\in\treekl,k,k'=1,\ldots,\icm \} $ to be the set of symbols
%%and consider the formal determinant $ \fdet\big( \Id + u_\star \big) $, where
%%\begin{align}
%%	\label{e.ustar}
%%	u_\star
%%	:=
%%	\Big( \sideset{}{^{\lli}_\symb}\sum
%%		\parityc(\wordkl[\word][k][k_1]) \cdots \parityc(\wordkl[\word'][k_\ell][k'])
%%		\
%%		\wordkl[\word][k][k_1]
%%		\cdots 
%%		\wordkl[\word'][k_\ell][k']	
%%	\Big)_{k,k'=1}^{\icm},
%%\end{align}
%%and the sum  $ \sum^{\lli} $ goes over all $ \wordkl[\word][k][k_1] \lli \ldots \lli \wordkl[\word'][k_\ell][k'] \in \symb $.
%Express $ \fdet(\Id+u_\star) $ as a formal sum of monomials:
%\begin{align}
%	\label{e.isle.fdet.}
%	\fdet(\Id+u_\star)
%	=
%	\sum%_{\pmonomial=\monomial_1\fdot\monomial_2\fdot\cdots \fdot \monomial_\ell} 
%	\fcoef_{\pmonomial} \, \pmonomial
%	\in
%	\C\bkk{\angg{\symb}/\cyclic}.
%\end{align}
%Here $ \fcoef_\pmonomial \in \C $,
%and the sum goes over all monomials $ \pmonomial=\monomial_1\fdot\cdots \fdot \monomial_\ell \in \C\bkk{ \angg{\symb}/\cyclic } $, where each $ \monomial_i\in \angg{\symb}/\cyclic $ is of the form
%%\begin{align}
%$
%%	\label{e.monomial}
%	\wordkl[\word^{(1)}][k_\ell][k_1] \cdots \wordkl[\word^{(\ell)}][k_{\ell-1}][k_\ell].
%%	\in \C\bkk{ \angg{\symb}/\cyclic }.
%$
%%\end{align}

Let us reformulate Proposition~\ref{p.isle}.
Consider a generic $ \monomial_i=\wordkl[\word^{(1)}][k_\ell][k_1] \cdots \wordkl[\word^{(\ell)}][k_{\ell-1}][k_\ell] \in \C\bkk{ \angg{\symb}/\cyclic } $.
Mimicking the definition of preferred traces in Definition~\ref{d.prefer}, we call such a $ \monomial_i \in \angg{\symb}/\cyclic $  \tdef{non-preferred} if $ \wordkl[\word^{(i-1)}][k_{i-1}][k_{i}] \lli \wordkl[\word^{(i+1)}][k_{i}][k_{i+1}]  $ for some $ i $, under the cyclic convention $ \wordkl[\word^{(\ell+1)}][k_{\ell}][k_{\ell+1}] := \wordkl[\word^{(1)}][k_{\ell}][k_{1}] $.
By Proposition~\ref{p.det.prop}\ref{p.det.formal}--\ref{p.det.sub},
Proposition~\ref{p.isle} is equivalent to the following.

\begin{customprop}{\ref*{p.isle}'}
\label{p.isle.}
For any non-preferred  $ \monomial\in\angg{\symb}/\cyclic $, 
and any monomial $ \pmonomial = (\cdots \monomial'\fdot \monomial \fdot \monomial'' \cdots ) \in \C\bkk{\angg{\symb}/\cyclic} $ that contains $ \monomial $ as a factor,
the corresponding coefficient in \eqref{e.isle.fdet.} vanishes, namely $ \fcoef_{\pmonomial} = 0 $.
\end{customprop}
\begin{proof}
%Throughout this proof, 
%\begin{center}
%	\tdef{product of symbols} means
%	`first taking the product in $ \C\angg{\symb} $,
%	and then performing cyclic identification'.
%\end{center}
%The result is in $ \angg{\symb}/\cyclic $.
%%For example, the product of $ X_1X_2 $ and $ X_3X_4 $ is equal to $ X_1X_2X_3X_4 = X_2X_3X_4X_1 = X_3X_4X_1X_2 = X_2X_3X_4X_1 \in \angg{\symb}/\cyclic $.

Throughout this proof, irrelevant indices will be denoted by a $ * $.
For example $ \wordkl[\word][*][*] $.

Fix any non-preferred $ \monomial=(\cdots \wordkl[{\word^{(1)}}][*][k_0] \, \wordkl[{\word^{(2)}}][k_0][*] \cdots)\in\angg{\symb}/\cyclic $, 
where $ \wordkl[{\word^{(1)}}][*][k_0] \lli \wordkl[{\word^{(2)}}][k_0][*] $.
Recalling the definitions of $ \lli $ from Definition~\ref{d.<<i}\ref{d.<<i.hypo} and Definition~\ref{d.<<}, we find that there exists $ i\in\{0,\ldots,m\} $ such that all words in $ \wordkl[{\word^{(1)}}][*][k] $ consists of letters $ \leq i $, and all words in $ \wordkl[{\word^{(2)}}][k][*] $ consists of letters $ > i $.
The case $ i=0 $ corresponds to all words in $ \wordkl[{\word^{(1)}}][*][k] $ being empty, and likewise for $ i=m $. 
Set
\begin{align*}
	\symb_1 := \big\{ \wordkl[\word][*][*] \in \symb :
	 \word \subset [1,i] \big\},
	\quad
	\symb_2 := \big\{ \wordkl[\word][*][*] \in \symb :
	\word \subset (i,m] \big\},
\end{align*}
and consider
\begin{align}
	\label{e.isle.u1}
	u_{1}
	&:=
	\Big( \kdelta_{k'=k_0} \sideset{}{^{\lli}_{\symb_{1}}}{\sum} 
		\parityc(\wordkl[\word][k][*]) \cdots \parityc(\wordkl[\word'][*][k_0])
		\
		\wordkl[\word][k][*]
		\cdots 
		\wordkl[\word'][*][k_0]
	\Big)_{k,k'=1}^{\icm}
	\in
	\Mat_{\icm}(\C\angg{\symb_1}),
\\
	\label{e.isle.u2}
	u_{2}
	&:=
	\Big( \kdelta_{k=k_0} \sideset{}{^{\lli}_{\symb_{2}}}{\sum} 
		\parityc(\wordkl[\word][k_0][*]) \cdots \parityc(\wordkl[\word'][*])
		\
		\wordkl[\word][k_0][*]
		\cdots 
		\wordkl[\word'][*]
	\Big)_{k,k'=1}^{\icm}
	\in
	\Mat_{\icm}(\C\angg{\symb_2}).
\end{align}
Here $ \kdelta $ denotes the Kronecker delta so that $ u_1 $ and $ u_2 $ are nonzero only in the $ k_0 $-th column and $ k_0 $-th row, respectively.
Recall $ u_\star $ from \eqref{e.ustar} and set $ u_3 := u_\star - u_1 - u_2 - u_1u_2 $.
Namely, $ (u_3)_{kk'} $ is the sum of all those
$
	\parityc(\wordkl[\word][k][*]) \cdots \parityc(\wordkl[\word'][*])
	\
	\wordkl[\word][k][*]
	\cdots 
	\wordkl[\word'][*] 
$
in \eqref{e.ustar} that are not already accounted 
by the summand in \eqref{e.isle.u1} or by the summand in \eqref{e.isle.u2} or by concatenating the summands in \eqref{e.isle.u1}--\eqref{e.isle.u2}.
Schematically express $ u_3 $ as
\begin{align}
	\label{e.isle.u3}
	u_3
	=
	\Big( 
	\sideset{}{^{\lli}_{\ldots}}\sum
		\parityc(\wordkl[\word][k][*]) \cdots \parityc(\wordkl[\word'][*])
		\
		\wordkl[\word][k][*]
		\cdots 
		\wordkl[\word'][*]
	\Big)_{k,k'=1}^{\icm}
	\in
	\Mat_{\icm}(\C\angg{\symb}).
\end{align}
Let us note a few useful properties of $ u_3 $.
\begin{enumerate}[leftmargin=20pt, label=(\roman*)]
%\item \label{p.factor.u3.1} The sum in \eqref{e.isle.u3} is nonempty.
\item \label{p.factor.u3.2} When $ k=k_0 $ necessarily $ \wordkl[\word][k][*] \notin \symb_{2} $.
\item \label{p.factor.u3.3} When $ k'=k_0 $ necessarily $ \wordkl[\word'][*] \notin \symb_{1} $.
\item \label{p.factor.u3.4} Whenever the index $ k_0 $ appears in \eqref{e.isle.u3} in the middle, namely
	$
		(\wordkl[\word][k][*] \cdots \wordkl[\word'][*]) = ( \cdots \wordkl[\word''][*][k_0]\, \wordkl[\word'''][k_0][*] \cdots ),
	$
	necessarily $ (\wordkl[\word''][*][k_0],\wordkl[\word'''][k_0][*]) \notin \symb_1 \times \symb_{2} $.
\end{enumerate}
The property~\ref{p.factor.u3.2} holds because if $ \wordkl[\word][*]\in\symb_{2} $,
the condition $ \wordkl[\word'][k][*] \lli \cdots \lli \wordkl[\word'][*]  $ would force all the $ \wordkl[\word][*][*] $'s to be in $ \symb_2 $.
The properties~\ref{p.factor.u3.3}--\ref{p.factor.u3.4} hold for similar reasons.

We now factorize the formal determinant by mimicking the procedures in Example~\ref{ex.u-u},
specifically \eqref{e.islefact.ex1}--\eqref{e.islefact.ex2}.
Set $ v= u_1+u_2+u_1u_2 $ and $ v'=\sum_{i,j\geq 0} (-1)^{i+j}u^i_2 u^j_1 u_3 $.
It is straightforward to check that $ (1+v)(1+v')-1= u_1+u_2+u_1u_2+u_3 $.
Applying Proposition~\ref{p.det.prop}\ref{p.det.factorize} with this choice of $ (u,v) $ gives
$
	\fdet(\Id+u)
	=
	\fdet( \Id + u_1 + u_2 + u_1 u_2 ) \fdot \fdet( \Id + \sum_{i,j\geq 0} (-1)^{i+j}u^i_2 u^j_1 u_3 ).
$
For the first determinant on the right side, further apply Proposition~\ref{p.det.prop}\ref{p.det.factorize} with $ (u,v)\mapsto (u_1,u_2) $.
These factorization gives
\begin{align}
	\label{e.p.isle.fdet}
	\fdet(\Id+u)
	=
	\fdet\big( \Id + u_1 \big) \fdot \fdet\big( \Id + u_2 \big) \fdot \fdet \Big( \Id + \sum\nolimits_{i,j\geq 0} (-1)^{i+j}u^i_2 u^j_1 u_3 \Big).
\end{align}

We now argue that the right side of \eqref{e.p.isle.fdet} does not contain the given $ \monomial $.
For $ \fdet( \Id + u_1 ) $ and $ \fdet( \Id + u_2 ) $, the symbols involved are entirely from $ \symb_{1} $ or $ \symb_{2} $ and hence cannot produce $ \monomial $.
Next, use \eqref{e.isle.u1}--\eqref{e.isle.u3} to express 
\begin{align}
	\label{e.p.isle.type}
	\sum_{i,j\geq 0} (-1)^{i+j}u^i_2 u^j_1 u_3 
	=
	\Big(
		\sum\pm \big(\cdots\textcircled{2}\cdots\big)\big(\cdots\textcircled{1}\cdots\big)\big(\cdots\textcircled{3}\cdots\big)
	\Big)_{k,k'=1}^{\icm}.
\end{align}
\begin{enumerate}[leftmargin=20pt]
\item $ (\cdots\textcircled{1}\cdots) $ is a product of symbols from $ \symb_{1} $, or $ (\cdots\textcircled{1}\cdots)=1 $, which corresponds to $ j=0 $;
\item $ (\cdots\textcircled{2}\cdots) $ is a product of symbols from $ \symb_{2} $, or $ (\cdots\textcircled{2}\cdots)=1 $, which corresponds to $ i=0 $;
\item $ (\cdots\textcircled{3}\cdots) $ is an element from the summand of \eqref{e.isle.u3};
\item the sum on the right side could be empty, in which case $ \sum_{\emptyset} := 0 $.
\end{enumerate}
Hence $ \fdet(\Id+\eqref{e.p.isle.type}) = \sum \alpha_{\monomial_1 \fdot \cdots \fdot \monomial_\ell} \monomial_1 \fdot \cdots \fdot \monomial_\ell $, 
where $ \alpha_{\monomial_1 \fdot \cdots \fdot \monomial_\ell}\in\C $ and each $ \monomial_i $ is of the form
\begin{align*}
	\big(\cdots\textcircled{2}\cdots\big)\big(\cdots\textcircled{1}\cdots\big)\big(\cdots\textcircled{3}\cdots\big)
	\cdots
	\big(\cdots\textcircled{2}\cdots\big)\big(\cdots\textcircled{1}\cdots\big)\big(\cdots\textcircled{3}\cdots\big),
\end{align*}
interpreted with cyclic identification, namely as an element of $ \angg{\symb}/\cyclic $.
It is now readily checked from Properties~\ref{p.factor.u3.2}--\ref{p.factor.u3.4} that no $ \monomial_i $ of this type can be equal to the given $ \monomial $.
\end{proof}

\subsection{Factorization of formal determinants --- proof of Proposition~\ref{p.det.prop}\ref{p.det.factorize}}
\label{s.isle.pf1}

The first step of the proof is to establish the faithfulness of formal determinants.
A skeptic may question our definition of formal determinants.
Modding out cyclic permutations is indeed necessary, but is it sufficient to \emph{faithfully} represent numerical determinants?
The next lemma answers this question affirmatively for two symbols $ \{X,Y\} $,
and the generalization to finitely many symbols is straightforward.
Let $ \Mat_d(\C) $ denote the set of $ d\times d $ matrices over $ \C $.
\begin{lem}
\label{l.det}
For any polynomial $ g\in \C[ \MonomialXY/\cyclic ] $,
if $ g(\tr(A),\tr(B),\tr(A^2),\ldots) = 0 $ for all $ A,B\in\Mat_d(\C) $ and for all $ d\in\Z_{> 0} $, then necessarily $ g=0 $.
\end{lem}
\begin{proof}
We begin with a reduction to linear algebra.
To set up notation, we define the \tdef{$X$ power} and \tdef{$Y$ power} of monomials in $ (\MonomialXY/\cyclic) $ by counting the total numbers of $ X $ and of $ Y $.
For example, the $ X $ power and $ Y $ power of $ \pmonomial = XX\fdot XY \fdot Y $ are respectively $ 3 $ and $ 2 $.
Note that here we are concerned with \emph{polynomials}, namely elements in $ \C[ \MonomialXY/\cyclic ] $.
Fix $ g_*\in\C[ \MonomialXY/\cyclic ] $, and let $ i_* $ and $ j_* $ denote the respective highest $ X $ power and $ Y $ power of monomials in $ g_* $. 
Consider the subset in $ \C[\angg{X,Y}/\cyclic] $:
\begin{align*}
	\calM := \big\{ \pmonomial = \monomial_1 \fdot \cdots \fdot \monomial_\ell  \, 
	\big|
	\monomial_i\in (\angg{X,Y}/\cyclic),
	\
	\pmonomial \text{ has } X \text{ power } \leq i_* \text{ and } Y \text{ power } \leq j_* \}
\end{align*}
and the $ \C $-linear space thus spanned: $ \C\calM = \bigoplus_{\pmonomial\in\calM} \C \pmonomial \subset \C[ \MonomialXY/\cyclic ] $.
Indeed $ g_* \in \C\calM $.
The evaluation $ g \mapsto g(\tr(A),\tr(B),\tr(A^2),\ldots) $ can be viewed as a $ \C $-linear map $ \C\calM \to \C[ A_{ij},B_{ij} ] $.
Namely, for fixed $ d\in\Z_{> 0} $,
we view entries of $ A,B \in \Mat_d(\C) $ as indeterminates, and accordingly the evaluation as a $ \C $-linear map
\begin{align*}
	\varphi:\,
	\C\calM \longrightarrow \C[ A_{ij},B_{ij} ],
	\qquad
	g \longmapsto g\Big( \sum A_{ii}, \sum B_{ii}, \sum A_{ij}A_{ji}, \ldots \Big).
\end{align*}

It suffices to show that $ \varphi $ is injective for a large enough $ d $.
It turns out $ d=i_*+j_* $ suffices, and we fixed such a $ d $.
We will achieve the goal by showing that $ \varphi(\calM) $ is $ \C $-independent.
Fix any $ \pmonomial\in\calM $.
To alleviate heavy notation we will often work with the example $ \pmonomial = XY\fdot XX \fdot Y $,
but the proof applies to general $ \pmonomial $.
Express $ \varphi(\pmonomial) $ as
\begin{align}
	\label{e.l.det.1}
	\varphi(\pmonomial) 
	=  
	\sum A_{i_1i'_1} B_{i'_1 i_1} \sum A_{i_2i'_2} A_{i'_2 i_2} \sum B_{i_3i_3}.
\end{align}
The right side is a linear combination of monomials in $ \C[ A_{ij},B_{ij} ] $.
Namely, letting $ \LAMonomials $ denote the set of monomials in $ \C[ A_{ij},B_{ij} ] $, we have
%\begin{align}
%	\label{e.l.det.2}
$
	\varphi(\pmonomial) 
	=
	\sum_{\LAmonomial\in\LAMonomials} c(\pmonomial,\LAmonomial) \, \LAmonomial,
$
%	\qquad
%	(\vecj,\veck) = (j_i,k_i)_{i=1}^5,
%\end{align}
where the coefficient $ c(\pmonomial,\LAmonomial) \in \Z_{\geq 0} $ counts how many times $ \LAmonomial $ appears in \eqref{e.l.det.1}. %and the sum is finite.
The proof strategy is to identify an $ \LAmonomial_\pmonomial\in\LAMonomials $ for each $ \pmonomial\in\calM $, such that $ c(\pmonomial,\LAmonomial_\pmonomial) \neq 0 $ but $ c(\pmonomial',\LAmonomial_\pmonomial) = 0 $ for all $ \pmonomial'\in\calM\setminus\set{\pmonomial} $.
Once this construction is done for every $ \pmonomial\in\calM $, since $ \LAMonomials \subset\C[ A_{ij},B_{ij} ]  $ is $ \C $-independent,
the desired $ \C $-independence of $ \varphi(\calM) $ follows.

To construct such an $ \LAmonomial_\pmonomial $, we specialize the indices in \eqref{e.l.det.1} into numbers in order, start from $ 1 $, and increase by $ 1 $ each time.
More explicitly, we specialize $ (i_1,i'_1,i_2,i'_2,i_3) \mapsto (1,2,3,4,5) $, which gives rise to $ \LAmonomial_\pmonomial = A_{12}B_{21} A_{34}A_{43} B_{55} $.
The assumption $  i_*+j_* \leq d $ ensures sufficient room to increase the index. %In this example $ 2+3 \leq i_*+j_* \leq d $.
A priori, such a construction does not uniquely specify $ \LAmonomial_\pmonomial $, because there are multiple ways of expressing $ \pmonomial \in \calM\subset \angg{X,Y}/\cyclic $, for example $ \pmonomial = Y\fdot XX \fdot YX = XX \fdot Y \fdot YX $.
We fix an arbitrary way of expressing $ \pmonomial $ so that $ \LAmonomial_\pmonomial:\calM\to\LAMonomials $ is a map.
(This fixing is actually unnecessary but makes the proof cleaner.) 
The condition $ c(\pmonomial,\LAmonomial_\pmonomial) \geq 1 $ holds by construction.
Conversely, given any $ \LAmonomial\in\{ \LAmonomial_\pmonomial: \pmonomial\in\calM \} $, we can identify a \emph{unique} $ \pmonomial\in\calM $ such that $ c(\pmonomial,\LAmonomial) \neq 0 $.
For example for $ \LAmonomial = A_{34}A_{12}B_{55} B_{21} A_{43} $,
we read the numerical indices, start from $ 1 $ and increase by $ 1 $ whenever possible, 
`group' terms together whenever the index loops, and continue onto the next largest index after each grouping.
The result produces the groups $ (A_{12}B_{21}) $, $ (A_{34}A_{43}) $, and $ (B_{55}) $, which tells us that $ c(\pmonomial,\LAmonomial) \neq 0 $ for $ \pmonomial=XY\fdot XX \fdot Y $ and for this $ \pmonomial $ \emph{only}.
We have obtained the desired $ \LAmonomial_\pmonomial $ and hence completed the proof.
\end{proof}

We now prove Proposition~\ref{p.det.prop}\ref{p.det.factorize} for the special case when $ u $ and $ v $ are each a symbol.

\begin{customprop}{\ref{p.det.prop}\ref{p.det.factorize}'}
\label{p.det.}
For any symbols $ X,Y $,
$	
	\fdet(\Id + X) \fdot \fdet(\Id + Y) = \fdet(\Id + (X + Y + XY)). % \fdet(\Id + ((1+X)(1+Y)-1)) =
$
\end{customprop}
\begin{proof}
We begin with a reduction.
Set
$ f:= \fdet(\Id + X) \fdot \fdet(\Id + Y) - \fdet(\Id + X+Y+XY))
\in\C\bkk{\MonomialXY} $ and decompose $ f = \sum_{i,j \geq 0} f_{ij} $, where $ f_{ij} \in \C[ \MonomialXY/\cyclic ] $ is a linear combination of monomials with $ X $ power $ i $ and $ Y $ power $ j $.
Note that, while $f$ is a formal power series and can contain infinitely many monomials, each $f_{ij}$ is a \emph{polynomial}, with finitely many monomials only.
It suffices to show $ f_{ij}=0 $, for all $ i,j \in \Z_{\geq 0} $.

To show $ f_{ij}=0 $, we appeal to the fact that $ f $ evaluates to zero on finite dimensional operators.
Fix $ d\in\Z_{>0} $. %and identify linear operators on $ \C^d $ as $ d\times d $ matrices $ \Mat_d(\C) $.
For any $ A,B\in\Mat_d(\C) $ and complex variables $ z,w\in\C $,
we evaluate $ f $ at $ (zA,wB) $ by substituting $ X\mapsto \tr(zA) $, $ Y\mapsto \tr(wB) $, $ XX\mapsto \tr(z^2A^2) $, $ XY \mapsto \tr(zwAB) $, etc.
In general such evaluation produces an infinite sum, but when the operators $A,B$ are finite-dimensional, the sum is in fact finite.
This is so because the quantity $ \D_n(A) $, given in \eqref{e.simon}, is equal to $ \tr(\wedge^kA) $, which vanishes when $ k>d^2 $ and $ A\in\Mat_d(\C) $.
The evaluation gives
\begin{align*}
	\sum\nolimits_{i,j} z^{i} w^{j} f_{ij}(\tr(A),\tr(B),\tr(A^2),\ldots) 
	=
	\det(\Id + zA) \det(\Id+wB) - \det(\Id + zA+wB+zwAB)
	=
	0.
\end{align*}
Again, the sum is \emph{finite} thanks to the preceding observation.
Viewing the left side as a \emph{polynomial} in $ (z,w)\in\C^2 $, we deduce that
$
	f_{ij}(\tr(A),\tr(B),\tr(A^2),\ldots) = 0,
$
for all $ i,j \in\Z_{\geq 0} $.
This holds for all $ A,B \in \Mat_d(\C) $ and $ d \in \Z_{> 0} $, so by Lemma~\ref{l.det} with $ g=f_{ij} $, we have $f_{ij}=0$.
Given $f_{ij}=0$ for all $i,j\in\Z_{\geq 0}$, it follows that $f=0$.
\end{proof}

%\begin{cor}
%\label{c.det}
%For any $ u,v \in \C\Monomials $,
%$	
%	\fdet(\Id + u) \fdot \fdet(\Id + v) = \fdet(\Id + (u + v + uv)).
%$
%\end{cor}
%
\begin{proof}[Proof of Proposition~\ref{p.det.prop}\ref{p.det.factorize}]
Recall $ \varphi_u $ and its lift from \eqref{e.fdet.phiu}--\eqref{e.fdet.phiu.}.
Let $ \pi: \C\bkk{\MonomialXY} \to \C\bkk{\MonomialXY/\cyclic} $ denote the projection.
Proposition~\ref{p.det.} asserts that,
\begin{align}
	\label{e.det.factorize}
	\pi\big( \fdet(\Id+X) \fdot \fdet(\Id+Y) - \bar\varphi_{X+Y+XY}(\fdet(\Id+X_0)) \big) = 0,
\end{align}
where $ X,Y,X_0 $ are distinct symbols,  $ \bar\varphi_{X+Y+XY}: \C\bkk{\angg{X_0}} \to \C\bkk{\MonomialXY} $. Our goal is to convert \eqref{e.det.factorize} into its counterpart that concerns $ u,v $. To this end, consider the algebra homomorphism $ \varphi_{u,v}: \C\MonomialXY \longrightarrow \C\Monomials $,
\begin{align*}
	\underbrace{X\cdots X}_{n} \underbrace{Y\cdots Y}_{n'} \cdots
	\mapsto
	\sum u_{k_0k_1} \cdots u_{k_{n-1}k_n} v_{k_n k_{n+1}} \cdots v_{k_{n+n'-1}k_{n+n'}} \cdots,
\end{align*}
where the last index of the summand is $ k_0 $, and the sum runs through $ k_0,k_1,\ldots \in 1,\ldots,\icm $.
It is readily checked that $ \C\angg{X} \xrightarrow{\text{inclusion}} \C\MonomialXY \xrightarrow{\varphi_{u}} \C\Monomials $ gives the same map as $ \C\angg{X} \xrightarrow{\varphi_{u,v}}\C\Monomials $. Lifting this relation gives the first of the following commutative diagrams.
Similar considerations produce the other two diagrams.
\begin{center}
\begin{tikzcd}
	\C\bkk{\angg{X}} \arrow[r,"\text{inclusion}"]  \arrow[rd,"\bar\varphi_{u}"] 
 	&\C\bkk{\MonomialXY}  \arrow[d,"\bar\varphi_{u,v}"] \\
	&\C\bkk{\Monomials}
\end{tikzcd}
\hfil
\begin{tikzcd}
	\C\bkk{\angg{Y}} \arrow[r,"\text{inclusion}"]  \arrow[rd,"\bar\varphi_{v}"] 
 	&\C\bkk{\MonomialXY}  \arrow[d,"\bar\varphi_{u,v}"] \\
	&\C\bkk{\Monomials}
\end{tikzcd}
\hfil
\begin{tikzcd}
	\C\bkk{\angg{X_0}} \arrow[r,"\bar\varphi_{X+Y+XY}"]  \arrow[rd,"\bar\varphi_{u+v+uv}"] 
	&\C\bkk{\MonomialXY}  \arrow[d,"\bar\varphi_{u,v}"] \\
	&\C\bkk{\Monomials}
\end{tikzcd}

\begin{tikzcd}
	\C\bkk{\MonomialXY}  \arrow[r, "\bar\varphi_{u,v}"] \arrow[d, "\pi"] & \C\bkk{\Monomials} \arrow[d, "\pi'" ] 
\\
	\C\bkk{\MonomialXY/\cyclic}  \arrow[r, "\psi_{u,v}"] 
	& \C\bkk{\Monomials/\cyclic}
\end{tikzcd}
\end{center}
In the above diagram, it is readily checked that $ \ker(\pi) \subset \ker(\pi'\circ \bar\varphi_{u,v}) $, so there exists $ \psi_{u,v} $ that makes the diagram commute.
Applying $ \psi_{u,v} $ to \eqref{e.det.factorize} and using the commutative diagrams conclude the proof.
\end{proof}

\section{Determinantal analysis: asymptotics of the traces}
\label{s.asymptotics}

Here we extract the $ N \to \infty $ rate of a generic preferred trace, and identify the dominant term in \eqref{e.det.expanded}.

Recall the assumptions \eqref{e.HLc.xa} and \eqref{e.nondeg.sd} imposed on $ (\vecx,\veca) $.
For the purpose of the steepest descent analysis to be performed later,
it is convenient to impose a slightly stricter version of those assumptions.
Doing so loses access to some borderline configurations only, and the losses will be compensated by the approximation arguments in Section~\ref{s.pftrw}.
Throughout Section~\ref{s.asymptotics} the following assumption is in action.
\begin{assu}
\label{assu.strict}
Under Convention~\ref{con.wordkl}.
\begin{enumerate}[leftmargin=20pt, label=(\alph*)]
\item 
For all $ i $ and $ k $,
$ s_1 < \ldots < s_m $, \ $ d_1 > \ldots > d_m $, \ 
$ s_i \neq \ics_{k} $, \ $ d_i \neq \icd_{k} $, \\
$
	(x_i,a_i) \in \hyp( \parab[1\cdots\icm](0) )^\circ \setminus \hyp( \parab[1\cdots\icm](t)).
$
\item \label{assu.slopes1}
For all all $ \wordkl $ and $ j\in\word $, $ \partial_y \rwf(y)|_{y=x_j^-} \neq \partial_y \rwf(y)|_{y=x_j^+}  $.
\item \label{assu.slopes2}
For all $ \wordkl[\word][k_1][k_2] $ and $ \wordkl[\wordsub][k_2][k_3] $,
$ {\xr}(\wordkl[\word][k_1][k_2]) \neq {\xl}(\wordkl[\wordsub][k_2][k_3]) $.
\end{enumerate}
\end{assu}

\subsection{Critical points and critical values}
\label{s.asymptotics.contour}
Consider a preferred trace, defined in Definition~\ref{d.prefer},
\begin{align}
	\label{e.ptrace}
	\monomial
	=
	\tr\big( \inddown_{\ick_n} \opuci{ \word^{(1)} } \inddown_{\ick_1} \opuci{ \word^{(2)} } \cdots \inddown_{\ick_{n-1}} \opuci{ \word^{(n)} } \inddown_{\ick_n} \big).
\end{align}
This trace permits a contour integral expression that will be spelled out in Section~\ref{s.a.steep}.
The integrand is a product of the functions $ \fnSl_{ki}, \fnSr_{i'k'},\fnrw_{ii'},\fnrw_{\ick\ick'} $ given in \eqref{e.fncontour} and some rational functions.
The rational functions are $ N $-independent and hence do not directly contribute to the $ N\to\infty $ asymptotics.

To extract the asymptotics of $ \fnSl_{ki}, \fnSr_{i'k'},\fnrw_{ii'},\fnrw_{\ick\ick'} $, set
\begin{subequations}
\label{e.steep.expfn}
\begin{align}
	\label{e.steep.expfn.L}
	\expSl(z;t,s,d) &:= s \log(2-z) - d \log z + \tfrac{t}{2}(1-z),&
	\expSl_{k i}(z) &:= \expSl(z;t,-\ics_{k}+s_i,-\icd_{k}+d_i)
\\
	\label{e.steep.expfn.R}
	\expSr(z;t,s,d) &:= -s \log z + d \log (2-z) + \tfrac{t}{2}(1-z),&
	\expSr_{k'{i'}}(z) &:= \expSr(z;t,s_{{i'}}-\ics_{k'},d_{i'}-\icd_{k'}),
\\
	\label{e.steep.expfn.rw}
	\exprw(z;s,d) &:= s \log(2-z) + d \log z,&
	\exprw_{i{i'}}(z) &:= \exprw(z;-s_{i}+s_{{i'}},d_i-d_{i'}),
\\
	\label{e.steep.expfn.rw.}
	&&
	\exprw_{\ick{\ick'}}(z) &:= \exprw(z;-\ics_{k}+\ics_{{k'}},\icd_k-\icd_{k'}),
\end{align}
\end{subequations}
so that
\begin{subequations}
\label{e.steep.exp}
\begin{align}
%\label{e.steep.exp}
	\fnSl_{k i}(z) &= z^{-1} e^{-N\expSl_{k i}(z)},
	& 
	\fnSr_{{i'}k'}(z) &= (2-z)^{-1} e^{-N\expSr_{k'{i'}}(z)},
\\
%\label{e.steep.exp.}
	\fnrw_{i{i'}}(z) &= z^{-1} e^{-N\exprw_{i{i'}}(z)},
	& 
	\fnrw_{\ick{\ick'}}(z) &= z^{-1} e^{-N\exprw_{\ick{\ick'}}(z)}.
\end{align}
\end{subequations}
Such expressions are suitable for steepest descent.
To proceed, differentiate \eqref{e.steep.expfn} in $ z $ to obtain the critical points:
\begin{subequations}
\label{e.zs}
\begin{align}
	\label{e.zls}
	\zls = \zls(t,s,d) &:= \tfrac{1}{t} ( t+s-d - \sqrt{ (t+s-d)^2 +4 td } ),
\\
	\label{e.zrs}
	\zrs =  \zrs(t,s,d) &:= \tfrac{1}{t} ( t-s+d - \sqrt{ (t-s+d)^2 +4 ts } ),
\\
	\label{e.zrws}
	\zrws = \zrws(s,d) &:= \tfrac{2d}{s+d}.
\end{align}
\end{subequations}
Specializing the $ z $'s and $ \Phi $'s into the relevant parameters gives the critical points and critical values:
\begin{align}
	\label{e.zs.specialized}
	\zls[k i] &:= \zls(t,-\ics_{k}+s_i,-\icd_{k}+d_i),
	&
	\zrs[{i'}k'] &:= \zrs(t,s_{{i'}}-\ics_{k'},d_{i'}-\icd_{k'}),
\\
	\label{e.zs.specialized.}
	\zrws[i{i'}] &:= \zrws(-s_{i}+s_{{i'}},d_i-d_{i'}),
	&
	\zrws[\ick{\ick'}] &:= \zrws(-\ics_{k}+\ics_{{k'}},\icd_k-\icd_{k'}),	
\\
	\notag
	\expSls[k i] &:= \expSl_{k i}(\zls[k i]),\qquad
	\expSrs[{i'}k'] := \expSr_{{i'}k'}(\zrs[{i'}k']),
	&
	\exprws[i{i'}] &:= \exprw_{i{i'}}(\zrws[i{i'}]), \qquad
	\exprws[\ick{\ick'}] := \exprw_{\ick{\ick'}}(\zrws[\ick{\ick'}]).
\end{align}
Next, according to the nested structure of the contour integrals in Definitions~\ref{d.opu} and \ref{d.opu.extended}, set
\begin{itemize}[leftmargin=10pt]
\item[$ \circ $]
	$ \expfn_\star(\wordkl[\wordsub]) := \expSls[k \wordsub_1] + \exprws[\wordsub_1\wordsub_2] + \cdots + \exprws[\wordsub_{\ell-1}\wordsub_n] + \expSrs[\wordsub_\ell k'] $,
	where $ \wordkl[\wordsub]\in\islesetkl $ and $ \ell := |\wordsub| >0 $,
\\
	$ \expfn_\star(\wordkl[\emptyset]) := \exprws[\ick\ick'] $, where $ k<k' $,
\item[$ \circ $] $ \expfnword[\wordkl[\word][k_0][k_\ell] ] := \sum_{i=1}^{\ell'} \expfn_\star(\wordkl[\wordsub^{(i)}][k_{i-1}][k_i]) $, where 
$ \wordkl[\word][k_0][k_{\ell'}] = \wordkl[\wordsub^{(1)}][k_0][k_1] \cupi \ldots \cupi \wordkl[\wordsub^{(\ell')}][k_{\ell-1}][k_{\ell'}] $, and
\item[$ \circ $] $ \expfnword[\monomial] := \sum_{i=1}^n \expfnword[\wordkl[{\word^{(i)}}][k_{i-1}][k_i]] $, where $ \monomial $ and $ n $ are as in \eqref{e.ptrace}.
\end{itemize}
Recall that $ \norm{\monomial} $ counts the total number of $ \opuci{\ldots} $ involved in $ \monomial $.

\begin{prop}
\label{p.trace}
Under Assumption~\ref{assu.strict}, there exists a $ c=c(t,\icvecx,\icveca,\vecx,\veca)<\infty $ such that for any preferred trace $ \monomial $,
\begin{align*}
	|\monomial| \leq e^{c\norm{\monomial}} \, \exp\big( -N \expfnword[\monomial]  \big),
\end{align*}
and for a fixed preferred trace $ \monomial $ there exists a $ c=c(t,\icvecx,\icveca,\vecx,\veca,\monomial)<\infty $ such that
\begin{align*}
	\monomial \geq (N+1)^{-c} \exp\big( -N\expfnword[\monomial]  \big).
\end{align*}
\end{prop}
\noindent{}%
Proposition~\ref{p.trace} is proven in Section~\ref{s.a.steep} by steepest descent.

\subsection{The geometric/probabilistic meanings of the critical points and critical values}
\label{s.asymptotics.critical}
We begin with $ \zrws[i{i'}] $ and $ \exprws[i{i'}] $.
Let $\rategeo(v) := v\log v + (v+1) \log(2/(1+v)) $, $ v\in[0,1] $, denote the rate function of the sum of i.i.d.\ $ \Geo(2) $ variables.
Substituting in $ z=\zrws[i{i'}] = \frac{2(d_i-d_{i'})}{-s_i+s_{i'}+d_i-d_{i'}} $ in \eqref{e.steep.exp} gives $ \exprws[i{i'}] = (-s_i+s_{i'}) \rategeo(\tfrac{d_i-d_{i'}}{-s_i+s_{i'}}) $, which is consistent with $ \rwop^{n} $ being the transition kernel of a geometric walk.
Translating these identities into the $ (x,a) $ coordinate gives
\leavevmode
\begin{align}
	\label{e.geomeaning.rw}
	\zrws[i{i'}] = 1 - \tfrac{-a_i+a_{i'}}{-x_i+x_{i'}} = 1 - (\text{slope}_{ii'}), 
	\qquad
	\exprws[i{i'}] %= (-x_i+x_{i'}) \rateber( \tfrac{-a_i+a_{i'}}{-x_i+x_{i'}} ) 
	= (-x_i+x_{i'}) \rateber(\text{slope}_{ii'}),
\end{align}
where $($slope$ {}_{ii'}) $ denotes the slope, in the $ (x,a) $ coordinate system, of the line joining $ (x_i,a_i) $ and $ (x_{i'},a_{i'}) $.
%The last identity can be obtained by direct calculations.
%Alternatively, one can use the fact that, a geometric walk (with increments $ \sim \Geo(2) $) in the $ (s,d) $ system
%is equivalent to a Bernoulli walk (with unbiased $ \pm 1 $-valued increments) in the $ (x,a) $ system.

Next we turn to $ \zls[k i], \expSls[k i] $ and $ \zrs[{i'}k'], \expSrs[{i'}k'] $.
Fix any $ d_i<\icd_{k} $ and $ s_{i'} < \ics_{k'} $.
Consider the configuration depicted in Figures~\ref{f.zls-interpretation}--\ref{f.zrs-interpretation}.
Let $($slope$ {}_\triangleleft) $ denote the slope of the straight line between $ x_i $ and $ \xl $ in Figure~\ref{f.zls-interpretation}, and similar for $($slope$ {}_\triangleright) $.
Straightforward (though tedious) calculations from \eqref{e.steep.expfn.L}--\eqref{e.steep.expfn.R} and \eqref{e.zls}--\eqref{e.zrs} gives
\begin{align}
	\label{e.geomeaning.sl}
	\zls[k i] &= 1 -(\text{slope}_\triangleleft),
	&
	\expSls[k i] &= \int_{\xl}^{x_i} \d y\, \rateber(\text{slope}_\triangleleft) - \int_{\xl}^{\icx_{k}} \d y \rateber( \partial_y \parab[k](t,y) ),
\\
	\label{e.geomeaning.sr}
	\zrs[{i'}k'] &= 1 +(\text{slope}_\triangleright),
	&
	\expSrs[{i'}k'] &= \int_{x_{i'}}^{\xr} \d y\, \rateber(\text{slope}_\triangleright) - \int_{\icx_{k'}}^{\xr} \d y \rateber( \partial_y \parab[k'](t,y) ),
\end{align}
where the integrals assume the usual sign convention: $ \int_a^b = - \int_b^a $.
The rates $ \expSls[k i] $ and $ \expSrs[{i'}k'] $ can be geometrically encoded as follows. 
Evaluate the integral of $ \rateber( \partial_y\, \Cdot\, ) $ along the solid curves depicted in Figures~\ref{f.zls-interpretation}--\ref{f.zrs-interpretation},
sign each right-going segment positive, and sign each left-going segment negative.

Based on the preceding discussion, we next give a geometric description of $ \expfnword[\monomial] $, where $ \monomial $ is a preferred trace.
Recall $ \xl(\wordkl[\wordsub]) $ and $ \xr(\wordkl[\wordsub]) $ from \eqref{e.xL.xR}.
First, for any $ \wordkl[\wordsub]\in\islesetkl $, the rate $ \expfn_\star(\wordkl[\wordsub]) $ is obtained by evaluating $ \raterw(\partial_y\Cdot) $ along the curve depicted in Figure~\ref{f.rate-v}, signed positive along right-going segments and negative along left-going segments.
Next, consider $ \wordkl = \wordkl[\wordsub^{(1)}][k][k_1] \cupi \ldots \cupi \wordkl[\wordsub^{(n)}][k_{n-1}][k'] $, where each $ \wordsub^{(i)} $ is an isle.
Since consecutive isles are $ \lli $-ordered here, by definition we have 
$ \xr(\wordkl[\wordsub^{(i)}][k_{i-1}][k_{i}]) < \xl(\wordkl[\wordsub^{(i+1)}][k_{i}][k_{i+1}]) $.
This being the case, the curves can be joined together to form a longer curve; see Figure~\ref{f.join} for an example.
Hence the rate $ \expfnword[\wordkl] $ is also described by Figure~\ref{f.rate-v}
upon the replacement $ (\rwf[\wordkl[\wordsub]], \xl(\wordkl[\wordsub]), \xr(\wordkl[\wordsub])) \mapsto (\rwf, \xl(\wordkl), \xr(\wordkl)) $.
A minor difference here is that $ \rwf $ may touch some $ \parab[\ick''](t) $ in between $ [\xl(\wordkl), \xr(\wordkl)] $, since $ \wordkl $ is in general not an isle.
Next, consider the preferred trace $ \monomial $ as in \eqref{e.ptrace}.
For each $ \wordkl[\word^{(i)}][k_{i-1}][k_i] $ therein, set
\begin{align}
\label{e.graph.notation}	
	\xl_{i} &:= \xl(\wordkl[\word^{(i)}][k_{i-1}][k_i]),
	\quad
	\xr_{i} := \xr(\wordkl[\word^{(i)}][k_{i-1}][k_i]),
	\quad
	F_i := \rwf[ \wordkl[\word^{(i)}][k_{i-1}][k_i] ].
\end{align}
The rate $ \expfnword[\monomial] $ is similarly obtained by joining curves.
Recall that $ \monomial $ being preferred means $ \wordkl[\word^{(i)}][k_{i-1}][k_{i}] \not\lli \wordkl[\word^{(i+1)}][k_{i}][k_{i+1}] $
for all $ i $, under the cyclic convention.
Hence, unlike in the preceding, the joined curve must backtrack between each $ \xr_{i} $ and $ \xl_{i+1} $.
Also, by the cyclic nature of $ \monomial $ the joined curve is closed; see Figure~\ref{f.rate-trace} for an example.
Set $ \raterw(f|_{U}) := \int_U\d y\,\raterw(\partial_y f) $.
Formally,
\begin{align*}
%	\label{e.exp.graph}
	\expfnword[\monomial]
	=
	\raterw(F_1|_{[\xl_1,\xr_1]}) - \raterw( \parab[k_1](t)|_{[\xl_2,\xr_1]} ) + \raterw(F_2|_{[\xl_2,\xr_2]}) - \ldots 
%	+ \raterw(F_n|_{[\xl_{n},\xr_{n}]}) 
	- \raterw( \parab[k_n](t)|_{[\xr_n,\xl_1]} ).
\end{align*}

\begin{figure}
\centering
\begin{minipage}[t]{.49\linewidth}
	\begin{minipage}[b]{\linewidth}
	\centering
	\frame{\includegraphics[width=.495\linewidth]{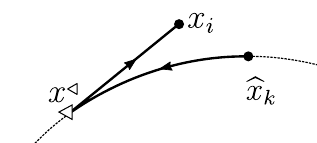}}\hfill%
	\frame{\includegraphics[width=.495\linewidth]{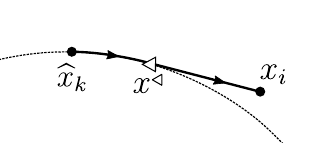}}
	\end{minipage}
	\caption{The geometry for \eqref{e.geomeaning.sl}. %
	The dashed line is the graph of $ \parab[k](t) $. %
	A straight line connects $ (x_i,a_i) $ to $ \parab[k](t) $ and hits the latter tangentially at $ y=\xl $. %
	The cases $ \xl< \icx_{k} $ and $ \icx_{k}<\xl $ are shown. %
	}
	\label{f.zls-interpretation}%
\end{minipage}
\hfill
\begin{minipage}[t]{.49\linewidth}
	\begin{minipage}[b]{\linewidth}
	\centering
	\frame{\includegraphics[width=.495\linewidth]{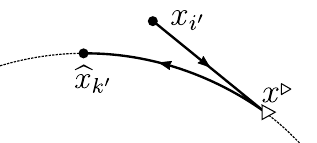}}\hfill%
	\frame{\includegraphics[width=.495\linewidth]{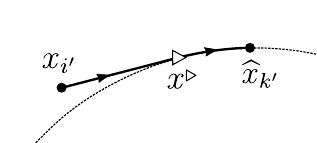}}
	\end{minipage}
	\caption{The geometry for \eqref{e.geomeaning.sr}. %
	The dashed line is the graph of $ \parab[k'](t) $. %
	A straight line connects $ (x_{i'},a_{i'}) $ to $ \parab[k'](t) $ and hits the latter tangentially at $ y=\xr $. %
	The cases $ \icx_{k'}< \xr $ and $ \xr<\icx_{k'} $ are shown. %
	}
	\label{f.zrs-interpretation}%
\end{minipage}
\\
\begin{minipage}[t]{.35\linewidth}
\frame{\includegraphics[height=78pt]{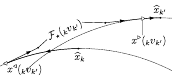}}
\caption{The curve for $ \expfn_\star(\wordkl[\wordsub]) $. The dashed curves are $ \parab[k](t) $ and $ \parab[k'](t) $.}
\label{f.rate-v}
\end{minipage}
\hfill
\begin{minipage}[t]{.63\linewidth}
	\frame{\includegraphics[height=78pt]{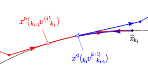}}
	\hfill
	\frame{\includegraphics[height=78pt]{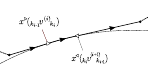}}
	\caption{Joining the curves.
	In the left figure, the blue and red curves cancel each other between $ \xl(\wordkl[\wordsub^{(i+1)}][k_{i}][k_{i+1}]) $ and $ \icx_{k_i} $. %
	}
	\label{f.join}%
\end{minipage}\\
\frame{\includegraphics[width=.9\linewidth]{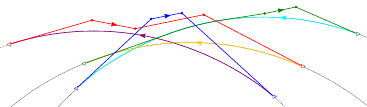}}
\caption{An example of the curve of $ \expfnword[\monomial] $. %
Here $ \icm=3 $, $ n=3 $, $ k_3=1 $, $ k_1=2 $, and $ k_2=3 $. %
The segments of the curve that coincide with $ F_1 $, $ \parab[k_1](t) $, $ F_2 $, $ \parab[k_2](t) $, $ F_3 $, and $ \parab[k_3](t) $
are respectively colored red, yellow, green, cyan, blue, and purple.}
\vspace{-10pt}
\label{f.rate-trace}
\end{figure}

\subsection{Identifying the dominant term}
\label{s.asymptotics.dominant}  
The next task is to identify the candidates for the dominant term in \eqref{e.det.expanded}.
By Proposition~\ref{p.trace}, this task boils down to minimizing the quantity $ \expfnword[\monomial] $ among preferred traces $ \monomial $.

The main step is to argue that $ \expfnword[\monomial] $ can be bounded below by rates that involve $ k=k' $ only.
Referring to the description of $ \expfnword[\monomial] $ in \eqref{e.graph.notation}, we see that the rate generally involves curves that start and end with different $ k $'s, namely $ k_{i-1}\neq k_{i} $.
Proposition~\ref{p.nondiag} below asserts that such a rate can be bounded below by those that involve only $ k_{i-1}= k_{i} $.
Let us prepare some notation.
Recall that $ f $ passes through a letter $ j $ if $ f(x_j)=a_j $
and that $ \wordF(f) \in \wordset $ denotes the words formed by all letters that $ f $ passes through.
Let
\begin{align}
	\label{e.HLspk}
	\HLspk(t;\icx,\ica) := \big\{ f\in\Lip : \parab[\icx,\ica](t) \leq f \leq \parab[\icx,\ica](0),\, f(x_j) \leq a_j, \,\forall j \big\},
	\quad
	\HLspk(k) := \HLspk(t;\icx_{k},\ica_{k}).
\end{align}
This space depends on $ (\vecx,\veca) $, but we omit the dependence in the notation.
The condition $  f(x_j) \leq a_j $ for all $ j $ is relevant because it is implied by the condition $ \max_{k} f_{k}(x_j) = a_j $  in \eqref{e.raterw.xa}.

\begin{prop}
\label{p.nondiag}
There exist finite sets $ \HLspkk(k) \subset \HLspk(k) \setminus \{ \parab[k](t) \} $, $ k=1,\ldots,\icm $, 
that depend only on $ (t,\icvecx,\icveca,\vecx,\veca) $ such that the following holds.
\begin{enumerate}[leftmargin=20pt, label=(a)]
\item \label{p.nondiag.1}
Given any preferred trace $ \monomial $ and $ k_{i}, \word^{(i)} $ as in~\eqref{e.ptrace},
there exists $ \til{F}_i \in \HLspkk(k_i) $ such that $ \wordF(\til{F}_1) \cup \ldots \cup \wordF(\til{F}_n) \supset \word^{(1)}\cup \ldots \cup \word^{(n)} $
and that
$
	\expfnword[\monomial]
	=
	\raterwRel{ \til{F}_{1} }{ \parab[k_1](t) } + \ldots +\raterwRel{ \til{F}_{n} }{ \parab[k_n](t) }.
$
\end{enumerate}
\end{prop}
\noindent{}%
This proposition immediately implies 
\begin{align}
	\label{e.p.nondiag}
	\expfnword[\monomial]
	\geq 
	\min \Big\{ \sum_{i=1}^n \raterwREl{ f_{1} }{ \parab[k_i](t) } \,:\, 
	%\raterwREl{ f_{1} }{ \parab[k_1](t) } + \ldots +\raterwREl{ f_{n} }{ \parab[k_n](t) } :  
	f_i \in \HLspk(k_i), \ \bigcup_{i=1}^n \wordF(f_i) \supset \bigcup_{i=1}^n \wordF(\word^{(i)}) \Big\}.
\end{align}
\begin{customprop}{\ref*{p.nondiag}}[continued]
\begin{enumerate}[leftmargin=20pt, label=(b)]
\item[]
\item \label{p.nondiag.2} The inequality in \eqref{e.p.nondiag} is strict unless $ k_1=k_2=\ldots=k_n $.
\end{enumerate}
\end{customprop}

\begin{proof}
The strategy is to apply a downward induction on $ n $ to the $ (F_i,\xl_i,\xr_i,k_i)_{i=1}^n $ given in \eqref{e.graph.notation}.
For the induction to work, however, we need to consider a \emph{broader} class of data.
Consider $ (F_i,\xl_i,\xr_i,k_i)_{i=1}^n \in (\Lip\times\R^2\times\{1,\ldots,\icm\})^n $ such that,
under the cyclic convention $ \xl_{n+1}:=\xl_n $, $ k_{0} := k_n $, etc.,
\begin{enumerate}[leftmargin=10pt]
\item[$ \circ $] $ F_i $ satisfies \eqref{e.rwfn.space} with $ (k,k')\mapsto(k_{i-1},k_{i}) $, 
\item[$ \circ $] $ \wordF(F_i)\neq\emptyset $ whenever $ k_{i-1}\geq k_{i} $, 
\item[$ \circ $] $ (x_j,a_j) \notin \hyp(F_i)^\circ $\, for all $ i $ and $ j $,
\item[$ \circ $] $ \xl_i = \inf\{ y\in\R : F_i(y)\neq \parab[k_{i-1}](t) \} $,
	$ \xr_i = \sup\{ y\in\R : F_i(y)\neq \parab[k_{i}](t) \} $, and
\item[$ \circ $] $ \xl_i<\xr_i $, $ \xl_{i+1} < \xr_i $. 
\end{enumerate}
Note that we do \emph{not} assume $ F_i $ to be tangent to $ \parab[k_i](t) $ at $ \xl_i $ or tangent to $ \parab[k_{i+1}](t) $ at $ \xr_i $.
The specific data in \eqref{e.graph.notation} satisfies these conditions.
In particular, $ (x_j,a_j) \neq \hyp(F_i)^\circ $ follows from the $ \IHC[k_{i-1}k_{i}] $ satisfied by $ \word^{(i)} $.

The starting point of the induction is to identify an intersection of two $ F_i $'s.
Fix any data $ (F_i,\xl_i,\xr_i,k_i)_{i=1}^n $ as described previously.
By the cyclic nature of this data, we assume without loss of generality that $ k_1 $ is the minimum of the $ k_i $'s.
In particular, $ k_1 \leq k_2 $ and $ k_n \geq k_1 $, and hence $ \wordF(F_1) \neq \emptyset $.
Let $ \letterlast $ denote the last letter in $ \wordF(F_1) $.
We would like to identify an intersection of $ F_2 $ and $ F_1|_{[x_{\letterlast},\xr_1]} $.
To this end we divide $ \R^2=\{(x,a)\} $ into two regions. 
Let $ L\in\Lip $ be the function that matches $ F_1 $ on $ y \in [x_{\letterlast},\infty) $ and extends linearly in $ y\in (-\infty,x_{\letterlast}) $.
We refer to `$ \in\hyp(L) $' as being \tdef{inside of $ L $} and `$ \in\R^2 \setminus \hyp(L) $' as being \tdef{outside of $ L $}.
First, the condition $ \xl_2 < \xr_1 $ implies that $ (\xl_2,F_2(\xl_2)) =  (\xl_2,\parab[k_1](t,\xl_2)) $ is inside of $ L $. See Figure~\ref{f.L}.
\begin{description}[leftmargin=10pt]
\item[Case 1] $ (\xr_2,F_2(\xr_2)) $ is outside of $ L $. See Figure~\ref{f.case1}.\\
In this case $ F_2 $ must intersect with $ L $ within $ y\in(-\infty,\xr_1) $.
\item[Case 2] $ (\xr_2,F_2(\xr_2)) $ is inside of $ L $. See Figure~\ref{f.case2}.\\
We appeal to $ F_2|_{[\xr_2,\infty)} = \parab[k_2](t)|_{[\xr_2,\infty)} $.
Using~\eqref{e.rwfn.space.>n} for $ (f,k,k')\mapsto (F_2,k_1,k_2) $ gives $ F_2(\xr_2) = \parab[k_2](t,\xr_2) \geq \parab[k_1](t,\xr_2) $.
Note that, if $ k_1<k_2 $, the functions $ \parab[k_1](t) $ and $ \parab[k_2](t) $ intersect at exactly one point, to the left of the intersection we have $ \parab[k_1](t) > \parab[k_2](t) $, and to the right of the intersection we have $ \parab[k_1](t) < \parab[k_2](t) $.
Combining these observations with $ \parab[k_2](t,\xr_2) \geq \parab[k_1](t,\xr_2) $ gives
$
	F_2|_{(\xr_2,\infty)} = \parab[k_2](t)|_{(\xr_2,\infty)} \geq \parab[k_1](t)|_{(\xr_2,\infty)} \text{ if } k_1< k_2,
$
and indeed $ F_2|_{(\xr_2,\infty)} = \parab[k_2](t)|_{(\xr_2,\infty)} = \parab[k_1](t)|_{(\xr_2,\infty)} $ if $ k_1= k_2 $.
Given this property, if $ (\xr_2,F_2(\xr_2)) $ is inside of $ L $, 
then $ F_2 $ and $ L $ must intersect within $ y\in(-\infty,\xr_1) $ when $ k_1<k_2 $,
and intersect at $ y=\xl_1 $ when $ k_1=k_2 $.
\end{description}
Pick the rightmost intersection of $ F_2 $ and $ L $ within $ y\in(-\infty,\xr_1] $ and let $ x_\circ $ be the horizontal coordinate of the intersection.
In Case~2, when $ k_1=k_2 $, necessarily $ x_\circ=\xl_1 $.
The condition $ (x_{\letterlast},a_{\letterlast}) \notin \hyp( F_2 )^\circ $ implies $ x_\circ \geq x_{\letterlast} $.
Since $ L = F_1 $ within $ y\in[x_{\letterlast},\xr_1] $, the point $ x_\circ $ is also an intersection of $ F_2 $ and $ F_1 $.
To summarize
\begin{description}
\item [Case~1] $ F_1 $ and $ F_2 $ intersect at $ y=x_\circ \in [x_{\letterfirst},\xr_1) $.
\item[Case~2 and $ \boldsymbol{k_1<k_2} $] $ F_1 $ and $ F_2 $ intersect at $ y=x_\circ \in [x_{\letterfirst},\xr_1) $.
\item[Case~2 and $ \boldsymbol{k_1=k_2} $] $ F_1 $ and $ F_2 $ intersect at $ y=x_\circ=\xr_1 $.
\end{description}

To further the induction, follow Figures~\ref{f.case1.rewired} and \ref{f.case2.rewired} to `rewire' $ F_1 $ and $ F_2 $.
Doing so produces $ F'_1 $ and $ \til{F}_2 $ with
\begin{align*}
	&\raterw(F_1|_{[\xl_1,\xr_1]}) - \raterw( \parab[k_1](t)|_{[\xl_2,\xr_1]} ) 
	+ \raterw(F_2|_{[\xl_2,\xr_2]}) 
	+ \ldots 
	- \raterw( \parab[k_n](t)|_{[\xr_n,\xl_1]} )
\\	
	=&
	\raterwREl{ \til{F}_2 }{ \parab[k_1](t) }
	+
	\Big( 
		\raterw(F'_1|_{[\xl_1,{\xr_1}']}) - \raterw( \parab[k_2](t)|_{[\xl_3,{\xr_1}']} ) 
%		+ \raterw(F_3|_{[\xl_3,\xr_3]}) 
		+ \ldots 
		- \raterw( \parab[k_n](t)|_{[\xr_n,\xl_1]} )		
	\Big),
\end{align*}
where $ {\xr_1}':=\xr_2 $ in Case~1 and $ {\xr_1}':=x_\circ $ in Cases~2.
By construction, the data $ (F'_1,F_3,\ldots,\xl_1,{\xr_1}',\xl_3,\xr_3,\ldots ,k_2,k_3,\ldots) $
satisfies the required conditions and has length $ n-1 $.
The downward induction hence continues until reaching $ n=1 $.
This induction takes any preferred trace $ \monomial $ as an input and outputs
$ \til{F}_1,\ldots,\til{F}_n $ such that $ \expfnword[\monomial] = \raterwRel{ \til{F}_{1} }{ \parab[k_1](t) } + \ldots +\raterwRel{ \til{F}_{n} }{ \parab[k_n](t) }. $
By construction, each $ \til{F}_k \in \HLspk(k) \setminus \{ \parab[k](t) \} $.

\ref{p.nondiag.1}\
It suffices to show that the $ \til{F}_k $'s can be chosen from a finite set.
Even though the induction takes any preferred trace $ \monomial $ as an input,
it actually suffices to consider those inputs where $ k_1,\ldots,k_n $ are distinct.
To see why, view $ k_0, k_1,k_2,\ldots $ as a list of integers under the cyclic identification $ k_{i} := k_{i \mod n} $.
We say such a list has repeated numbers if $ k_i=k_j $ for some $ i\neq j \mod n $.
The given list may have repeated numbers, but can always be decomposed into 
$\ldots$, $k_{n_1}$, $k_{n_1+1}$, $\ldots$, $k_{n_2-1}$, $k_{n_2}$, $k_{n_2+1}$, $\ldots$ ,$k_{n_3-1}$, $k_{n_3}$, $\ldots$,
where each $ \{k_{n_i},\ldots,k_{n_{i+1}}-1\} $ have no repeated numbers and $ k_{n_i}=k_{n_{i+1}} $.
This observation reduces the inputs to those with distinct  $ k_1,\ldots,k_n $.
There are only finitely many such inputs once $ (t,\icvecx,\icveca,\vecx,\veca) $ is fixed,
and hence the resulting $ \til{F}_k $'s form a finite set.

\ref{p.nondiag.2}\
Run the induction for any preferred trace $ \monomial $, regardless of whether $ k_1,\ldots,k_n $ are distinct.
In each step of the induction,
we seek to perform a surgery on $ \til{F}_2 $ in a small neighborhood of $ x_\circ $ to reduce $ \raterwRel{ \til{F}_2 }{ \parab[k_1](t) } $.
If $ k_1<k_2 $,
in both Cases~1 and 2, $ x_\circ \in [ x_{\letterlast},\xl_1) $, which implies $ F_2(x_\circ) = F_1(x_\circ) > \parab[k_1](t,x_\circ) $.
Given this property, it is possible to perform the surgery while keeping the resulting function in $ \HLspk(k_1) $ and passing through $ \wordF(\til{F}_2)\setminus\{\letterlast\} $; see Figure~\ref{f.smooth} for an illustration.
Note that even in the case $ x_{\letterlast} = x_{\circ} $, 
we can forgo $ \letterlast $ in $ \til{F}_2 $ because $ F'_1 $ already passes through $ \letterlast $.
Therefore, the inequality in \eqref{e.p.nondiag} is strict unless $ k_1=k_2 $ throughout the induction.
The last scenario forces all the $ k_i $'s to be the same.
\end{proof}

\begin{figure}[h]
\begin{minipage}[t]{.27\linewidth}
\frame{\includegraphics[height=100pt]{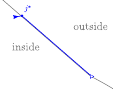}}
\caption{The inside and outside regions. The gray curve is $ L $. The blue curve is part of $ F_1 $.}
\label{f.L}
\end{minipage}
\hfill
\begin{minipage}[t]{.36\linewidth}
\frame{\includegraphics[height=100pt]{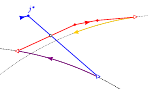}}
\caption{Illustration of Case~1. %
The blue, purple, red, and yellow curves are respectively parts of $ F_1 $, $ \parab[k_2](t) $, $ F_2 $, and $ \parab[k_3](t) $.}
\label{f.case1}
\end{minipage}
\hfill
\begin{minipage}[t]{.36\linewidth}
\frame{\includegraphics[height=100pt]{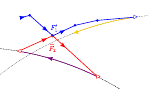}}
\caption{Rewired curves, Case~1.}
\label{f.case1.rewired}
\end{minipage}
\\
\begin{minipage}[t]{.405\linewidth}
\frame{\includegraphics[width=\linewidth]{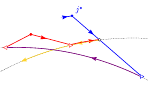}}
\caption{Illustration of Case~2. %
The blue, purple, red, and yellow curves are respectively parts of $ F_1 $, $ \parab[k_2](t) $, $ F_2 $, and $ \parab[k_3](t) $.}
\label{f.case2}
\end{minipage}
\hfill
\begin{minipage}[t]{.335\linewidth}
\begin{minipage}[b]{\linewidth}
\frame{\includegraphics[width=\linewidth]{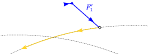}}\\
\frame{\includegraphics[width=\linewidth]{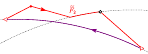}}
\end{minipage}
\caption{Rewired curves, Case~2.}
\label{f.case2.rewired}
\end{minipage}
\hfill
\begin{minipage}[t]{.245\linewidth}
\frame{\includegraphics[width=\linewidth]{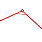}}
\caption{The dashed curve is the modified one.}
\label{f.smooth}
\end{minipage}
\vspace{-10pt}
\end{figure}

Proposition~\ref{p.nondiag} reduces minimizing $ \expfnword[\monomial_1]+\ldots+\expfnword[\monomial_\ell] $ among preferred terms 
to minimizing $ \sum \raterwRel{ f_{k_i} }{ \parab[k_i](t) } $.
This minimization is similar to \eqref{e.raterw.xa}, except that in the former the $ f_{k_i} $'s may share the same $ k_i $.
The next lemma shows that such $ f_{k_i} $'s can be combined into a single function.

\begin{lem}
\label{l.repeat}
For any $ f_1,\ldots,f_n \in \HLspk(k) $,
\begin{align*}
	\raterwREl{ f_{1} }{ \parab[k](t) } + \ldots + \raterwREl{ f_{n} }{ \parab[k](t) } 
	\geq 
	\min \Big\{  \raterwREl{ f }{ \parab[k](t) }  
	:
	f \in \HLsp(k),
	\
	\wordF(f) \supset \cup_{i=1}^n \wordF(f_{i}) \Big\},
\end{align*}
and the inequality is strict unless $ (\hyp(f_1)^\circ \setminus \hyp( \parab[k](t) )), \ldots , (\hyp(f_n)^\circ \setminus \hyp( \parab[k](t) )) $ are disjoint.
\end{lem}
\begin{proof}
Take $ n=2 $, and the case $ n>2 $ follows by induction. 
Consider $ \max\{f_1,f_2\} $ and $ \min\{f_1,f_2\} $.
Both functions belong to $ \HLspk(k) $ and the former passes through $ \wordF(f_1)\cup\wordF(f_2) $.
Further
$
	\raterwRel{f_1}{\parab[k](t) } + \raterwRel{f_2}{\parab[k](t) }
	=
	\raterwRel{\max\{f_1,f_2\}}{\parab[k](t) } + \raterwRel{\min\{f_1,f_2\}}{\parab[k](t) }	
$.
Forging the last term gives the desired inequality, and the last term is zero only when $ (\hyp(f_1)^\circ \setminus \hyp( \parab[k](t) )) $ and $ (\hyp(f_2)^\circ \setminus \hyp( \parab[k](t) )) $ are disjoint.
\end{proof}

Proposition~\ref{p.nondiag} and Lemma~\ref{l.repeat} together show that the minimum of $ \expfnword[\monomial_1]+\cdots+\expfnword[\monomial_\ell] $ among preferred, non-degenerate terms is bounded below by $ \raterw((\icvecx,\icveca)\xrightarrow{\scriptscriptstyle t}(\vecx,\veca)) $, given in \eqref{e.raterw.xa}.
We call a minimizer $ \{F_k\}_{k} $ of \eqref{e.raterw.xa} a \tdef{minimizer of $ (\icvecx,\icveca)\xrightarrow{\scriptscriptstyle t} (\vecx,\veca) $}.
Conversely, the next lemma asserts that any minimizer of $ (\icvecx,\icveca)\xrightarrow{\scriptscriptstyle t} (\vecx,\veca) $ is given by $ F_k = \rwf[\wordkl[\wordF(F_k)][k][k]] $ and $ \wordF(F)\in\treekl[kk] $.
Note that for such a minimizer, $ \wordF(F_k) = \emptyset $ only if $ F_k = \parab[k](t) $, in which case $ \raterwRel{F_k}{\parab[k](t)} = 0 $.
We hence adopt the convention.
\begin{con}
\label{con.no.empty}
Any $ F_k $ with $ \wordF(F_k) = \emptyset $ is removed from a minimizer $ \{F_k\}_k $ of $ (\icvecx,\icveca)\xrightarrow{\scriptscriptstyle t} (\vecx,\veca) $.
\end{con}
\noindent{}%
Recall from Section~\ref{s.results.matching}  the function $ \Rwf $ and that $ F_k = \Rwf( \wordkl[\wordF(F_k)][k][k] ) $, for any minimizer $ \{F_k\}_{k} $ of $ (\icvecx,\icveca)\xrightarrow{\scriptscriptstyle t} (\vecx,\veca) $.

\begin{lem}
\label{l.raterw.minimizer.}
Under Convention~\ref{con.no.empty},
for any $ F_k $ in a minimizer of $ (\icvecx,\icveca)\xrightarrow{\scriptscriptstyle t} (\vecx,\veca) $,
$ F_k = \Rwf( \wordkl[\wordF(F_k)][k][k] ) =\rwf[\wordkl[\word][k][k]] $, and $ \wordF(F_k) \in \treekl[kk] $.
\end{lem}
\noindent{}%
This lemma is contained in the results of Lemma~\ref{l.raterw.minimizer}--\ref{l.raterw.minimizer.word}. 
See Section~\ref{s.a.raterw} for the proof.

\subsection{Bounding terms and identifying the dominant term}
\label{e.asymptotics.bound}

\begin{prop}
\label{p.term.bd}
Under Assumption~\ref{assu.strict},
there exists $ c=c(t,\icvecx,\icveca,\vecx,\veca) $ such that 
\begin{align*}
	|\gterm|
	\leq
	\exp\big( - N\raterw((\icvecx,\icveca)\xrightarrow{t}(\vecx,\veca)) \big)
	\,
	\exp\big( c \norm{\gterm} - \tfrac{N}{c}\, (\norm{\gterm}-c)_+ \big),
\end{align*}
for all preferred, non-degenerate terms $ \gterm $.
\end{prop}

\begin{proof}
Throughout the proof $ c=c (t,\icvecx,\icveca,\vecx,\veca) $
and $ \raterw^* := \raterw((\icvecx,\icveca)\xrightarrow{\scriptscriptstyle t}(\vecx,\veca)) $.
Recall that $ \norm{\gterm} $ counts the total number $ \opuci{\ldots} $ in $ \gterm $.
Fix any preferred, non-degenerate term $ \gterm = \monomial_1\cdots\monomial_\ell $.
The first step is to derive a bound on $ \expfnword[\gterm] $.
Proposition~\ref{p.nondiag} applied to these $ \monomial_i $'s gives
%\begin{align*}
$
	\expfnword[\gterm]
	=
	\expfnword[\monomial_1] + \ldots + \expfnword[\monomial_\ell]
	\geq
	\sum_{i=1}^{\norm{\gterm}} \raterwRel{ f_{i} }{ \parab[k_i](t) },
$
%\end{align*}
for some $ f_{i} \in \HLspkk(k_i) $,
and the union of all the $ \wordF(f_{i}) $'s exhausts all letters $ 1,2,\ldots,m $. 
After reindexing the $ f_{i} $'s, we have $ \wordF(f_{1})\cup\dots\cup\wordF(f_{n'}) = 12\cdots m $, for some $ n' \leq m $.
Referring to \eqref{e.raterw.xa} and using Lemma~\ref{l.repeat} give
%\begin{align}
%	\label{e.gterm.ratebd.}
$
	\sum_{i=1}^{n'} \raterwRel{ f_{i} }{ \parab[k_i](t) }
	\geq
	\raterw^*.
$
%\end{align}
Since $ \HLspkk(k) $ is a finite subset of $ \HLspk(k_i) $
and since $ \parab[k](t) \notin \HLspkk(k) $,
$ \min\{ \raterwRel{f}{\parab[k](t)}: f \in \HLspkk(k) \}>0 $.
Hence
\begin{align*}
	\expfnword[\gterm]
	=
	\expfnword[\monomial_1] + \ldots + \expfnword[\monomial_\ell]
	\geq
	\raterw^*
	+
	\tfrac{1}{c}\, (\norm{\gterm}-n')_+
	\geq
	\raterw^*
	+
	\tfrac{1}{c}\, (\norm{\gterm}-m)_+.	
\end{align*}
Inserting this inequality into Proposition~\ref{p.trace} gives the desired result.
\end{proof}

We proceed to identify the dominant term, which will only be needed for obtaining the lower bound.
As is often the case in proving \acp{LDP}, the proof of the lower bound is tolerant to additional assumptions, and we impose the following.
\begin{assu}
\label{assu.lwbd}
The minimizer $ \{F_k\}_k $ of $ (\icvecx,\icveca)\xrightarrow{\scriptscriptstyle t} (\vecx,\veca) $ is unique,
the words do not share letters, namely $ \wordF(F_k) \cap \wordF(F_{k'}) = \emptyset $ for all $ k\neq k' $,
and each $ \wordF(F_k) \in \islesetkl[kk] $, under Convention~\ref{con.no.empty}.
\end{assu}
\noindent{}%
Under Assumption~\ref{assu.lwbd}, the next proposition shows that the dominant term is 
$ \gterm_\star := \prod_k \tr( \inddown_{\Ic{k}} \opuci{\wordF(F_k)} \inddown_{\Ic{k}} ). $
Recall from \eqref{e.DD} and \eqref{e.det.expanded} that $ \fcoef_{\gterm} := \sum_{n \geq 0} \frac{1}{n!}\coef_{n,\gterm} $ denotes the coefficient of $ \gterm $ in the Plemelj-like expansion.
Recall the parity $ \parityc(\wordkl) $ from the beginning of Section~\ref{s.updown}.

\begin{prop}
\label{p.term.dominanace}
Under Assumptions~\ref{assu.strict} and \ref{assu.lwbd},
\begin{align*}
	\lim_{N\to\infty} \frac{1}{N} \log \gterm_\star
	=
	- \raterw\big((\icvecx,\icveca)\xrightarrow{ t}(\vecx,\veca)\big),
	\qquad
	\fcoef_{\gterm_\star}
	:=
	\sum_{n \geq 0} \frac{1}{n!}\coef_{n,\gterm_\star} = \prod_{k} \parityc(\wordkl[\wordF(F_k)][k][k]),
\end{align*}
and, for any preferred, non-degenerate term $ \gterm \neq \gterm_\star $,
$
	\lim_{N\to\infty} \frac{1}{N} \log \gterm
	<
	- \raterw((\icvecx,\icveca)\xrightarrow{\scriptscriptstyle t}(\vecx,\veca)).
$
\end{prop}

\begin{proof}
Throughout the proof %$ c=c (t,\icvecx,\icveca,\vecx,\veca) $ and 
$ \raterw^* := \raterw((\icvecx,\icveca)\xrightarrow{\scriptscriptstyle t}(\vecx,\veca)) $.
Fix the unique minimizer $ \{ F_k = \Rwf( \wordkl[\wordF(F_k)][k][k] ) \}_k $
and fix a preferred, non-degenerate term $ \gterm $.
Under the given assumptions,
inspecting the conditions for `$ = $' to hold in Proposition~\ref{p.nondiag}\ref{p.nondiag.2} and Lemma~\ref{l.repeat} reveals that
$ \expfnword[\gterm]>\raterw^* $ unless $ \gterm = \gterm_\star $.
Conversely, Lemma~\ref{l.raterw.minimizer.} asserts that $ \expfnword[\gterm_\star]=\raterw^* $ and $ \wordF(F_k) \in \treekl[kk] $.
To find the coefficient $ \fcoef_{\gterm_\star} $,
recall that $ \DD_n $ is obtained by inserting $ K_{kk'\star} $ from \eqref{e.det.islse} into $ \D_n(\ldots) $ in \eqref{e.simon} and expanding.
The summand in \eqref{e.det.islse} contains $ \opuci{\wordF(F_k)}=\opuc{\wordF(F_k)} $,
since $ \wordF(F_k) \in \treekl[kk] \cap\islesetkl[kk] $.
During the insertion and expansion procedure,
if any $ \opuci{\word} $ with $ \word \notin \{\wordF(F_k)\}_k $ gets involved, the resulting term will not be $ \gterm_\star $.
Therefore, without affecting any contribution to $ \gterm_\star $, we delete all $ \opuci{\word} $ with $ \word \notin \{\wordF(F_k)\}_k $ from $ (K_{kk'\star})_{k,k'=1}^{\icm} $.
Namely, we replace $ K_{kk\star} \mapsto \parityc( \wordkl[\wordF(F_k)][k][k] ) \opuc{\wordF(F_k)} $, replace $ K_{kk'\star} \mapsto 0 $ when $ k\neq k' $, insert the result into $ \D_n(\ldots) $, and expand.
The resulting expansion reads
\begin{align*}
	\det\Big( \Id + \big( \kdelta_{k=k'} \ \parityc(\wordkl[\wordF(F_k)][k][k]) \ \opuc{\wordF(F_k)}  \big)_{k,k'=1}^{\icm}  \Big)
	=
	\prod\nolimits_{k} \big( 1 + \parityc(\wordkl[\wordF(F_k)][k][k]) \ \tr( \opuc{\wordF(F_k)}) + \ldots \big).
\end{align*}
Reading the coefficient of $ \gterm_\star $ from the right side gives the desired result.
\end{proof}

\section{Determinantal analysis: approximations and proof of \hyperref[t.rw]{Fixed-time Theorem}}
\label{s.pftrw}

The proof will frequently involve the discretized setting 
$ (\icvecx,\icveca) \xrightarrow{\scriptscriptstyle t} (\vecx,\veca) $.
Throughout the proof, the initial condition $ (\icvecx,\icveca) $ will always be assumed to satisfy \eqref{e.nondeg.sd}. 
%Doing so does not lose any generality.
%Neither the quantity $ \raterw((\icvecx,\icveca) \xrightarrow{\scriptscriptstyle t} (\vecx,\veca)) $ nor the function $ \parab[1\ldots\icm](0) $ is affected by removing any $ k $ with $ \ics_{k'}=\ics_{k} $ and $ \icd_{k'}>\icd_{k} $ or any $ k $ with $ \icd_{k}=\icd_{k'} $ and $ \ics_{k}<\ics_{k'} $.
The terminal condition $ (\vecx,\veca) $ will have varying assumptions, to be specified below.

We will prove \hyperref[t.rw]{Fixed-time Theorem} by establishing the upper and lower bound separately.
For the upper bound, we will consider a massif initial condition first and then a general initial condition.
Recall that we endow $ \Lip $ with the metric $ \dist(f_1,f_2) := \sum_{n=1}^\infty 2^{-n} \sup_{x\in[-n,n]} | f_1(x) - f_2(x) | $.
\begin{thmfixedtimeup}
\begin{enumerate}[leftmargin=20pt,label=(\alph*)]
\item[] 
\item \label{t.rw.upper.1}
Fix any $ (\icvecx,\icveca) $,
and start the \ac{TASEP} from the massif initial condition $ \hh_N(0) = \parab[1\ldots\icm](0) $.
For any $ f\in\Lip $ that satisfies the Hopf--Lax condition $ \parab[1\ldots\icm](t) \leq f \leq \parab[1\ldots\icm](0) $,
\begin{align*}
%	\label{e.trw.upbd.goal}
	\limsup_{\delta\to 0} \limsup_{N\to\infty} \,
	\frac{1}{N} \log \P_{\parab[1\cdots\icm](0)}\big[ \dist(\hh_N(t),f)<\delta \big]  
	\leq 
	- \raterw\big( \parab[1\ldots\icm](0) \xrightarrow{t} f \big).
\end{align*}
\item \label{t.rw.upper.2}
For any $ f,g\in\Lip $,
\begin{align*}
	\limsup_{\delta\to 0}
	\limsup_{N\to \infty} \ \sup_{\dist(\hh_N(0),g)<\delta} 
	\Big\{ 
		\frac{1}{N} \log \P_{\hh_N(0)}[ \dist(\hh_N(t),f) < \delta ] 
	\Big\}
%	\leq
%	(\text{left side of } \eqref{e.upbd.1})
	\leq
	- \raterw( g \xrightarrow{ t} f).
\end{align*}
\end{enumerate}
\end{thmfixedtimeup}

\begin{proof}
\ref{t.rw.upper.1}\
Fix an $ \e>0 $.
Refer to the definition \eqref{e.raterw.gf} of $ \raterw( \parab[1\ldots\icm](0) \xrightarrow{t} f ) $.
Pick $ (\vecx,\veca)\!=\!(x_i,f(x_i))_{i=1}^m\! $ such that
\begin{align}
	\label{e.trw.upbd.1}
	\raterw\big( \parab[1\ldots\icm](0) \xrightarrow{t} f \big) - \e
	<
	\raterw\big( \parab[1\ldots\icm](0) \xrightarrow{t} (\vecx,\veca) \big)
	=
	\raterw\big( (\icvecx,\icveca) \xrightarrow{t} (\vecx,\veca) \big).
\end{align}
Recall that we call a minimizer $ \{F_k\}_k $ of \eqref{e.raterw.xa} \tdef{a minimizer of $ (\icvecx,\icveca)\xrightarrow{\scriptscriptstyle t} (\vecx,\veca) $}. 

\medskip\noindent{}%
\hypertarget{s.pftrw.upper.add}{\textbf{Step 1: adding and removing points.}}\
We will modify $ (\vecx,\veca) $ to obtain certain desirable properties.
Throughout the modification,
the inequality \eqref{e.trw.upbd.1} will always be maintained, and
the points $ (x_1,a_1),\ldots, (x_m,a_m) $ will always be on the graph of $ f $.

Consider an open interval $ U $ on which $ \partial_y f|_{U}= -1 $ or $ \partial_y f|_{U}= 1 $ or $ f|_{U}=\parab[k](t)|_{U} $ for some $ k $.
We seek to remove any $ x_{i_0}\in U $ from $ (\vecx,\veca) $.
For each such $ x_{i_0} $, consider the maximal interval $ U=(\underline{x},\overline{x})\ni x_{i_0} $ on which $ \partial_y f|_{U}=-1 $ or $ \partial_y f|_{U}= 1 $ or $  f|_{U}=\parab[k](t)|_{U} $,
and add the endpoints $ (\underline{x},f(\underline{x})) := (\underline{x},\underline{a}) $ and $ (\overline{x},f(\overline{x})) := (\overline{x},\overline{a}) $ to $ (\vecx,\veca) $.
Adding points only increases the right side of \eqref{e.trw.upbd.1}, so \eqref{e.trw.upbd.1} continues to hold after adding the points.
Next, remove all points $ (x_{i_0},a_{i_0}) $ such that $ a_{i_0} \in U $.
Removing points generally decreases the right hand side of \eqref{e.trw.upbd.1}, but \emph{not} in these cases.
To see why, consider the case $ \partial_y f|_{U}=-1 $, and take any minimizer $ \{F_k\}_k $ of $ (\icvecx,\icveca)\xrightarrow{\scriptscriptstyle t} (\vecx,\veca) $ for the \emph{post-removal} $ (\vecx,\veca) $.
Take an $ F_{k_0} $ that passes through $ (\underline{x},f(\underline{x})) $.
This function $ F_{k_0} $ passes through all $ (x_{i_0},a_{i_0}) $, despite the fact that it has been removed from $ (\vecx,\veca) $.
This is so because $ F_{k_0}(\bar{x}) \leq f(\bar{x}) $ forces $ \partial_x F_{k_0}|_{(\underline{x},\overline{x})} = -1 $.
For any other $ F_k $, using $ F_k(\underline{x})\leq f(\underline{a}) $, $ F_k(\bar{x})\leq f(\bar{a}) $, and $ F_k\in\Lip $ yields that $ F_k(x_{i_0}) \leq a_{i_0} $.
Hence $ \{F_k\}_k $ is also a minimizer of $ (\icvecx,\icveca) \xrightarrow{\scriptscriptstyle t}( (\vecx,\veca)\cup\{(x_{i_0},a_{i_0})\}) $, and therefore removing $ (x_{i_0},a_{i_0}) $ does not decrease the right side of \eqref{e.trw.upbd.1}.
The cases $ \partial_y f|_{U}=1 $ and $ \partial_y f|_{U}=\parab[k](t)|_{U} $ follow similar arguments.

\medskip\noindent{}%
\hypertarget{s.pftrw.upper.perturb}{\textbf{Step 2: perturbing $ \boldsymbol{s_i} $.}}\
We proceed to perturb the $ (\vecx,\veca) $ obtained at the end of Step~\hyperlink{s.pftrw.upper.add}{1}.
The properties stated in the first paragraph of Step~\hyperlink{s.pftrw.upper.add}{1} will still be maintained.
For clarity, the pre-perturbation variables will be labeled `pre', for example $ (\vecx^\text{pre},\veca^\text{pre}) $.
The perturbation will be done in the $ (s,d) $ coordinates.
Turn $ f $ into a function in the $ (s,d) $ system as $ \phi(s):= \inf\{ \frac{1}{2}(x+f(x)) : x\in\R \} $,
which is decreasing and right-continuous-with-left-limit.
Fix any $ \delta_0>0 $.
We seek to perturb $ (s^\text{pre}_i,d^\text{pre}_i) $ such that the following holds for the post-perturbation points $ (s_i,d_i)=(s_i, \phi(s_i)) $.
We call a letter $ i $ \tdef{deep} if $ a_i=\parab[k](t,x_i) $ for some $ k $.
\begin{enumerate}[leftmargin=20pt, label=(\Roman*)]
\item \label{enu.pertur.1} The function $ \phi $ is continuous at each $ s_i $.
\item \label{enu.pertur.2} $ |s_i-s^\text{pre}_i|+|d_i-d^\text{pre}_i| < \delta_0 $; if $ d^\text{pre}_i \neq d^\text{pre}_j $ then $ d_i \neq d_j $.
\item \label{enu.pertur.2.5} A letter $ i $ is deep only if $ \phi $ is non-continuous at $ s^\text{pre}_i $.
\item \label{enu.pertur.4} $ s_i \neq \ics_k $, for all $ i $ and $ k $.
\end{enumerate}
Since $ \phi $ is non-continuous for at most countably many $ s $,
around any neighborhood of $ s^\text{pre}_i $ there exists an $ s_i $ such that \ref{enu.pertur.1} holds.
If $ \phi $ is non-continuous at $ s^\text{pre}_i $, a small change in $ s $ could introduce an uncontrolled change in $ \phi $.
By the construction in Step~\hyperlink{s.pftrw.upper.add}{1}, any such $ s^\text{pre}_i $ is an endpoint of a maximal interval on which $ \partial_y f = -1 $, and the condition~\ref{enu.pertur.2} can be satisfied by choosing the suitable direction to perturb; see Figure~\ref{f.pfup-0}.
Similarly considerations show that \ref{enu.pertur.2.5} can be achieved.
Lastly, \ref{enu.pertur.4} is clearly achievable.

Perturbing $ s_i $ may change the set $ \alphabetkl $.
Let $ \alphabet^\text{pre}(kk') $ denote the corresponding pre-perturbation set.
We would like the set to remain the same as much as possible.
Referring to the definition~\ref{e.alphabetkl} of $ \alphabetkl $, one sees that the perturbation can be done without changing $ \alphabetkl $ except when $ s^\text{pre}_{i-1}=s^\text{pre}_{i} = \ics_{k'} $ for some $ k $.
These are the $ s_i^\text{pre} $ at which $ \phi $ is non-continuous that happens to coincide with some $ \ics_{k'} $.
We call the letter $ i $ \tdef{$ k' $-aligned} or simply \tdef{aligned}.
The perturbation gives $ s_{i-1}<s^\text{pre}_{i-1} $ and $ s^\text{pre}_i<s_i $ (see Figure~\ref{f.pfup-0}), so
\begin{enumerate}[leftmargin=20pt, label=(\Roman*)]
\setcounter{enumi}{4}
\item \label{enu.pertur.alphabet}
$ \alphabetkl = \alphabet^\text{pre}(kk') \setminus\{ i \} $, where $ i $ is the unique (if exists) $ k' $-aligned letter.
\end{enumerate}
We require two more conditions.
\begin{enumerate}[leftmargin=20pt, label=(\Roman*)] % 
\setcounter{enumi}{5}
\item \label{enu.pertur.2.75} Any deep $ i $ is larger than $ 1 $, the precedent letter $ i-1 $ is not deep, and for any $ \word\supset\{i-1,i\} $ we have $ \UDc(\wordkl)_{i} = \down $.
\item \label{enu.pertur.3} For all $ \wordkl[\word][k_1][k_2] $ and $ \wordkl[\wordsub][k_2][k_3] $, $ \xr(\wordkl[\word][k_1][k_2]) \neq \xl(\wordkl[\wordsub][k_2][k_3]) $.
\end{enumerate}
The condition should hold for all $ \wordkl $ under Convention~\ref{con.wordkl} with respect to the \emph{post-perturbation} $ \wordsetkl $.
The condition~\ref{enu.pertur.2.75} holds for all small enough $ \delta_0>0 $ because by \ref{enu.pertur.2.5}
and by the construction in Step~\hyperlink{s.pftrw.upper.add}{1},
for any deep $ i $, $ (x^\text{pre}_{i-1},a^\text{pre}_{i-1}) $ and $ (x^\text{pre}_{i},a^\text{pre}_{i}) $ sit at the endpoints of a maximal open interval $ U $
on which $ \partial_y f|_{U} = -1 $.
As for \ref{enu.pertur.3}, note that perturbing $ s_i $ always changes $ \xr(\wordkl) $ and changes $ \xl(\wordkl) $ unless $ \phi(s_i)=\icd_k $.
Hence \ref{enu.pertur.3} can be achieved.

By Lemma~\ref{l.raterw.lsc}, the right side of \eqref{e.trw.upbd.1} is \ac{lsc} in $ (\vecx,\veca) $.
Now fix $ \delta_0 $ small enough so that \eqref{e.trw.upbd.1} holds for the post-perturbation points.

\medskip\noindent{}%
\hypertarget{s.pftrw.upper.d+-}{\textbf{Step 3: constructing $ \boldsymbol{d_i^\pm} $.}}\
Given the $ s_i $ and $ d_i=\phi(s_i) $ previously obtained,
we proceed to construct $ (d_i^-,d^+_i)\supset d_i $.
Let $ (x_i^\pm,a_i^\pm)=(s_i-d^\pm_i,s_i+d^\pm_i) $ denote the resulting point in the $ (x,a) $ coordinates.

We would like the resulting $ (\vecs,\vecd^{\vecrho}) $, where $ \vecrho\in\{\pm\}^m $,
to satisfy Assumption~\ref{assu.strict}. 
This is however not always possible due to the following scenarios.
\begin{description}[leftmargin=5pt] %,font=\normalfont\space
\item[Steep] If $ d_i=d_{i+1} $, necessarily $ d_i^-> d_{i+1}^+ $. We call $ (i,i+1) $ and $ (d_i^-,d_{i+1}^+) $ \tdef{steep}.
	Note that by the construction in Step~\hyperlink{s.pftrw.upper.add}{1} and by \ref{enu.pertur.2},
	the scenario $ d_i=d_{j} $ is possible only when $ |i-j|=1 $.
\item[Deep] If $ i $ is deep, necessarily $ (x_i^-,a^-_i) \in \hyp(\parab[k](t))^\circ $.
\end{description}
Further, recall that $ \alphabetkl $ consists of all $ i $'s such that $ d_i \leq \icd_k $ and $ s_i \leq \ics_{k'} $, and note that $ d_i=\icd_k $ is possible.
Hence replacing $ d_i $ with $ d_i^\pm $ could change $ \alphabetkl $.
\begin{description}[leftmargin=5pt] %,font=\normalfont\space
\item[High] If $ d_i=\icd_{k} $, necessarily $ d_i^+> \icd_{k} $. We call $ i $ and $ d_i^+ $ \tdef{$ k $-high} or just \tdef{high}.
\end{description}
We call $ \vecrho\in\{\pm\}^m $ \tdef{legal} if none of the above happens,
namely $ (\rho_i,\rho_{i'}) \neq (-,+) $ whenever $ (i,i') $ is steep,
$ \rho_i = + $ whenever $ i $ is deep,
and $ \rho_i=- $ whenever $ i $ is high.

We seek to construct $ d_i^\pm $ so that the following desired properties hold for all legal $ \vecrho $.
Let $ \alphabet^{\vecrho}(kk')$, $\vecUDc^{\vecrho}(\wordkl)$, $\isleset^{\vecrho}(kk') $ denote the corresponding sets and variables for $ (\vecs,\vecd)\mapsto (\vecs,\vecd^{\vecrho}) $.
Fix $ \delta_1>0 $.
Construct $ (d_i^-,d^+_i)\ni d_i $ such that, for any legal $ \vecrho $,
\begin{enumerate}[leftmargin=20pt, label=(\roman*)]
\item \label{enu.d+-.0} $ |d_i^+-d_i| < \delta_1 $ and $ |d_i-d_i^-| < \delta_1 $.
\item \label{enu.d+-.1} Assumption~\ref{assu.strict} holds for $  (\vecs,\vecd) \mapsto (\vecs,\vecd^{\vecrho}) $.
%\item \label{enu.d+-.2} The property \ref{enu.pertur.3} holds for $ (\vecs,\vecd) \mapsto (\vecs,\vecd^{\vecrho}) $.
\item \label{enu.d+-.3} $ \alphabet^{\vecrho}(kk') = \alphabet(kk') $, $ \vecUDc^{\vecrho}(\wordkl) = \vecUDc(\wordkl) $, $ \isleset^{\vecrho}(kk') = \islesetkl $.
\item \label{enu.d+-.4} If $ d_j < \icd_k $ then $ d_j^+ < \icd_k $; if $ d_j>\icd_k $ then $ d^-_j > \max\{ d^+_i: d_i=\icd_k \} $.
\end{enumerate}
Given \ref{enu.pertur.1}--\ref{enu.pertur.4} and \ref{enu.pertur.3},
it is straightforward to check that \ref{enu.d+-.0}--\ref{enu.d+-.3} are achievable.
The property \ref{enu.d+-.4} follows from \ref{enu.d+-.0} for all small enough $ \delta_1 $.
By Lemma~\ref{l.raterw.lsc}, the right side of \eqref{e.trw.upbd.1} is \ac{lsc} in $ (\vecx,\veca) $.
Hence for $ \delta_1>0 $ small enough, the inequality \eqref{e.trw.upbd.1} holds for $ (\vecx,\veca)\mapsto (\vecx^{\vecrho},\veca^{\vecrho}) $.

\medskip\noindent{}%
\textbf{Step 4: applying the determinantal analysis.}\
We begin by bounding $ \P_{\parab[1\cdots\icm](0)}[ \dist(\hh_N(t),f)<\delta ] $.
By \ref{enu.pertur.1}, for all $ \delta $ small enough, $ \{ \dist(\hh_N(t),f) < \delta \} \subset \{ \hd_N(t) \in (d^-_i,d^+_i], i=1,\ldots,m \} $.
On the right side, forgo all deep $ d^-_i $. 
More precisely, let $ \deep $ denote the set of deep letters, and consider the modified pinning probability
\begin{align*}
	\ppin':= \P_{\parab[1\cdots\icm](0)}\big[ \hd_N(t) \in (d^-_i,d^+_i] \text{ for } i\notin\deep \text{, and } \hd_N(t) \leq d^+_i \text{ for }i\in\deep \big].
\end{align*}
Forgoing conditions only makes the event bigger, so
\begin{align}
	\label{e.<ppin'}
	\text{for all } \delta \text{ small enough,}\qquad
	\P_{\parab[1\cdots\icm](0)}[\dist(\hh_N(t),f) < \delta ] \leq \ppin'.
\end{align}
We call $ \vecrho\in\{\pm\}^m $ \tdef{non-deep} if $ \vecrho|_{\deep} = (+,\ldots,+) $.
The inclusion-exclusion formula gives $ \ppin' = \sumie' \punder(\vecs,\vecd^{\vecrho},N) $,
where the operator $ \sumie' $ is modified from $ \sumie $ (defined in \eqref{e.sumie})
by restricting the sum to non-deep $ \vecrho $.

We seek to apply the determinantal analysis with $ (\vecs,\vecd) \mapsto (\vecs,\vecd^{\vecrho}) $ for each non-deep $ \vecrho $.
The up-down iteration and isle factorization require some inputs: $ \vecUDc $, the $ kk' $-th alphabet, and the set of $ kk' $-th isles, for all $ k,k' $.
These variables and sets are defined only for legal $ \vecrho $.
However, since the up-down iteration and isle factorization are purely algebraic/combinatorial procedures, we can simply \emph{construct} a set of inputs and feed them into the procedures.
To construct the inputs,
within $ (\vecs,\vecd) $, for each aligned letter $ i $,
replace $ (s_i,d_i) \mapsto (s^\text{pre}_{i},d^\text{pre}_{i}) $,
and let $ (\vecs^\times,\vecd^\times) $ denote the result.
Namely, we undo the perturbation in Step~\hyperlink{s.pftrw.upper.perturb}{2} for aligned letters.
Let $ \alphabet^\times(kk'),\vecUDc^\times(\wordkl),\isleset^\times(kk') $ denote the corresponding sets and variables.
This construction together with \ref{enu.pertur.alphabet} gives $ \alphabet^\times(kk') = \alphabetkl \cup \{ i \} $, where $ i $ is the unique (if exists) $ k' $-aligned letter.
This added letter $ i $ can be combined with existing ones to form longer words, but those preexisting words in $ \wordsetkl $ are unaffected.
Hence
\begin{enumerate}[leftmargin=10pt]
\item[$ \circ $] $ \alphabet^\times(kk') = \alphabetkl \cup \{ i \} $, where $ i $ is the unique (if exists) $ k' $-aligned letter,
\item[$ \circ $] $ \wordsetkl \subset \wordset^\times(kk') $,
	and any word in $ \wordset^\times(kk')\setminus\wordsetkl $ must contain the $ k' $-aligned letter,
\item[$ \circ $] $ \vecUDc(\wordkl) = \vecUDc^\times(\wordkl) $ for all $ \word\in\wordsetkl $, and
\item[$ \circ $] $ \isleset(kk') \subset \isleset^\times(kk') $.
\end{enumerate}
By \ref{enu.d+-.3}, for all legal $ \vecrho $, these properties hold for $ \alphabetkl \mapsto \alphabet^{\vecrho}(kk') $, $ \vecUDc(\wordkl) \mapsto \vecUDc^{\vecrho}(\wordkl) $, etc.

Fix any non-deep $ \vecrho \in \{\pm\}^{m} $.
Apply the up-down iteration and isle factorization with $ (\vecs,\vecd) \mapsto (\vecs,\vecd^{\vecrho}) $
and with $ \alphabet^\times(kk'),\vecUDc^\times(\wordkl),\isleset^\times(kk') $ being the inputs.
Doing so requires modification of the trimming procedure in Section~\ref{s.geo.trim}.
For fixed $ k,k' $, examine if there exists a high letter $ i_* $ with $ d_{i_*}=\icd_k $, and in case such letters are not unique choose the smallest $ i_* $; examine if there exists a (necessarily unique) $ k' $-aligned letter $ i^* $.
Instead of applying Lemma~\ref{l.trimming} with $ \letterfirst = \min\{j: d_j^{\rho_j} \leq \icd_k \} $ and $ \letterlast = \max\{j: s_j \leq \ics_{k'} \} $, here we apply it with $ \letterfirst':=\min\{i_*,\letterfirst\} $ and $ {\letterlast}' := \max\{ i^*,\letterlast\} $.
The choices apply because  by Property~\ref{enu.d+-.4} we have $ i_* < \min\{j: d_j^{\rho_j} \leq \icd_k \} $, and because $ s_{i^*}>s^\text{pre}_{i^*} = \ics_{k} $.
The choices ensure $ [\letterfirst',{\letterlast}'] = \alphabet^\times(kk') $.

We examine the resulting terms from the application of up-down iteration and isle factorization.
Any term that involves steep $ (d^-_i,d^+_{i+1}) $ or high $ d_i^+ $ or an aligned letter is automatically zero.
To see why, consider steep $ d_i^-<d^+_{i+1} $, which happens if and only if $ d_i=d_{i+1} $.
By the construction in Step~\hyperlink{s.pftrw.upper.add}{1},
the points $ x_i $ and $ x_{i+1} $ must sit at the endpoints of a maximum open interval $ U $ on which $ \partial_y f|_U = 1 $.
Hence, for any word $ \word $ that contains $ i $ and $ i+1 $, $ \UDc(\wordkl)_i = \up $ and $ \UDc(\wordkl)_{i+1} = \down $.
From this and $ d^-_i<d^+_{i+1} $, it is readily checked from Definition~\ref{e.opu} that $ \opuc[k][k']{\word} = 0 $.
Next, it is readily checked that, for a $ k $-high or a $ k' $-aligned letter $ i $, we have $ \UDc^\times(\wordkl)_i=\up $.
Using the identity \eqref{e.trimming} shows that any term that contains such a letter $ i $ is zero.
Consequently, it suffices to consider terms that only involve legal $ \vecd^{\vecrho} $ and involve no aligned letters.
Any such term involves only words in $ \{\islesetkl[kk']\}_{k,k'} $ and satisfies Assumption~\ref{assu.strict}.

Next we examine a property related to degenerate terms.
Since the inclusion-exclusion sum has been modified to $ \sumie' $ here,
the notion of degenerate terms (defined in Definition~\ref{d.gterm}) should change accordingly.
We call a generic term \tdef{modified-degenerate} if it does not not use all non-deep letters.
The application of $ \sumie' $ kills modified-degenerate terms, but not necessarily degenerate terms.
We claim that, however, in the current setup, any nonzero non-modified-degenerate term $ \gterm $ is automatically a non-degenerate term.
To see why, fix a non-modified-degenerate term $ \gterm $ and any deep letter $ i $.
By \ref{enu.pertur.2.75}, the letter $ i-1 $ is non-deep, and hence $ i-1\in\word $ for some $ \wordkl $ involved in $ \gterm $.
By Proposition~\ref{p.tree.prop}\ref{p.IHC}, the word $ \wordkl $ satisfies $ \IHC $.
If $ i\in\wordsetkl $, then by the last property in \ref{enu.pertur.2.75}, the letter $ i $ must already be in $ \word $,
otherwise $ \wordkl $ cannot not satisfy $ \IHC $.
If $ i\notin\wordsetkl $, then by \ref{enu.pertur.alphabet} the letter is $ k' $-aligned, which forces $ \gterm=0 $.

We are now ready to bound $ \ppin' $.
For each non-deep $ \vecrho $,
invoke \eqref{e.det.expanded} for $ (\vecs,\vecd) \mapsto (\vecs,\vecd^{\vecrho}) $
and, as said previously, with $ \alphabet^\times(kk'),\vecUDc^\times(\wordkl),\isleset^\times(kk') $ being the inputs.
The result gives an expansion of $ \punder(\vecs,\vecd^{\vecrho},N) $ into sums of generic terms.
Apply $ \sumie' $ to the result.
The conclusion of the previous paragraph asserts that $ \sumie' $ kills all degenerate terms. 
Let $ \DDprnd_n $ denote the quantities obtained by removing non-preferred and degenerate terms from $ \DD_n $.
We have
\begin{align}
	\label{e.upbd.expansion}
	\ppin' = \Sumie' \sum_{n=0}^\infty \frac{1}{n!} \DD_n(\vecs,\vecd^{\vecrho},N)
	=
	\sum_{n=0}^\infty  \Sumie' \frac{1}{n!} \DDprnd_n(\vecs,\vecd^{\vecrho},N).
\end{align}
Swapping $ \sumie' $ and $ \sum_n $ requires the absolute convergence of the right side of \eqref{e.upbd.expansion}, which we verify next.
As discussed previously,
each nonzero term in \eqref{e.upbd.expansion} only involves words in $ \{\islesetkl[kk']\}_{k,k'} $ and satisfies Assumption~\ref{assu.strict}.
Proposition~\ref{p.term.bd} hence applies to the term.
Recall that $ \norm{\gterm} $ counts the total number of $ \opuci{\ldots} $ involved in $ \gterm $.
Recall the definition of $ \DD_n $ from \eqref{e.DD}, \eqref{e.det.islse}, and \eqref{e.simon}--\eqref{e.trace}.
For each $ \gterm $ in $ \DD_n $, we have $ \norm{\gterm} \geq n $, and there are at most $ n! c^n $ terms in $ \DD_n $, for some $ c=c(\icm,m) $.
These properties continue to hold for $ \DDprnd_n $ as it is obtained from $ \DD_n $ by removing terms.
Hence, the application of Proposition~\ref{p.term.bd} gives
\begin{align*}
	\big|\tfrac{1}{n!}\DDprnd_n(\vecs,\vecd^{\vecrho},N) \big|
	\leq
	\exp\big( - N\raterw((\icvecx,\icveca)\xrightarrow{t}(\vecx^{\vecrho},\veca^{\vecrho})) \big)
	\,
	\exp\big( c n - \tfrac{N}{c}\, (n-c)_+ \big),	
\end{align*}
for some $ c=c(t,\icvecx,\icveca,\vecx^{\vecrho},\veca^{\vecrho}) $.
This inequality verifies the absolute convergence of the right side of \eqref{e.upbd.expansion} and gives
\begin{align*}
	\limsup_{N\to\infty} \frac{1}{N}\log \ppin' 
	\leq 
	- \min_{\vecd^{\vecrho}\text{ legal }} \big\{  \raterw\big((\icvecx,\icveca)\xrightarrow{t}(\vecx^{\vecrho},\veca^{\vecrho})\big) \big\}
	\leq
	- \raterw\big( \parab[1\cdots \icm](0) \xrightarrow{t}f\big) + \e.
\end{align*}
Combining this result with \eqref{e.<ppin'} gives the desired result of Part~\ref{t.rw.upper.1}.

\bigskip

\ref{t.rw.upper.2}\
Fix a compact $ \calK \subset \Lip $.
With Proposition~\ref{p.HLc.topo.weak.HLc}\ref{p.HLc.topo}, standard point-set topology arguments leverage Part~\ref{t.rw.upper.1} into
\begin{align}
	\label{e.trw.upbd.goal.}
	\limsup_{N\to\infty} \,
	\frac{1}{N} \log \P_{\parab[1\cdots\icm](0)}[ f\in \calK ]  
	\leq 
	- \inf_{f\in\calK}\raterw\big( \parab[1\ldots\icm](0) \xrightarrow{\scriptscriptstyle t} f \big).
\end{align}
%Fix any $ \e>0 $.
%Let $ B(f,\delta):= \{ f'\in\Lip : \dist(f',f)<\delta_f \} $,
%and $ \HLsp|_t := \{ f'\in\Lip : \parab[1\ldots\icm](t)\leq f\leq\parab[1\ldots\icm](0) \} $,
%which should be viewed as the fixed-time analog of $ \HLsp $.
%For any $ f\in\calK\cap (\HLsp|_t) $, by Part~\ref{t.rw.upper.1}, there exists $ \delta_{\e,f}>0 $ such that
%\begin{align}
%	\label{e.trw.upbd.goal.1}
%	\limsup_{N\to\infty} \,
%	\frac{1}{N} \log \P_{\parab[1\cdots\icm](0)}\big[ \hh_N(t) \in B(f,\delta_{\e,f}) \big]  
%	\leq 
%	- \raterw\big( \parab[1\ldots\icm](0) \xrightarrow{\scriptscriptstyle t} f \big) + \e.
%\end{align}
%For any $ f\in\calK\setminus (\HLsp|_t) $,
%with $ \HLsp|_t $ being closed,  
%find a small enough $ \delta_{f}>0 $ be such that $ B(f,\delta_{f}) \subset \Lip\setminus\HLsp|_t $.
%Proposition~\ref{p.HLc.topo.weak.HLc}\ref{p.HLc.topo} gives
%\begin{align}
%	\label{e.trw.upbd.goal.2}
%	\limsup_{N\to\infty} \,
%	\frac{1}{N} \log \P_{\parab[1\cdots\icm](0)}\big[ \hh_N(t) \in B(f,\delta_{f}) \big]  
%	=
%	- \infty
%	=
%	- \raterw\big( \parab[1\ldots\icm](0) \xrightarrow{\scriptscriptstyle t} f \big).
%\end{align}
%Since $ \calK $ is compact, the open cover $ \{ B(f,\delta_{f,\e}), B(f,\delta_f)\} $ of $ \calK $ reduces to a finite one.
%Hence \eqref{e.trw.upbd.goal.1}--\eqref{e.trw.upbd.goal.2} imply \eqref{e.trw.upbd.goal.}.
%
%We now begin the proof.
Next, fix any $ \e>0 $ and any $ g,f\in\Lip $.
Lemma~\ref{l.ic} gives an estimate of the stability of $ \P_{\hh_N(0)}[\Cdot] $ in $ \hh_N(0) $.
Apply the lemma with the given $ \e,g,f $ to obtain the $ \delta $.
Now, construct a massif $ \parab[1\ldots\icm](0) $ such that $ \dist(g,\parab[1\ldots\icm](0)) < \min\{\delta,\e\} $.
The result of applying Lemma~\ref{l.ic} gives, for all $ g\in\Lip $ with $ \dist(\hh_N(0),g)<\delta $,
$
	\P_{\hh_N(0)}[ \dist(\hh_N(t),f) < \e ]
	\leq
	\P_{\parab[1\cdots\icm](0)}[ \dist(\hh_N(t),f) \leq 2\e ] + c\exp(-\e^{-1} N),
$
where $ c $ is a universal constant.
The set $ \{ \til{f} \in\Lip: \dist(\til{f},f) \leq 2\e \} $ is compact, so by \eqref{e.trw.upbd.goal.}
\begin{align*}
	\limsup_{N\to \infty} 
	\sup_{\dist(\hh_N(0),g)<\delta} 
	\Big\{ 
		\frac{1}{N} \log \P_{\hh_N(0)}[ \dist(\hh_N(t),f) < \e ] 
	\Big\}
	\leq
	- \min\Big\{ 
		\inf_{\dist(\til{f},f) \leq 2\e}
		\raterw( \parab[1\ldots\icm](0) \xrightarrow{ t} \til{f} ), \e^{-1} 
	\Big\}.
\end{align*}
On the right side, sending $ \e \to 0 $ and using the fact that $ \raterw( \Cdot \xrightarrow{\scriptscriptstyle t} \Cdot) $
is \ac{lsc} (from Lemma~\ref{l.raterw.lsc}) gives $ -\raterw( g \xrightarrow{\scriptscriptstyle t} f) $.
On the left side, note that the quantity decreases when $ \delta $ and $ \e $ decrease.
The desired result follows.
\end{proof}

\begin{figure}[h]
\centering
\begin{minipage}[t]{.51\linewidth}
\centering
\frame{\includegraphics[width=.76\linewidth]{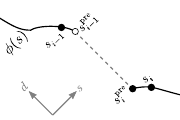}}
\caption{Perturbation of $ s_i $ when $ \phi $ is non-continuous at $ s_i^\text{pre} $.}
\label{f.pfup-0}
\end{minipage}
\hfil
\begin{minipage}[t]{.47\linewidth}
\centering
\frame{\includegraphics[width=.8\linewidth]{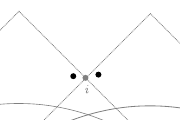}}
\caption{%
Perturbation in a borderline configuration.
The gray point is $ (b_i,y_i) $.
The two black points are the duplicated and perturbed points.%
}
\label{f.pflw-1}
\end{minipage}
\\
\begin{minipage}[t]{.51\linewidth}
\centering
\begin{minipage}[b]{.76\linewidth}
\frame{\includegraphics[width=\linewidth]{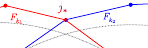}}\\
\frame{\includegraphics[width=\linewidth]{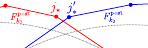}}
\end{minipage}
\caption{Duplicating and perturbing a shared letter in $ \{F_k\}_k $.}
\label{f.duplicate.perturb}
\end{minipage}
\hfil
\begin{minipage}[t]{.47\linewidth}
\centering
\begin{minipage}[b]{.8\linewidth}
\frame{\includegraphics[width=\linewidth]{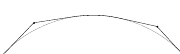}}\\
\frame{\includegraphics[width=\linewidth]{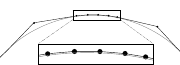}}
\end{minipage}
\caption{Adding points to make $ \wordF(F_k) $ an isle.}
\label{f.addtoisle}
\end{minipage}
\end{figure}

\begin{thmfixedtimelow}
%\label{t.rw.lower}
For any $ f,g\in\Lip $,
\begin{align*}
	- \raterw( g \xrightarrow{ t} f)
	\leq
	\liminf_{\delta\to 0}
	\liminf_{N\to \infty} \ \inf_{\dist(\hh_N(0),g)<\delta} 
	\Big\{ 
		\frac{1}{N} \log \P_{\hh_N(0)}[ \dist(\hh_N(t),f) < \delta ] 
	\Big\}.
\end{align*}
\end{thmfixedtimelow}

\begin{proof}
\textbf{Step 1: discretizing the initial condition.}\
Fix any $ \e>0 $.
Refer to \eqref{e.raterw.gxa}--\eqref{e.raterw.gf}.
We find that there exists $ \icvecx=(\icx_1<\ldots<\icx_{\icm}) $,
such that, with $ \ica_k := g(\icx_k) $, $ \raterw( \parab[1\ldots\icm](0) \xrightarrow{\scriptscriptstyle t} f) < \raterw( g \xrightarrow{\scriptscriptstyle t} f) + \e $.
%\begin{align}
%%	\label{e.lwbd.1}
%	\raterw\big( \parab[1\ldots\icm](0) \xrightarrow{t} f \big) < \raterw\big( g \xrightarrow{t} f\big) + \e.
%\end{align}
Next, apply Lemma~\ref{l.ic} with the given $ \e,g,f $ to obtain the $ \delta=:\delta_0 $.
Add points to $ \{\icx_k\}_k $ if necessary (with $ \ica_k := g(\ica_k) $) so that $ \dist(\parab[1\ldots\icm](0),g) < \delta_0 $.
The inequality $ \raterw( \parab[1\ldots\icm](0) \xrightarrow{\scriptscriptstyle t} f) < \raterw( g \xrightarrow{\scriptscriptstyle t} f) + \e $ continues to hold
because adding points to $ \{\icx_k\}_k $ only decreases the left side.
The result of applying Lemma~\ref{l.ic} gives, for some universal $ c<\infty $,
\begin{align}
	\label{e.lwbd.ic}
	\dist(\hh_N(0),g)<\delta_0 \text{ implies }
	\P_{\parab[1\cdots\icm](0)}\big[ \dist(\hh_N(t),f) < \e \big]
	\leq
	\P_{\hh_N(0)}\big[ \dist(\hh_N(t),f) < 2\e \big] + c\,e^{-\e^{-1} N}. 	
\end{align}

Pick $ \vecy=(y_1<\ldots<y_{m'}) $ and $ \delta_1>0 $ such that, with $ b_i:=f(y_i) $,
\begin{align}
	\label{e.lwbd.constr.1}
	\big\{ |\hh_N(t,y_1)-b_1| < \delta_1, \ldots , |\hh_N(t,y_{m'})-b_{m'}| < \delta_1  \big\}
	\subset
	\big\{ \dist(\hh_N(t),f) < \e \big\}.
\end{align}
Such $ \vecy $ and $ \delta_1 $ exist because $ \hh_N\in\Lip $ and because $ \dist $ measures the uniform norm over compact intervals.
Further, referring to \eqref{e.raterw.gf}, by making the mesh of $ \vecy $ small enough and making $ [y_1,y_{m'}] $ a wide enough interval,
we have $ \raterw( \parab[1\ldots\icm](0) \xrightarrow{\scriptscriptstyle t} (\vecy,\vecb)) < \raterw( \parab[1\ldots\icm](0) \xrightarrow{\scriptscriptstyle t} f) +\e $.
Hence
\begin{align}
	\label{e.lwbd.constr.2}
	\raterw\big( \parab[1\ldots\icm](0) \xrightarrow{t} (\vecy,\vecb)\big)
	<
	\raterw\big( g \xrightarrow{t} f\big) + 2\e.
\end{align}

We seek to modify $ (\vecy,\vecb) \mapsto (\vecx,\veca) $ so that the following holds.
\begin{enumerate}[leftmargin=20pt, label=(\Roman*)]
\item \label{enu.lwbd.a} The points $ (\icvecx,\icveca) $ and $ (\vecx,\veca) $ satisfy Assumption~\ref{assu.strict}.
\item \label{enu.lwbd.b} The condition \eqref{e.lwbd.constr.1} holds for $ (\vecy,\vecb) \mapsto (\vecx,\veca) $ and for some $ \delta_1 \mapsto \delta_2 $.
\item \label{enu.lwbd.c} The condition\eqref{e.lwbd.constr.2} holds for $ (\vecy,\vecb) \mapsto (\vecx,\veca) $.
\item \label{enu.lwbd.d} The minimizer of  $ (\icvecx,\icveca)\xrightarrow{\scriptscriptstyle t} (\vecx,\veca) $ is unique.
\item \label{enu.lwbd.e} For the unique minimizer $ \{F_k\} $ of $ (\icvecx,\icveca)\xrightarrow{\scriptscriptstyle t} (\vecx,\veca) $,
	the words $ \{\wordF(F_k)\}_k $ are mutually disjoint.
\item \label{enu.lwbd.f} For the unique minimizer $ \{F_k\} $ of $ (\icvecx,\icveca)\xrightarrow{\scriptscriptstyle t} (\vecx,\veca) $,
	the words $ \{\wordF(F_k)\}_k $ are isles.
\end{enumerate}
Property~\ref{enu.lwbd.b} holds as long as the perturbation is small enough.
Fix a minimizer $ \{\til{F}_k\}_k $ of $ (\icvecx,\icveca)\xrightarrow{\scriptscriptstyle t} (\vecy,\vecb) $.
Properties~\ref{enu.lwbd.a} and \ref{enu.lwbd.c} are achievable except in the case depicted in Figure~\ref{f.pflw-1}.
Namely, $ (y_i,b_i) $ sits on the boundary of $ \hyp(\parab[k](0))\setminus \hyp(\parab[k](t))^\circ $ for two $ k $'s,
and the letter $ i $ is shared by two $ \wordF(\til{F}_k) $'s.
We circumvent the issue by duplicating $ (y_i,b_i) $ and perturbing accordingly; see Figure~\ref{f.pflw-1} for an illustration.
The resulting $ (\vecx,\veca)=(x_i,a_i)_{i=1}^{m} $ may have $ m\geq m' $, 
but the desired properties \ref{enu.lwbd.a}--\ref{enu.lwbd.c} hold.

We continue to modify $ (\vecx,\veca) $ to meet \ref{enu.lwbd.d}.
Property~\ref{enu.lwbd.d} is achievable.
To see why, view $ \raterwRel{ \Rwf(\wordkl[\word][k][k]) }{ \parab[k](t) } $ as a function of $ (\vecx,\veca) $,
denoted $ \phi_{k,\word}(\vecx,\veca) $, 
with the convention that $ \phi_{k,\word}(\vecx,\veca):=+\infty $ whenever $ \raterwRel{ \Rwf(\wordkl[\word][k][k]) }{ \parab[k](t) } $ is undefined
or whenever $ (\Rwf(\wordkl[\word][k][k]))|_{y=x_j} > a_j $ for some $ j $ (see Remark~\ref{r.Rwf}).
It is readily verified that for those $ (\vecx,\veca) $ satisfying Assumption~\ref{assu.strict},
either $ \phi_{k,\word}(\vecx,\veca) $ is analytic or $ =+\infty $,
and $ \phi_{k,\word}= \phi_{k',\word'} $ as functions only when $ k=k' $ and $ \word=\word' $.
Hence, there exists an arbitrarily small perturbation of $ (\vecx,\veca) $ that satisfies \ref{enu.lwbd.d} while maintaining \ref{enu.lwbd.a}--\ref{enu.lwbd.c}.

The unique minimizer $ \{F_k\}_k $ may not satisfy \ref{enu.lwbd.e},
so we seek to further modify $ (\vecx,\veca) $.
By Lemma~\ref{l.raterw.minimizer.word.}, the only way to have shared letters in $ \{\wordF(F_k)\}_k $ is $ \wordF(F_{k_1})\cap \wordF(F_{k_2}) = \{j_*\} $, with $ k_1<k_2 $ and $ j_*= \wordF(F_{k_1})_{|\wordF(F_{k_1})|} = \wordF(F_{k_2})_1 $.
In this case we modify  $ (\vecx,\veca) $ by duplicating $ (x_{j_*},a_{j_*}) $ and perturbing accordingly.
Slightly abusing notation, we let $ (x_{j'_*},a_{j'_*}) $ denote the duplicated point, 
let $ \alphabet' := \alphabet\cup\{j'_*\} $, with the convention $ j_*<j'_*<j_*+1 $,
and label post-modification functions and variables by `post'.
We perturb $ (x_{j_*},a_{j_*}) $ and $ (x_{j'_*},a_{j'_*}) $ in such a way
that $ j_* \notin \wordF(F^\text{post}_{k_2}) $ and $ j'_* \notin \wordF(F^\text{post}_{k_1}) $
as depicted in Figure~\ref{f.duplicate.perturb}, 
which is achievable because $ k_1<k_2 $.
As long as the perturbation is small enough, $ \{F^\text{post}_k\}_k $ remains the unique minimizer.
To see why, note that since $ \{F_k\}_k $ was the \emph{unique} minimizer,
the only potential competitor is the following $ \{f_k\}_k $.
Let $ \word^{(k_1)} := (\wordF(F_{k_1})\setminus\{j_*\})\cup\{j_*'\} $,
$ \word^{(k_2)} := (\wordF(F_{k_2})\setminus\{j'_*\})\cup\{j_*\} $,
\begin{align*}
	f_{k_1} 
	&:= 
	\mathrm{argmin} \big\{ \raterwREl{f}{\parab[k_1](t)} :\, \wordF(f) \supset \word^{(k_1)}, \, f(x_j) \leq a_j \text{ for all } j\in\alphabet'\big\},
\\
	f_{k_2} 
	&:= 
	\mathrm{argmin} \big\{ \raterwREl{f}{\parab[k_2](t)} :\, \wordF(f) \supset \word^{(k_2)}, \, f(x_j) \leq a_j \text{ for all } j\in\alphabet'\big\},
\end{align*}
and $ f_k := F^\text{post}_k $ for $ k\neq k_1,k_2 $.
Namely, we attempt to swap the letters $ j_* $ and $ j'_* $ between 
$ F^\text{post}_{k_1} $ and $ F^\text{post}_{k_2} $ under the constraint $ f(x_j) \leq a_j $, for all $ j\in\alphabet' $.
Given that $ k_1<k_2 $, simple geometric considerations show that the last constraint forces either $ \{j_*,j'_*\} \in \wordF(f_{k_1}) $ or $ \{j_*,j'_*\} \in \wordF(f_{k_2}) $ to hold.
In either case, a rewiring argument similar to those in the preceding shows that $ \{f_k\}_k $ underperforms $ \{F^\text{post}_{k}\}_k $.

The unique minimizer may not have already satisfied \ref{enu.lwbd.f}.
Add points to $ (\vecx,\veca) $ as depicted in Figure~\ref{f.addtoisle} to make \ref{enu.lwbd.f} hold.
These added points can be made arbitrarily close to $ \parab[k](t) $,
and doing so controls the change in $ \raterwRel{F_k}{\parab[k](t)} $ within an arbitrarily small amount.
Since $ \{F_k\}_k $ was the \emph{unique} minimizer before adding the points,
it remains the unique minimizer and \ref{enu.lwbd.a}--\ref{enu.lwbd.e} continue to hold
as long as the added points are close enough to $ \parab[k](t) $.

Given the preceding $ (\vecx,\veca) $, we proceed to construct the corresponding $ s_i $ and $ d^\pm_i $.
As always, $ s_i := \frac12(x_i+a_i) $ and $ d_i := \frac12(-x_i+a_i) $.
For $ \delta_3>0 $, let $ d_i^\pm := d_i\pm\delta_3 $, and accordingly $ (x^\pm_i,a^\pm_i):=(x_i\pm\delta_3,a_i\pm\delta_3) $.
For $ \vecrho\in\{\pm\}^m $, let $ \alphabet^{\vecrho}(kk'),\vecUDc^{\vecrho}(\wordkl),\Rwf^{\vecrho}(\wordkl[\word][k][k]),\isleset^{\vecrho}(kk') $
denote the corresponding sets, variables, and functions with $ (\vecx,\veca) \mapsto (\vecx^{\vecrho},\veca^{\vecrho}) $.
By the construction of $ (\vecx,\veca) $,
the following holds for all $ \delta_3>0 $ small enough and all $ \vecrho\in\{\pm\}^m $.
\begin{enumerate}[leftmargin=20pt, label=(\roman*)]
\item \label{enu.lwbd.1}
The points $ (\icvecx,\icveca) $ and $ (\vecs,\vecd^{\vecrho}) $ satisfy Assumptions~\ref{assu.strict} and \ref{assu.lwbd}, and the unique minimizer of \eqref{e.raterw.xa} is given by $ \{\Rwf^{\vecrho}(\wordkl[\word^{(k)}][k][k])\}_k $, for some word $ \word^{(k)} $ that does not depend on $ \vecrho $.
\item \label{enu.lwbd.2}
The sets and variables $ \alphabet^{\vecrho}(kk'),\vecUDc^{\vecrho}(\wordkl),\isleset^{\vecrho}(kk') $ do not depend on $ \vecrho $.
\item \label{enu.lwbd.3}
$ \raterw( (\icvecx,\icveca) \xrightarrow{\scriptscriptstyle t} (\vecx^{\vecrho},\veca^{\vecrho})) < \raterw( g \xrightarrow{\scriptscriptstyle t} f)+2\e $.
\item \label{enu.lwbd.4}
$ \cap_{i=1}^m \{ d^-_i < \hd_N(t,s_i) \leq \hd_N(t,s_i^-) \leq d^+_i \} \subset \{ \dist(\hh_N(t),f) < \e \} $.
\end{enumerate}

\medskip\noindent{}%
\textbf{Step 2: applying the determinantal analysis.}\
Given \ref{enu.lwbd.1}--\ref{enu.lwbd.2}, we apply the up-down iteration and isle factorization with $ (\vecs,\vecd) \mapsto (\vecs,\vecd^{\vecrho}) $.
Doing so gives an expansion of $ \punder(\vecs,\vecd^{\vecrho},N) $.
Apply $ \sumie $ to the expansion to get
\begin{align*}
%	\label{e.lwbd.expansion}
	\ppin = \Sumie \sum_{n=0}^\infty \frac{1}{n!} \DD_n(\vecs,\vecd^{\vecrho},N)
	=
	\sum_{n=0}^\infty  \Sumie \frac{1}{n!} \DDprnd_n(\vecs,\vecd^{\vecrho},N).
\end{align*}
As before, $ \DDprnd_n $ is obtained from $ \DD_n $ (defined in \eqref{e.DD}) by removing any non-preferred and degenerate terms.
Under  Assumptions~\ref{assu.strict} and \ref{assu.lwbd},
Proposition~\ref{p.term.dominanace}
asserts that $ \gterm_\star := \prod_k \tr( \inddown_{\Ic{k}} \opuci{\word^{(k)}} \inddown_{\Ic{k}} ) $ is the dominant term.
Further, as argued previously in the proof of the upper bound, each $ \DDprnd_n $ contains at most $ n!c^n $ terms, for some $ c=c(\icm,m) $, and each term $ \gterm $ in $ \DDprnd_n $ has $ \norm{\gterm} \geq n $.
Applying Propositions~\ref{p.term.bd} and \ref{p.term.dominanace} gives
\begin{align}
	\label{e.lwbd.expansion}
	\ppin =
	\prod\nolimits_{k} \parityc(\word^{(k)}) \cdot \Sumie \gterm_{\star}(\vecs,\vecd^{\vecrho},N) + \Sumie R(\vecs,\vecd^{\vecrho},N). 
\end{align}
Here $ R(\vecs,\vecd^{\vecrho},N) $ is exponentially subdominant to $ \gterm_{\star}(\vecs,\vecd^{\vecrho},N) $.

We next identify the dominant term among the candidates $ \gterm_{\star}(\vecs,\vecd^{\vecrho},N) $, $ \vecrho\in\{\pm\}^m $.
Recall from Proposition~\ref{p.term.dominanace} that each of them has rate
$ \raterw( (\icvecx,\icveca) \xrightarrow{\scriptscriptstyle t} (\vecx^{\vecrho},\veca^{\vecrho})) $,
which can be written as $ \sum_k \raterwRel{ \Rwf^{\vecrho}( \wordkl[\word^{(k)}][k][k] ) }{ \parab[k] } $.
Let $ \vecUDc(k) := \vecUDc^{\vecrho}(\wordkl[\word^{(k)}][k][k]) $, which is $ \vecrho $-independent by \ref{enu.lwbd.2}.
Varying $ \vecrho $ in the rate $ \raterwRel{ \Rwf^{\vecrho}( \wordkl[\word^{(k)}][k][k] ) }{ \parab[k] } $ amounts to varying the $ d_i $'s.
It is readily checked that the rate increases in $ d_i $ if $ \UDc(k)_i = \up $, and decreases in $ d_i $ if $ \UDc(k)_i = \down $.
Also, note that under Assumption~\ref{assu.lwbd}, the set of words $ \{\word^{(k)}\}_k $ partitions $ \{1,\ldots,m\} $.
Hence the unique dominant term in \eqref{e.lwbd.expansion} is $ \gterm_{\star}(\vecs,\vecd^{\vecrho_\star},N) $, where $ (\vecrho_\star)_i := - $ if $ \UDc(k)_i=\up $ and  $ (\vecrho_\star)_i := + $ if $ \UDc(k)_i=\down $, for the unique $ k $ such that $ \word^{(k)}\ni i $.
Recall that the inclusion-exclusion sum comes with the sign factor $ \prod_{i=1}^m \rho_i $,
and for $ \vecrho_\star $ this sign factor is $ \prod_{i} (\rho_\star)_i = \prod_{k} \parityc(\word^{(k)}) $.
Hence
$
	\lim_{N\to\infty} \frac{1}{N}\log\ppin = - \raterw( (\icvecx,\icveca) \xrightarrow{\scriptscriptstyle t} (\vecx^{\vecrho_\star},\veca^{\vecrho_\star})).
$
Combining the last result with \ref{enu.lwbd.3}--\ref{enu.lwbd.4} and \eqref{e.lwbd.ic} gives the desired result.
\end{proof}

\section{Hydrodynamic large deviations: elementary solutions}
\label{s.matching}

Here we prove Proposition~\ref{p.matching} and establish the relevant properties of elementary solutions.
Then, in the last subsection, we combine the preceding results to conclude \hyperref[t.main]{Main Theorem}.

\subsection{Wedge initial conditions}
\label{s.matching.icm=1}
Fix $ t>0 $, $ (\icx_{1},\ica_{1}) $, and $ (\vecx,\veca)=(x_i,a_i)_{i=1}^m $ that satisfy the discretized Hopf--Lax condition~\eqref{e.HLc.xa}.
Recall $ \Rwf:=\Rwf( \wordkl[12\cdots m][1][1] ) $, the backward Hopf--Lax evolution $ \HLb $, and the elementary solution $ h_\star(\tau) := \HLb_{t-\tau}(\Rwf) $ from Section~\ref{s.results.matching}.

Let $ u_\star := \partial_x h_\star $.
The evolution of $ u_\star $ can be described by a geometric transformation applied to the graph of $ u_\star(t) $.
To begin, the graph of $ u_\star(t) = u_\star(t,\Cdot) = \partial_x \rwfsymb $ consists
of segments that are flat or coincide with $ \partial_x \parab[1](t) $ and finitely many jumps.
We leave $ u $ undefined at the jumps.
The \tdef{complete graph} of $ u_\star(t) $, denoted \tdef{$ \CG(u_\star(t)) $}, 
is the curve in $ \R^2=\{(x,u)\} $ formed by the graph of $ u_\star(t) $ 
augmented by vertical line segments at the jumps of $ u_\star(t) $.
By the method of characteristics, $\CG(u_\star(s)) $ can be obtain from $ \CG(u_\star(t)) $ by a shear-and-cut transformation;
see \cite[Ch\,2]{whitham11}.
First apply the shear $ (x,u)\mapsto (x+(t-\tau)u,u) $ to $ \CG(u_\star(t)) $.
The sheared graph may have overhangs, which should be rectified by vertical cuts that conserve area.
See Figure~\ref{f.sheerandcut} for an illustration.
The shear-and-cut transformation satisfies the semigroup property.
Namely, for $ \tau<\tau'\in[0,t) $, applying the transformation to $ u_\star(t) \mapsto u_\star(\tau') $ and then to $ u_\star(\tau')\mapsto u_\star(\tau) $ 
produces the same result as applying the transformation to $ u_\star(t) \mapsto u_\star(\tau) $.

\begin{figure}[h]
\centering
\begin{minipage}[t]{.327\linewidth}
\frame{\includegraphics[width=\linewidth]{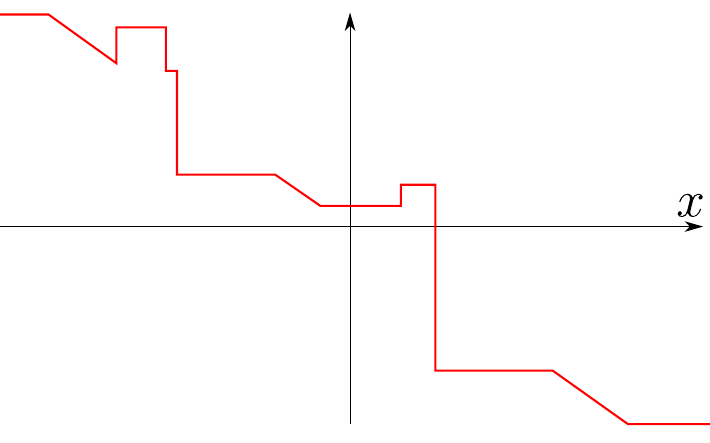}}
\end{minipage}
\hfill
\begin{minipage}[t]{.327\linewidth}
\frame{\includegraphics[width=\linewidth]{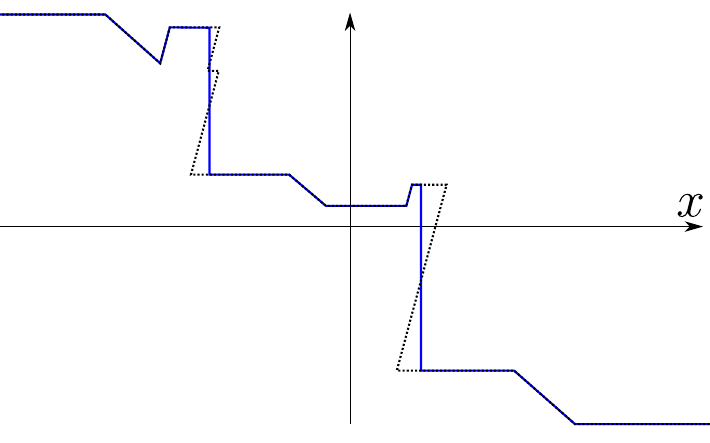}}
\end{minipage}
\hfill
\begin{minipage}[t]{.327\linewidth}
\frame{\includegraphics[width=\linewidth]{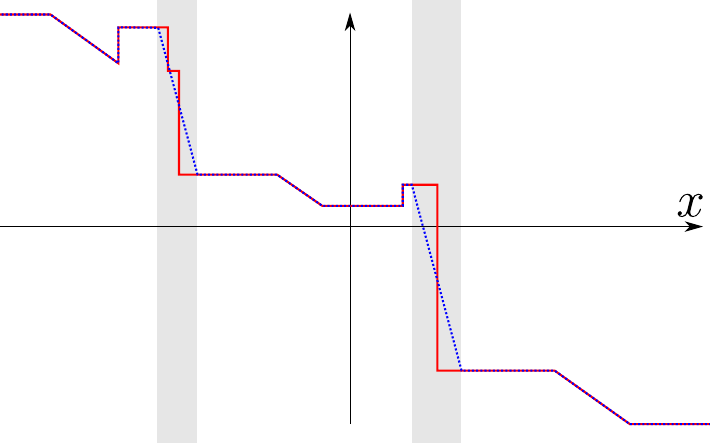}}
\end{minipage}
\caption{An realization of the sheer-and-cut procedure. 
The red curves in the left and right figures are $ \CG(u_\star(t)) $.
In the middle figure, the dashed curve is obtained by applying the (backward) sheer to $ \CG(u_\star(t)) $, and the blue curve is $ \CG(u_\star(\tau)) $.
The right figure concerns the proof of Proposition~\ref{p.raterw.sg}: The blue dashed curve is $ \CG(v) $ and the gray zones are $ U_i\times\R $.}
\label{f.sheerandcut}
\vspace{-10pt}
\end{figure}

Let us summarize some properties of $ u_\star $.
For $ \tau\in(0,t) $ and $ |x-\icx_{1}|\leq \tau $, 
$ \CG(u_\star(\tau)) $ consists of line segments of slopes $ 1/(t-\tau) $ or $ 0 $, 
vertical line segments, and segments that coincide with $ \partial_x \parab[1](\tau) $,
and for $ |x-\icx_{1}| > \tau $, $ u_\star(\tau,x)=\partial_x \parab[1](\tau,x) $.
For any $ \tau\in(0,t) $, the function $ u_\star(\tau) $ has \emph{downward jumps} only.
Namely, at every jump the left limit is larger than the right limit.
Every weak solution of the backward Burgers equation $ \partial_\tau u = -\frac12\partial_x (u^2) $ is also a weak solution of the forward equation \eqref{e.burgers}.
Let $ \shock_1(\tau)<\ldots<\shock_{n(\tau)}(\tau)\in( -\tau+\icx_{1},\icx_{1}+\tau) $ label the locations of these jumps.
At the terminal time $ t $, the $ \shock_i(t) $'s are exactly those $ x_i $'s where $ \Rwf $ has a $ \wedge $ kink (concave kink);
going backward in time, these jumps merge but never branch, so $ n(\tau) $ decreases as $ \tau $ decreases.
Each trajectory is uniformly Lipschitz, and is analytic except when merges happen, which happen only finitely many times in $ [0,t] $.

Consider a piecewise analytic weak solution $ u $ of \eqref{e.burgers} with finitely many piecewise jumps, and assume the trajectories of the jumps are piecewise analytic. 
It is readily verified that $ \partial_\tau \rateber(u) - \partial_x \rateberc(u) $ is supported around the jumps. Further using the Rankine–Hugoniot condition shows that the quantity is positive if and only if the jump is downward.
For $ u_\star $ this condition always holds in $ 0<\tau<t $.
Namely, all jumps of $ u_\star(\tau) $ in $ 0<\tau<t $ correspond to antishocks/non-entropic shocks.

\begin{proof}[Proof of Proposition~\ref{p.matching} for wedge initial conditions]
Given the preceding discussion, the proof follows by calculus, as follows.
Recall (see \eqref{e.raterwrel} and \eqref{e.raterw.gf}) that 
$ \raterw( \parab[1](0) \xrightarrow{\scriptscriptstyle \tau } h_\star(\tau) ) = \raterwRel{ h_\star(\tau) }{ \parab[1](\tau) } = \int_{\R} \d x \, ( \rateber(u_\star(\tau)) - \rateber(\partial_x \parab[1](\tau)) ) $.
Differentiate this quantity in $ \tau $ to get
\begin{align*}
%	\label{e.matching.icm=1.1}
	\partial_\tau \raterw( \parab[1](0) \xrightarrow{\tau } h_\star(\tau) ) 
	=
	\int_{-t+\icx_{1}}^{\icx_{1}+t} \d x \, \partial_\tau  \rateber(u_\star(\tau,x))
	-
	\int_{-t+\icx_{1}}^{\icx_{1}+t} \d x \, \partial_\tau  \rateber(\partial_x \parab[1](\tau,x)),
\end{align*}
where the first integrand is interpreted in the weak sense.
For the second integrand, since $ \partial_x \parab[1](t,x) $ is a continuous and piecewise smooth solution of \eqref{e.burgers},
$ \partial_\tau \rateber(\partial_x \parab[1](\tau,x)) = \partial_x \rateberc(\partial_x \parab[1](\tau,x)) $.
This identity turns the resulting integral into
$ \rateberc(\partial_x \parab[1](\tau,x))|_{x=-t+\icx_{1}}^{x=\icx_{1}+t} $.
Further using $ \partial_x\parab[1](\tau,\icx_1\pm t) = u_\star(\tau,\icx_1\pm t) $ turns the last expression into
$
%	\alpha_2 = 
%	\rateberc(u_\star(\tau,x))|_{x=-t+\icx_{1}}^{x=\icx_{1}+t}
%	=
	\int_{-t+\icx_{1}}^{\icx_{1}+t} \d x \, \partial_x \rateberc(u_\star),
$
interpreted in the weak sense.
Hence
%\begin{align*}
$
	\partial_\tau 
	\raterw( \parab[1](0) \xrightarrow{ \scriptscriptstyle \tau } h_\star(\tau) ) 
	=
	\int_{\R} \d x \, ( \partial_\tau \rateber(u_\star) + \partial_x \rateberc(u_\star) ).
$
%\end{align*}
As mentioned previously, the integrand is non-negative.
Integrating in $ \tau $ concludes the desired result.
\end{proof}

\subsection{A subadditive property and an additive property}

First, by \hyperref[t.rw]{Fixed-time Theorem} and the submultiplicativity of Markov transition probabilities,
\begin{align}
	\label{e.sg.<=}
	\raterw\big( (\icx_1,\ica_1) \xrightarrow{t} (\vecx,\veca) \big)
	=
	\raterw\big( \parab[1](0) \xrightarrow{t} h_\star(t) \big)
	\leq
	\raterw\big( \parab[1](0) \xrightarrow{\tau}f \big)
	+
	\raterw\big( f \xrightarrow{t-\tau} (\vecx,\veca) \big).
\end{align}
%To see why, recall that \hyperref[t.rw]{Fixed-time Theorem} gives the \ac{LDP} of $ \hh_N(t) $ from a generic $ \hh_N(0) \approx g $.
%Applying the contraction principle reduces the \ac{LDP} to finite dimensional distributions of the terminal conditions:
%\begin{align*}
%	\lim_{\delta\to 0}\limsup_{N\to\infty} \sup_{ \dist(\hh_N(0),g) < \delta}
%	\Big| 
%		\frac{1}{N} \log \P_{\hh_N(0)}\big[ |\hh_N(t,x_i)-a_i|<\delta, i=1,\ldots,m \big]
%		+ \raterw\big(g\xrightarrow{t} (\vecx,\veca)\big)
%	\Big|
%	= 0.
%\end{align*}
%The left side of \eqref{e.sg.<=} is the rate of the event that the \ac{TASEP} starts from $ \hh_N(0)=\parab[1](0) $
%and ends up being approximately $ a_i $ at $ x_i $ at time $ t $.
%The entire right side is the rate of the event
%that the \ac{TASEP} starts from $ \hh_N(0)=\parab[1](0) $,
%be approximately $ f $ at time $ \tau $, and ends up being approximately $ a_i $ at $ x_i $ at time $ t $.
%The latter even is smaller than the former, and hence the inequality \eqref{e.sg.<=}.
%
%The next proposition shows that \eqref{e.sg.<=} becomes an equality when $ f=h_\star(\tau) $
%and characterizes $ \raterw(h_\star(\tau)\xrightarrow{\scriptscriptstyle t-\tau} h_\star(t)) $ in terms of shocks.
Set $ b_i(\tau):= h_\star(\tau,\shock_i(\tau)) $.

\begin{prop}
\label{p.raterw.sg}
The inequality in \eqref{e.sg.<=} becomes an equality when $ f=h_\star(\tau) $.
Further, there exist $ F_1,\ldots,F_{n(\tau)} \in \Lip $ such that $ \max\{ F_1,\ldots, F_{n(\tau)},\HLf_{t-\tau}(h_\star(\tau)) \} = h_\star(t) $, and
\begin{align*}
	\raterw\big( h_\star(\tau) \xrightarrow{t-\tau} (\vecx,\veca) \big)  %\scriptscriptstyle 
	= 
	\raterw\big( (\shock_i(\tau),b_i(\tau))_{i=1}^{n(\tau)} \xrightarrow{t-\tau} (\vecx,\veca) \big)
	=
	\sum_{i=1}^{n(\tau)} \raterw\big( \parab[\shock_i(\tau),b_i(\tau)](0) \xrightarrow{t-\tau} F_i \big).
\end{align*}
\end{prop}
\begin{proof}
The key is to consider $ g := \HLf_{t-\tau}(\HLb_{t-\tau}( h_\star(t))) $.
Throughout the proof, $ \tau $ and $ t $ are fixed, so $ g=g(x)\in\Lip $.
Telescope the quantity
$
	\raterw( \parab[1](0) \xrightarrow{\scriptscriptstyle t} h_\star(t) )
	=
	\raterwRel{ h_\star(t) }{ \parab[1](t) }
$
into
\begin{align}
	\label{e.raterw.sg.decomp}
	\raterw\big( \parab[1](0) \xrightarrow{t} h_\star(t) \big)
%	=
%	\raterwREl{ h_\star(t) }{ \parab[1](t) }
	=
	\raterwREl{ h_\star(t) }{ g }
	+
	\raterwREl{ g }{\parab[1](t) },
\end{align}
where $ \raterwRel{ h_\star(t) }{ g } $ is defined the same way as in \eqref{e.raterwrel}, but with $ g $ replacing $ \parab[\icx,\ica](t) $.

We begin by preparing some notation and tools.
We call $ (x,u) \mapsto (x+(t-\tau)u,u) $ the \tdef{backward shear}, as it corresponds to the backward evolution,
and call $ (x,u) \mapsto (x-(t-\tau)u,u) $ the \tdef{forward shear}.
For given $ f_1,f_2\in \Lip $ such that $ f_1(x)=f_2(x) $ for all large enough $ |x| $,
we have
\begin{align}
	\label{e.raterw.sg.formula}
	\raterwREl{f_1}{f_2}
	:=
	\int_{\R} \d x \, \big( \rateber( \partial_x f_1 ) - \rateber( \partial_x f_2 ) \big)
	=
	\int_{\Omega(u_1/\!\!/ u_2)} \d x \d u \ \rateber'(u).
\end{align}
Here $ \Omega(u_1 /\!\!/ u_2) \subset \R^2 = \{(x,u)\} $ is the region bounded by $ \CG(u_1) $ and $ \CG(u_2) $,
signed according to whether $ u_1 $ is above or below $ u_2 $.

Set $ v:= \partial_x g $.
Let us establish some properties of $ u_\star(t) $, $ u_\star(\tau) $, and $ v $. First,
\begin{align}
	\label{e.raterw.sg.cg1}
	\CG(u_\star(t))
	&\mapstochar\xrightarrow{\text{ backward shear } }
	(\ldots) %G_\times
	\mapstochar\xrightarrow{\text{ cuts } }
	\CG(u_\star(\tau))
	\mapstochar\xrightarrow{\text{ forward shear } }
	\CG(v),
\\
	\label{e.raterw.sg.cg2}
	\CG(\partial_x \parab[1](t))
	&\mapstochar\xrightarrow{\text{ backward shear } }
	\CG(\partial_x \parab[1](\tau))
	\mapstochar\xrightarrow{\text{ forward shear } }
	\CG(\partial_x \parab[1](t)).
\end{align}
The first two steps in \eqref{e.raterw.sg.cg1} make up the cut-and-sheer procedure described in the second paragraph of this subsection.
Next, note that $ \CG(u_\star(\tau)) $ consists of lines of slopes $ -1/(t-\tau) $, flat lines, vertical lines that correspond to downward jumps of $ u_\star(\tau) $, and segments that coincide with $ \partial_x \parab[1](\tau) $.
This being the case, upon the forward sheer, there is no need to perform the cuts.
Likewise, it is readily checked that there is no need for cuts in either step in \eqref{e.raterw.sg.cg2}.
Moving on, we turn to the region of difference $ \Omega(u_\star(t)/\!\!/ v) $ between $ u_\star(t) $ and $ v $.
Refer to \eqref{e.raterw.sg.cg1}.
Since the forward and backward shears cancel each other, the differences between $ \CG(u_\star(t)) $ and $ \CG(v) $ all come from the vertical lines introduced in the cuts in \eqref{e.raterw.sg.cg1}.
Within $ \CG(u_\star(\tau)) $, these vertical cuts/lines are located at $ \shock_1(\tau),\ldots,\shock_{n(\tau)}(\tau) $, and the vertical lines turn into lines of slope $ -1/(t-\tau) $ upon the forward shear.
Therefore, we have the \emph{disjoint} intervals $ U_i := (\shock_i(\tau)-(t-\tau)u_\star(\tau,\shock_i(\tau)^-),\shock_i(\tau)-(t-\tau)u_\star(\tau,\shock_i(\tau)^+)) $, $ i=1,\ldots,n(\tau) $, such that $ \Omega(u_\star(t)/\!\!/ v) \subset \cup_{i=1}^n (\bar{U_i}\times\R) $ and $ v(x)|_{x\in U_i} = -x/(t-\tau)+ $constant; see Figure~\ref{f.sheerandcut} for an illustration.

Next, we leverage the final conclusion of the last paragraph into properties of $ h_\star(t) $ and $ g $.
We claim that
\begin{align}
	\label{e.raterw.sg.misc}
	h_\star(t) \geq g,
	\qquad
	g|_{U_i} = \parab[\shock_i(\tau),b_i(\tau)](t-\tau)|_{U_i},
	\qquad
	h(t)=g \text{ off } (U_1 \cup \ldots \cup U_{n(\tau)}),
\end{align}
where $ b_i(\tau) := h_\star(\tau,\shock_i(\tau)) $.
The first claim $ h_\star(t) \geq g=\HLf_{t-\tau}(\HLb_{t-\tau}(h_\star(t))) $ is readily checked from \eqref{e.HLf} and \eqref{e.HLb}.
The property $ v(x)|_{x\in U_i} = -x/(t-\tau)+ $constant together with $ |v(x)|_{x\in U_i}<1 $ gives $ g(x)|_{U_i} = \parab[\shock_i(\tau),\beta_i](t-\tau,x) $ for some $ \beta_i\in\R $.
Undoing the forward shear in \eqref{e.raterw.sg.cg1} gives $ \HLb_{t-\tau}(g) = h_\star(t-\tau) $, thereby 
$ \beta_i= \HLb_{t-\tau}(\parab[\shock_i(\tau),\beta_i](\tau))|_{x=\shock_i(\tau)}=h_\star(\tau,\shock_i(\tau)) = b_i(\tau) $.
Move on. Recall that those cuts in \eqref{e.raterw.sg.cg1} that transform $ (\ldots) $ into $ \CG(u_\star(t-\tau)) $ must conserve areas,
and note that the forward shear turns $ (\ldots) $ and $ \CG(u_\star(t-\tau)) $ respectively into $ \CG(u_\star(t)) $ and $ \CG(v) $.
Hence $ \int_{\Omega(u_\star(t)/\!\!/ v)\cap U_i} \d x \d u = \int_{U_i} \d x (u_\star(t,x) - v(x)) = 0 $.
This property implies $ h_\star(t) = g $ off $ U_1 \cup \ldots \cup U_{n(\tau)} $.

Return to \eqref{e.raterw.sg.decomp}.
Let $ F_i := h_\star(t)|_{x\in U_i} + \parab[\shock_i(\tau),b_i(\tau)](t-\tau)|_{x\notin U_i} $.
By \eqref{e.raterw.sg.misc}, $ \max\{F_1,\ldots,F_{n(\tau)}, g\} = h_\star(t) $, and
\begin{align*}
	\raterwRel{ h_\star(t) }{ g } 
	= 
	\sum_{i=1}^{n(\tau)} \int_{U_i} \d x \, \big( \rateber(\partial_x h_\star(t)) - \rateber(\partial_x g) \big) 
	=
	\sum_{i=1}^{n(\tau)} \raterwREl{ F_i }{ \parab[\shock_i(\tau),b_i(\tau)](t-\tau) }.
\end{align*} 
Next, use \eqref{e.raterw.sg.formula} to express the last term in \eqref{e.raterw.sg.decomp} as $ \int_{\Omega(v/\!\!/ \partial_x\parab[1](t))} \d x \d u \, \rateber'(u) $.
Within this integral, apply the change of variables $ (x,u)\mapsto (x+(t-\tau)u,u) $.
Doing so does not change the integrand $ \partial_u \rateber $ (because it depends only on $ u $)
but turns the integration domain into $ \Omega(u_\star(\tau) /\!\!/ \partial_x\parab[1](\tau)) $
because of \eqref{e.raterw.sg.cg1}--\eqref{e.raterw.sg.cg2}.
The resulting integral is recognized as 
$ \raterwRel{h_\star(\tau)}{\parab[1](\tau)} = \raterw(  \parab[1](0) \xrightarrow{\scriptscriptstyle \tau} h_\star(\tau) ) $. 
Using the preceding results on the right side of \eqref{e.raterw.sg.decomp} gives 
\begin{align}
	\label{e.raterw.sg.final}
	\raterw\big( \parab[1](0) \xrightarrow{t} h_\star(t) \big)
	=
	\raterw\big(  \parab[1](0) \xrightarrow{\tau} h_\star(\tau) \big)
	+
	\sum\nolimits_{i=1}^{n(\tau)} \raterwREl{ F_i }{ \parab[\shock_i(\tau),b_i(\tau)](t-\tau) }.
\end{align}
By \eqref{e.raterw.sg.misc}, the last term in \eqref{e.raterw.sg.final} is a candidate of a minimizer of 
$ \raterw( h_\star(\tau) \xrightarrow{\scriptscriptstyle t-\tau} (\vecx,\veca) ) $.
By \eqref{e.sg.<=}, the left side of \eqref{e.raterw.sg.final}
is $ \leq (  \parab[1](0) \xrightarrow{\scriptscriptstyle \tau} h_\star(\tau) ) + \raterw( h_\star(\tau) \xrightarrow{\scriptscriptstyle t-\tau} (\vecx,\veca) ) $,
so the last term in \eqref{e.raterw.sg.final} is a minimizer. The desired result follows.
\end{proof}

\subsection{General initial conditions}
\label{s.matching.general}

Recall $ \ldq $ from Section~\ref{s.results.matching}.
Fix any $ g\in\ldq $
and $ (\vecx,\veca)=(x_i,a_i)_{i=1}^{m} $ with $ (x_i,a_i)\in\hyp(g)\setminus\hyp(\HLb_{t}(g))^\circ $.
Let $ (\icx_{k},F_{k})_{k=1}^{\icm} $ be a minimizer of $ \raterw(g\xrightarrow{\scriptscriptstyle t} (\vecx,\veca)) $ in \eqref{e.raterw.gxa}.
Consider 
\begin{align}
	\label{e.skeleton.singlelayer}
	h_\star(\tau,x) := \max\big\{ h_{\star 1}(\tau,x),\ldots,h_{\star\icm}(\tau,x), \HLf_{\tau}(g)(x) \big\},
	\
	h_{\star k}(\tau) := \HLb_{t-\tau}(F_{k}),
	\quad
	\tau\in[0,t].
\end{align}
Recall from Section~\ref{s.results.notation} that $ F_k = \Rwf( \wordkl[\wordF(F_k)][k][k] ) $.

Let us verify that $ h_\star(\tau)\in\ldq $.
First, each $ \partial_x h_{\star k}(\tau) $ evolves according to the description in Section~\ref{s.matching.icm=1}.
Let $ \shock_{k1}(\tau)<\ldots<\shock_{kn_{k}(\tau)}(\tau) $  be the shocks.
Off these shocks, the function $ h_{\star k} $ solves \eqref{e.intburgers} classically.
From these observations, we see that there exists a finite partition of $ [0,t]\times\R $ into regions with piecewise analytic boundaries, such that, within each region, the function $ h_{\star k}(\tau,x) $ is equal to one of the following
\begin{align}
	\label{e.skeleton.form}
	\parab[\icx,\ica](\tau+\gamma,x),
	\quad
	\ica'-\parab(t-\tau,x-\icx'),
	\quad
	\alpha x - \tfrac12(1-\alpha^2)\tau + \beta,
	\quad
	\icx,\ica,\icx',\ica',\beta\in\R,
	\
	\gamma\geq 0,
	\
	\alpha\in[-1,1],
\end{align}
and $ \ica'-\parab(t-\tau,x-\icx') $ only shows up in a bounded region.
Given that $ g \in \ldq $, similar considerations show that $ (\HLf_{\tau}(g))(x) $ enjoys the same property.
Therefore, $ h_\star(\tau)\in\ldq $ for all $ \tau\in[0,t] $.

Let us verify that $ h_\star $ solves \eqref{e.intburgers} almost surely.
The intersection of finitely many functions of the form \eqref{e.skeleton.form}
makes up finite many piecewise analytic curves in $ [0,t]\times\R $.
Therefore, off a finite number of piecewise analytic curves, 
the function $ h_\star(\tau,x) $ coincides with a function of the form \eqref{e.skeleton.form},
and hence solves \eqref{e.intburgers} classically.

Next we show that the shocks of each $ \partial_x h_{\star k} $ survive the maximum in \eqref{e.skeleton.singlelayer}.
This is the major step toward proving Proposition~\ref{p.matching}.
Let $ \shock_{k1}(\tau)<\ldots<\shock_{kn_{k}(\tau)}(\tau) $ denote the locations of the shocks of $ \partial_x h_{\star k}(\tau) $.
%Proposition~\ref{p.shock.survive} in the following asserts that these shocks persist in \eqref{e.skeleton.singlelayer}.
Let us prepare some notation and tools.
Set $ \ica_{k} := g(\icx_{k}) $ and $ \parab[k](t) := \parab[\icx_{k},\ica_{k}](t) $ as always, and consider
\begin{align}
	\label{e.raterw.1}
	\raterw\big( (\icx,\ica) \xrightarrow{t} (x_j,a_j)_{j\in\word}|_{\star} \big)
	:=
	\min\big\{ 
		\raterwRel{ f }{ \parab[\icx,\ica](t) }
		:
		\wordF(f) \supset \word, \ f(x_j) \leq a_j, j=1,\ldots,m
	\big\},
\end{align}
where the $ |_\star $ means to satisfy $ f(x_j) \leq a_j $ for all $ j $.
The analog of \eqref{e.sg.<=} reads
\begin{align}
	\tag{\ref*{e.sg.<=}'}
	\label{e.sg.<=.}
	\raterw\big( (\icx_k,\ica_k) \xrightarrow{t} (x_j,a_j)_{j\in\word}|_{\star} \big)
	\leq
	\raterw\big( \parab[1](0) \xrightarrow{\tau}f \big)
	+
	\raterw\big( f \xrightarrow{t-\tau} (x_j,a_j)_{j\in\word}|_{\star} \big).
\end{align}
Next, set $ b_{ki}(\tau) := h_\star(\tau,\shock_{ki}(\tau)) $.
Applying Proposition~\ref{p.raterw.sg} with $ (\vecx,\veca) \mapsto (x_i,a_i)_{i\in\wordF(F_k)} $ and $ (\icx_{1},\ica_{1}) \mapsto (\icx_{k},\ica_k) $ gives some $ F_{k1},\ldots,F_{kn(\tau)}\in\Lip $ such that $ \max\{ F_{k1},\ldots,F_{kn(\tau)},\HLf_{t-\tau}(h_{\star k}(\tau))\} = F_k $ and
\begin{align}
	\label{e.raterw.sg.icm>1}
	\raterw\big( \parab[k](0) \xrightarrow{t} F_k \big)
	=
	\raterw\big( \parab[k](0)  \xrightarrow{\tau} h_{\star k}(\tau) \big)
	+
	\sum_{i=1}^{n_k(\tau)}
	\raterw\big( (\shock_{ki}(\tau),b_{ki}(\tau)) \xrightarrow{t-\tau} F_{ki} \big).
\end{align}
Being a minimizer of $ \raterw(g\xrightarrow{\scriptscriptstyle t} (\vecx,\veca)) $, the function $ F_k $ must satisfy $ F_k(x_j) \leq a_j $, for all $ j $.
Hence, for all $ k $ and $ i $,
\begin{align}
	\label{p.shock.survive.Fki}
	F_{ki}(x_j) \leq a_j,\qquad j=1,\ldots,m.
\end{align}

\begin{prop}
\label{p.shock.survive}
For all $ \tau\in(0,t) $, $ k_0\in\{1,\ldots,\icm\} $, and $ i_0\in\{1,\ldots,n_{k_0}(\tau)\} $,
\begin{align*}
	h_{\star k_0}(\tau,\shock_{k_0i_0}(\tau))
	=:
	b_{k_0i_0}(\tau)
	> 
	\max\big\{ \HLf_{\tau}(g)(x), \  h_{\star k}(\tau,x): \  k\neq k_0 \big\} \big|_{x=\shock_{k_0i_0}(\tau)}.
\end{align*}
\end{prop}
\begin{proof}
To simplify notation, we will often denote points by their indices.
For example $ (\icx_k,\ica_k) = k $, $ (\shock_{ki}(\tau),b_{ki}(\tau)) = ki $, and $ (x_j,a_j)_{j\in\wordF(F_k)} = \wordF(F_k) $.
Also, we use $ |_f $ to emphasize that a point is on the graph of $ f $.
For example $ (\shock_{ki}(\tau),b_{ki}(\tau)) = (\shock_{ki}(\tau),h_{\star k}(\shock_{ki}(\tau))) = ki|_{h_{\star k}(\tau)} $.

Assume the contrary, and consider the case $ b_{k_0i_0}(\tau) \leq h_{\star k_1}(\tau,\shock_{k_0i_0}) $ for some $ k_1 \neq k_0 $.
The other case $ b_{k_0i_0}(\tau) \leq \HLf_{\tau}(g)(\shock_{k_0i_0}) $ can be treated similarly.
Set $ A_0:= \raterw( k_0i_0 \xrightarrow{ \scriptscriptstyle t-\tau} F_{k_0i_0} ). $

The idea is to derive a contradiction by `transporting $ A_0 $ from $ k_0 $ to $ k_1 $'.
More precisely, $ A_0 $ appears on the right side of \eqref{e.raterw.sg.icm>1} for $ k=k_0 $,
and we seek to transport the contribution of $ A_0 $ into \eqref{e.raterw.sg.icm>1} for $ k=k_1 $. 
The first step is to ``move $ A_0 $ out of $ k_0 $''. 
Consider \eqref{e.raterw.sg.icm>1} for $ k=k_0 $.
On the right side, single out $ A_0 $, and combine the rest into a single expression $ B $. 
Let $ \wordsub^{(0)} := \wordF(F_{k_0i_0}) $.
Using \eqref{e.sg.<=.} for $ f=h_{\star k}(\tau) $ and using \eqref{p.shock.survive.Fki} for $ k=k_0 $ and all $ i\neq i_0 $ give
\begin{align}
	\label{e.p.shock.survive.1}
	B \geq \raterw\big( k_0{} \xrightarrow{ \ t\ } \wordF(F_k)\setminus \wordF(F_{k_0i_0})|_{\star} \big).
\end{align}
Note that $ B $ contains the term $ \raterw( k_0 \xrightarrow{\scriptscriptstyle \tau} h_{\star k_0}(\tau) ) $,
and note that $ h_{\star k_0}(\tau) $ has a $ \wedge $ kink at $ y=\xi_{k_0i_0} $ because $ \xi_{k_0i_0} $ is an antishock.
We modify $ h_{\star k_0}(\tau) $ in a small neighborhood of $ \xi_{k_0i_0} $ as depicted in Figure~\ref{f.smooth}
to reduce $ \raterw( k_0 \xrightarrow{\scriptscriptstyle\tau} (\ldots) ) = \raterwRel{ (\ldots) }{ \parab[k_0](\tau) } $.
Let $ B' $ denote the post-modification $ B $. We have $ B>B' $.
Further, by making the neighborhood small enough, we ensure that the function $ \max\{ \HLf_{t-\tau}(f_0), \parab[\xi_{ki},b_{ki}](t), k\neq k_0 \} $ passes through all letters in $ \wordF(F_k)\setminus \wordF(F_{k_0i_0}) $, so that \eqref{e.p.shock.survive.1} holds also for $ B\mapsto B' $.
These procedures altogether give
\begin{align}
	\label{e.p.shock.survive.2}
	\raterw\big( \parab[k_0](0) \xrightarrow{t} F_{k_0} \big)
	=
	B+A_0
	>
	\raterw\big( k_0 \xrightarrow{\ t\ } \wordF(F_{k_0})\setminus\wordF(F_{k_0i_0})|_{\star} \big) 
	+
	A_0.
\end{align}

Having ``moved $ A_0 $ out of $ k_0 $'', we proceed to ``move $ A_0 $ into $ k_1 $''.
View $ \raterw( (\shock_{k_0i_0}(\tau),b) \xrightarrow{\scriptscriptstyle t-\tau} \word(F_{k_0i_0})|_{\star} ) =:\Phi(b) $
as a function of $ b $, and note that $ A_0 = \Phi(b_{k_0i_0}) $.
Set $ b_1:=h_{\star k_1}(\shock_{k_0i_0}(\tau)) $ and recall from the second paragraph of this proof that $ b_1\geq b_{k_0i_0}(\tau) $ by assumption.
We claim that $ \Phi $ is decreasing on $ [b_{k_0i_0},b_1] $.
Let $ f_b \in \Lip $ be the unique minimizer of $ \Phi(b) $, more explicitly
$
	f_b := \mathrm{argmin} \{ \raterwRel{ f }{ \parab[\shock_{k_0i_0}(\tau),b](\tau) } \}.
$
As $ b $ increases from $ b_{k_0i_0} $ to $ b_1 $, the word $ \wordF(f_b) $ may increase (gain more letters).
Break $ [b_{k_0i_0}, b_1] $ into subintervals $ [b_0,b_1],[b_1,b_2],\ldots $ on which $ \wordF(f_b) $ remains constant.
On each subinterval, it is readily checked that $ \Phi(b) $ decreases when $ b $ increases.
The claim follows, so in particular $ A_0 = \Phi(b_{k_0i_0}) \geq \Phi(b_1) = \Phi(h_{\star k_1}(\shock_{k_0i_0})) $.
Now, take \eqref{e.raterw.sg.icm>1} for $ k=k_1 $, add $ A_0 $ to both sides of the equation, and on the right side of the result use $ A_0 \geq\Phi(h_{\star k_1}(\shock_{k_0i_0})) $.
We have
\begin{align*}
	\raterw\big( \parab[k_1](0) \xrightarrow{t} F_{k_1} \big) + A_0
	\geq
	\raterw\big( \parab[k_1](0)  \xrightarrow{\tau} h_{\star k_1}(\tau) \big)
	+
	\sum_{i=1}^{n_{k_1}(\tau)}
	\raterw\big( k_1i|_{h_{\star k_1}(\tau)} \xrightarrow{t-\tau} F_{k_1i} \big)
	+
	\Phi\big(h_{\star k_1}(\shock_{k_0i_0})\big).
\end{align*}
On the right side, further use \eqref{e.sg.<=.} for $ f=h_{\star k_1}(\tau) $ and use \eqref{p.shock.survive.Fki} for $ k=k_1 $ and for all $ i $ and for $ (k,i)=(k_0,i_0) $.
We see that the right side is $ \geq \raterw( k_1 \xrightarrow{\scriptscriptstyle t } \wordF(F_{k_1})\cup\wordF(F_{k_0i_0}) |_\star ) $.

Combining the last result with \eqref{e.p.shock.survive.2} gives
\begin{align*}
	\sum_{k=k_0,k_1}\raterw\big( \parab[k](0) \xrightarrow{t} F_{k} \big)
	>
	\raterw\big( k_0 \xrightarrow{t} \wordF(F_{k_0})\setminus \wordF(F_{k_0i_0})|_{\star} \big)
	+
	\raterw\big( k_1 \xrightarrow{t} \wordF(F_{k_1})\cup\wordF(F_{k_0i_0})|_{\star} \big).
\end{align*}
On both sides, add $  \raterw( k \xrightarrow{\scriptscriptstyle t } \wordF(F_{k}) ) $, for all $ k\neq k_0,k_1 $.
The left side of the result is $ \raterw((\icvecx,\icveca)\xrightarrow{\scriptscriptstyle t }(\vecx,\veca)) $,
because $ (\icx_k,F_k)_k $ is a minimizer of it;
the right side of the result is $ \geq\raterw((\icvecx,\icveca)\xrightarrow{\scriptscriptstyle t }(\vecx,\veca)) $, a contradiction.
\end{proof}

\begin{proof}[Proof of Proposition~\ref{p.matching} for LdQ initial conditions]
The expression $ \partial_\tau \rateber(\partial_x h_\star) + \partial_x \rateberc(\partial_x h_\star) $
is supported around shocks of $ \partial_x h_\star(\tau) $, or equivalently kinks of $ h_\star(\tau) $.
Proposition~\ref{p.shock.survive} asserts that the kinks of $ h_{\star k}(\tau) $, all of which correspond to antishocks, are presented in $ h_\star(\tau) $.
There are other kinks of $ h_\star $ that come from intersections of the functions $ h_{\star 1}(\tau),\ldots,h_{\star \icm}(\tau), \HLf_{\tau}(g) $; see \eqref{e.skeleton.singlelayer}.
These functions are piecewise analytic (LdQ in fact),
and the intersection must give $ \vee $ kinks (convex kinks) since we are taking the maximum in \eqref{e.skeleton.singlelayer}.
These $ \vee $ kinks correspond to entropic shocks and do not contribute to $ \rateJV $.
Therefore, $ \rateJV(\partial_t h_{\star}) = \sum_{k=1}^{\icm} \rateJV(\partial_t h_{\star k}) $.
The result of Section~\ref{s.matching.icm=1} gives $ \rateJV(\partial_x h_{\star k})  = \raterw( \parab[k](0) \xrightarrow{ \scriptscriptstyle t } h_{\star k}(t) ) = \raterw( \parab[k](0) \xrightarrow{ \scriptscriptstyle t } F_k ) $.
Hence $ \rateJV(\partial_t h_{\star}) = \sum_{k=1}^{\icm} \raterw( \parab[k](0) \xrightarrow{ \scriptscriptstyle t } F_k ) = \raterw( g \xrightarrow{ \scriptscriptstyle t } h_\star(t) ) $.
\end{proof}

\subsection{Proof of \hyperref[t.main]{Main Theorem}}
\label{s.matching.pfmain}
Combine the fixed-time \ac{LDP} from \hyperref[t.rw]{Fixed-time Theorem}
and the general statement Lemma~\ref{l.discrete.to.continuous} to go from fixed-time to full \acp{LDP}.
We have that $ \hh_N $ satisfies the \ac{LDP} with rate function
\begin{align*}
	\Rate(h) 
	:=
	\liminf_{ \norm{\mathbf{t}} \to 0 } \ \sum_{i=1}^{n} \raterw\big( h(t_{i-1}) \xrightarrow{t_i-t_{i-1}} h(t_i) \big)
	\ \
	\text{ if } h(0)=\hic,
	\qquad
	\Rate(h):=+\infty \ \ \text{ otherwise}.
\end{align*}
By the definition of $ \raterw $, this rate function $ \Rate(h)=+\infty $ whenever $ h\notin\HLsp $; see \eqref{e.raterwrel}.
For an $ h\in\HLsp $, by Proposition~\ref{p.matching.} 
and the fact that $ \raterw(g \xrightarrow{ \scriptscriptstyle t} f) $ is \ac{lsc} in $ g, f $ (from Lemma~\ref{l.raterw.lsc}),
the rate $ \Rate(h) $ coincides with the $ \RAte(h) $ given in Definition~\ref{d.Rate}.

Having settled \hyperref[t.main]{Main Theorem}, we show that $ \RAte|_{\Elem}=\rateJV\circ\partial_x|_{\Elem} $.
Fix an $ h_\star \in \Elem $ with layers $ 0=t_0<\ldots<t_n=T $.
Applying \hyperref[t.rw]{Fixed-time Theorem} with $ (t,g,f) = (t_{i}-t_{i-1},h_\star(t_{i-1}),h_\star(t_i)) $, $ i=1,\ldots,n $
and applying \hyperref[t.main]{Main Theorem} for a small ball around $ h_\star $
give $ \sum_{i=1}^n \raterw(h_\star(t_{i-1})\xrightarrow{ \scriptscriptstyle t_{i}-t_{i-1}} h_\star(t_i)) \leq \RAte(h_\star) $.
By Proposition~\ref{p.matching}, the left side is $ \rateJV(\partial_x h_\star) $,
but the right side is the liminf of $ \rateJV\circ\partial_x|_{\Elem} $.
Hence $ \rateJV(\partial_x h_\star) = \RAte(h_\star) $.

\appendix

\section{The Hopf--Lax space and related properties}
\label{s.a.HLsp}
We begin with a preliminary version of Proposition~\ref{p.HLc.topo.weak.HLc}\ref{p.HLc.topo}.
\begin{lem}
\label{l.HLc}
Initiate the \ac{TASEP} from an arbitrary $ \hh_N(0) $.
Partition $ [0,T] $ into subintervals $ [0,t_1],[t_1,t_2],\ldots,[t_{n-1},T] $.
Given any $ \e>0 $ and $ r<\infty $, there exists $ c=c(\e,t_0,\ldots,t_{n},r)>0 $ such that
\begin{align*}
%	\label{e.l.HLc}
	\P\big[ 
		\HLf_{t_i-t_{i-1}}(\hh_N(t_{i-1})) |_{[-r,r]} -\e \leq \hh_N(t_i)|_{[-r,r]}, \
		i=1,\ldots,n
		%\leq \hh_N(t_{i-1})\big|_{[-r,r]}
	\big]
	\geq 
	1 - e^{-\frac{1}{c}N^{2}}.
\end{align*}
\end{lem}
\begin{proof}
Let us first assume $ n=1 $ and $ \hh_N(0)=\parab[\icx,\ica](0) $, the wedge initial condition.
By \cite[Theorem~1.2]{johansson00},
\begin{align}
	\label{e.l.HLc.1}
	\P_{\parab[1\cdots\icm](0)}\big[ \hh_N(t,x) \geq \parab(t,x)-\e \big] \geq 1 - \exp\big(-\tfrac{1}{c(\e,t,x)}N^2\big),
\end{align}
for fixed $ \e>0 $, $ t>0 $, and $ x\in(-t+\icx,\icx+t) $.
Next, set $ m := \lceil\frac{2t}{\e}\rceil $, and partition $ (-t+\icx,\icx+t) $ into $ \icm $ evenly spaced intervals as $ -t+\icx=x_0<x_1 <\ldots<x_{m-1}<x_{m} =\icx+t $.
This partition has a mesh at most $ \e $.
Apply the union bound to \eqref{e.l.HLc.1} for $ x=x_i $, $ i=1,\ldots,m-1 $. 
The resulting probability is bounded below by $ 1 - (m-1)\exp(-\tfrac{1}{c(\e,t)}N^2) $.
Since $ m\leq c(\e,t) $, the factor $ (m-1) $ can be absorbed into the exponential.
We have
%\begin{align}
%	\label{e.l.HLc.2}
$
	\P_{\parab[1\cdots\icm](0)}[ \hh_N(t,x_i) \geq\parab[\icx,\ica](t,x_i)-\e, \ i=1,\ldots,m ] \geq 1 - \exp(-\tfrac{1}{c(\e,t)}N^2).
$
%\end{align}
Further, since $ \hh_N(t)\in\Lip $ and since $ \parab[\icx,\ica](t)|_{|x-\icx|>t} =  \parab[\icx,\ica](0)|_{|x-\icx|>t} =  \hh_N(0)|_{|x-\icx|>t} $,
\begin{align}
	\label{e.l.HLc.3}
	\P_{\parab[1\cdots\icm](0)}\big[ \parab[\icx,\ica](t)-2\e \leq \hh_N(t) \big] 
	\geq
	1 - \exp\big(-\tfrac{1}{c(\e,t)}N^2\big).
\end{align}

The next step is to leverage \eqref{e.l.HLc.3} into a statement for a general initial condition $ \hh_N(0) $.
Fix arbitrary $ x_*\in\R $. Recall from~\eqref{e.HLf} that $ \HLf_t(\hh_N(0))(x_*) $ is defined as an infimum, so
there exists $ \icx_* $ such that $ \HL_t(\hh_N(0))(x_*) \geq \parab[\icx_*,\ica_*-\e](t,x_*) $, where $ \ica_* := \hh_N(0,\icx_*) $.
Consider another \ac{TASEP}, denoted $ \hh_N' $, with the wedge initial condition $ \hh'_N(0) = \parab[\icx_*,\ica_*-\e](0) $.
The initial conditions are ordered, namely $ \hh_N(0) \geq \hh_N'(0) $, because $ \hh_N(0,\icx_*) = \ica_* \geq \hh_N'(0,\icx_*) $, $ \hh_N(0)\in\Lip $, and $ \partial_x \hh_N'(0,x) = \sgn(\icx-x) $.
Under the basic coupling, $ \hh_N(t) \geq \hh'_N(t) $ for all $ t \geq 0 $.
Combining this coupling result with \eqref{e.l.HLc.3} for $ (\icx,\ica)\mapsto (\icx_*,\ica_*-\e) $ gives
%\begin{align*}
%	\P_{\hh_N(0)}\big[ \hh_N(t,x_*) \geq \HL_{t}(\hh_N(0))(x_*)-3\e \big] 
%	\geq
%	1 - \exp\big(-\tfrac{1}{c(\e,t)}N^2\big).
%\end{align*}
$	
	\P_{\hh_N(0)}[ \hh_N(t,x_*) \geq \HL_{t}(\hh_N(0))(x_*)-3\e ] 
	\geq
	1 - \exp(-\tfrac{1}{c(\e,t)}N^2).
$
Applying the union-bound argument that leads to \eqref{e.l.HLc.3}
yields the desired result for $ n=1 $, with $ \e \mapsto 4\e $ and $ T\mapsto t $.

The desired result for $ n>1 $ follows by applying the result for $ n=1 $ with $ (\hh_N(0),\hh_N(t)) \mapsto (\hh_N(t_{i-1}),\hh_N(t_i)) $ 
and taking the union bound over $ i=1,\ldots,n $.
\end{proof}

The following bound will come in handy, and the proof is straightforward from \eqref{e.HLf}.
\begin{align}
	\label{e.hl.cmp}
	\text{ For } r>t>0,
	f_1, f_2 \in \Lip \text{ with }  f_1|_{[-r,r]} \leq f_2|_{[-r,r]},
	\ \
	\HLf_{t}(f_1))|_{[-(r-t),r-t]} \leq (\HLf_{t}(f_2))|_{[-(r-t),r-t]}.
\end{align}

Recall that we endow the space $ \Lip $ with the metric $ \dist(f_1,f_2) := \sum_{n=1}^\infty 2^{-n} \sup_{x\in[-n,n]} | f_1(x) - f_2(x) | $,
and the space $ \Dsp([0,T],\Lip) $ with the metric $ \dist_{[0,T]}(h_1,h_2) := \sup_{t\in[0,T]} \dist(h_1(t),h_2(t)) $.

\begin{proof}[Proof of Proposition~\ref{p.HLc.topo.weak.HLc}\ref{p.HLc.topo}]
The first step is to establish an approximation result.
Fix $ \e_*>0 $.
For some $ \e,n,r $ to be specified later, consider an evenly spaced partition $ [0,t_1],\ldots[t_{n-1},T] $ of $ [0,T] $, namely $ t_i= \frac{iT}{n} $, and a realization of $ \hh_N $ that satisfies the approximate Hopf--Lax condition 
\begin{align}
	\label{e.p.HLc.1}
	\HL_{t_i-t_{i-1}}(\hh_N(t_{i-1}))|_{[-r,r]} -\e \leq \hh_N(t_i)|_{[-r,r]} \leq \hh_N(t_{i-1})|_{[-r,r]}, \qquad i=1,\ldots,n.
\end{align}
For such a realization $ \hh_N $, we claim that there exists $ g\in\HLsp $ such that $ \dist_{[0,T]}(\hh_N,g)< \e_* $. 
Let $ \Mid\{a,b,c\} := a+b+c - \max\{a, b, c\} - \min\{a, b, c\} $.
Define $ g $ on $ [t_{i-1},t_i] $, inductively in $ i $ as $ g(0) := \hh_N(0) $ and
\begin{align}
	\label{e.p.HLc.2}
	g(t) := \Mid \big\{ \HL_{t-t_{i-1}}(g(t_{i-1})), \, \hh_N(t_i)+i\e, \, g(t_{i-1})  \big\}, \qquad t\in(t_{i-1},t_i].
\end{align}
It is straightforward (though tedious) to check that $ g\in\HLsp $.
We claim that, for $ r_i:=r-t_{i-1} $,
\begin{align}
	\label{e.p.HLc.3}
	\hh_N(t_{i})|_{[-r_i,r_i]} \leq g(t_i)|_{[-r_i,r_i]} \leq \hh_N(t_i)|_{[-r_i,r_i]}+i\e,
	\qquad
	i=1,\ldots,n.
\end{align}
Setting $ t=t_1 $ in \eqref{e.p.HLc.2} and using \eqref{e.p.HLc.1} for $ i=1 $ prove \eqref{e.p.HLc.3} for $ i=1 $.
To progress we use induction.
Assume \eqref{e.p.HLc.3} holds for $ i\ge 1 $.
Apply \eqref{e.hl.cmp} with $ (f_1,f_2,t,r) \mapsto ( g(t_i),\hh_N(t_i)+i\e,t_{i+1}-t_i,r_i) $ and use the first inequality in \eqref{e.p.HLc.1}.
We have $ \HLf_{t_{i+1}-t_i}(g(t_i))|_{[-r_{i+1},r_{i+1}]} \leq \hh_N(t_{i+1})|_{[-r_{i+1},r_{i+1}]} + (i+1)\e $.
Combining this result with \eqref{e.p.HLc.2} for $ t=t_{i+1} $ gives $ g(t_{i+1})|_{[-r_{i+1},r_{i+1}]} = \min\{ \hh_N(t_{i+1})+(i+1)\e, g(t_i) \}|_{[-r_{i+1},r_{i+1}]} $,
which proves the second inequality in \eqref{e.p.HLc.3}.
Within the last expression, using $  g(t_i)|_{[-r_i,r_i]}\geq\hh_N(t_i)|_{[-r_i,r_i]} \geq \hh_N(t_{i+1})|_{[-r_i,r_i]} $ proves the first inequality in \eqref{e.p.HLc.3}.
By \eqref{e.p.HLc.3}, we have $ \dist(\hh_N(t_i),g(t_i)) \leq n\e + 2^{-\lfloor (r-T)_+ \rfloor} $, for $ i=1,\ldots,n $.
The next step is to leverage this bound, which holds for $ t_1,\ldots,t_n $, into a bound that holds for all $ t\in[0,T] $.
First, recall from Remark~\ref{r.Lip}\ref{r.Lip.} that $ g\in\HLsp $ is necessarily $ \frac12 $-Lipschitz in time.
Next, the function $ \hh_N $ decreases in $ t $ for each fixed $ x $, so
for all $ t\in[t_{i-1},t_i] $ we have $ |\hh_N(t,x)-\hh_N(t_{i-1},x)| \leq \hh_N(t_{i-1},x)-\hh_N(t_{i-1},x) $.
Combining this inequality with \eqref{e.p.HLc.1} gives 
$
	|\hh_N(t,x)-\hh_N(t_{i-1},x)| \leq \hh_N(t_{i-1},x) - \HL_{T/n}(\hh_N(t_{i-1}))(x) +\e.
$
The last expression, by~\eqref{e.HLf}, is be bounded by $ -\parab(\frac{T}{n},0) +\e = \tfrac{T}{2n}+\e $. 
Altogether these properties give
$	
	\dist_{[0,T]}(\hh_N,g) \leq (n\e + 2^{-\lfloor (r-T)_+ \rfloor}) + (\frac{T}{2n}) + (\frac{T}{2n} + \e) =: \e'. %= (n+1)\e + 2^{-\lfloor L-T \rfloor} + \tfrac{T}{n}.
$
Now choose large enough $ n,L $ and small enough $ \e $ so that $ \e' < \e_* $. 
These choices of $ n,L,\e $ depend only on $ \e_*,T $.

We now prove the desired statement.
Assume without loss of generality $ \hh_N(0)=0 $, and accordingly replace $ \HLsp $ with $ \HLsp':= \HLsp \cap \{ h : h(0,0)=0 \} $.
Recall from Remark~\ref{r.Lip}\ref{r.Lip.} that any $ h\in\HLsp $ is uniformly Lipschitz, so $ \HLsp' $ is compact. 
For the given open set $ \calO $, there exists $ \e_*>0 $ such that $ \{ h\in\Dsp([0,T],\Lip) : \dist_{[0,T]}(h,\HLsp') < \e_* \} \subset \calO. $
By the preceding construction, as soon as $ \hh_N $ satisfies \eqref{e.p.HLc.1},
we have $ \dist_{[0,T]}(\hh_N,\HLsp) < \e_* $, which implies $ \hh_N \in \calO $.
By Lemma~\ref{l.HLc}, the premise \eqref{e.p.HLc.1} holds up to probability $ \exp(-\frac{1}{c(\e,n,L,T)}N^2) $.
The choice of $ \e,n,L $ depend only on $ \e_*,T $, and $ \e_* $ depends only on $ \calO $.
Hence $ c=c(\calO,T) $. This completes the proof.
\end{proof}

\begin{proof}[Proof of Proposition~\ref{p.HLc.topo.weak.HLc}\ref{p.weak.HLc}]
Since elementary solutions are dense in $ \HLsp $, it suffices to prove that every weak solution  $ h $ lives in $ \HLsp $.
To this end, we begin with a reduction.
Recall from~\eqref{e.HLf} that $ \HL_{t}(h(t_0))(x) $ is defined as a supremum, so proving $ h\in\HLsp $ amounts to proving $ h(t_0,\icx) + \parab(t,x-\icx) \leq h(t_0+t,x) \leq h(t_0,x) $, for all $ t_0 < t_0+t\in[0,T]$ and $ \icx,x \in \R $.
Since \eqref{e.intburgers} is invariant under shifts in spacetime, without loss of generality we assume $ t_0=0 $ and $ \icx =0 $.
The goal can be reduced to showing
\begin{align}
	\label{e.weak->HLc.goal}
	h(0,0) + \parab(t,x) \leq h(t,x) \leq h(0,x),
	\qquad
	t\in[0,T], \ \icx \in \R.
\end{align}

The idea for showing the second inequality in~\eqref{e.weak->HLc.goal} is to use $ \partial_\tau h \leq 0 $ from \eqref{e.intburgers}. 
Roughly speaking, we seek to integrate this inequality along $ [0,t]\times\{x\} $ to get $ h(t,x)\leq h(0,x) $.
However, recall that for a weak solution \eqref{e.intburgers} holds only almost everywhere, and $ \partial_\tau h \leq 0 $ can fail on zero-measure sets in $ [0,T]\times\R $.
To circumvent this issue, we consider a thin corridor $ \Omega_\e = [0,t]\times[x-\e,x+\e] $, integrate the inequality ($ \partial_\tau h \leq 0 $ a.e.) over $ \Omega_\e $, and divide the result by $ 2\e $.
Doing so gives $ \frac{1}{2\e} \int_{|y-x|\le\e} \d y \, h(t,y) - \frac{1}{2\e} \int_{|y-x|\le\e} \d y \, h(0,y) \leq 0 $.
Since $ h $ is continuous (see Remark~\ref{r.Lip}\ref{r.Lip.express}), sending $ \e\to 0 $ yields the second inequality in \eqref{e.weak->HLc.goal}.

To show the first inequality in \eqref{e.weak->HLc.goal}, consider the characteristic velocity $ v:= \Mid\{ -1, x/t,1 \} $, the line in $ [0,t]\times\R $ that passes through $ (t,x) $ with velocity $ v $, and a thin corridor $ \Omega'_\e $ around this line, namely $ \Omega'_\e:=\{ (\tau,y) : \tau\in[0,t], |y-x-(t-\tau) v| \leq \e \} $.
On both sides of \eqref{e.intburgers} add $ v\partial_x h $, integrate the result over $ \Omega'_\e $, and divide the result by $ 2\e $. 
We get
$
	\frac{1}{2\e} \int_{\Omega'_\e} \d \tau \d y \, ( \partial_\tau h + v\partial_yh )  
	=
	\frac{1}{2\e} \int_{\Omega'_\e} \d \tau \d y \, (\frac{-1}{2}(1+v^2) +\frac{1}{2}(v+\partial_y h)^2 ) ).
$
For the integrand on left side we have $ (\partial_\tau h + v\partial_yh)(\tau,y) = \partial_\tau ( h(\tau,y+\tau v)) $ by the chain rule, which applies since $ h $ is Lipschitz in $ (\tau,y) $.
After the change of variables $ y+\tau v\mapsto t $, the left integral evaluates to $ \frac{1}{2\e}\int_{|y-x|<\e} \d y \, h(t,y) - \frac{1}{2\e} \int_{|y-x+tv|<\e} \d y \, h(0,y) $.
The integrand on the right side is at least $ -\frac12(1+v^2) $, so the right integral is bounded below by $ -\frac{t}{2}(1+v^2) $.
Combining the preceding resulting and sending $ \e\to 0 $ give $ h(t,x) \geq h(0,x-tv) - \frac{t}{2}(1+v^2) $.
When $ |x|\leq t $, the right side evaluates to $ h(0,0) + \parab(t,x) $.
When $ \pm x >t $, $ v=\pm 1 $. 
In this case, we use $ h(0,x-tv) \geq h(0,0)-|x-tv| = h(0,0)+t\mp x $ to obtain $ h(t,x) \geq h(0,0)\mp x = h(0,0)-|x| = h(0,0)+\parab(t,x) $. 
This concludes the first inequality in~\eqref{e.weak->HLc.goal}.
\end{proof}

\section{Properties of the operator $ \opu[k][k']{\ldots} $}
\label{s.a.opu.properties}

Let $ \conj $ act on $ \ell^2(\Z) $ by multiplication by $ (1+\mu^2) $.
Recall $ \opu[k][k']{(\ldots)_{k_1}(\ldots)_{k_2}\ldots} $ from Definition~\ref{d.opu.extended}.

\begin{lem}
\label{l.trace-class}
The operator $ \conj^{k} \inddown_{\ick}\opu{(\ldots)_{k_1}(\ldots)_{k_2}\ldots}\inddown_{\ick'} \conj^{-k'} $ is trace-class.
\end{lem}
\begin{proof}
By Lemma~\ref{l.opu.id} and \eqref{e.flip}, the operator can be expressed as a linear combination of products of operators of the form  $ \conj^{k_0} \inddown_{\ick_0}\opu{\word_{\ups\ldots\up}} \inddown_{\ick'_0} \conj^{-k'_0} $, $ \word\in\wordset $, $ k_0,k'_0\in\{1,\ldots,\icm\} $, so it suffices to show the latter are trace-class.
The statement is straightforward to show when $ \word=\emptyset $.
We consider $ |\word|=n>0 $, in which case 
\begin{align*}
	\conj^{k_0} &\inddown_{\ick_0}\opu{\word_{\ups\ldots\up}} \inddown_{\ick'_0} \conj^{-k'_0}
	=
	\conj^{k_0} \inddown_{\ick_0}\opuMQR{\word_{\ups\ldots\up}} \inddown_{\ick'_0} \conj^{-k'_0}
\\
	&=\big( \conj^{k_0}\inddown_{\ick_0}  \Slop_{-t,-\ics_{k_0}+s_{1}} \conj \big)
	\big(
		\conj^{-1}
		\rwop^{-s_1+s_{\word_1}}
		\indup_{\word_1} 
		\rwop^{-s_{\word_1}+s_{\word_2}}
		\cdots 
		\indup_{\word_{n}}
		\rwop^{-s_{\word_{n}} +s_1}
	\big) 
	\big( \Srop_{-t,-s_1+\ics_{k_0'}} \inddown_{\ick_0'} \conj^{-k_0'} \big),
\end{align*}
where we have dropped the irrelevant scaling by $ N $.
View the last expression as the product of three factors, denoted by $ A_1,A_2 $, and $ A_3 $, respectively.
Note that $ s_1 \leq s_{\word_n} $.
As is readily checked from \eqref{e.rwop}, we have $ \rwop^{-s_1+s_{\word_1}}(\mu,\mu')|_{\mu<\mu'}=0 $ and $ \rwop^{-s_{\word_{n}} +s_1}(\mu,\mu')|_{\mu>\mu'-|s_1-s_{\word_n}|}=0 $.
Hence $ A_2(\mu,\mu') $ is supported on $ \mu \geq d_{\word_1} $ and $ \mu' \geq d_{\word_n}-|s_1-s_{\word_n}| := \lambda $.
Therefore, 
$ 
	\conj^{k_0} \inddown_{\ick_0}\opu{\word_{\ups\ldots\up}} \inddown_{\ick'_0} \conj^{-k'_0} 
	=
	(A_1 \indup_{\word_1})(A_2)( \ind_{ \geq \lambda} A_3 ).
$
From \eqref{e.rwop}--\eqref{e.Srop}, it is straightforward to check that the last three factors are Hilbert--Schmidt.
The product is hence trace-class.
\end{proof}

\section{Steepest descent}
\label{s.a.steep}

Here we use steepest descent to prove Proposition~\ref{p.trace}, which concerns a preferred trace $ \monomial $.

From Definition~\ref{d.opu.extended}, we have the following contour integral expression.
\begin{align*}
	\monomial = \tr\big( \inddown_{\ick_n} \opuci{ \word^{(1)} } \inddown_{\ick_1} \opuci{ \word^{(2)} } \cdots \inddown_{\ick_{n-1}} \opuci{ \word^{(n)} } \inddown_{\ick_n} \big)
	=
	\prod_{i=1}^{n}
	\oint \frac{\d z_{i}}{2\pi\img} \oint \frac{\d z'_{i}}{2\pi\img}\, \til{U}_i(z_{i},z'_{i};\icd_{k_{i-1}},\icd_{k_{i}}) \til{V}_i(z'_{i},z_{i+1}).
\end{align*}
Here we adopt the cyclic convention $ k_0:=k_n $, $ z_{n+1}:=z_1 $, etc., and let $ \til{U}_i(z_i,z'_i;\icd_{k_{i-1}},\icd_{k_i}) $ be the right side of \eqref{e.opu.extended} without the $ z_0 $ and $ z'_\ell $ integrals, with $ (z_0,z'_\ell,\icmu,\icmu',k_0,k_n) \mapsto (z_i,z'_i,\icd_{k_{i-1}},\icd_{k_i},k_{i-1},k_{i}) $, and with $ (\vecUD^{(1)},\ldots,\vecUD^{(\ell)}) $ being specialized to the $ \vecUDc $ variables (defined in \eqref{e.UDc}).
\begin{description}[leftmargin=10pt]
\item[When the last isle in \tdef{$ \word^{(i)}$} is not \tdef{$ \emptyset $}]
We set $ \til{V}_i(z_i',z_{i+1}) := \frac{z_{i+1}}{2-z_i'-z_{i+1}} $. The $ z'_i $ contour does not enclose $ 2-z_{i+1} $ and the $ z_{i+1} $ contour does not enclose $ 2-z'_i $.
\item[When the last isle in \tdef{$ \word^{(i)}$} is \tdef{$ \emptyset $}]
We set $ \til{V}_i(z_i',z_{i+1}) := \frac{z_{i+1}}{z'_i-z_{i+1}} $ and the $ z'_{i} $ contour encloses $ z_{i+1} $.
\end{description}
The contours are counterclockwise loops that satisfy all conditions in Definitions~\ref{d.opu} and \ref{d.opu.extended} and the preceding ones.

It will be more convenient to work with circular contours.
We let all contours be circles $ \{z=re^{\img\theta}:\theta\in\R\} $ that pass through the corresponding critical points in \eqref{e.zs}--\eqref{e.zs.specialized.}.
For example, in Definition~\ref{d.opu}, $ r=\zls[k\word_1]$, $\zrws[\word_1\word_2]$, \ldots, $\zrws[\word_{n-1}\word_n]$,$ \zrs[\word_n k'] $ respectively for $ z_0,z_1,\ldots,z_n $ if $ \word\neq\emptyset $, and $ r=\zrws[\ick\ick'] $ for $ z_0 $ if $ \word=\emptyset $.
Such contours satisfy the required conditions except possibly for those $ z $'s that involve $ \fnSr $.
Take the conditions in Definition~\ref{d.opu} for example.
The conditions \ref{d.opu.1}--\ref{d.opu.3} are readily checked from \eqref{e.UDc}, \eqref{e.geomeaning.rw}--\eqref{e.geomeaning.sr}.
As for the condition \ref{d.opu.4}, when $ n>1 $ and $ \UDc(\wordkl)_{\word_n}=\down $, we have $ \zrws[\word_{n-1}\word_n] > 2-\zrs[\word_n k'] $ (from \eqref{e.UDc} and \eqref{e.geomeaning.rw}--\eqref{e.geomeaning.sr}).
This inequality shows that the $ z_{n-1} $ contour does contain $ 2-z_n $ when the angle of $ z_n=\zrs[\word_n k'] e^{\img\theta} $ is small.
However, because $ \zrws[\word_{n-1}\word_n]<2 $, the $ z_{n-1} $ contour would not contain $ 2-z_n $ when the angle of $ z_n $ becomes larger.
We address this issue by simply evaluating the residue at $ z_{n-1} = 2-z_n $ when $ 2-z_n $ is outside the $ z_{n-1} $ contour, and use the notation 
$
	\mathrm{Res}_{z_{n-1} = 2-z_{n}}| {}_{\{|2-z_{n}| \geq \zrws[\word_{n-1}\word_n] \}}
$ 
to encode this action.
Similarly, when $ n=1 $ and  $ \UDc(\wordkl)_{\word_n}=\down $, we introduce $ \mathrm{Res}_{z_{1} = 2-z_{0}} | {}_{\{|2-z_{0}| \geq \zrs[\word_{1}k'] \}} $ to accommodate the possible violation of the condition \ref{d.opu.5}.
From \eqref{e.opuci} and Definition~\ref{d.prefer}, it is readily checked that the above scenario holds generally: The required conditions (a contour encloses a variable) always hold when the corresponding angle is small enough, and the residue operator acts only when the angle becomes larger.
Recall (from Definitions~\ref{d.opu} and \ref{d.opu.extended}) that the integrand is a product of the functions $ \fnSl,\fnSr $, and $ \fnrw $ and some $ N $-independent rational functions, and recall from \eqref{e.steep.expfn} that the former can be expressed as exponentials of the $ \expfn $'s.
Following the preceding descriptions, we schematically express $ \monomial $ as
\begin{align}
	\label{e.steep.formula.}
	\monomial 
	&= \prod \big( 1 + \mathrm{Res}\big|_\text{condition} \big) \cdot \prod \oint_{\text{circle}}\frac{\d z}{2\pi\img} \cdot \prod e^{-N \expfn } \cdot \prod \text{(rational function)}
\\
	\label{e.steep.formula}
	&=
	\sum \prod \mathrm{Res} \cdot \prod \oint_{\text{circle}}\frac{\d z}{2\pi\img} \cdot \Big( \prod e^{-N \expfn } \cdot \prod \text{(rational function)} \Big)\Big|_\text{conditions}.	
\end{align}
Each residue operator acts contingent on a given condition, for example $ |2-z_{n}| \geq \zrws[\word_{n-1}\word_n]  $ or $ |2-z_{0}| \geq \zrs[\word_{1}k'] $.

To analyze \eqref{e.steep.formula}, we need to estimate the $ \expfn $'s along those circles.
Recall the $ \expfn $'s from \eqref{e.steep.expfn}.
For $ \expSl(z;t,s,d) $ and $ \expSr(z;t,s,d) $, Assumption~\ref{assu.strict} gives 
\begin{align}
	\label{e.HLc.contour}
	(s-d,s+d) = (x,a) \in \hyp( \parab(0) )^\circ \setminus \hyp( \parab(t) ).
\end{align}
Recall $ \zls=\zls(t,s,d) $, $ \zrs=\zrs(t,s,d) $, and $ \zrws=\zrws(s,d) $ from \eqref{e.zs}.
In the $ (x,a) $ coordinates we have
\begin{align}
\label{e.zs.xa}
	\zls(t,\tfrac{x+a}2,\tfrac{-x+a}2) = \tfrac{1}{t} ( t+x - \sqrt{ t^2+x^2+2ta } ),
	\quad
	\zrs(t,\tfrac{x+a}2,\tfrac{-x+a}2) = \tfrac{1}{t} ( t-x - \sqrt{ t^2+x^2+2ta } ).
\end{align}
It readily verified from \eqref{e.zs.xa} that $ \zls,\zrs \in (0,2) $ under \eqref{e.HLc.contour}, and from \eqref{e.zrws} that $ \zrws \in (0,2) $ for all $ 0<s,d $.
\begin{lem}
\label{l.steep}
For fixed $ s,d $ that satisfy \eqref{e.HLc.contour}, consider $ \expSl(\theta) := \expSl(\zls e^{\img\theta};t,s,d) $ and  $ \expSr(\theta) := \expSl(\zrs e^{\img\theta};t,s,d) $; for fixed $ 0<s,d $, consider $ \exprw(\theta) = \exprw(\zrws e^{\img\theta};s,d) $.

The function $ \expfn $ (= $ \expSl,\expSr $, or $ \exprw $) satisfies $ \partial_\theta \expfn(0)=0 $ and $ \partial^2_\theta \expfn(0)\in(0,\infty) $. The real part $ \Re(\expfn) $ is even in $ \theta $ and strictly increasing on $ \theta\in[0,\pi] $.
\end{lem}
\begin{proof}
Throughout this proof $ z=z(\theta) = \radius e^{\img\theta} $, where $ \radius=\zls,\zrs,\zrws $.
All but the last statement are straightforwardly verified, and we will show the last statement for $ \expSl,\expSr $, and $ \exprw $ separately.

We begin with $ \expSl $.
Let $ d_* := -t+|s|-2\sqrt{t|s|} $.
First, it is straightforward to verify that \eqref{e.HLc.contour} implies $ 0<-s < t $ and $ d\in(d_*,0) $. 
Let $ q_{\triangleleft}(\theta) := \frac{t}{2} - 2|s| \,|2-z(\theta)|^{-2} $. 
Direct calculations give $ \Re(\partial_\theta \expSl) = \zls \sin\theta \, q_{\triangleleft}(\theta) $.
Since $ r\in(0,2) $, we have $ q_{\triangleleft}(\theta) \geq q_{\triangleleft}(0) $, and it suffices to show $ q_{\triangleleft}(0) >0 $.
Write $ q_{\triangleleft}(0) = \frac{t}{2} - 2|s|(2-\zls)^{-2}  $ and recall that $ \zls $ depends on $ (t,s,d) $.
Fixing $ t,s $ but varying $ d $, we view $ q_{\triangleleft}(0) $ as a function of $ d $.
It is readily checked from \eqref{e.zs} that $ \zls $ strictly decreases in $ d $,
so the value of $ q_{\triangleleft}(0) $ for $ d \in (d_*,0) $ is strictly larger than the value at $ d=d_* $.
The latter evaluates to $ \frac{t}{2} - 2|s|(2\sqrt{|s|/t})^{-2} = 0 $. 
Hence $ q_{\triangleleft}(0) > 0 $.

The proof for $ \expSr $ parallels the preceding one. 
Set $ q_{\triangleright}(\theta) := \frac{t}{2} - 2|d| \,|2-z(\theta)|^{-2}  $.
Straightforward calculations give $ \Re(\partial_\theta \expSr) = \zrs \sin\theta \, q_{\triangleright}(\theta) $.
By varying $ s $ while keeping $ t,d $ fixed, similar arguments as in the preceding shows that $ q_{\triangleright}(0)>0 $.
The proof for $ \exprw $ parallels the preceding ones. 
Here $ 0<s,d $. Set $ q_{\trw}(\theta) := 2s\, |2-z(\theta)|^{-2} $. 
Straightforward calculations give $ \Re(\partial_\theta \expSr) = \zrws \sin\theta  \, q_{\trw}(\theta) $.
That $ \zrws\in(0,2) $ implies $ q_{\trw}(\theta) >0 $.
\end{proof}

\begin{proof}[Proof of Proposition~\ref{p.trace}]
Consider the expression \eqref{e.steep.formula}. 
Let $ R $ denote the total number of residue operators so that the sum in \eqref{e.steep.formula} has $ 2^R $ terms.
Parameterize the contours by $ z=z(\theta)=z_\star e^{\img\theta} $ and accordingly $ \expfn(\theta)=\expfn(z(\theta)) $.
We have and will often omit specifying the superscript ($ \triangleleft,\triangleright,\trw $) and the subscript ($ k\word_1,\word_1\word_2, \ldots $) in $ \expfn $, $ z_\star $, etc.
Throughout the proof, we write $ c=c(t,\icvecx,\icveca,\vecx,\veca) $ for a generic positive constant that depends only on $ t,\icvecx,\icveca,\vecx,\veca $.

We begin with the upper bound.
Recall from just before Proposition~\ref{p.trace} that $ \expfnword[\monomial] = \sum \expfn_\star $, which is equal to $ \sum \expfn(0) $ under the current notation.
Lemma~\ref{l.steep} gives $ \prod |e^{-N \expfn(\theta) }| \leq e^{-N \expfnword[\monomial] } $.
We claim that this property remains true under the action of the residue operators up to some multiplicative constants.
Take the scenario described in Definition~\ref{d.opu}\ref{d.opu.4} for example.
The residue operator turns $ \oint \frac{\d z_{n-1}}{2\pi \img} \fnrw_{\word_{n-1}\word_n}(z_{n-1}) \frac{-(2-z_n)}{2-z_n-z_{n-1}} \fnSr_{\word_n k'}(z_{n}) $ into
\begin{align}
	\label{e.pf.p.trace.1}
	\fnrw_{\word_{n-1}\word_n}(2-z_{n}) \cdot (2-z_n) \cdot \fnSr_{\word_n k'}(z_{n}) \big|_{ \{ |2-z_n| \geq \zrws[\word_{n-1}\word_n] \} } 	
	=
	\fnSr_{\word_{n-1} k'}(z_{n}) \big|_{ \{ |2-z_n| \geq \zrws[\word_{n-1}\word_n] \} }.
\end{align}
Write $ z_{n-1}=z_{n-1}(\theta_{n-1}) $ and $ z_n=z_n(\theta_n) $ and let $ (\theta_{n-1},\theta_n)=(\alpha,\beta)\in(0,\frac{\pi}{2})^2 $ be the angles where the  circles $ \Circ:=\{ {z_{n-1}(\theta_{n-1})}\} $ and $ \Circ':=\{ 2-{z_{n}(\theta_{n})}\} $ intersect, more explicitly $ \zrws[\word_{n-1}\word_n] e^{\img\alpha} = (2-\zrs[\word_nk'] e^{\img\beta}) $.
The condition $ |2-z_n| \geq \zrws[\word_{n-1}\word_n] $ in \eqref{e.pf.p.trace.1} is equivalent to $ |\theta_n| \geq \beta $, so the right side of \eqref{e.pf.p.trace.1} is bounded in absolute value from above by $ \exp(-N \Re(\expSr_{\word_{n-1} k'}(\beta))) $.
On the other hand, specializing \eqref{e.pf.p.trace.1} at $ \theta_n=\beta $ and taking absolute values give
\begin{align*}
	\exp( -N \Re(\exprw_{\word_{n-1}\word_n}(\alpha)) ) \cdot |2-\zrs[\word_nk'] e^{\img\beta}| \cdot \exp( -N \Re(\expSr_{\word_n}(\beta)) )
	=
	\exp(-N \Re(\expSr_{\word_{n-1} k'}(\beta)) ).
\end{align*}
The left side is $ \leq c \exp( -N (\exprw_{\word_{n-1}\word_n}(0)+\expSr_{\word_{n-1} k'}(0)) = c \exp( -N (\exprws[\word_{n-1}\word_n]+\expSrs[\word_{n-1} k']) ) $, whereby
\begin{align*}
	\big| \eqref{e.pf.p.trace.1} \big|
	\leq
	c \exp( -N (\exprws[\word_{n-1}\word_n]+\expSrs[\word_{n-1} k']) ).
\end{align*}
To summarize, a residue operator effectively combines two $ \fncontour $ factors, and the combined factor is exponentially smaller than the critical value of the pre-combined one.

We proceed to complete the proof of the upper bound.
By the preceding paragraph, the $ \expfn $ factors in \eqref{e.steep.formula} altogether is bounded above by $ c^R \exp(-N\expfnword[\monomial]) $.
The remaining rational functions are $ N $-independent, and are either bounded or integrable along the circular contours.
Let $ M\in\Z_{> 0} $ denote the total number of $ z $ variables in \eqref{e.steep.formula.}.
We have that $ |\monomial| = |\eqref{e.steep.formula}| \leq \sum c^{R+M} \exp(-N\expfnword[\monomial]) = 2^R c^{R+M} \exp(-N\expfnword[\monomial]) $.
Recall that $ \norm{\monomial} $ counts the total number of $ \opuci{\ldots} $ involved in $ \monomial $, and note that $ R $, the number of residue operators, is at most the number of $ \fnSr $ factors involved.
Under Convention~\ref{con.wordkl}, it is straightforward to check that $ R+M \leq c\,\norm{ \monomial } $.
This concludes the desired upper bound.

Turning to the lower bound, we fix a preferred trace $ \monomial $ and view $ R,M $ as being fixed too.
Among the $ 2^{R} $ terms in the sum in \eqref{e.steep.formula}, there is a distinguished one that does not involve residue operators.
By the preceding discussion, those that involve residue operators are exponentially smaller than $ e^{-N \expfnword[\monomial] } $.
For the distinguished one, use Lemma~\ref{l.steep} and the change of variables $ \theta\mapsto \theta/\sqrt{N} $. 
The result evaluates to $ \geq c^{-M} N^{-M/2} e^{-N \expfnword[\monomial] } $.
\end{proof}

\section{Properties of $ \raterw(\Cdot\xrightarrow{ \scriptscriptstyle t}\Cdot) $}
\label{s.a.raterw}

Except in Lemma~\ref{l.raterw.minimizer.word.}, in this section we only assume the basic conditions \eqref{e.HLc.xa} and \eqref{e.nondeg.sd} on $ (\icvecx,\icveca) $ and $ (\vecx,\veca) $. Take any minimizer $ \{F_k\}_{k} $ of \eqref{e.raterw.xa} under Convention~\ref{con.no.empty} and set $ \word^{(k)} := \wordF(F_k) $.
Under this convention, the list $ \{\word^{(k)} \}_k $ does \emph{not} contain any empty word.
Recall from Section~\ref{s.results.matching} the function $ \Rwf(\wordkl[\word][k][k]) $ and that $ F_k = \Rwf(\wordkl[\word^{(k)}][k][k]) $.

\begin{lem}
\label{l.raterw.minimizer}
Fix any $ \word=\word^{(k)} $ and $ F_k $ in the minimizer.
The word $ \word=\word^{(k)} $ decomposes into $ \wordkl[\word][k][k]=\wordkl[\wordsub^{(1)}][k][k] \cupi \ldots \cupi \wordkl[\wordsub^{(n)}][k][k] $,
where each $ \wordkl[\wordsub^{(i)}][k][k] \in \islesetkl[kk] $, and $ F_k = \Rwf(\wordkl[\word][k][k]) = \rwf[\wordkl[\word][k][k]] $.
\end{lem}

\begin{proof}
Throughout this proof $ \Rwf(y)=(\Rwf(\wordkl[\word][k][k]))(y) $ and $ \rwfsymb(y)=(\rwf[\wordkl[\word][k][k]])(y) $.

The two statements are equivalent.
To see why, recall that $ \rwfsymb $ is the minimizer of $ \raterwRel{f}{\parab[k](t)} $ under the constraints in \eqref{e.rwfn.space},
while $ \Rwf $ is the minimizer of the same quantity  under the constraints in \eqref{e.raterw.xa}.
If $ \wordkl[\word][k][k] $ decomposes into isles in $ \islesetkl[kk] $, the constraints posed by $ \parab[k'](t) $, $ k'=1,\ldots,\icm $, in \eqref{e.rwfn.space} is effective only for $ k'=k $, so $ \Rwf = \rwfsymb $.
Conversely, if $ \Rwf=\rwf $, those constraints posed by $ \parab[k'](t) $, $ k'=1,\ldots,\icm $, in \eqref{e.rwfn.space} are not effective for all $ k'\neq k $, so the decomposition of $ \wordkl[\word][k][k] $ into isles consists solely of isles in $ \islesetkl[kk] $.

It suffices to prove the first statement.
Assume the contrary: $ \wordkl[\word][k][k] = \wordkl[\wordsub][k][k'] \cupi \wordkl[\wordsub'][k'][k] $, for some $ k'\neq k $.
Consider $ k'>k $, and the case $ k'<k $ can be treated similarly.
The first step is to compare $ \Rwf(y) $ and $ \parab[k'](t,y) $ for various $ y $.
By Convention~\ref{con.wordkl}, $ \wordsub' \neq \emptyset $.
Setting $ \letterlast := \wordsub'_{|\wordsub'|} $, we have $ \Rwf(x_{\letterlast}) = a_{\letterlast} \geq \parab[k'](t,x_{\letterlast}) $.
Next, recall that $ \xl:=\xl(\wordkl[\wordsub'][k'][k])<x_{\letterlast} $ denotes the left point where $ \rwfsymb $ merges tangentially with $ \parab[k'](t) $, and note that $ \Rwf \leq \rwfsymb $, because \eqref{e.rwfn.space} imposes more `$ \geq $-type' constraints than \eqref{e.raterw.xa}.
Hence $ \Rwf(\xl) \leq \rwfsymb(\xl) = \parab[k'](t,\xl) $.
%
%\item []
Finally, because $ k<k' $, for all $ y $ large enough $ \Rwf(y) = \parab[k](t,y) < \parab[k'](t,y) $.
%\end{itemize}
Given these properties, we set $ \alpha := \sup\{ y\in [\xl,x_{\letterlast}] : \Rwf(y) \leq \parab[k'](t,y) \} $
and $ \beta := \inf\{ y\in[x_\letterlast,\infty) : \Rwf(y) \leq \parab[k'](t,y) \} $
so that $ \Rwf|_{[\alpha,\beta]} \geq \parab[k'](t)|_{[\alpha,\beta]} $ and $ \Rwf|_{y=\alpha,\beta} = \parab[k'](t)|_{y=\alpha,\beta} $. 
Further,
\begin{align}
	\label{e.l.raterw.minimizer.1}
	(\partial_y \Rwf)(\beta^+) < (\partial_y \parab[k'](t))(\beta).
\end{align}
To see why, note that $ \Rwf $ is linear in $ (x_{\letterlast},\xr(\wordkl[\word'][k'][k])) $ and then merges tangentially with $ \parab[k](t) $ for $ y \geq \xr(\wordkl[\word'][k'][k]) $.
Hence, for all $ y>x_{\letterlast} $, $ \partial_y \Rwf(y) \leq \partial_y\parab[k](t,y) $, which is strictly smaller than $ \partial_y\parab[k'](t,y) $ since $ k<k' $.
Taking into account the possibility that $ \beta=x_{\letterlast} $ leads to \eqref{e.l.raterw.minimizer.1}.

We now derive a contradiction by a `rewiring' argument similar to the proof of Proposition~\ref{p.nondiag}.
Recall that $ F_k = \Rwf $ and that $ \{F_{k''}\}_{k''} $ minimizes \eqref{e.raterw.xa}.
Rewire the functions $ F_k=\Rwf $ and $ \parab[k](t) $ at $ y=\alpha,\beta $ to get
$
	\til{F}_k := \Rwf(y) \ind_{y\notin(\alpha,\beta)} + \parab[k'](t,y)\ind_{y\in[\alpha,\beta]},
$
and
$
	\til{F}_{k'} := \Rwf(y) \ind_{y\in(\alpha,\beta)} + \parab[k'](t,y)\ind_{y\notin(\alpha,\beta)}.
$
These new functions conserve $ \wordF(F_k) $ and $ \raterwRel{F_k}{\parab[k](t)} $,
namely $ \wordF(\til{F}_k)\cup\wordF(\til{F}_{k'}) = \wordF(F_k) $
and $ \raterwRel{ F_k }{ \parab[k](t) } = \raterwRel{ \til{F}_k }{ \parab[k](t) } +\raterwRel{ \til{F}_{k'} }{ \parab[k'](t) } $.
By Lemma~\ref{l.repeat}, the last quantity can be `combined' with $ \raterwRel{ F_{k'} }{ \parab[k'](t) } $, and the result does not increase.
As for $ \raterwRel{ \til{F}_k }{ \parab[k](t) } $, given \eqref{e.l.raterw.minimizer.1},
we modify $ \til{F}_k $ in a small neighborhood of $ y=\beta $ as depicted in Figure~\ref{f.smooth}, and let $ \widehat{F}_k $ denote the modified function.
The modified function satisfies $ \raterwRel{\widehat{F}_k}{ \parab[k](t) } < \raterwRel{\til{F}_k}{ \parab[k](t) } $, $ \widehat{F}_k \leq \til{F}_k $, and $ \wordF(\widehat{F}_k) \cup \wordF(\til{F}_{k'}) = \wordF(\til{F}_k) \cup \wordF(\til{F}_{k'}) $. 
Altogether, these procedures strictly decreases the value of $ \sum_{k''} \raterwRel{({\ldots})_{k''}}{\parab[k''](t)} $ while satisfying the constraints in \eqref{e.raterw.xa}, a contradiction.
\end{proof}

\begin{lem}
\label{l.raterw.minimizer.word}
Fix any $ \word=\word^{(k)} = \wordF(F_k) $.
\begin{enumerate*}[label=(\alph*)]
\item \label{l.raterw.minimizer.word.1}
The word contains all letters in $ [\word_1,\word_{|\word|}] $.
\quad
\item \label{l.raterw.minimizer.word.2}
$ \word\in\treekl[kk] $.
\end{enumerate*}
\end{lem}
\begin{proof}
\ref{l.raterw.minimizer.word.1}\
Fix any $ j\in(\word_1,\word_{|\word|})\cap\Z $.
That $ \{F_{k''}\}_{k''} $ minimizes \eqref{e.raterw.xa} requires $ \cup_{k''}\word^{(k'')} = 12\cdots m $,
so $ j \in \word^{(k')} $ for some $ k' $.
If $ j\notin\word $, there exist $ \alpha<\beta $
with $ x_j \in (\alpha,\beta) \subset [x_{\word_1},x_{\word_{|\word|}}] $
such that $ F_{k'}|_{(\alpha,\beta)} >F_{k}|_{(\alpha,\beta)} $
and $ F_{k'}|_{y=\alpha,\beta} = F_{k}|_{y=\alpha,\beta} $.
A similar rewiring procedure as in the proof of Lemma~\ref{l.raterw.minimizer} produces a contradiction.

\ref{l.raterw.minimizer.word.2}\
Follow the up-down iteration corresponding to $ \treekl[kk] $, starting with $ \wordfullkl[kk] $.
Recall that, at each step of the iteration, a child chooses to delete some and keep some letters from the activated letters.
In each step of the iteration, follow the child that chooses to keep any letter that is in $ \word $ and delete any letter that is not in $ \word $.
Since, by Part~\ref{l.raterw.minimizer.word.1}, the word $ \word $ is made up of consecutive letters,
this procedure eventually leads to $ \wordsub = \{\ldots j_1\}\cup\word\cup\{j'_1\ldots\} \in \treekl[kk] $,
where the $ j $'s are excess letters, with $ \ldots,j_1 < \word_1 $ and $ \word_{|\word|} < j'_1,\ldots $.
These excess letters have never activated throughout the iteration, so $ \vecUDc(\wordsub)|_{\{\ldots j_1,j'_1,\ldots\}} = (\down\ldots\down) $.
Invoking geometric argument as depicted in Figure~\ref{f.wing} shows that the last property contradicts with the $ \IHC[kk] $ that $ \wordsub $ satisfies.
Hence those excess letters cannot exist and $ \word=\wordsub \in \treekl[kk] $.
\end{proof}

Under Assumption~\ref{assu.strict},
the following lemma gives a stronger version of Lemma~\ref{l.raterw.minimizer.word}\ref{l.raterw.minimizer.word.1}.
The proof is similar to that of Lemma~\ref{l.raterw.minimizer.word}\ref{l.raterw.minimizer.word.1}, so we omit it.

\begin{lem}
\label{l.raterw.minimizer.word.}
List the words $ \word^{(k)} $ in the ascending order of $ k $.
Under Assumption~\ref{assu.strict}, the words share a letter only if they are consecutive in the list,
and the shared letter must be the last letter in the precedent word and the first letter in the succeeding word.
\end{lem}

\begin{lem}
\label{l.raterw.lsc}
The rates $ \raterw((\icvecx,\icveca) \xrightarrow{ \scriptscriptstyle t} (\vecx,\veca)) $,
$ \raterw(g \xrightarrow{ \scriptscriptstyle t} (\vecx,\veca)) $,
and $ \raterw(g \xrightarrow{ \scriptscriptstyle t} f) $
are \ac{lsc} in $ (\icvecx,\icveca,\vecx,\veca)\in\R^{2\icm+2m} $,
in $ (g,\vecx,\veca)\in\Lip\times\R^{2m} $,
and in $ (g,f) \in (\Lip)^2 $, respectively.
\end{lem}
\begin{proof}
Take any sequences $ (\icvecx^{(n)},\icveca^{(n)}) \to (\icvecx,\icveca) $
and $ (\vecx^{(n)},\veca^{(n)}) \to (\vecx,\veca) $.
Accordingly, $ \wordset^{(n)}(kk') $ and $ \Rwf^{(n)}(\wordkl[\word][k][k]) $ depend on $ n $.
We call $ \wordkl[\word][k][k] $ \tdef{$n$-admissible} if $ \word \in \wordsett^{(n)}(kk') $ and $ (\Rwf^{(n)}(\wordkl[\word][k][k]))(x^{(n)}_j) \leq a^{(n)}_j $, for all $ j $,
with the convention $ \vecx^{(\infty)} := \vecx $, etc.
Indeed, any $ \wordkl[\word][k][k] $ that is not $ \infty $-admissible is not $ n $-admissible for all large enough $ n $. 
Hence there exists $ n_0\in\Z_{> 0} $ such that the following holds: For all $ n\geq n_0 $ and $ \wordkl[\word][k][k] $, being $ n $-admissible implies being $ \infty $-admissible.
For each $ n\geq n_0 $, take a minimizer $ \{ F^{(n)}_k \}_k $ of $ \raterw((\icvecx^{(n)},\icveca^{(n)}) \xrightarrow{ \scriptscriptstyle t} (\vecx^{(n)},\veca^{(n)}) ) $ and set $ f^{(n)}_k := \Rwf(\wordkl[\wordF(F^{(n)}_k)][k][k]) $.
We have
$
	\sum_k \raterwRel{ f^{(n)}_k }{ \parab[k](t) } \geq \raterw((\icvecx,\icveca) \xrightarrow{ \scriptscriptstyle t} (\vecx,\veca) ).
$
It is readily checked that $ \lim_{n\to\infty} | \raterwRel{ f^{(n)}_k }{ \parab[k](t) } - \raterwRel{ F^{(n)}_k }{ \parab[k]^{(n)}(t) } | = 0 $.
This concludes that $ \raterw((\icvecx,\icveca) \xrightarrow{ \scriptscriptstyle t} (\vecx,\veca)) $ is \ac{lsc}.
The other two statements follow from this one by approximations.
\end{proof}

\section{Stability in the initial condition}
\label{s.a.ic}

Initiate the \ac{TASEP} from two deterministic initial conditions $ \hh_N(0) $ and $ \hh'_N(0) $.
Couple the dynamics together by the basic coupling, and let $ \hh_N(t) $ and $ \hh'_N(t) $ denote the resulting processes.

\begin{lem}
\label{l.ic.}
For any given $ t,\e>0 $, there exist a $ \delta=\delta(t,\e)>0 $ and a universal $ c<\infty $ such that $ \dist(\hh_N(0),\hh'_N(0))< \delta $
implies $ \P[ \dist(\hh_N(t),\hh'_N(t)) < \e ] \geq 1- c\exp( -\e^{-1} N ) $.
\end{lem}
\begin{proof}
All \ac{TASEP} height functions in this proof are coupled by the basic coupling.
Fix $ \e>0 $. We say an event occurs \tdef{With High Probability (WHP)} if the probability is $ \geq 1 - c\exp(-N/\e) $, for some universal $ c $.

Recall that the metric $ \dist $ controls the sup norm over compact intervals.
Fix $ r<\infty $.
The main step of the proof is to show that $ \hh_N(0,x) $, for large $ |x| $, does not affect $ \hh_N(t,x)|_{|x|\leq r} $ WHP.
To state this precisely, 
for an $ \alpha>0 $ to be specified later, let $ r':=r+\alpha $.
Consider the \ac{TASEP} $ \hh^{\icx}_N(t) $ with the wedge initial condition $  h^{\icx}_N(0) = \parab[\icx,\hh_N(0,\icx)](t) $,
and similarly for $ \hh^{\tL,\icx}_N(t) $ with $ h^{\tL,\icx}_N(0) = \parab[\icx,\hh_N(0,-r')-(\icx+r')](0) $
and for $ \hh^{\tR,\icx}_N(t) $ with $ h^{\tR,\icx}_N(0) = \parab[\icx,\hh_N(0,r')+(\icx-r')](0) $.
Define the lower envelope of $ \hh_N $ to be
$ \hh_N^-(t,x) := \max\{ \hh^{\icx}_N(t,x) : |\icx| \leq r' \} $ and
the upper envelope of $ \hh_N $ to be
\begin{align}
	\label{e.l.ic.0}
	\hh_N^+(t,x) := \max\Big\{ \max_{\icx<-r'}\big\{ \hh^{\tL,\icx}_N(t,x) \big\}, \ \hh_N^-(t,x), \ \max_{\icx>r'}\big\{ \hh^{\tR,\icx}_N(t,x) \big\} \Big\}.
\end{align}
Indeed, $ \hh^-_N $ and $ \hh^+_N $ sandwich $ \hh_N $, and they all evolve according to the \ac{TASEP} dynamics.
We claim that $ \hh_N^-(t)|_{[-r,r]} = \hh_N^+(t)|_{[-r,r]} $ WHP.
Once the claim is proven, the desired statement follows by considering $ \hh_N^++\delta $ and $ \hh_N^--\delta $
and using the order-preserving property of the basic coupling.

To prove the claim, consider the dynamics of $ \hh^{\icx}_N $ in terms of particles.
The first particle proceeds as a unit-rate (in the pre-scaled units) Poisson process.
For large $ \alpha $, WHP the particle travels less than $ N\alpha $ within $ [0,Nt] $, which gives $ \hh^{\icx}_N(t)|_{x\geq\icx+\alpha} = \hh^{\icx}_N(0)|_{x\geq\icx+\alpha} $.
The same argument together with the particle-hole duality gives $ \hh^{\icx}_N(t)|_{x\leq\icx-\alpha} = \hh^{\icx}_N(0)|_{x\leq\icx-\alpha} $.
Use these properties for $ \icx=-r',r' $ gives, WHP, $ \hh_N^-(t)|_{[-r,r]} \geq \max\{ \parab[-r',\hh_N(0,-r')](0), \parab[r',\hh_N(0,r')](0) \}|_{[-r,r]} $.
The last function is bounded below by the inner maxima in \eqref{e.l.ic.0} when restricted to $ x\in[-r,r] $, because $ \hh^{\tL,\icx}_N(t,x) \leq \hh^{\tL,\icx}_N(0,x) = \hh_N(0,-r')-(\icx+r') - (x-\icx) =  \parab[-r',\hh_N(0,-r')](0,x) $ and similarly $ \hh^{\tR,\icx}_N(t) \leq \parab[r',\hh_N(0,r')](0,x) $.
Dropping the inner maxima in \eqref{e.l.ic.0} proves the claim.
\end{proof}

For convenience of referencing we convert Lemma~\ref{l.ic.} into the following statement.
\begin{lem}
\label{l.ic}
Fix any $ t,\e>0 $, $ f,g\in\Lip $.
There exist a $ \delta=\delta(t,\e)>0 $ and a universal $ c<\infty $ such that\\
$ \dist(\hh_N(0),g)< \delta $ and $ \dist(\hh'_N(0),g)< \delta $
imply $ \P_{\hh_N(0)}[ \dist(\hh_N(t),f) < \e ] \leq \P_{\hh'_N(0)}[ \dist(\hh_N(t),f) < 2\e ] + c \exp(-\e^{-1}N) $.
\end{lem}

\bibliographystyle{alphaabbr}
\bibliography{tasepLDP-det}

\end{document}